\documentclass{gtart_a}
\pdfoutput=1
\usepackage{graphicx}

\title[Pseudoholomorphic punctured spheres in $\mathbb{R} \times (S^{1}\times  S^{2})$]{Pseudoholomorphic punctured spheres in $\mathbb{R} \times (S^{1}\times  S^{2})$:\\Moduli space parametrizations}

\author{Clifford Henry Taubes}
\givenname{Clifford Henry}
\surname{Taubes}
\address{Department of Mathematics\\
Harvard University\\
Cambridge, MA 02138\\USA}
\email{chtaubes@math.harvard.edu}
\urladdr{}

\volumenumber{10}
\issuenumber{}
\publicationyear{2006}
\papernumber{44}
\lognumber{0793}
\startpage{1855}
\endpage{2054}

\doi{}
\MR{}
\Zbl{}


\received{5 April 2005}
\revised{}
\accepted{25 October 2006}
\published{7 November 2006}
\publishedonline{7 November 2006}
\proposed{Rob Kirby}
\seconded{Jim Bryan, Ron Stern}
\corresponding{}
\editor{CPR}
\version{}

\arxivreference{}  

\AtBeginDocument{\let\bar\wbar\let\tilde\wtilde\let\hat\what}
\numberwithin{equation}{section}

\newcommand{\step}[1]{\medskip{{\bf #1}}\qua\ignorespaces}
\newcommand{\substep}[1]{\medskip{{\bfseries\itshape #1}}\qua\ignorespaces}

\def\itaubes#1{\addtocounter{equation}{1}
\begin{itemize}
\leftskip25pt
\item
\noindent\llap{\hbox to 53.5pt{\rm\hypertarget{eq:#1}%
{(\theequation)}\hss}}\ignorespaces}

\def\enditaubes{\end{itemize}}

\def\qtaubes#1{\addtocounter{equation}{1}\par\leftskip38pt
\noindent\llap{\hbox to 38pt{\rm\hypertarget{eq:#1}%
{(\theequation)}\hss}}\ignorespaces}

\def\endqtaubes{\par\medskip\leftskip0pt}

\def\eqreft#1#2{\hyperlink{eq:#1.#2}{(#1--#2)}}

\let\subsectionold\subsection
\def\subsection#1{\subsectionold{#1}\setobjecttype{Subs}}

\DeclareMathOperator{\dist}{dist}
\DeclareMathOperator{\Maps}{Maps}
\DeclareMathOperator{\Aut}{Aut}
\DeclareMathOperator{\Autt}{\mathbb{A}ut}
\DeclareMathOperator{\Hom}{Hom}
\DeclareMathOperator{\old}{old}
\DeclareMathOperator{\Crit}{Crit}
\DeclareMathOperator{\cokernel}{cokernel}
\DeclareMathOperator{\kernel}{kernel}
\DeclareMathOperator{\constant}{constant}
\DeclareMathOperator{\sign}{sign}
\DeclareMathOperator{\domain}{domain}
\DeclareMathOperator{\inti}{int}
\DeclareMathOperator{\Cyc}{Cyc}
\DeclareMathOperator{\Auttt}{\hat{A}ut}
\DeclareMathOperator{\Sym}{Sym}
\DeclareMathOperator{\new}{new}
\DeclareMathOperator{\Vertt}{Vert}
\DeclareMathOperator{\Arc}{Arc}
\DeclareMathOperator{\Auth}{\hat{A}ut}
\DeclareMathOperator{\maps}{maps}
\DeclareMathOperator{\Map}{Map}

\makeatletter
\def\cnewtheorem#1[#2]#3{\newtheorem{#1}{#3}[section]
\expandafter\let\csname c@#1\endcsname\c@theorem}

\newtheorem{theorem}{Theorem}[section]
\cnewtheorem{lemma}[theorem]{Lemma}
\cnewtheorem{proposition}[theorem]{Proposition}
\newtheorem{point}{Point}
\newtheorem{scenario}{Scenario}

\theoremstyle{definition}
\cnewtheorem{definition}[theorem]{Definition}

\makeatother  

\newcommand{\bit}{\begin{itemize}}
\newcommand{\eit}{\end{itemize}}

\begin{document}

\begin{asciiabstract}
This is the second of two articles that describe the moduli spaces of
pseudoholomorphic, multiply punctured spheres in R x (S^1 x S^2) as
defined by a certain natural pair of almost complex structure and
symplectic form. The first article in this series described the local
structure of the moduli spaces and gave existence theorems. This
article describes a stratification of the moduli spaces and gives
explicit parametrizations for the various strata.
\end{asciiabstract}

\begin{htmlabstract}
This is the second of two articles that describe the moduli spaces of
pseudoholomorphic, multiply punctured spheres in <b>R</b> &times;
(S<sup>1</sup> &times; S<sup>2</sup>) as defined by a certain natural
pair of almost complex structure and symplectic form. The first
article in this series described the local structure of the moduli
spaces and gave existence theorems. This article describes a
stratification of the moduli spaces and gives explicit
parametrizations for the various strata.
\end{htmlabstract}

\begin{abstract}
This is the second of two articles that describe the moduli spaces of
pseudoholomorphic, multiply punctured spheres in $\mathbb{R} \times
(S^1 \times S^2)$ as defined by a certain natural pair of almost
complex structure and symplectic form. The first article in this
series described the local structure of the moduli spaces and gave
existence theorems. This article describes a stratification of the
moduli spaces and gives explicit parametrizations for the various
strata.
\end{abstract}
\maketitle

\section{Introduction}\label{sec:1}

This is the second of two articles that describe the moduli spaces
of multiply punctured, pseudoholomorphic spheres in $\mathbb{R}
\times  (S^1 \times  S^2)$ as defined using
the almost complex structure, $J$, for which
\begin{equation}\label{eq1.1}
\begin{split}
J\cdot \partial _{s}&=\frac{1 }{\surd 6 (1 +3\cos  ^4\theta )^{1 /
2}} \bigl((1 - 3\cos ^{2}\theta )\partial _{t} + \surd 6 \cos \theta
\partial _{\varphi }\bigr)
\\
J\cdot \partial _{\theta }&=\frac{1 }{\surd 6(1 + 3\cos  ^4\theta
)^{1 / 2}} \biggl(-\surd 6 \cos \theta  \sin \theta
\partial _{t} + (1 - 3 \cos ^{2}\theta )\frac{1 }{ {\sin\theta }}\partial _{\varphi }\biggr) .
\end{split}
\end{equation}
Here, $s$ is the Euclidean coordinate on the $\mathbb{R}$ factor of
$\mathbb{R}\times  (S^1 \times  S^2)$, while $t$ is
the $\mathbb{R}/(2\pi \mathbb{Z})$ coordinate on the $S^1$ factor
and $(\theta , \varphi ) \in  [0, \pi ] \times \mathbb{R}/(2\pi \mathbb{Z})$
are the usual spherical coordinates
on the $S^2$ factor. A change of coordinates shows that this
almost complex structure is well defined near the $\theta  = 0$
and $\theta = \pi $ cylinders, and that the latter are
pseudo-holomorphic with a suitable orientation.

This almost complex structure arises naturally in the following
context: A smooth, compact and oriented 4 dimensional manifold with
non-zero second Betti number has a 2--form that is symplectic on the
complement of its zero set, this a disjoint set of embedded circles
(see, eg Taubes \cite{T1}, Honda \cite{Ho}, Gay and Kirby
\cite{GK}). There are indications that certain sorts of closed,
symplectic surfaces in the complement of the zero set of such a
2--form code information about the differential structure of the
4--manifold (Taubes \cite{T2}). Meanwhile, \cite{T1} describes the complement
of any such vanishing circle in one of its tubular neighborhoods as
diffeomorphic to $(0, \infty ) \times (S^1 \times S^{2})$ via a
diffeomorphism that makes all of the relevant symplectic surfaces
pseudoholomorphic with respect to either the almost complex structure
in \eqref{eq1.1} or its push-forward by the two-fold covering map that
sends $(t, \theta , \varphi )$ to the same point as $(t+\pi , \pi
-\theta , -\varphi )$.

The investigation of these symplectic surfaces in the differential
topology context lead the author to study the pseudoholorphic
subvarieties in $\mathbb{R} \times  (S^1 \times  S^2)$
in general with a specific focus on the multiply
punctured spheres. This series of articles reports on this study.
In particular, the first article in this series \cite{T3} defined a
topology on the set of pseudoholomorphic subvarieties and defined
them as elements of moduli spaces that are much like those
introduced by Hofer \cite{H1,H2,H3}, Hofer--Wysocki--Zehnder 
\cite{HWZ2,HWZ1,HWZ3} and
Eliashberg--Hofer--Givental \cite{EGH} in a closely related context. The
multiply punctured sphere moduli spaces were then proved to be
smooth manifolds and formulae were given for their dimensions.
Finally, \cite{T3} describes necessary and sufficient conditions to
guarantee the existence of a moduli space component with
prescribed $|s|  \to \infty$ asymptotics in
$\mathbb{R}  \times (S^1 \times  S^2)$.
The details of much of this are summarized below for the benefit
of those who have yet to see \cite{T3}.

This second article in the series describes the multiply punctured sphere
moduli spaces in much more detail as it describes the set of components and
provides something of a parametrization for each component.

Note that the article \cite{T4}, a prequel to this series, provided this
information for certain disk, cylinder and thrice-punctured sphere moduli
space components.

\subsection{The background}\label{sec:1a}

The almost complex structure in \eqref{eq1.1} is compatible with
the symplectic form
\begin{equation}\label{eq1.2}
\omega  = d\bigl(e^{-\surd 6s} \alpha \bigr),
\end{equation}
where $\alpha $ is the following contact 1--form on $S^1 \times S^2$:
\begin{equation}\label{eq1.3}
\alpha \equiv - (1 - 3 \cos ^{2}\theta ) dt - \surd 6 \cos \theta
\sin^2\theta  d\varphi.
\end{equation}
In this regard, the standard product metric on $\mathbb{R} \times  (S^1 \times  S^2)$
is related to the bilinear
form $\omega (\cdot , J\cdot )$ using the rule
\begin{equation}\label{eq1.4}
\frac{1}{\surd 6 (1 + 3\cos  ^4\theta )^{1/2}}e^{\surd 6s}
\omega (\cdot , J\cdot ) = ds^{2} + dt^{2} +
d\theta ^{2} + \sin^{2}\theta  d\varphi ^{2}.
\end{equation}
On a related note, the form $\omega $ is self-dual and harmonic
with respect to the product metric in \eqref{eq1.4}, this a
consequence of the various strategically placed factors of $\surd
6$.

Following Hofer, a pseudoholomorphic subvariety in $\mathbb{R} \times  (S^1 \times  S^2)$
is defined to be a closed subset, $C$, that lacks isolated points and has the following
properties:

\itaubes{1.5}
\textsl{The complement in $C$ of a countable, nowhere accumulating subset is a
2--dimensional submanifold whose tangent space is $J$--invariant.}

\item
$\smallint _{C \cap K}  \omega  < \infty $ \textsl{when $K \subset  \mathbb{R}\times
(S^1 \times  S^2)$ is an open set with compact closure}.

\item
$\smallint _{C} d\alpha  < \infty $.

\end{itemize}

A pseudoholomorphic subvariety is said to be `reducible' when the removal of
a finite set of points disconnects it.

As explained in \cite[Section~2]{T4}, any pseudoholomorphic
subvariety intersects some sufficiently large $R$ version of the
$|s|   \ge R$  portion of $\mathbb{R} \times  (S^1 \times  S^2)$ as an embedded union of
disjoint, half-open cylinders. In particular, if $E$ denotes such a
cylinder, then the restriction of s to $E$ defines a smooth, proper
function with no critical points. Moreover, the limit as $|s|   \to  \infty $
of the constant s slices of $E$
converge in $S^1 \times S^2$ pointwise as a multiple
cover of some `Reeb orbit', this an embedded, closed orbit of the
vector field
\begin{equation}\label{eq1.6}
\hat {\alpha } \equiv (1 - 3 \cos ^{2}  \theta )  \partial _{t} +
\surd 6 \cos \theta \partial _{\varphi }.
\end{equation}
A subset $E  \subset  C$ of the sort just described is said to be an `end' of $C$.

In the case that an irreducible subvariety is not a $\theta =0$ or
$\theta =\pi $ cylinder, considerations of the convergence of the
constant $s$ slices of any given end to the limiting Reeb orbit led in  \cite[Section~1]{T3} to
the association of a 4--tuple of asymptotic data to the end in
question. To elaborate, such a 4--tuple has the form $(\delta , \varepsilon,  (p, p'))$
with $\delta $ either $-1$, $0$ or $1$; with $\varepsilon $ one of
the symbols $\{-, +\}$; and with  $(p, p')$ being an ordered pair of integers
that are not both zero and are constrained to obey:

\itaubes{1.7}
$|\frac{p'}{p}|  >\sqrt {\frac{3}{2}} $ \textsl{in the case that $p < 0$ and $\delta  = 0$}.

\item
\textsl{If $\delta =\pm 1$, then $p < 0$}.

\item
\textsl{If $\delta  = 1$, then $\frac{p'}{p} < -\sqrt {\frac{3}{2}} $ if
$\varepsilon  = +$, and $\frac{p'}{p} > -\sqrt{\frac{3}{ 2}} $ if $\varepsilon  = -$}.

\item
\textsl{If $\delta = -1$, then $\frac{p'}{p} >\sqrt {\frac{3 }{2}} $
if $\varepsilon  = +$, and $\frac{p'}{p} <\sqrt {\frac{3 }{ 2}} $ if $\varepsilon  = -$}.
\end{itemize}

To explain the meaning of the 4--tuple assignment, first note that
the angle $\theta $ is constant on any integral curve of the
vector field in \eqref{eq1.6}. This understood, the case $\delta
= 0$ signifies that this constant value of $\theta $ on the limit
Reeb orbit for the given end is an angle in $(0,\pi )$.
Meanwhile, the case $\delta  = 1$ signifies that this constant
value is 0, and $\delta = -1$ signifies that the $|s|    \to
\infty $ limit of $\theta $ on the end is $\pi $. In all of these
cases, the appearance of $+$ for the parameter $\varepsilon $
signifies that $s \to  \infty $ on the given end, while the
appearance of $-$ for $\varepsilon $ signifies that $s \to
-\infty $ on the end. An end with $\varepsilon  =+$ is said to be
a `concave side' end, and an $\varepsilon  = -$ end is said to be
a `convex side' end. Finally, the integers $p$ and $p'$ are the
respective integrals of the 1--forms $\frac{1 }{ {2\pi }}dt$ and
$\frac{1}{ {2\pi }}d\varphi $ around any given constant $|s| $
slice of the end granted that the latter are oriented by the
pull-back of the 1--form $-\alpha $.

In the case that $\delta  = 0$, the $|s| \to \infty $
limit of $\theta $ on the end in question is determined by the integer pair
 $(p, p')$  as this limiting angle obeys
\begin{equation}\label{eq1.8}
p'(1 - 3 \cos ^{2}\theta )- p\surd 6 \cos \theta  = 0 \text{ and }p'\cos \theta  \ge 0.
\end{equation}
In this regard, keep in mind that any ordered pair,  $(p, p')$ ,
of integers with at least one non-zero defines a unique angle in
$(0, \pi )$  via \eqref{eq1.8} in the case that $p \ge  0$. In
the case that $p < 0$, such a pair defines an angle in  $(0, \pi
)$  if and only if $| \frac{{p' } }{p}|  > \sqrt {\frac{3 }{
2}} $. The angle so defined is also unique.

In the case that $\delta =\pm 1$, the pair  $(p, p')$  determine the rate of
convergence of the angle $\theta $ to its limiting value of 0 or $\pi $. To
be explicit, results from \cite[Sections~2 and~3]{T4} can be used to verify
that
\begin{equation}\label{eq1.9}
\sin \theta =\surd 6 \hat {c}e^{(\sqrt {\frac{3 }{ 2}}  +
\delta  \frac{{p' } }{p})s}(1+ o(1))
\end{equation}
as $|s| \to \infty $ on the given end with $\hat {c}$ some positive constant.

Aside from the collection of 4--tuples from its ends, the
subvariety also defines a pair, (\c{c}$_ -$, \c{c}$_{ + })$, of
non-negative integers, these being the respective numbers
(counting multiplicity) of intersections between the the $\theta  = 0$
and $\theta =\pi $ cylinders and the subvariety. In this
regard, keep in mind that these cylinders are pseudoholomorphic.
Also keep in mind a consequence of the analysis in \cite[Section~2]{T4}:
There are at most a finite set of intersection points
between any two distinct, irreducible pseudoholomorphic
subvarieties in $\mathbb{R} \times  (S^1 \times  S^2)$. Finally, keep in mind that any intersection point
between distinct pseudoholomorphic subvarieties has positive local
intersection number (McDuff \cite{M}).

Granted what has just been said, an irreducible pseudoholomorphic subvariety
that is not a $\theta  = 0$ or $\theta =\pi $ cylinder defines an
example of what is a called here an asymptotic data set, this the set whose
elements consist of the ordered pair (\c{c}$_{ - }$, \c{c}$_{ + })$ and the
collection of 4--tuples from its ends. In general, the term `asymptotic data
set' refers to a certain sort of set that consists of one ordered pair of
non-negative integers, here (\c{c}$_{ - }$, \c{c}$_{ + })$; and some number
of 4--tuples that have the form $(\delta , \varepsilon ,  (p, p'))$ where
$\delta $ is $-1$, $0$ or $1$, $\varepsilon $ is either $-$ or $+$,
and  $(p, p')$  is an ordered pair of integers that obey the rules in~\eqreft17. A set,
$\hat{A}$, as just described is deemed an asymptotic data set in the event that
it obeys five additional constraints. Here are the first two:
\begin{equation}\label{eq1.10}
\sum _{(\delta ,\varepsilon ,(p,p')) \in \hat{A}}
\varepsilon p = 0
\qquad\text{and}\qquad
\sum _{(\delta ,\varepsilon ,(p,p')) \in \hat{A}} \varepsilon p' +\text{\c{c}}_{ + }
+ \text{\c{c}}_{ - } = 0.
\end{equation}
The constraint in \eqref{eq1.10} follows when $\hat{A}$ comes
from an irreducible pseudoholomorphic subvariety by using Stokes'
theorem when considering the integrals of $\frac{1}{2\pi}dt$ and
$\frac{1}{{2\pi}}d\varphi $ on any sufficiently large, constant
$|s| $ slice of the subvariety.

Here is the third constraint:

\qtaubes{1.11}
\textsl{If $\hat{A}$ has two 4--tuples and {\rm\c{c}}$_{ + } = \text{\rm\c{c}}_{ - } = 0$,
then the 4--tuple integer pairs are relatively prime.}
\endqtaubes

Indeed, any $\hat{A}$ with two 4--tuples and \c{c}$_{ + } = \text{\c{c}}_{ - } = 0$
labels a moduli space of pseudoholomorphic cylinders. All of the latter are
described in \cite[Section~4]{T4} and none violate~\eqreft1{11}.

The two remaining constraints refer to a set, $\Lambda _{\hat{A}}$,
of distinct angles in $[0, \pi ]$ that is defined from $\hat{A}$. This set
contains the angle 0 if \c{c}$_{ + } > 0$ or if $\hat{A}$ contains a
$(1,\ldots)$ element, it contains the angle $\pi $ if \c{c}$_{ - } > 0$ or
if $\hat{A}$ contains a $(-1,\ldots)$ element, and its
remaining angles are those that are defined through~\eqref{eq1.8} by the integer
pairs from the (0,$\ldots)$ elements in $\hat{A}$. This
understood, here are the remaining constraints:

\itaubes{1.12}
\textsl{$\Lambda _{\hat{A}}$ is a single angle if and only if the latter is in $(0, \pi)$,
$\hat{A}$ has only two 4--tuples, and these are $(0,+,P)$ and $(0,-,P)$ with
$P =  (p, p')$  being a relatively prime integer pair.}

\item
\textsl{If $\Lambda _{\hat{A}}$ has more than one angle, then neither of its maximal
or minimal elements is defined via~\eqref{eq1.8} by the integer pairs of any $(0,+,\ldots)$
element from $\hat{A}$.}
\end{itemize}

The constraints in~\eqreft1{12} are consequences of two facts noted in
\cite[(4.21)]{T4} and in \cite[Section 2.E]{T3} about the pull-back of the angle $\theta $
when the latter is non-constant on an irreducible pseudoholomorphic
subvariety: First, this pull-back has neither local maxima nor local minima
in  $(0, \pi )$ . Second, if its $s \to \infty $ limit on a concave side
end is neither 0 nor $\pi $, then its restriction to any constant $s$ slice of
the end takes values both greater and less than this limit.

Given an asymptotic data set $\hat{A}$, use $\mathcal{M}_{\hat{A}}$ to
denote the set of irreducible, pseudoholomorphic, multiply punctured spheres
that give rise to $\hat{A}$. Grace $\mathcal{M}_{\hat{A}}$ with the
topology where a basis for the neighborhoods of any given $C  \in  \mathcal{M}_{\hat{A}}$
is given by sets that are indexed by positive real
numbers where the version that is defined by $\kappa  > 0$ consists of those
$C'  \in   \mathcal{M}_{\hat{A}}$ with
\begin{equation}\label{eq1.13}
\sup_{z \in C} \dist(z, C') + \sup_{z \in C' } \dist(C, z) < \kappa.
\end{equation}
The following theorem restates \cite[Theorem~1.2 and Proposition~2.5]{T3}:

\begin{theorem}\label{thm:1.1}
If $\hat{A}$ is an asymptotic data set and $\mathcal{M}_{\hat{A}}$ is non-empty,
the latter is a smooth manifold of dimension
\begin{equation*}
N_{ + }+ 2(N_{ - }+\hat {N}+\text{\rm\c{c}}_{ - }+\text{\rm\c{c}}_{ + }-1),
\end{equation*}
where $N_{ + }$, $N_{ - }$ and $\hat {N}$ are the respective numbers of
$(0,+ \ldots)$, $(0,-,\ldots)$ and $(\pm 1,\ldots)$ elements in $\hat{A}$.
\end{theorem}

Note that by virtue of \eqreft1{12}, the sum $N_{ - }+\hat
{N}+\text{\c{c}}_{ -}+\text{\c{c}}_{ + }$ is at least 1 for any
asymptotic data set. As explained in \cite{T4}, the story when $N_{ -
}+\hat {N}+\text{\c{c}}_{ - }+\text{\c{c}}_{ + }= 1$ is as
follows:
\begin{equation}\label{eq1.14}
\textsl{$\mathcal{M}_{\hat{A}}$ consists of
$\mathbb{R}$--invariant cylinders when $N_{ - }+\hat{N}+
\text{\c{c}}_{ - }+\text{\c{c}}_{ + }= 1$.}
\end{equation}
The components of the moduli spaces of pseudoholomorphic,
$\mathbb{R}$--invariant cylinders are of two sorts. First, there
are two single element components, the $\theta  = 0$ cylinder and
the $\theta =\pi $ cylinder. The second sort of component is a
circle. In this regard, each circle component is labeled by a
relatively prime pair of integers, $(p,p')$, with at least one
non-zero and with $| \frac{{p' } }{ p}|  > \sqrt {\frac{3 }{
2}} $ in the case that $p < 0$. The circle consists of the
subvarieties of the form $\mathbb{R} \times  \gamma $ where
$\gamma   \subset  S^1 \times S^2$ is an orbit of the vector
field $\hat {\alpha }$ where $\theta $ is given by the pair  $(p,
p')$  via \eqref{eq1.8}. The circle parameter on the moduli space
can be taken to be the constant value in $\mathbb{R}/(2\pi
\mathbb{Z})$ along $\gamma $ of $p't - p\varphi$.

A description of $\mathcal{M}_{\hat{A}}$ in the case that
$N_{ -}+\hat {N}+\text{\c{c}}_{ - }+\text{\c{c}}_{ + } = 2$ is given in the next
subsection. \fullref{sec:1c} starts the story that is told in subsequent
sections about the general case.

\subsection{The space $\mathcal{M}_{\hat{A}}$ when its dimension is $N_{ + }+2$}\label{sec:1b}

This subsection is divided into two parts. The first provides an explicit
parameterization for $\mathcal{M}_{\hat{A}}$ in the case that
$N_{ -}+\hat {N}+\text{\c{c}}_{ - }+\text{\c{c}}_{ + }= 2$ and so
$\dim(\mathcal{M}_{\hat{A}}) = N_{ + }+2$. The second part describes some aspects of the
subvarieties that map near the frontier in the parametrizing space.

Before starting, note that the story in this case is simpler than
for cases when the dimension is greater than $N_{ + }+2$ by virtue
of the fact that function $\theta $ lacks critical points with
$\theta $ values in  $(0, \pi )$  on the model curves of
subvarieties in the $N_{ - }+\hat {N}+\text{\c{c}}_{ - }+\text{\c{c}}_{ +}= 2$
versions of $\mathcal{M}_{\hat{A}}$. This is a
consequence of \cite[Proposition~2.11]{T3}. The same proposition
implies that the tautological map from any such model curve to
$\mathbb{R} \times  (S^1 \times  S^2)$ is an
immersion that is transversal to the $\theta  = 0$ and $\theta =\pi $ loci.

\step{Part 1}
The story here and also in the $N_{ - }+\hat {N}+\text{\c{c}}_{ - }+\text{\c{c}}_{+ }> 2$
case presented later requires the introduction of a graph with
labeled vertices and labeled edges. The simplicity in the case at hand stems
from the fact that the graph in question is linear.

To describe the graph in the $N_{ - }+\hat {N}+\text{\c{c}}_{ - }+\text{\c{c}}_{ +}= 2$
case, first agree to view a linear graph as a closed interval in $[0, \pi ]$
whose vertices define a finite subset that includes the endpoints.
Let $T^{\hat{A}} \subset  [0, \pi ]$ denote the graph in
question. Each vertex of $T^{\hat{A}}$ is labeled with its
corresponding angle, and the set of these vertex angles is the set $\Lambda_{\hat{A}}$
as defined in the previous subsection. Meanwhile, each
edge of $T^{\hat{A}}$ is labeled by an integer pair using the rules
that follow. In the statement of these rules and subsequently, the letter
`$e$' is used to signify an edge, and $Q_{e}$ is used to signify an integer
pair that is associated to the edge $e$. Here are the rules:

\itaubes{1.15}
\textsl{If $e$ starts the graph at an angle in  $(0, \pi )$, then $Q_{e}$
is the integer pair from the $(0,-,\ldots)$  element in $\hat{A}$
that define this minimal angle via \eqref{eq1.8}.}

\item
\textsl{ If 0  is the smallest angle on $e$, then one and only one of the following
 situations arise: There is a $(1,-,\ldots)$  element in $\hat{A}$,
 or there is a $(1,+,\ldots)$  element in $\hat{A}$,  or $\text{\c{c}}_{+ } = 1$.
 In the these respective cases, $Q_{e}$  is the integer pair from the $(1,-,\ldots)$
 element in $\hat{A}$, or minus the integer pair from the $(1,+,\ldots)$
 element in $\hat{A}$ , or $(0, -1)$  when $\text{\c{c}}_{ + } = 1$.}

\item
\textsl{ Let $o$  denote a bivalent vertex, let $\theta _{o}$  denote its angle,
 and let $e$  and $e'$  denote its incident edges with the convention that
 $\theta _{o}$  is the largest angle on $e$.  Then $Q_{e} - Q_{e'  }$
 is the sum of the integer pairs from the $(0,+,\ldots)$  elements in $\hat{A}$
 that define $\theta _{o}$  via \eqref{eq1.8}.}
\end{itemize}

According to \cite[Theorem~1.3]{T3}, the moduli space $\mathcal{M}_{\hat{A}}$
is non-empty if and only if the following condition holds:

\qtaubes{1.16}{\sl
 Let  $\hat{e}$ denote an edge in $T^{\hat{A}}$. Then $p{q_{\hat{e} }}'-p'q_{\hat{e} } > 0$
 in the case that  $(p, p')$   is an integer pair that defines the
 angle of a bivalent vertex on $\hat{e}$. Moreover, if all of

\begin{enumerate}\leftskip38pt

\item[\rm(a)]
${q_{\hat{e} }}' < 0$,

\item[\rm(b)]
neither vertex on $\hat{e}$  has angle 0  or $\pi $  and

\item[\rm(c)]
the version of \eqref{eq1.8}'s integer $p'$  for one of the
vertex angles on $\hat{e}$ is positive,
\end{enumerate}
 hold, then both versions of $p'$  are positive.}
\endqtaubes

This condition is assumed in what follows.

The graph $T^{\hat{A}}$ is now used as a blueprint of sorts to define
a space, $O^{\hat{A}}$, that plays a fundamental role in what
follows. The definition of the latter follows in three steps.

\substep{Step 1}
Let $\hat{A}_{ + } \subset \hat{A}$ denote the set of $(0,+,\ldots)$ elements and let
\begin{equation}\label{eq1.17}
\mathbb{R}^{\hat{A}} \subset  \Maps(\hat{A}_{ + };\mathbb{R})
\end{equation}
denote the subspace where distinct elements in $\hat{A}_{+}$ have
distinct images in $\mathbb{R}/(2\pi \mathbb{Z})$ when their
respective integer pair components define the same angle in
\eqref{eq1.8}.

Let $\mathbb{R}_{ - }$ denote an extra copy of the affine line $\mathbb{R}$.

\substep{Step 2}
View the space $\Maps(\hat{A}_{ + }; \mathbb{Z})$ as a group using
addition in $\mathbb{Z}$ to give the composition law. Of
interest is an action of this group on
\begin{equation}\label{eq1.18}
\mathbb{R}_{ - }  \times \Maps(\hat{A}_{+};\mathbb{R}).
\end{equation}
This action is trivial on $\mathbb{R}_{ - }$. To describe the
action on the second factor in \eqref{eq1.18}, note first that
$\Maps(\hat{A}_{ + }; \mathbb{Z})$ is generated by a set
$\{z_{u}:u  \in \hat{A}_{ + }\}$ where $z_{u}(\hat{u}) = 0$
unless $\hat{u}= u$ in which case $z_{u}(u) = 1$. The action of
$z_{u}$ on a given $x  \in \Maps(\hat{A}_{ + }; \mathbb{R})$ is
as follows: First, $(z_{u}\cdot x)(\hat{u}) = x(\hat{u})$ in the
case that the integer pair from $\hat{u}$ defines an angle via
\eqref{eq1.8} that is less than that defined by $u$'s integer
pair. Such is also the case when the two angles are equal and
$\hat{u}  \ne u$. Meanwhile, $(z_{u}\cdot x)(u) = x - 2\pi $.
Finally, if the integer pair from $\hat{u}$ defines an angle that
is greater than that defined by $u$'s integer pair, then
$(z_{u}\cdot x)(\hat{u})$ is obtained from $x(\hat{u})$ by adding
\begin{equation}\label{eq1.19}
-2\pi   \frac{{p_u} ' p_{\hat{u}} - p_u p_{\hat{u}} ' }{{q_{\hat{e}}} '
p_{\hat{u}} - {q_{\hat{e}}} p_{\hat{u}} ' },
\end{equation}
where $(p_{u}, {p_{u}}')$ is the integer pair entry
of the element $u$, $(p_{\hat{u}}, {p_{\hat{u}}}')$ is that of
$\hat{u}$, and $\hat{e}$ is the edge in $T^{\hat{A}}$ whose
largest angle vertex has the angle that is defined via
\eqref{eq1.8} by this same $(p_{\hat{u}},{p_{\hat{u}}}')$.

The action just described of $\Maps(\hat{A}_{ + }; \mathbb{Z})$
commutes with the action of $\mathbb{Z}  \times   \mathbb{Z}$
that is defined as follows: An integer pair $N = (n, n')$ acts as
translation by $-2\pi (n'q_{e} - n{q_{e}}')$ on $\mathbb{R}_{ -}$; here
$(q_{e}, {q_{e}}')$ is the integer pair that is
associated to the edge in $T^{\hat{A}}$ with the smallest
angle vertex. Meanwhile, $N$ acts on any $x  \in \Maps(\hat{A}_{ +}; \mathbb{R})$
so that the resulting map, $N\cdot x$, sends any
given $u  \in \hat{A}_{ + }$ to the point that is obtained from $x(u)$ by adding
\begin{equation}\label{eq1.20}
-2\pi \frac{n' p_{\hat{u}} - np_{\hat{u}} ' }{{q_{\hat{e}}} ' p_{\hat{u}} -
q_e p_{\hat{u}} ' }.
\end{equation}

\substep{Step 3}
Granted the definitions in the preceding steps, set
\begin{equation}\label{eq1.21}
O^{\hat{A}} \equiv  [\mathbb{R}_{ - }  \times
\mathbb{R}^{\hat{A}}]/[(\mathbb{Z}  \times  \mathbb{Z})  \times \Maps(\hat{A}_{ + };
\mathbb{Z})].
\end{equation}
As is explained in \fullref{sec:3a}, $O^{\hat{A}}$ is a smooth
manifold. Of particular interest here is its quotient by the
evident action of the group, $\Aut^{\hat{A}}$, whose
elements are the 1--1 self maps of $\hat{A}_{ + }$ that permute only
elements with identical 4--tuples. To elaborate, this group action
is induced from the action on $\mathbb{R}^{\hat{A}}$ whereby a given
$\iota   \in  \Aut^{\hat{A}}$ acts by composition. Thus,
\begin{equation}\label{eq1.22}
(\iota \cdot x)(u) = x(\iota (u)) \text{ for each }u  \in \hat{A}_{ + }.
\end{equation}
In what follows, $\hat{O}^{\hat{A}}   \subset  O^{\hat{A}}$
denotes the subset of points with trivial $\Aut^{\hat{A}}$ stabilizer.

The following theorem explains the relevance of these constructs:

\begin{theorem}\label{thm:1.2}
There is a diffeomorphism from $\mathcal{M}_{\hat{A}}$ to
$\mathbb{R}\times\hat{O}^{\hat{A}}/\Aut^{\hat{A}}$.
\end{theorem}

There are diffeomorphisms between $\mathcal{M}_{\hat{A}}$ and
$\mathbb{R} \times \hat{O}^{\hat{A}}/\Aut^{\hat{A}}$ that grant direct
geometric interpretations to the
$\mathbb{R}$ parameter and to all of the parameters in the
$\hat{O}^{\hat{A}}/\Aut^{\hat{A}}$ factor. Indeed,
\fullref{thm:3.1} asserts that the diffeomorphism in question can be
chosen so as to intertwine the $\mathbb{R}$ action on $\mathcal{M}_{\hat{A}}$
that translates the subvarieties by a
constant amount along the $\mathbb{R}$ factor in $\mathbb{R} \times  (S^1 \times  S^2)$
with the $\mathbb{R}$
action on its factor in $\mathbb{R} \times \hat{O}^{\hat{A}}/\Aut^{\hat{A}}$.
Furthermore, Propositions \ref{prop:3.4} and \ref{prop:3.5}
describe diffeomorphisms of the latter sort that interpret
the image in $\hat{O}^{\hat{A}}/\Aut^{\hat{A}}$ for
any given subvariety in terms of the $|s| \to \infty $ limits on the
subvariety and its behavior near the 0 and $\pi $ loci.

For example, what follows describes how this
$\hat{O}^{\hat{A}}/\Aut^{\hat{A}}$ data determines Reeb orbit
limits for a particular choice of diffeomorphism in \fullref{thm:1.2}.
To set the stage, keep in mind that $\theta $ is constant on any
Reeb orbit and that the Reeb orbits at a given $\theta    \in (0,
\pi )$ are distinguished as follows: Let  $(p, p')$  denote the
relatively prime integer pair that defines $\theta $ via
\eqref{eq1.8}. The $\mathbb{R}/(2\pi \mathbb{Z})$ valued function
$p\varphi  - p't$ is constant on any Reeb orbit at angle $\theta
$, and its values distinguish these orbits.

The set of $s \to   \infty $ limits in  $(0, \pi )$  of $\theta $'s
restriction to any subvariety in $\mathcal{M}_{\hat{A}}$ is the set of
angles for the bivalent vertices in $T^{\hat{A}}$. Here is how to
obtain information about the corresponding Reeb orbit limits of the $s \to
\infty $ slices: Each element in $\hat{A}_{ + }$ whose integer pair defines
a given angle in $\Lambda _{\hat{A}}$ labels a map from
$O^{A}$ to $\mathbb{R}/(2\pi \mathbb{Z})$ that is obtained as
follows: First take the $\mathbb{R}/(2\pi \mathbb{Z})$ reduction of
the image via a given map in $\mathbb{R}^{\hat{A}}$ of the
element in $\hat{A}_{ + }$ and then multiply the result by
$({q_{e}}'p - q_{e}p')$ where $e$ denotes the edge that ends at the
given vertex and  $(p, p')$  denotes the relatively prime pair of
integers that defines the angle in question. Now, let n denote the
number of $\hat{A}_{ + }$ elements whose integer pair defines the
given angle. The unordered set of n points in $\mathbb{R}/(2\pi
\mathbb{Z})$ so defined by the image in $O^{\hat{A}}/\Aut^{\hat{A}}$
of any chosen subvariety from $\mathcal{M}_{\hat{A}}$ is precisely
the set of values for $p\varphi - p't$ on those Reeb orbits at the
given angle that arise as $s \to \infty $ limits of the constant s slices of the chosen
subvariety.

\step{Part 2} The space $\hat{O}^{\hat{A}}$ is compact in the
case that the pairs from distinct $(0,+,\ldots)$ elements in
$\hat{A}$ define distinct angles via \eqref{eq1.8}. Thus,
$\mathcal{M}_{\hat{A}}/\mathbb{R}$ is compact in this case.
However, when two or more $(0,+,\ldots)$ elements of $\hat{A}$
have integer pairs that give the same angle, then the quotient
$\mathcal{M}_{\hat{A}}/\mathbb{R}$ is no longer compact. This
part of the subsection concerns the latter situation. In
particular, what follows is a brief introduction to the sorts of
subvarieties that inhabit the frontier of
$\mathcal{M}_{\hat{A}}/\mathbb{R}$. A more detailed description of
the frontier is contained in \fullref{sec:9}.

To set the stage, remark that $\hat{O}^{\hat{A}}$ sits in
$O^{\hat{A}}$ and the latter sits in the compact space,
$\underline {O}^{\hat{A}}$, that is obtained by replacing
$\mathbb{R}^{\hat{A}}$ in \eqref{eq1.21} with the whole of
$\Maps(\hat{A}_{ + }; \mathbb{R})$. As the $\Aut^{\hat{A}}$
action on $O^{\hat{A}}$ extends to one on $\underline
{O}^{\hat{A}}$, the space
$\underline{O}^{\hat{A}}/\Aut^{\hat{A}}$ provides a
compactification of $\mathcal{M}_{\hat{A}}/\mathbb{R}$. This
compactification is geometrically natural since each point in
$\underline {O}^{\hat{A}}$ can be used to parametrize some
pseudoholomorphic, multiply punctured sphere. However, a point in
$\underline {O}^{\hat{A}}-\hat{O}^{\hat{A}}$ together with a
point in $\mathbb{R}$ always parametrizes a sphere with less than
$N_{ + }+N_{ - }+\hat {N}$ punctures. Even so, the subvarieties
that are parametrized by the points in $\underline
{O}^{\hat{A}}-\hat{O}^{\hat{A}}$ are geometric limits of
sequences in $\mathcal{M}_{\hat{A}}$.

To elaborate on this last remark, suppose that $y  \in \mathbb{R}\times\underline{O}^{\hat{A}}$
and that $C$ is the pseudoholomorphic, punctured sphere that is parametrized
by $y$. Meanwhile, let $\{y_{j}\}$ denote any sequence in
$\mathbb{R} \times  \underline {O}^{\hat{A}}$ that
converges to $y$ and let $\{C_{j}\}$ denote the corresponding
sequence of pseudoholomorphic subvarieties. As explained in
\fullref{sec:9c}, the latter converges pointwise to $C$ in $\mathbb{R}
\times  (S^1 \times  S^2)$ in the sense that
\begin{equation}\label{eq1.23}
\lim_{j \to \infty } \bigl(\sup_{z \in C} \dist(z, C_{j}) + \sup_{z \in C_j }
\dist(C, z)\bigr) = 0.
\end{equation}
In fact, more is true: Let $C_{0}$ denote the model curve for $C$,
this a punctured sphere together with a proper, almost everywhere
1--1 pseudoholomorphic map to $\mathbb{R} \times  (S^1 \times  S^2)$
whose image is $C$. As noted above, $C$ has at
worst immersion singularities. Thus, the punctured sphere $C_{0}$
has a canonical `pull-back' normal bundle, $N \to  C_{0}$, with
an exponential map from a fixed radius disk subbundle in $N$ to
$\mathbb{R} \times  (S^1 \times  S^2)$ that
immerses this disk bundle as a regular neighborhood of $C$. Let
$N_{1} \subset  N$ denote this disk bundle. Now, suppose
that $R$ is very large number, chosen to insure, among other things,
that the $|s|   \ge  R$ part of $C$ lies far out on
the ends of $C$. Let $C_{0j}$ denote the corresponding model curve
for $C_{j}$. Then the $|s| \le  R$ portion of
each sufficiently large $j$ version of $C_{0j}$ maps to the
corresponding portion of $N_{1}$ so that the composition with the
exponential map gives the tautological map to the $|s| \le  R$ part of $C_{j}$.
Meanwhile, the composition of this
map to $N_{1}$ with the projection from $N_{1}$ to $C_{0}$ defines
a proper covering map of the $|s| \le  R$ part
of $C_{0j}$ over this same part of $C_{0}$. In this regard, the
degree of this covering can be 1 or greater; in all cases, the
covering is unramified.

To describe the $|s|   \ge R$ part of $C_{j}$,
suppose that $E  \subset  C$ is an end and let $\gamma \subset  S^1 \times S^2$
denote the corresponding
closed Reeb orbit. This is to say that the constant $|s| $ slices of $E$
converge pointwise to $\gamma $ as $|s| \to \infty $. Thus, $E$ lies in a small,
constant radius tubular neighborhood of one of the very large
$|s| $ sides of $\mathbb{R} \times \gamma$ as an embedded cylinder.
A component of the $|s|   \ge R$ part of $C_{j}$ lies in this same tubular
neighborhood. When $y$ is in $O^{\hat{A}}$, then the
corresponding part of $C_{0j}$ is a cylinder. However, if $y$ is in
the boundary of $\underline {O}^{A}$, there is at least one end of
$C$ where the corresponding part of $C_{0j}$ is a sphere with at
least three punctures. In any case, there is a tubular
neighborhood projection onto $\mathbb{R} \times\gamma $ that restricts
to the nearby part of $C_{j}$ so as to
define a proper, possibly ramified covering map from the
corresponding $|s|   \ge R$ part of $C_{0j}$
onto the relevant $|s|   \ge R$ part of
$\mathbb{R} \times \gamma $.

More is said in \fullref{sec:9c} about this compactification.

\subsection{The case that $\mathcal{M}_{\hat{A}}$ has dimension
greater than $N_{ + }+2$}\label{sec:1c}

The discussion here is meant to provide a brief overview of the story in the
case that $N_{ - }+\hat {N}+\text{\c{c}}_{ - }+\text{\c{c}}_{ + } = k+2 > 2$.

Necessary and sufficient conditions for a non-empty
$\mathcal{M}_{\hat{A}}$ are given in \cite[Theorem~1.3]{T3}. As
remarked previously, when non-empty, then $\mathcal{M}_{\hat{A}}$
is a smooth manifold whose dimension is $N_{ + }+2k+2$. The
structure of $\mathcal{M}_{\hat{A}}$ in the $k > 0$ case is
rather more complicated then in the $k = 0$ case. To simplify
matters to some extent, the space $\mathcal{M}_{\hat{A}}$ is
viewed here as an open subset in a somewhat larger space whose
extra elements consist of what are called `multiply covered'
subvarieties. In this regard, a multiply covered subvariety
consists of an equivalence class of elements of the form $(C_{0},
\phi )$ where $C_{0}$ is a connected, complex curve with finite
Euler characteristic and $\phi $ is a proper, pseudoholomorphic
map from $C_{0}$ to $\mathbb{R} \times  (S^1 \times  S^2)$ whose
image is a pseudoholomorphic subvariety in the sense of
\eqreft15. Here, the equivalence relation puts $(C_{0},\phi )
\sim  (C_{0}, \phi ')$ in the case that $\phi'$ is obtained from
$\phi $ by composing with a holomorphic diffeomorphism of
$C_{0}$. If $\phi $ is almost everywhere 1--1 onto its image, then
$(C_{0}, \phi )$ defines a pseudoholomorphic subvariety as
described in \eqreft15. The added elements consist of the
equivalence classes of pairs $(C_{0},\phi )$ where $\phi $ maps
$C_{0}$ to $\phi (C_{0})$ with degree greater than 1. Of interest
here is the case where $C_{0}$ and $\phi (C_{0})$ are punctured
spheres. Moreover, with $\hat{A}$ given, then $C_{0}$ must have
$N_{ + }+N_{ - }+\hat {N}$ punctures. Furthermore, the ends of
$C_{0}$ and the points where $\theta $'s pull-back is 0 and $\pi
$ must define the data set $\hat{A}$. In what follows
${\mathcal{M}^{*}}_{\hat{A}}$ is used to denote this larger space.

The topology on ${\mathcal{M}_{ \hat{A}}}^{*}$ is defined as
follows: A basis for the neighborhoods of any given ($C_{0}$, $\phi )$ are
sets \{$\mathcal{U} _{\kappa }$\}$_{\kappa > 0}$, where $\mathcal{U}
_{\kappa }$ consists of the equivalence classes of elements (${C_{0}}'$,
$\phi' $) that obey the following: There exists a diffeomorphism
$\psi \co C_{0} \to {C_{0}}'$ such that
\begin{equation}\label{eq1.24}
\sup_{z \in C} \bigl(\dist(\phi (z), (\phi ' \circ \psi )(z)) + r_{z}(\psi )\bigr) < \kappa.
\end{equation}
Here, $r_{z}(\psi )$ is the ratio of the norm at $z$ of the
$\Hom(T_{0,1}C_{0}; T_{1,0}{C_{0}}')$ part of $\psi $'s
differential to that of the $\Hom(T_{1,0}C_{0};T_{1,0}{C_{0}}')$ part.
The theorem that follows describes the
local topology of ${\mathcal{M}_{ \hat{A}}}^{*}$. This
theorem introduces the subspace $\mathcal{R} \subset {\mathcal{M}_{ \hat{A}}}^{*}$
whose elements are the
equivalence classes $(C_{0}, \phi )$ where $\phi $ agrees with
its pull-back under some non-trivial holomorphic diffeomorphism.
Thus $\phi =\phi  \circ \psi $ where $\psi \co  C_{0} \to C_{0}$ is a non-trivial,
holomorphic diffeomorphism.

The description that follows of ${\mathcal{M}^{*}}_{\hat{A}}$
speaks of a `smooth orbifold'. This term is used here to denote a Hausdorff
space with a locally finite, open cover by sets of the form $B/G$, where $B$ is
a ball in a Euclidean space and where $G$ is finite group acting on $B$.
Moreover, these local charts are compatible in the following sense: Let $U$
and $U'$ denote two sets from the cover that overlap. Let $\lambda \co  B  \to
 U$ and $\lambda '\co  B'  \to  U'$ denote the quotient maps and let $\Omega
 \subset  B$ and $\Omega '  \subset  B'$ denote respective components
of the $\lambda $ and $\lambda '$ inverse images of $U  \cap  U'$. Then,
there is a diffeomorphism, $h\co  \Omega  \to  \Omega '$ such that
$\lambda =\lambda ' \circ h$. A smooth map between smooth orbifolds is
a map that lifts near any given point in the domain to give a smooth,
equivariant map between Euclidean spaces. A diffeomorphism is a smooth
homeomorphism with smooth inverse.

\begin{theorem}\label{thm:1.3}

The space ${\mathcal{M}^{*}}_{\hat{A}}$ has the structure of a smooth manifold of
dimension $N_{ + }+2(N_{ - }+\hat {N}+\text{\c{c}}_{\hat{A}}-1)$
on the complement of $\mathcal{R}$ and,
overall, it has the structure of a smooth orbifold.
Moreover, $\mathcal{M}_{\hat{A}}$ embeds
in ${\mathcal{M}_{ \hat{A}}}^{*}$ as an open set.

\end{theorem}

This theorem is proved in \fullref{sec:5}.

A more detailed view of ${\mathcal{M}_{ \hat{A}}}^{*}$ is
provided via a decomposition as a stratified space where each
stratum intersects ${\mathcal{M}_{ \hat{A}}}^{*}-\mathcal{R}$
as a smooth submanifold that extends across
$\mathcal{R}$ as an orbifold.  The details of this are in
Sections~\ref{sec:5}--\ref{sec:9}. 
What follows is a brief outline of the story.

The strata of ${\mathcal{M}^{*}}_{\hat{A}}$ are indexed in part
by a subset, $B$, of $(0,-,\ldots)$ elements from $\hat{A}$ along
with a non-negative integer, $c$, that is no greater than $N_{ - }+\hat
{N}+\text{\c{c}}_{ - }+\text{\c{c}}_{ + }-2-| B| $. Here, $|B|$
denotes the number of elements in the set $B$. With $B$ and $c$ fixed, introduce
$S_{B,c} \subset  {\mathcal{M}_{ \hat{A}}}^{*}$ to denote
those equivalence classes $(C_{0}, \phi )$ for which the $\phi $ pull-back
of $\theta $ has precisely c critical points where $\theta $ is neither 0
nor $\pi $, and where the following condition holds:

\qtaubes{1.25}
\textsl{
The 4--tuples in $B$ correspond to the ends in $C_{0}$ where $s$ is unbounded from below,
where the $s \to -\infty $ limit of $\theta $ is neither 0 nor $\pi $,
and where this limit is achieved at all sufficiently large values of $|s|$.}
\endqtaubes

In general, $\mathcal{S} _{B,c}$ is a union of strata. To describe the latter,
introduce $d  \equiv N_{ + }+|B| +c$ and use $I_{d}$
to denote the space of $d$ unordered points in  $(0, \pi )$ , thus the $d$--fold
symmetric product of the interval  $(0, \pi )$ . The space $I_{d}$ has a
stratification whose elements are labeled by the partitions of $d$ as a sum of
positive integers. A partition of d as $d_{1}+\cdots+d_{m}$ corresponds to the
subset in $I_{d}$ where there are precisely $m$
distinct $\theta $ values, with $d_{1}$ points supplying one $\theta $
value, $d_{2}$ supplying another, and so on. If $\mathfrak{d} = (d_{1}, \ldots ,d_{m})$
denotes such a partition, use $I_{d,d}$ to denote the corresponding stratum.

Now define a map, $f\co  \mathcal{S} _{B,c} \to  I_{d}$ by assigning
to any given pair $(C_{0}, \phi )$ the angles of the critical
values of $\phi^*\theta $ in  $(0, \pi )$  along with the angles
that are defined via \eqref{eq1.8} using the integer pairs from
the $(0,+,\ldots)$ elements in $\hat{A}$ and from the 4--tuples
from $B$. Each stratum in $\mathcal{S} _{c}$ has the form
$f^{-1}(I_{d,\mathfrak{d}})$ where $\mathfrak{d}$ is a partition
of the integer $d$. The stratum $f^{-1}(I_{d,\mathfrak{d}})
\subset  \mathcal{S} _{B,c}$ is denoted by $\mathcal{S}
_{B,c,\mathfrak{d}}$ in what follows.

As explained in Sections~\ref{sec:6} and~\ref{sec:8}, each component of each stratum
is homeomorphic to a space that has the schematic form $\mathbb{R}
\times \mathbb{O}/\Autt$ where $\Autt$ is a
certain finite group and $\mathbb{O}$ is a certain tower of circle
bundles over a product of simplices. The locus $\mathcal{R}$
appears here as the image of the points where the $\Autt$
action is not free.

By way of an example, suppose that the integer pairs from the
$(0,+,\ldots)$ elements in $\hat{A}$ define distinct angles via
\eqref{eq1.8}. In this case, the version of $\mathbb{O}$ for any
component of the $N_{ +}+2(k+1)$ dimensional stratum is a tower
of circle bundles over a $k$--dimensional product of simplices.
Moreover, $\Autt$ is always trivial when $k > 0$. To say more,
reintroduce the set $\Lambda _{\hat{A}}$ and let $\{\theta _{ -
}, \theta _{ + }\}$ denote the respective minimal and maximal
angles in $\Lambda _{\hat{A}}$. Then any $k$--dimensional simplex
product that appears in the definition of an $N_{ +}+2(k+1)$
dimensional version of $\mathbb{O}$ is a component of the space
of $k$ distinct angles in $(\theta _{ - }, \theta _{ +
})-\Lambda_{\hat{A}}$.

In this example, the codimension 1 strata consist of the subvarieties from
${\mathcal{M}^*}_{\hat{A}}$ where the pull-back of $\theta $
to the model curve has $k$ non-degenerate critical points with either $k-1$
distinct critical values, none in $\Lambda _{\hat{A}}$, or $k$
distinct critical values with one in $\Lambda _{\hat{A}}-(\theta _{ - }, \theta _{ + })$.
What follows is meant to give a
rough picture of the manner in which the coincident top dimensional strata
join along the codimension 1 stratum.

In the case where there is a critical value in $\Lambda _{\hat{A}}$,
the picture depends on where the angle comes from. For example, in the
case that the angle comes from some $(0,+,\ldots)$ element
in $\hat{A}$, three coincident top dimensional strata glue across the
codimension 1 strata via a fiber bundle version of the `pair of pants' that
joins two circles to one. This happens fiberwise as respective fiber circles
for two of the associated top strata are joined across a codimension 1
stratum to a fiber circle in the third. The following is a schematic
drawing:
\begin{equation}
\label{eq1.26}\includegraphics{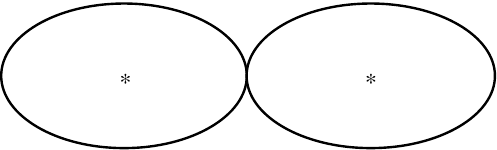}
\end{equation}
In this picture, the interior of each oval, minus its center point, and the
exterior of the union of the ovals correspond to the three codimension 0
strata. The two ovals minus their intersection point correspond to the
codimension 1 strata; the central intersection point is a codimension two
stratum.

If the angle comes from some $(0,-,\ldots)$ element, then
there are two incident top dimensional strata; and the gluing comes from an
identification between a fiber circle in one top strata and a corresponding
circle in the other. There is no fancy stuff here.

The story for the case where two critical values coincide is more involved
by virtue of the fact that there are various ways for this to happen. In the
first, the critical values coincide but the two critical points are in
mutually disjoint components of the critical locus. In this case again, two
codimension 0 strata glue across the codimension 1 stratum via an
identification of a fiber circle in one stratum with that in the other.

In the remaining cases, the two critical points share the same component of
the critical locus. What follows is a schematic picture for the first of
these cases:
\begin{equation}
\label{eq1.27}\includegraphics[width=1.5in]{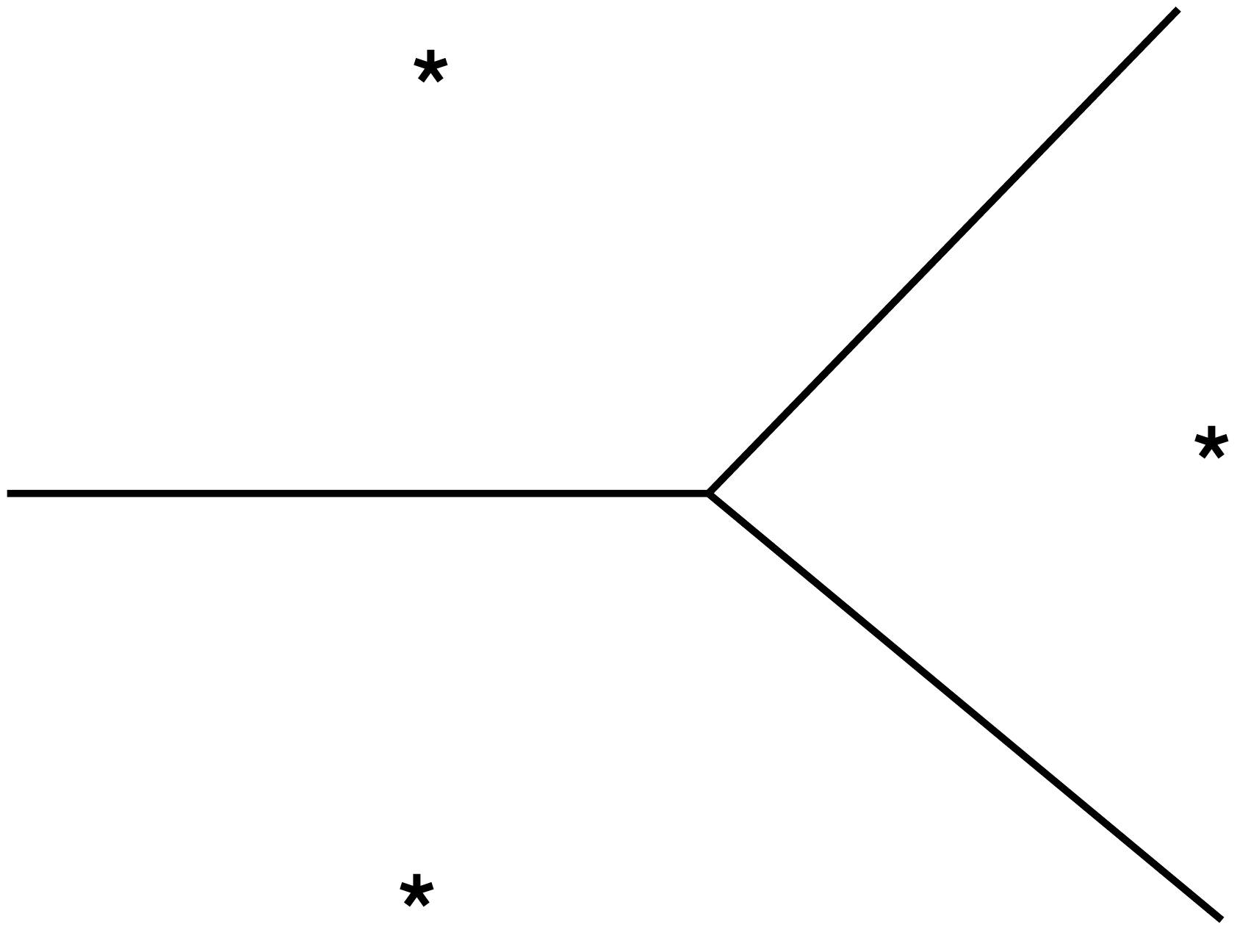}
\end{equation}
Here, the complement of the three points and the three rays corresponds to
three codimension 0 strata. The rays minus the origin correspond to
codimension 1 strata. The origin and the point at $\infty $ correspond to
codimension 2 strata.

The following drawing depicts the second of the cases under consideration.
\begin{equation}
\label{eq1.28}\includegraphics[width=2in]{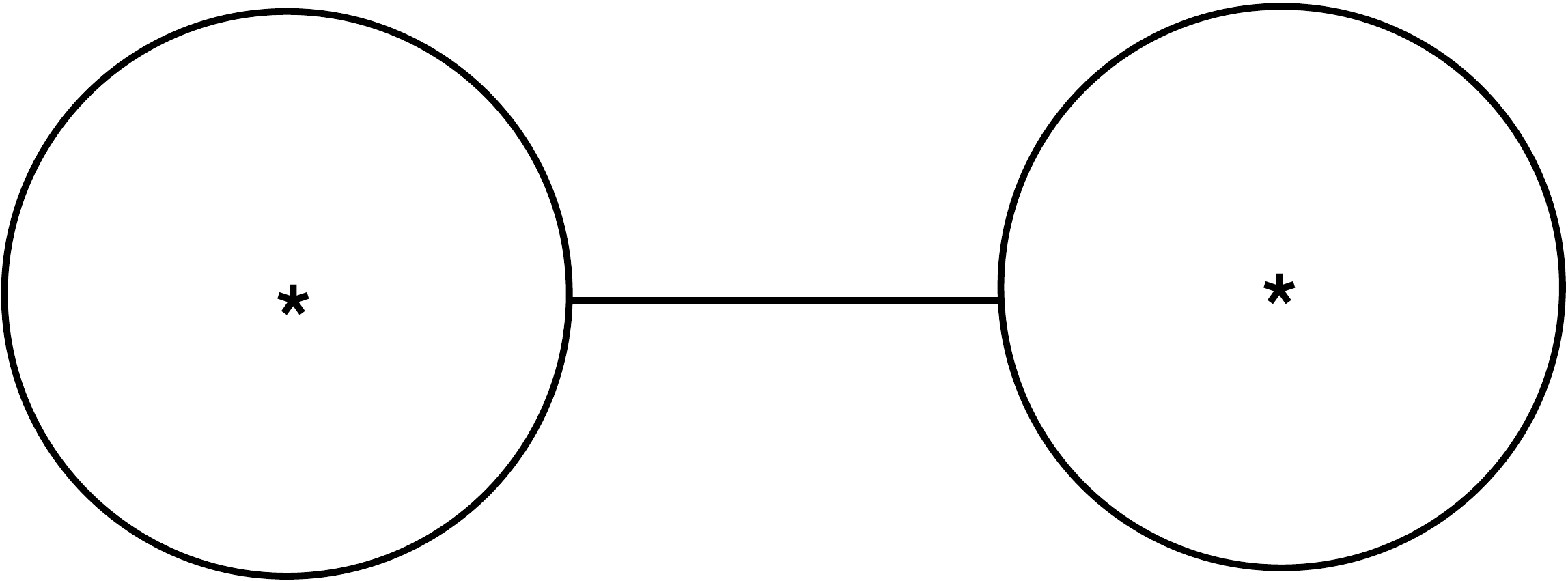}
\end{equation}
There are three top dimensional strata depicted here; these correspond to
the interiors of the two circles minus their centers, and the exterior
region. Meanwhile, the codimension 1 strata correspond to the interior of
the horizontal arc, and the complement in the two circles of the
endpoints of this same arc. The end points of the horizontal arc depict
codimension 2 strata.

\fullref{sec:9} contains additional details about all of this. \fullref{sec:9} also
describes a compatible, stratified compactification of ${\mathcal{M}^{*}}_{\hat{A}}$.

\subsection{A table of contents for the remaining sections}\label{sec:1d}

\fullref{sec:2} constitutes a digression of sorts to accomplish two
tasks. The first is to explain how to use the level sets of the
function $\theta $ on a given subvariety to construct a certain
sort of graph with labeled edges and vertices. The second task
explains how this graph can be used to define a canonical set of
parametrizations for the subvariety. As explained in a later
section, the associated graph and canonical parametrizations
provide the map that identifies a given component of a given
stratum in ${\mathcal{M}^{*}}_{\hat{A}}$ with a particular
version of $\mathbb{R}\times  \mathbb{O}/\Autt$.

\fullref{sec:3} elaborates on the story told above for the case where
$N_{ -}+\hat {N}+\text{\c{c}}_{ - }+\text{\c{c}}_{ + } = 2$. In particular, this section
describes a map that provides the diffeomorphism in \fullref{thm:1.2}. The proof
that the map is a diffeomorphism is started in this section.

\fullref{sec:4} starts with a digression to describe how the parametrizations from
\fullref{sec:2} can be used to distinguish distinct elements in
${\mathcal{M}^{*}}_{\hat{A}}$. These results are then used to prove that the map
from \fullref{sec:3} is 1--1. The final part of this section proves that the map is
proper. This completes the proof of \fullref{thm:1.2}.

\fullref{sec:5} proves \fullref{thm:1.3}; and it proves that each component of each
stratum of ${\mathcal{M}^{*}}_{\hat{A}}$ is a suborbifold.

\fullref{sec:6} uses the graphs from \fullref{sec:2} to parametrize certain
sorts of slices of the strata of ${\mathcal{M}^{*}}_{\hat{A}}$. A graph $T$
determines a certain submanifold, $\mathcal{M}^{*}_{\hat{A},T}$;
these are the fibers for a map that
fibers a given component of a given strata over a product of
simplices. Theorems~\ref{thm:6.2} and \ref{thm:6.3} identify each $\mathcal{M}^{*}_{\hat{A},T}$
as the product of $\mathbb{R}$ with the
quotient by a certain finite group of an iterated tower of circle
bundles over simplices. This section ends with description of a
particular map that realizes the identification in \fullref{thm:6.2}.

\fullref{sec:7} completes the proof of Theorems~\ref{thm:6.2} and
\ref{thm:6.3}. The arguments here first use the results from the
beginning of \fullref{sec:4} to prove that the map from
\fullref{sec:6} is one-to-one. The map is then proved to be
proper. The proof of the latter assertion requires the analysis of
limits of sequences in ${\mathcal{M}^{*}}_{\hat{A}}$, and the analysis
of such limits is facilitated by the use of \fullref{sec:2}'s
canonical parametrizations. In particular, the latter are used to
replace some of the hard analysis in similar compactness theorems from
Bourgeois--Eliashberg--Hofer--Wysocki--Zehnder \cite{BEHWZ} with
topology.

\fullref{sec:8} describes the stratification of ${\mathcal{M}^{*}}_{\hat{A}}$
in greater detail. In particular, this section describes each
component of each stratum as a fiber bundle over a product of simplices with
the typical fiber being the space $\mathcal{M}^{*}_{\hat{A},T}$
from \fullref{sec:6}. The section then explains how the graphs that arise can be
used to classify the strata.

\fullref{sec:9} is the final section. This section describes how the codimension 0
strata fit together across the codimension 1 strata and associated
codimension 2 strata. This section also describes a certain stratified
compactification of ${\mathcal{M}^{*}}_{\hat{A}}$. In particular,
\fullref{sec:9c} says more about \fullref{sec:1b}'s compactification of the
$N_{ -}+\hat {N}+\text{\c{c}}_{ - }+\text{\c{c}}_{ + } = 2$ moduli spaces. The section
ends with a description of the neighorhoods of the codimension 1 strata that
are added to make the compactification of ${\mathcal{M}^{*}}_{\hat{A}}$ in the
cases where $N_{ - }+\hat {N}+\text{\c{c}}_{ - }+\text{\c{c}}_{ +} > 2$.

\subsubsection*{Acknowledgements}
David Gay pointed out to the author that the use
of the critical values of $\theta $ to decompose ${\mathcal{M}^{*}}_{\hat{A}}$ has
very much a geometric invariant theory
flavor. He also pointed out that the techniques that are described in
\fullref{sec:4a} are similar to those used by Grigory Mikhalkin in other contexts.
The author also gratefully acknowledges the insight gained from
conversations with Michael Hutchings.

The author is supported in part by the National Science Foundation.


\setcounter{section}{1}
\setcounter{equation}{0}

\section{Background material}\label{sec:2}

The purpose of this section is to elaborate on various notions from
\cite{T3} that are used extensively in the subsequent subsections to
construct the parametrizations of any given component of the
multi-punctured sphere moduli space. The first subsection explains
how to use a subvariety in such a moduli space to define a certain
graph with labeled edges and vertices. The second describes some
useful parametrizations of various cylinders in any given
pseudoholomorphic subvariety. The third explains how a graph from
the first subsection is used to designate as `canonical' some of the
parametrizations from the second subsection.

\subsection{Graphs and subvarieties}\label{sec:2a}

The purpose of this subsection is to elaborate on the part of the discussion
in \cite[Section 2.G]{T3} that describes a method of associating a graph to a
pair $(C_{0}, \phi )$ where $C_{0}$ is a complex curve and $\phi $ is a
pseudoholomorphic map from $C_{0}$ into $\mathbb{R}  \times (S^1  \times  S^2)$
whose image is a pseudoholomorphic subvariety as defined in
\eqreft15. As discussed in \cite[Section 2.E]{T3}, the pull-back of $\theta $ to
$C_{0}$ has no local extreme points where its value is in $(0, \pi )$. This
understood, let $\Gamma    \subset C_{0}$ denote the union of the
non-compact or singular components of the level sets of the pull-back of
$\theta $.

The graph assigned to $(C_{0}, \phi )$ is denoted here as $T$. Its edges are
in a 1--1 correspondence with the components of $C_{0}- \Gamma $.
Given the correspondence and an edge, $e$, then $K_{e}$ is used in what
follows to denote $e$'s component of $C_{0}- \Gamma $. Each edge is
labeled by an ordered pair of integers, and when $e$ denotes a given edge,
$Q_{e}$ or $(q_{e}, {q_{e}}')$ is used to denote the labeling integer pair.
These integers are the respective integrals of $\frac{1 }{{2\pi}} dt$ and $\frac{1 }{ {2\pi }}d\varphi $
about any constant $\theta $ slice of $K_{e}$ when these components are oriented by the
pull-back of the form
\begin{equation}\label{eq2.1}
x = (1-3\cos^{2}\theta ) d\varphi -\surd 6 \cos\theta  dt.
\end{equation}
In this regard, note that this pull-back is non-zero along any such slice
because the form in \eqref{eq2.1} is a non-zero multiple of $J\cdot d\theta $.
(Here, the action of $J$ on the cotangent bundle is dual to its action on the
tangent bundle.) The integral of $x$ around any constant $\theta $ slice of
$K_{e}$ gives the $Q  \equiv  (q, q') = Q_{e}$ version of the function
\begin{equation}\label{eq2.2}
\alpha _{Q}(\theta ) = q' (1 - 3 \cos^{2}\theta )- q \surd 6\cos\theta.
\end{equation}
The $Q = Q_{e}$ version of $\alpha _{Q}$ is strictly positive on the
closure of $K_{e}$.

The monovalent vertices in $T$ are in a 1--1 correspondence with the following:

\itaubes{2.3}
 \textsl{The points in $C_{0}$  where $\theta $  is either 0  or $\pi $.}

\item
 \textsl{The ends of $C_{0}$  where the $|s|    \to   \infty $  limit of $\theta $  is 0  or $\pi $.}

\item
 \textsl{The convex side ends of $C_{0}$  where the $|s|    \to   \infty $  limit of $\theta $  is not achieved at finite $|s|$.}
\end{itemize}

With regards to the last two points, the discussion in \cite[Section~1.E]{T3}
noted the existence of some $R \gg  1$ such that the $|s|    \ge  R$
portion of $C_{0}$ is a disjoint union of embedded cylinders to which $|s| $ restricts as an unbounded,
proper map to $[R, \infty )$ without
critical points. Each such cylinder is called an `end'. The end is said to
be on the convex side of $C_{0}$ when s is negative on the end. Otherwise,
the end is said to be on the concave side of $C_{0}$. The angle $\theta $ on
any such end has a unique limit as $|s|    \to   \infty $.

To say more about the manner in which this limit is approached, note first
that the analysis used in \cite[Sections 2 and 3]{T4} proves a version of \eqref{eq1.9}
for a given end where the $|s|    \to   \infty $ limit of
$\theta $ is in $(0, \pi )$. In fact, the techniques from \cite{T4} can be used
to find coordinates $(\rho, \tau )$ for the end such that $\rho $ is
equal to a positive multiple of $s, \tau    \in \mathbb{R}/(2\pi\mathbb{Z})$, and $d\rho    \wedge  d\tau $
is positive. Moreover, when
written as a function of $\rho $ and $\tau $, the function $\theta $ has the
form
\begin{equation}\label{eq2.4}
\theta (\rho , \tau )=\theta _{E} + e^{- r\rho } (b \cos(n_{E}(\tau +\sigma )) + \hat{o}),
\end{equation}
where the notation is as follows: First, $\theta _{E}$ is the $s  \to \infty $ limit of $\theta $ on $E$.
Second, $b$ is a non-zero real number,
$n_{E}$ is a non-negative integer, but strictly positive if $E$ is on the
concave side of $C_{0}$, and $\sigma    \in   \mathbb{R}/(2\pi \mathbb{Z})$.
Third, $r > 0$ when $E$ is on the concave side, $r < 0$ when $E$ is on the convex
side of $C_{0}$; and in either case, $r$ is determined a priori by the integer
$n_{E}$ and $E$'s label in $\hat{A}$. Finally, $\hat{o}$ and its derivatives limit to
zero when $| \rho |    \to   \infty $.

The convex side ends that arise in the third point of \eqreft23 are those using
the $n_{E} = 0$ version of \eqref{eq2.4}.

The multivalent vertices of $T$ are in a 1--1 correspondence with the sets that
comprise a certain partition of the collection of non-point like components
in $\Gamma $. To describe these partition subsets, define a graph, $G$, whose
vertices are the non-point like components of $\Gamma $, and where two
vertices share an edge if the restrictions of $\theta $ to the corresponding
components of $\Gamma $ agree and if both of these components lie in the
closure of some component of $C_{0}- \Gamma $. The set of
components of $G$ defines the desired partition of the set of non-point like
components of $\Gamma $. In this regard, note that each compact, non-point
like component of $\Gamma $ defines its own partition subset.

As remarked at the outset, each vertex in $T$ has a label. These labels are
explained next. To start, a monovalent vertex that corresponds to a point in
$C_{0}$ where $\theta  = 0$ is labeled by $(p')$, where $p'$ is the positive
integer that gives the degree of tangency at the intersection point between
the $\theta  = 0$ cylinder and the image of any sufficiently small radius
disk about the given point in $C_{0}$. Meanwhile, a vertex that corresponds
to a point in $C_{0}$ where $\theta =\pi $ is labeled by $(-p')$ where $p' \ge  1$ is the analogous degree of tangency.

A monovalent vertex that corresponds to an end where the $|s|  \to   \infty $ limit of $\theta $
is 0 or $\pi $ is labeled by a 4--tuple of
the form $(\delta , \varepsilon , (p, p'))$, where $\delta  = 1$ if the
$\theta $ limit is 0 and $\delta  = -1$ if the $\theta $ limit is $\pi $.
Meanwhile, $\varepsilon    \in  \{+, -\}$ with $+$ appearing when the end
is on the concave side and $-$ appearing when the end is on the concave side.
Finally, the ordered pair $(p, p')$ are the integers that appear in \eqref{eq1.9}.

A monovalent vertex that corresponds to a convex side end where the $|s|    \to \infty$ limit of $\theta $
is neither 0 nor $\pi $ is labeled by
a 4--tuple of the form $(0,-, (p, p'))$ where the ordered pair of integers is
either $+$ or $-$ the pair that labels its incident edge. The sign here is that
of the constant $b$ in \eqref{eq2.4}.

The labeling of the multivalent vertices is more involved. To elaborate,
each such vertex is first assigned the angle in $(0, \pi )$ of the
components of its corresponding partition subset. In addition, each is
assigned a certain graph of its own, this also a graph with labeled vertices
and edges. When $o$ denotes a vertex, its graph is denoted by $\underline{\Gamma }_{o}$.
The graph $\underline {\Gamma }_{o}$ is a certain
closure of the union, $\Gamma _{o}$, of the components of the partition
subset that is assigned to $o$. The vertices of $\underline {\Gamma }_{o}$
that lie in $\Gamma _{o}$ consist of the critical points of $\theta $ on
$\Gamma _{o}$. The vertices in $\underline {\Gamma }_{o}- \Gamma _{o}$ are in 1--1 correspondence
with the set of ends of $C_{0}$
where all sufficiently large $|s| $ slices intersect $\Gamma_{o}$. Each vertex in
$\underline {\Gamma }_{o}$ is labeled with an
integer. This integer is zero when the vertex lies in $\Gamma _{o}$. The
integer is positive when the vertex corresponds to a concave side end of
$C_{0}$ and it is negative when the vertex corresponds to a convex side end.
In either case, the absolute value of this integer is the greatest common
divisor of the respective integrals of $\frac{1 }{ {2\pi }} dt$ and
$\frac{1 }{ {2\pi }}d\varphi $ around any given constant $|s| $ slice in the corresponding end.

So as not to confuse the edges in $\underline {\Gamma }_{o}$ with $o$'s
incident edges in $T$, those in $\underline {\Gamma }_{o}$ are called
`arcs'. The arcs are in 1--1 correspondence with the components of the
complement in $\Gamma _{o}$ of the $\Gamma _{o}$'s $\theta $ critical
points. Note that each arc is oriented by the 1--form that appears in \eqref{eq2.1}.
As can be seen in \cite[(2.16)]{T3} and \eqref{eq2.4} here, each vertex in $\underline {\Gamma }_{o}$ has
an even number of incident half-arcs and with half
oriented so as to point towards the vertex and half oriented so as to point
away. Only vertices with non-zero integer label can have two incident half
arcs. Each arc is also labeled by two incident edges to the vertex $o$; these
correspond to the two components of $C_{0}- \Gamma $ whose closure
contains the arc's image in the locus $\Gamma _{o}$. Thus, one labeling
edge connects $o$ to a vertex with smaller angle and the other to a vertex
with larger angle.

An isomorphism between graphs $T$ and $T'$ of the sort just described consists
of, among other things, a homeomorphism between the underlying 1--complexes.
However, such an isomorphism must map vertices to vertices and edges to
edges so as to respect all labels. In particular, if $o$ is a bivalent vertex
in $T$ and $o'$ its image in $T'$, then the isomorphism induces an isomorphism
between $\underline {\Gamma }_{o}$ and $\underline {\Gamma}_{o'}$ that preserves their vertex
labels and arc orientations and
respects the labeling of their arcs by pairs of incident edges to the
vertices. An automorphism of a given graph $T$ is an isomorphism from $T$ to
itself.

As just described, the edges of $T$ correspond to the components of
$C_{0}- \Gamma $, the vertices to certain ends of $C_{0}$ and the
singular and non-compact level sets of $\phi ^*\theta $. Meanwhile, the
labeling of the edges and vertices correspond to other aspects of the ends
of $C_{0}$ and the $\phi ^*\theta $ level sets. This understood, a
`correspondence' in $(C_{0}, \phi )$ of a graph $T$ of the sort just
described signifies in what follows a particular choice for the
identification that were just described between the geometry of $\phi ^*\theta $ on
$C_{0}$ and the various edges, vertices and so on of $T$. Note
that any one correspondence of $T$ in $(C_{0}, \phi )$ is obtained from any
other by the use of an automorphism on $T$. The use of a subscript, `$C$',
on $T$, thus $T_{C}$, signifies in what follows a graph $T$ together with a
chosen correspondence in a specified pair $(C_{0}, \phi )$.

If $T$ has a correspondence in $(C_{0}, \phi )$, and if $\psi $ is a
holomorphic diffeomorphism of $C_{0}$, then $T$ also has a correspondence in
$(C_{0}, \phi '  \equiv   \phi  \circ \psi )$. More to the point,
there is an automorphism $\delta \co  T  \to T$ and respective
correspondences for $T$ in $(C_{0}, \phi )$ and in $(C_{0}, \phi ')$ such
that the inverse image via $\psi $ and $\delta $ intertwine the two
correspondences. For example, if $e \subset T$ is an edge and $K_{e}$ is
its corresponding component of the $\phi ^*\theta $ version of $C_{0}- \Gamma $, then
$\psi ^{- 1}(K_{e})$ is the component of the
$\phi '*\theta $ version of $C_{0}- \Gamma $ that corresponds to
$K_{\delta (e)}$.

Granted what has just been said, the isomorphism type of graph with a
correspondence in a given pair $(C_{0}, \phi )$ depends only on the image
of $(C_{0}, \phi )$ in ${\mathcal{M}^{*}}_{\hat{A}}$.

\subsection{Preferred parametrizations}\label{sec:2b}

Let $K \subset C_{0}- \Gamma $ denote a component. Since the
$\theta $ level sets in $K$ are circles, the angle $\theta $ and an affine
parameter on the $\theta $ level sets can be used to parametrize $K$ by an
open cylinder. In this regard, there is a set of preferred parametrizations.
To set the stage for their description, introduce the ordered pair $Q \equiv  (q, q')$
to denote the respective integrals of $\frac{1 }{{2\pi }} dt$ and $\frac{1 }{ {2\pi }}d\varphi $ around any given
constant $\theta $ slice of $K$ using the orientation that is defined by the
pull-back of the 1--form in \eqref{eq2.1}. Next, let $\theta _{0}$ denote the
infimum of $\theta $ on $K$ and let $\theta _{1}$ denote the supremum. In
what follows, $\sigma $ is used to denote the linear coordinate on $(\theta_{0}, \theta _{1})$ and $v$
is used to denote an affine coordinate for the circle $\mathbb{R}/(2\pi \mathbb{Z})$.

\begin{definition}\label{def:2.1}
A preferred parametrization for $K$  is a diffeomorphism from the cylinder
$(\theta _{0}, \theta _{1})  \times   \mathbb{R}/(2\pi \mathbb{Z})$  to $K$  whose composition with
the tautological immersion of $K$  into $\mathbb{R}  \times (S^1  \times  S^2)$  can be written in
terms of smooth functions $a$  and $w$  on the cylinder as the map that pulls back the coordinates
$(s, t, \theta , \varphi )$  as

\itaubes{2.5}
$s = a$,

\item
$t = q v + (1-3\cos^{2}\sigma ) $w$ \mod(2\pi \mathbb{Z})$,

\item
$\theta  = \sigma$,

\item
$\varphi  = q'v + \surd 6 \cos\sigma $w$ \mod(2\pi \mathbb{Z})$.
\enditaubes
\end{definition}

The set of preferred parametrizations for $K$ is in all cases non-empty.
Indeed, to construct such a parametrization, start by fixing a transversal
to the constant $\theta $ circle in $K$, this a properly embedded, open arc in
 $K$ that is parametrized by the restriction of $\theta $. Next, use $\theta $
as one coordinate and, for the other, use the line integral of $x/\alpha_{Q}$ along the constant
$\theta $ circles from the chosen arc. A
preferred coordinate system can be obtained from the latter by adding an
appropriate function of $\theta $ to the second coordinate.

Listed next are six important properties of the preferred parametrizations.

\step{Property 1}{All preferred parametrizations pull the
exterior derivative of the contact 1--form $\alpha $ back as
$\sin\sigma \alpha _{Q}(\sigma ) d\sigma \wedge dv$ where $\alpha
_{Q}$ is the function on $[0, \pi ]$ that appears in
\eqref{eq2.2}. This implies that $\alpha _{Q}$ is positive on the
interval $(\theta _{0}, \theta _{1})$, and that it is positive at an
endpoint of this interval if the value there of $\theta $ is attained on
the closure of $K$.}

With regards, to $\alpha _{Q}$, note as well that the 1--form $x$ in
\eqref{eq2.1} pulls back to any constant $\sigma $ circle in the
parametrizing cylinder as $\alpha _{Q} dv$.

\step{Property 2}{The fact that $K$ is pseudoholomorphic puts
certain demands on the pair $(a,w)$. In particular, $K$ is
pseudoholomorphic in $R$ $\times (S^1 \times S^2)$ if and only if
\begin{equation}\label{eq2.6}
\begin{aligned}
&\alpha _{Q} a_{\sigma }-\surd 6 \sin\sigma  (1 + 3 \cos^{2}\sigma
) w a_{v} = -\frac{{1 + 3\cos ^4\sigma } }{ {\sin \sigma
}}\Bigl(w_{v}-\frac{1 }{ {1 + 3\cos ^4\sigma }}\beta \Bigr)
\\
&(\alpha _{Q}w)_{\sigma }-\surd 6 \sin\sigma  (1 + 3
\cos^{2}\sigma ) ww_{v}=\frac{1 }{ {\sin \sigma }}a_{v},
\end{aligned}
\end{equation}
Here, $\beta $ is defined to be the function $p (1 - 3 \cos^{2}\sigma ) +p'\surd 6 \cos\sigma \sin^{2}\sigma $.
In this equation and
subsequently, the subscripts $\sigma $ and $v$ denote the partial derivatives
in the indicated direction. Because $\alpha _{Q}$ is nowhere zero on
$(\theta _{0}, \theta _{1})$ the system in \eqref{eq2.6} is a non-linear
version of the Cauchy--Riemann equations.}

Here is an important consequence of the second equation in \eqref{eq2.6}:

\qtaubes{2.7}
\textsl{The function $\sigma    \to   \alpha _{Q}(\sigma )
\int _{\mathbb{R}/ (2\pi \mathbb{Z})} w(\sigma , v) dv$  is constant on the interval\break $(\theta _{0}, \theta _{1})$.}
\endqtaubes

\step{Property 3}{Consider the behavior of $K$ where $\theta $
is near a given $\theta _{*} \in \{\theta _{0}, \theta _{1}\}$.  If
$\theta _{*}$ is not achieved by $\theta $ on the closure of $K$, then
there exists $\varepsilon > 0$ such that the portion of $K$ where $|
\theta -\theta _{*}| \le \varepsilon $ is properly embedded in an end
of $C_{0}$. In particular, the constant $\theta $ slices of this
portion of $K$ are isotopic to the constant $|s| $ slices when $\theta
$ is very close to $\theta _{*}$. Moreover, if $\theta_{*} \notin \{0,
\pi \}$, then such an end is on the convex side of $C_{0}$ and the
associated integer $n_{E}$ that appears in \eqref{eq2.4} is zero.}

Supposing still that $\theta _{*}   \in (0, \pi )$, let $(p, p')$
denote the relatively prime integer pair that defines $\theta _{*}$
via \eqref{eq1.8}. Then $(q, q') = m(p, p')$ with $m$ a non-zero integer, and it
follows from \eqreft25 and \eqreft27 that the $\mod(2\pi \mathbb{Z})$ reduction of any
$\sigma    \in (\theta _{0}, \theta _{1})$ version of
\begin{equation}\label{eq2.8}
\frac{1 }{ {2\pi m}}\alpha _{Q}(\sigma )\smallint _{\mathbb{R} /(2\pi \mathbb{Z})} w(\sigma, v) dv
\end{equation}
is the $\mathbb{R}/(2\pi \mathbb{Z})$ parameter that distinguishes the Reeb
orbit limit of the $\theta  \to   \theta _{*}$ circles in $K$.

\step{Property 4}{If $\theta _{*} \in \{\theta _{0}, \theta
_{1}\}$ is neither 0 nor $\pi $ and if $\theta _{*}$ is achieved on
the closure of $K$, then the complement of the $\theta $ critical
points in the $\theta =\theta _{*}$ boundary of this closure is the
union of a set of disjoint, embedded, open arcs. The closures of each
such arc is also embedded. However, the closures of more than two arcs
can meet at any given $\theta $-critical point.}

This decomposition of the $\theta =\theta _{*}$ boundary of $K$
into arcs is reflected in the behavior of the parametrizations in \eqreft25 as
$\sigma $ approaches $\theta _{*}$. To elaborate, each critical
point of $\theta $ on the $\theta =\theta _{*}$ boundary of the
closure of $K$ labels one or more distinct points on the $\sigma =\theta
_{*}$ circle in the cylinder $[\theta _{0}$, $\theta _{1}]\times   \mathbb{R}/(2\pi \mathbb{Z})$.
These points are called `singular
points'. Meanwhile, each end of $C_{0}$ that intersects the $\theta  =\theta _{*}$
boundary of the closure of $K$ in a set where $|s| $ is unbounded also labels one or more points on this same circle.
The latter set of points are disjoint from the set of singular points. A
point from this last set is called a `missing point'.

The complement of the set of missing and singular points is a disjoint set
of open arcs. Each point on such an arc has a disk neighborhood in
$(0, \pi)  \times   \mathbb{R}/(2\pi \mathbb{Z})$ on which the parametrization in
\eqreft25 has a smooth extension as an embedding into $\mathbb{R}  \times (S^1  \times  S^2)$ onto a disk in $C_{0}$.

As might be expected, the set of arcs that comprise the complement of the
singular and missing points are in 1--1 correspondence with the set of arcs
that comprise the $\theta =\theta _{*}$ boundary of the closure
 $K$. In particular, the extension to \eqreft25 along any given arc in the $\sigma=\theta _{*}$ boundary of
 $[\theta _{0}, \theta _{1}]\times   \mathbb{R}/(2\pi \mathbb{Z})$ provides a smooth parametrization of
the interior of its corresponding arc in the $\theta =\theta _{*}$ boundary of the closure of $K$.

\step{Property 5}
It is pertinent to what follows to say more about the behavior of the
parametrization near the singular points on the $\sigma =\theta _{*}$ circle in the case that
$\theta _{*} \in \{\theta _{0}, \theta _{1}\}$ is a value of $\theta $ on the
closure of $K$.

To set the stage, let $z_{*}$ denote any given point in $C_{0}$. Let
$\hat {t}$ and $\hat {\varphi }$ denote the functions, defined on a ball
centered at the image of $z_{*}$ in $S^1  \times  S^2$, that
vanish at the latter point and whose respective differentials are $dt$ and
$d\varphi $. Now introduce
\begin{equation}\label{eq2.9}
r \equiv   \frac{{\sin \theta } }{ {(1 + 3\cos ^4\theta )^{1 /2}}}\bigl((1 - 3 \cos^{2}\theta )
\hat {\varphi }-\surd 6 \cos\theta \hat {t}\bigr),
\end{equation}
a function that is defined on the given ball about the image of $z_{*}$ in $S^1  \times  S^2$.
Next, let $D_{*}   \subset C_{0}$ be a small radius disk with center $z_{*}$ on which the
pull-back of $r$ is well defined. Note that d$\theta    \wedge  dr$ is zero
on $D_{*}$ only at the critical points of $\theta $. Thus, $r$ and
$\theta $ define local coordinates on the complement in $D_{*}$ of its
$\theta $ critical points.

When a component $K \subset C_{0}- \Gamma $ whose closure
contains $z_{*}$ is given a preferred parametrization, there is a
function, $\hat {v}$, that is defined on any given contractible component of
$K \cap D_{*}$ whose differential pulls back to $(\theta _{0},\theta _{1})\times \mathbb{R}/(2\pi \mathbb{Z})$ as $dv$.
This understood, then \eqreft25 identifies $r$ on such a component of $K \cap D_{*}$ as
\begin{equation}\label{eq2.10}
\begin{split}
r& = \frac{{\sin \theta } }{ {(1 + 3\cos ^4\theta )^{1 / 2}}}
\alpha _{Q_K } (\theta ) (\hat {v} - v^{*})\\
&\quad-\frac{{\sin \theta } }{ {(1 + 3\cos ^4\theta )^{1 / 2}}}  \surd
6 (\cos\theta  - \cos\theta _{*})(1 + 3 \cos\theta  \cos\theta _{*}) w^{*}.
\end{split}
\end{equation}
Here, $v^{*}$ and $w^{*}$ are constants. In particular, if
$z_{*}$ is in $K$, then $v^{*}=\hat {v}(z_{*})$
and $w^{*}$ is the value of $w$ at the point in $(\theta _{0},\theta _{1})  \times \mathbb{R}/(2\pi \mathbb{Z})$
that parametrizes $z_{*}$ via \eqreft25. Such is also the case when $z_{*}$ is on the
boundary of $K$ and is not a critical point of $\theta $ provided that the
radius of $D_{*}$ is sufficiently small. With regard to this last
case, note that a small radius guarantees the following: The boundary of $K$'s
closure intersects $D_{*}$ as an embedded arc and any chosen $\hat{v}$ extends to this arc as a smooth function.

According to \eqref{eq2.9}, the 1--form $dr$ can be written as $dr = J\cdot d\theta + r d\theta$
where $|r| $ vanishes at the image of $z_{*}$. This and the fact that $\phi $ is pseudoholomorphic have the following
consequence: There exists a holomorphic coordinate, $u$, on $D_{*}$ such
that the pull-back of the complex function $\theta +ir$ has the form
\begin{equation}\label{eq2.11}
\theta  + i r = \theta _{*} + u^{m + 1}+\mathcal{O}\bigl(|u| ^{m + 2}\bigr).
\end{equation}
Here, $m  \ge 0$ is the degree at $z_{*}$ of the zero of $d\theta $.
The term in \eqref{eq2.11} designated as $\mathcal{O}(|u| ^{m + 2})$ is
such that it's quotient by $|u| ^{m + 2}$ is bounded as $u$
limits to zero. Now, given that $D_{*}$ has small radius, \eqref{eq2.11}
indicates that $K$ intersects $D_{*}$ in a finite number of components,
each contractible. This noted, then $r$ is given on each such component by a
version of \eqref{eq2.10} where $v^*$ is determined by the choice for $\hat{v}$ and the
$R/(2\pi   \mathbb{Z})$ coordinate of the given given $\sigma  =\theta _{*}$ singular point.
Meanwhile, $w^{*}$ is the limiting value of the function $w$ at the singular point that maps
to $z_{*}$. Together, \eqref{eq2.10} and \eqref{eq2.11} describe the behavior of a preferred
parametrization near this singular point.

\step{Property 6}
{In the case that $\theta _{*}   \in  \{\theta _{0}, \theta_{1}\}$ is either 0 or $\pi $
and $\theta $ takes value $\theta _{*}$ on the closure of $K$, then the map in \eqreft25 extends to the
$\sigma  =\theta _{*}$ boundary of the cylinder as a smooth map that sends
this boundary to a single point. This extended map factors through a
pseudoholomorphic map of a disk into $\mathbb{R}  \times (S^1  \times
S^2)$ with the $\sigma =\theta _{*}$ circle being sent to the
disk's origin.}

To say a bit more about this case, note that the the pair $(q, q')$ has $q = 0$
and $q' < 0$. In this regard, $-q'$ is the intersection number between the image
of the aforementioned disk and the $\theta    \in  \{0, \pi \}$
locus. As $q = 0$, the assertion in \eqreft27 implies that the $\mod(2\pi \mathbb{Z})$
reduction of any $\sigma    \in (\theta _{0}, \theta _{1})$
version of the expression in \eqref{eq2.8} gives the $t$--coordinate of the
intersection point between the disk and the $\theta  = \{0, \pi \}$
locus.

To end this subsection, note that the set of preferred parametrizations of a
given component $K \subset C_{0}- \Gamma $ constitutes a single
orbit for an action of the group $\mathbb{Z}\times   \mathbb{Z}$. To
elaborate, let $\psi $ denote a parametrization for $K$, and let $N  \equiv (n, n')$
denote an ordered pair of integers. Let $\phi _{N}$ denote the
diffeomorphism of $(\theta _{0}, \theta _{1})  \times   \mathbb{R}/(2\pi \mathbb{Z})$ that pulls
$(\sigma, v)$ back as
\begin{equation}\label{eq2.12}
\phi _{N}^*(\sigma, v) = \biggl(\sigma , v - 2\pi   \frac{{\alpha _N
(\sigma )} }{ {\alpha _Q (\sigma )}}\biggr).
\end{equation}
Then $\psi ^{N}   \equiv   \psi  \circ \phi ^{N}$ is also a
preferred parametrization. In this regard, if $(a, w)$ are the pair that
appears in $\psi $'s version of \eqreft25, then the $\psi ^{N}$ version is the
pair $(a^{N}, w^{N})$ given by
\begin{equation}\label{eq2.13}
\begin{split}
a^{N}(\sigma , v) &= a\biggl(\sigma , v - 2\pi   \frac{{\alpha _N
(\sigma )} }{ {\alpha _Q (\sigma )}}\biggr)  \qquad\text{and}\qquad\\
w^{N}(\sigma , v) &= w\biggl(\sigma
, v - 2\pi   \frac{{\alpha _N (\sigma )} }{ {\alpha _Q (\sigma
)}}\biggr) + 2\pi   \frac{{qn' - q' n} }{ {\alpha _{Q_e }
(\sigma )}} .
\end{split}
\end{equation}
The assignment of the pair $(N, \psi )$ to $\psi ^{N}$ defines a
transitive action of $\mathbb{Z}  \times   \mathbb{Z}$ on the set of preferred
parametrizations. In this regard, note that the stabilizer of any given
parametrization is the $\mathbb{Z}$ subgroup in $\mathbb{Z}  \times   \mathbb{Z}$
of the integer multiples of $Q$.

As all parametrizations that arise henceforth for any given component of
$C_{0}- \Gamma $ are preferred parametrizations, the convention
taken from here through the end of this article is that the word
`parametrization' refers in all cases to a preferred parametrization. The
qualifier `preferred' will not be written as its presence should be
implicitly understood.

\subsection{Graphs such as $\underline{\Gamma}_{o}$ and preferred parametrizations}\label{sec:2c}

Let $o$ denote a given multivalent vertex in $T_{C}$. This subsection
describes how the graph $\underline {\Gamma }_{o}$ is used to define
certain `canonical' parametrizations for those components of $C_{0}- \Gamma $ that are labeled by $o$'s
incident edges. This is preceded by a
description of various properties of $\underline {\Gamma }_{o}$ that are
used in subsequent parts of this article. The discussion here has seven
parts.

\step{Part 1}
Suppose here that $\gamma $ is an arc in $\underline {\Gamma }_{o}$ and
let $e \in  E_{ - }$ and $e'  \in  E_{ + }$ denote the incident edges
to $o$ that comprise its edge-pair label. Suppose that both $K_{e}$ and
$K_{e'}$ have been graced with parametrizations. View the interior of
$\gamma $ as an open arc in $\Gamma _{o}$ and so an open arc in $C_{0}$.
As explained in the previous subsection, the parametrizations of both
$K_{e}$ and $K_{e'}$ extend to a neighborhood of int$(\gamma )$. As
such, there is a `transition function' that relates one of these extensions
to the other. To describe this transition function, let $\hat {v}_{e}$
denote a lift to $\mathbb{R}$ of the $\mathbb{R}/(2\pi \mathbb{Z})$ valued
coordinate $v$ on the parametrizing cylinder for $K_{e}$. Meanwhile, let $\hat
{v}_{e'}$ denote a corresponding lift of the $\mathbb{R}/(2\pi
\mathbb{Z})$ valued coordinate on the parametrizing cylinder for
$K_{e'}$. Then the coordinate transition function has the form
\begin{equation}\label{eq2.14}
\alpha_{Q_e }\hat {v}_{e}=\alpha _{Q_{e' } } \hat {v}_{e'} -2\pi   \alpha _{N},
\end{equation}
where $N  \equiv  (n, n')$ is some ordered pair of integers and
$\alpha_{N} = n' (1 - 3\cos^{2}\theta )-\surd 6 n \cos\theta $ is the $Q= N$
version of the function that appears in \eqref{eq2.2}. In this regard, the pair
$N$ can be chosen so that the corresponding versions of $(a, w)$ that appear in
\eqreft25 are related via
\begin{equation}\label{eq2.15}
\begin{split}
a_{e}\biggl(\sigma ,\frac{{\alpha _{Q_{e' } } (\sigma )} }{
{\alpha _{Q_e } (\sigma )}}\hat {v}_{e'} + 2\pi
\frac{{\alpha _N (\sigma )} }{ {\alpha _{Q_e } (\sigma )}} \biggr) &=
a_{e'}(\sigma , \hat {v}_{e'}) ,
\\
w_{e}\biggl(\sigma ,\frac{{\alpha _{Q_{e' } } (\sigma )} }{
{\alpha _{Q_e } (\sigma )}}\hat {v}_{e'} + 2\pi
\frac{{\alpha _N (\sigma )} }{ {\alpha _{Q_e } (\sigma )}}\biggr) &=
w_{e'}(\sigma , \hat {v}_{e'})\\
+\frac{1 }{
{\alpha _{Q_e } (\sigma )}}({q_{e}}'q_{e'} &-
q_{e}{q_{e'}}')\hat {v}_{e'}-\frac{{2\pi } }{
{\alpha _{Q_e } (\sigma )}}(q_{e}n' - {q_{e}}'n).
\end{split}
\end{equation}
Here, $(q_{e}, {q_{e}}')$ comprise the pair $Q_{e}$ while $(q_{e'},{q_{e'}}')$ comprise $Q_{e'}$.
With regards to these formulae,
note that changing the lift $\hat {v}_{e}$ by adding $2\pi $ has the
effect of changing the integer pair from $N$ to $N - Q_{e}$. Meanwhile, a
change of the lift $\hat {v}_{e'}$ by the addition of $2\pi $
changes $N$ to $N + Q_{e'}$.

\step{Part 2}
As is explained momentarily, a parametrization of $K_{e}$ and a lift to
$\mathbb{R}$ of the $\mathbb{R}/(2\pi \mathbb{Z})$ valued coordinate on the
corresponding parametrizing cylinder determines a canonical ordered pair
whose first component is a parametrization of $K_{e'}$ and whose
second is a lift to $R$ of the $\mathbb{R}/(2\pi \mathbb{Z})$ valued coordinate
on the latter's parametrizing cylinder. Indeed, this canonical pair provides
the $N = (0, 0)$ version of \eqref{eq2.14} and \eqref{eq2.15}, and it is a consequence of
\eqref{eq2.12} and \eqref{eq2.13} that there is a unique pair of parametrization and lift
that makes both \eqref{eq2.14} and \eqref{eq2.15} hold with $N = (0, 0)$.

Keep in mind that this canonical parametrization and lift changes with a
change in the initial parametrization for $K_{e}$ and lift of its $\mathbb{R}/(2\pi \mathbb{Z})$ parameter.
To make matters explicit, suppose that $N =(n, n')$ is an integer pair and that the latter changes the parametrization
for $K_{e }$ as depicted in \eqref{eq2.12} and \eqref{eq2.13}. In addition, suppose that the
value of the new $\mathbb{R}$--lift, $\hat {v}^{N}_{e}$, of the $\mathbb{R}/(2\pi \mathbb{Z})$ 
coordinate on the parametrizing cylinder is related
to the original at any given point by
\begin{equation}\label{eq2.16}
{\hat{v}^{N}}_{e}=\hat {v}_{e}- 2\pi \frac{{\alpha _N
(\sigma )} }{ {\alpha _{Q_e } (\sigma )}}.
\end{equation}
It now follows from \eqref{eq2.14} and \eqref{eq2.15} that the canonical $e'$ pair of
parametrization and lift are changed in the analogous manner by this same
integer pair $N$. This is to say that the new parametrization of
$K_{e'}$ is related to the original via the $e'$ version of \eqref{eq2.12} and
\eqref{eq2.13} as defined using the pair $N$. Meanwhile, the $\mathbb{R}$--valued lift,
$\hat {v}^{N}_{e'}$, of the $\mathbb{R}/(2\pi \mathbb{Z})$ valued
coordinate on the $e'$ version of the parametrizing cylinder is obtained from
the old at any given point by the version of \eqref{eq2.16} that substitutes $e'$ for $e$.

With regards to $\mathbb{R}$--valued lifts of the $\mathbb{R}/(2\pi \mathbb{Z})$
coordinate on a given parametrizing cylinder, note that any such lift is
uniquely determined by the lift to $\mathbb{R}$ of the $\mathbb{R}/(2\pi \mathbb{Z})$
coordinate of any one point. In the applications below, a lift to $\mathbb{R}$ is chosen
for the $\mathbb{R}/(2\pi \mathbb{Z})$ coordinate of a
distinguished missing or singular point on the $\sigma =\theta _{o}$
boundary of the parametrizing cylinder. The latter lift is then used to
define the lift of the $\mathbb{R}/(2\pi \mathbb{Z})$ coordinate over the
whole cylinder.

This last very straightforward observation is used below in the following context:
Suppose that $\upsilon _{e}$ is a distinguished missing or singular point
on the $\sigma =\theta _{o}$ boundary circle of the parametrizing
cylinder for $e$, and suppose that a lift to $\mathbb{R}$ of the $\mathbb{R}/(2\pi
\mathbb{Z})$ coordinate of $\upsilon _{e}$ has been chosen. Now, let $\nu $
denote a path on this boundary circle that begins at $\upsilon _{e}$ and
ends at some other missing or singular point. In this regard, the end point
of $\nu $ may well be $\upsilon _{e}$. In any event, assume that $\nu $
can be written as a non-empty concatenation of segments that connect pairs
of missing points, or connect pairs of singular points, or connect a missing
point and a singular point. Note that any such segment can run either with
or against the defined orientation of the $\sigma =\theta _{o}$
boundary circle.

Let $\gamma '$ denote the image in $\underline {\Gamma }_{o}$ of the final
segment on $\nu $, and let $e'$ denote the vertex that labels $\gamma '$ with
$e$. The lift to $\mathbb{R}$ of the $\mathbb{R}/(2\pi \mathbb{Z})$ coordinate of
$\upsilon _{e}$ then defines a lift to $\mathbb{R}$ of the coordinate along
the whole of $\nu $ and thus a lift over $\gamma '$. In this regard, note
that the latter depends on the homotopy class rel end points of the path
$\nu $. In any event, this lift and the given parameterization of $K_{e}$
determines a canonical parameterization of $K_{e'}$. Of course, it
also determines a canonical lift to $\mathbb{R}$ of the $\mathbb{R}/(2\pi \mathbb{Z})$
coordinate of the end point of $\nu $.

\step{Part 3}
This and Part 4 of the subsection describe a generalization of these
constructions. In this regard, the preceding definition suffices when
$T_{C}$ is a linear graph with no automorphisms, but more is needed for a
more complicated graph.

This part of the story constitutes a digression whose purpose is to
elaborate on some aspects of the graph $\underline {\Gamma }_{o}$. To
start, let $e$ denote an incident edge to o. Then $e$ labels a certain circular
graph, $\ell _{oe}$, with oriented edges and labeled vertices that has an
immersed image in $\underline {\Gamma }_{o}$. In order to describe $\ell_{oe}$,
it convenient to first choose a parametrization of the component,
$K_{e}   \subset C_{0}- \Gamma $ that $e$ labels. Such a
parametrization identifies $\ell _{oe}$ with the $\sigma =\theta
_{o}$ circle in the associated parametrizing cylinder. In this regard, the
vertices of $\ell _{oe}$ are the set of missing and singular points on
this circle, and its `arcs' are then the arcs in this circle that connect
these points. The orientation of an arc is defined by the restriction of dv.
Meanwhile, the vertices are labeled by integers in the following manner: All
singular points are labeled by the integer 0. Any given missing point has
label $\pm m$ where $m$ is a positive integer and where the $+$ sign is used if
and only if the missing point corresponds to a concave side end of $C$. The
integer $m$ is the greatest common divisor of the respective integrals of
$\frac{1 }{ {2\pi }} dt$ and $\frac{1 }{ {2\pi }}d\varphi $
over any given constant $|s| $ slice of the end.

The map from $\ell _{oe}$ to $\underline {\Gamma }_{o}$ is then induced
by the extension of the parametrizing map. In this regard, recall that this
extension maps the complement of the set of missing points on the $\sigma = \theta _{o}$
circle to $\Gamma _{o}   \subset \underline {\Gamma}_{o}$.
This map from $\ell _{oe}$ to $\underline {\Gamma }_{o}$ has
the following properties: First, it sends vertices to vertices and arcs to
arcs so as to preserve the vertex labels and the arc orientations. Second it
is 1--1 on the complement of the vertices. Thus, the image of the map from
$\ell _{oe}$ to $\underline {\Gamma }_{o}$ is a closed, oriented path in
$\underline {\Gamma }_{o}$ that crosses no arc more than once. The arcs in
the image of $\ell _{oe}$ are those whose edge pair label contains e. In
what follows, the homology class in $H_{1}(\underline {\Gamma }_{o}$;
$\mathbb{Z})$ of the image of $\ell _{o,e}$ is denoted as $[\ell _{oe}]$.

Since the abstract circle $\ell _{oe}$ can be reconstituted (up to an
automorphism) from its image as an oriented path in $\underline {\Gamma
}_{o}$, the image is also denoted by $\ell _{oe}$. Note as well that
distinct parametrizations of $K_{e}$ define isomorphic versions of $\ell
_{oe}$ and that the associated maps to $\underline {\Gamma }_{o}$ are
compatible with any such isomorphism.

Four properties of $C_{0}$ are reflected in the manner in which the
collection $\{\ell _{oe}\}$ sit in $\underline {\Gamma }_{0}$. Their
descriptions involve sets $E_{ - }$ and $E_{ + }$ where $E_{ - }$ is the set
of incident edges to $o$ that connect to vertices with angles less than
$\theta _{o}$, while $E_{ + }$ is the set of incident edges that connect
to vertices with angles greater than $\theta _{o}$.

\step{Property 1}
\textsl{Each arc in $\underline {\Gamma }_{o}$  is contained in precisely two versions of $\ell_{o(\cdot )}$,
one labeled by  an edge from $E_{ + }$  and the other by an edge from $E_{ - }$.  These are the
edges that comprise the label of the arc.}

\step{Property 2}
\textsl{The collection $\{[\ell _{oe}]\}_e$ is
incident to $o$  generates $H_{1}(\underline {\Gamma }_{o}; \mathbb{Z})$  subject to the following
constraint:}
\begin{equation}\label{eq2.17}
\sum_{e \in E_ + } [\ell _{oe}] - \sum _{e \in E_ - } [\ell _{oe}]= 0.
\end{equation}
The first property is a direct consequence of the definitions. To explain
the second property, first note that the closed parametrizing cylinders for
the components of $C_{0}- \Gamma $ that are labeled by $o$'s incident
edges can be glued to $\underline {\Gamma }_{o}$ by identifying any given
version of $\ell _{o(\cdot )}   \subset \underline {\Gamma }_{o}$
with corresponding segment in the $\sigma =\theta _{o}$ boundary of
the relevant parametrizing cylinder. The result of this gluing is a sphere
with as many punctures as $o$ has incident edges. By construction this
multipunctured sphere deformation retracts onto $\underline {\Gamma}_{o}$.
Meanwhile, its homology is generated by the collection $\{[\ell_{oe}]\}$ subject to the one constraint in \eqref{eq2.17}.

The third property is actually a consequence of the first two, but can also
be seen to follow from the fact that $C_{0}$ is a smooth, irreducible curve.

\step{Property 3}
\textsl{Let $\hat{E}$  denote the set of $o$'s incident edges that appear in the pair labels of the
incident half-arcs to a given vertex in $\underline {\Gamma }_{o}$.
An equivalence relation on $\hat{E}$ is generated by equating the two edges that come from the edge
pair label of an incident half-arc to the given vertex. This relation defines just one equivalence class.}

The final property is a special case of the one just stated.

\step{Property 4}
\textsl{Suppose that $e$  and $e'$ are incident edges to $o$,  one in $E_{ + }$  and the other in $E_{ - }$,
and  suppose that $\gamma    \subset   \ell _{oe}   \cap   \ell _{oe'}$.  If $\gamma'$  follows
$\gamma $  in $\ell _{oe}$,  then $\gamma '$  cannot follow $\gamma $  in $\ell _{oe'}$
unless the vertex between them is bivalent.}

\step{Part 4}
With the preliminary digression now over, this part of the subsection
describes the advertised generalization of the definition in Part 2 of a
canonical parametrization. The following definition is needed to set the
stage:

\begin{definition}\label{def:2.2}

A `concatenating path set' is an ordered set, $\{\nu _{1},\nu _{2},\ldots , \nu _{N}\}$,
of labeled paths in $\underline {\Gamma }_{o}$  with the following properties:

\begin{itemize}

\item
The label of each $\nu _{k}$  consists of a specified direction of travel and a specified incident edge to $o$.

\item
If $e$  denotes the edge label to a given $\nu _{k}$,  then $\nu _{k}$  is entirely contained in
$\ell _{oe}$.\\
Moreover, $\nu _{k }$ is the concatenation of a non-empty, ordered set of arcs in
$\ell_{oe}$  that are crossed in their given order when the path is traversed from start to finish.
Note that the arcs that comprise $\nu _{k}$  need not be distinct and can be crossed in either direction.

\item
No two consecutive pairs $\nu _{k}$  and $\nu _{k + 1}$  are labeled by the same incident edge to $o$.

\item
For each $1  \le  j < N$, the final arc on $\nu _{j}$  is the starting arc on $\nu _{j + 1}$.

\end{itemize}

The concatenating path set defines a directed path in $\underline {\Gamma }_{o}$
by sequentially traversing the paths that comprise the ordered set $\{\nu _{1}, \ldots , \nu _{N}\}$
in their given order.
\end{definition}

Granted this definition, here is the context for the discussion that
follows: Suppose that one of $o$'s incident edges has been designated as the
`distinguished incident edge'. Let $e$ denote the latter, and suppose that a
parametrization is given for $K_{e}$, and that a vertex has been designated
as the `distinguished vertex' on the realization of $\ell _{oe}$ as the
$\sigma =\theta _{o}$ boundary circle of the parametrizing cylinder.
Let $\upsilon _{e}$ denote this distinguished vertex, and let $\upsilon  \in \underline {\Gamma }_{o}$
denote its image. Now, let $e'$ label
some incident edge to $o$ with $e' = e$ allowed. Suppose, in addition, that a
concatenating path set $\nu \equiv \{\nu _{1},\ldots, \nu _{N}\}$ has been chosen so that

\itaubes{2.18}
\textsl{The edge label of $\nu _{1}$  is $e$, and $\nu _{1}$  starts at $\upsilon $.}

\item
\textsl{The edge label of $\nu _{N}$  is not $e'$  but its final arc lies in $\ell _{oe'}$.}
\end{itemize}

Properties 3 and 4 from Part 3 can be used to construct this sort of concatenating path set.

With this data set, a  \textit{parametrizing algorithm} is described next whose input is a pair consisting of
a parametrization for $e$ and a lift to $\mathbb{R}$ of the $\mathbb{R}/(2\pi\mathbb{Z})$
coordinate of $\upsilon _{e}$ in the $\sigma =\theta_{o}$ boundary of the parametrizing
cylinder for $K_{e}$ and whose output
is a parametrization for $K_{e'}$ as well as a canonical $\mathbb{R}$--valued lift of the
$\mathbb{R}/(2\pi \mathbb{Z})$ parameter on the
associated parametrizing cylinder. This output parameterization is the
advertised `canonical' parametrization for $K_{e'}$.

To describe this algorithm, note first that the observations from the final
paragraph from Part 2 can be viewed as defining a  parametrizing subroutine that takes as input the
data:

\itaubes{2.19}
\textsl{A parametrization of a component of $C_{0}- \Gamma $  whose closure intersects $\Gamma _{o}$.}

\item
\textsl{A distinguished point on the $\sigma =\theta _{o}$  boundary of the parametrizing cylinder.}

\item
\textsl{A lift to $\mathbb{R}$  of the $\mathbb{R}/(2\pi \mathbb{Z})$  coordinate of this distinguished point.}

\item
\textsl{A non-trivial path in the $\sigma =\theta _{o}$  boundary of the parametrizing cylinder
that starts at the distinguished point and consists of segments that connect missing and/or
singular points on this boundary.}
\end{itemize}
and gives as output:

\itaubes{2.20}
\textsl{A parametrization of the as yet unparametrized component of $C_{0}- \Gamma $  whose closure
contains the image in $\Gamma _{o}$  of the interior of the final segment on the chosen path.}

\item
\textsl{A lift to $\mathbb{R}$  of the $\mathbb{R}/(2\pi \mathbb{Z})$  coordinate on the $\sigma =\theta _{o}$
boundary of the corresponding parametrizing cylinder for this second component of $C_{0}- \Gamma$.}
\end{itemize}

The parametrizing algorithm runs this subroutine $N$ consecutive times. The
first run uses as input the chosen parameterization for $K_{e}$, the point
$\upsilon _{e}$, the chosen lift of its $\mathbb{R}/(2\pi \mathbb{Z})$
coordinate, and the path $\nu _{1}$. The second run of the subroutine uses
as input the resulting parameterization from \eqreft2{20} of the component of
$C_{0}- \Gamma $ that shares the edge label with $\nu _{2}$, the
second to last vertex on $\nu _{1}$ with the lift of its $\mathbb{R}/(2\pi
\mathbb{Z})$ parameter from \eqreft2{20}, and the path $\nu _{2}$. In general,
the $(j+1)'$st run of the subroutine uses the output from the $j$'th run of the
subroutine as it takes as input the parameterization in \eqreft2{20} for the
component of $C_{0}- \Gamma _{ }$ that shares the edge label with
$\nu _{j + 1}$, the second to last vertex on $\nu _{j}$, the lift
of the $\mathbb{R}/(2\pi \mathbb{Z})$ coordinate for $\nu _{j}$ that
is obtained from \eqreft2{20}, and the path $\nu _{j + 1}$. Because of \eqreft2{18},
the $N$'th run of the subroutine parametrizes $K_{e'}$ and gives a lift
of the $\mathbb{R}/(2\pi \mathbb{Z})$ parameter on its parametrizing cylinder.
This parametrization of $K_{e'}$ is deemed `canonical'.

As just described, the canonical parameterization of a component of
$C_{0}- \Gamma $ whose closure intersects $\Gamma _{o}$ depends on the following input data:

\itaubes{2.21}
\textsl{A distinguished incident edge to $o$}

\item
\textsl{A parameterization of the corresponding component in $C_{0}- \Gamma $.}

\item
\textsl{A choice of a distinguished vertex in the $\sigma =\theta _{o}$  boundary of the parametrizing cylinder.}

\item
\textsl{A lift to $R$ of the $R/(2\pi \mathbb{Z})$  coordinate of the corresponding point in the
$\sigma =\theta _{o}$  boundary of the parametrizing cylinder.}

\item
\textsl{A choice of a concatenating path set that obeys the constraints in \eqreft2{18}.}
\end{itemize}

The dependence on this input data has a role in subsequent parts of this
article and so warrants the discussion that occupies the remaining portions
of this subsection.

\step{Part 5}
The question on the table now is that of the dependence of a canonical
parametrization on the data in \eqreft2{21}. What follows describes the situation
for the various listed cases, starting at the top and descending. In these
descriptions, the edge $e$ denotes the original distinguished edge, and the
edge $e'$ denotes the edge that labels the component of $C_{0}- \Gamma $ that
is to receive the canonical parameterization. Use $\{\nu_{1}, \ldots , \nu _{N}\}$
to denote the original concatenating path set.

\substep{Case 1}
Suppose that a new distinguished edge,
$\hat{e}$, has been selected. What follows describes input data for the
parametrization algorithm, now run with the edge $\hat{e}$, that supplies a
canonical parametrization for $K_{e'}$ that agrees with the original
one. In short, the $\hat{e}$ input data is obtained as follows: A concatenating
path, $\hat {\nu }$, is chosen to first parametrize $K_{{\hat{e}}}$ given
the original parametrization of $K_{e}$. This parametrization of
$K_{{\hat{e}}}$ is then used as the input to the algorithm to obtain the
$\hat{e}$ version of the canonical parametrization of $K_{e'}$. The
concatenating path that is used for this $\hat{e}$ parametrization of
$K_{e'}$ runs backword along $\hat {\nu }$ to its starting vertex on
$\ell _{oe}$, and then forward along the concatenating path that is used
by the parametrization algorithm to obtain the original canonical
parametrization of $K_{e'}$.

Here are the details: Choose a concatenating path set, $\{\hat{\nu}_{1}, \ldots , \hat{\nu }_{\hat{N}} \}$
whose end vertex lies on
$K_{{\hat{e}}}$ and let $\hat {\upsilon }$ denote its ending vertex. Let
$\hat {\nu }$ denote this path set, and use $\hat {\nu }$ with the
parametrizing algorithm to give $K_{{\hat{e}}}$ its canonical
parameterization. Since the final arc from $\hat {\nu }_{\hat {N}} $ lies on
$\ell _{o{\hat{e}}}$, the vertex $\hat {\upsilon }$ has a canonical
lift to the $\sigma =\theta _{o}$ circle of the parametrizing
cylinder for $K_{{\hat{e}}}$. Use the latter for the new distinguished
vertex, and use the value given by the parametrizing algorithm for the
$\mathbb{R}$--valued lift of its $\mathbb{R}/(2\pi \mathbb{Z})$ parameter.

Now, take the following for the new concatenating path set: The first path
in this set consists solely of the final arc in $\hat {\nu }_{\hat {N}} $,
but traveled in the direction opposite to that used by $\hat {\nu }$. The
second constituent path in the new concatenating path set is $\hat {\nu}_{\hat {N}}$,
but traveled in the direction opposite to that by $\hat{\nu}$. The third is
$\hat{\nu }_{N - 1}$, also traveled in the reverse
direction from end to start. Continue in this vein using the paths in $\hat
{\nu }$ in reverse to return to the original distinguished vertex on $\ell
_{oe}$. However, instead of using the first constituent path, $\hat {\nu
}_{1}$, of $\hat {\nu }$ in reverse here, use instead the concatenation
that first takes $\hat {\nu }_{1}$ in reverse and then adds to its end the
first path, $\nu _{1}$, in the concatenating path set that lead from the
original distingushed vertex to $\ell _{oe'}$. This understood,
take the subsequent constituent path in the new concatenating path set to be
the second of the constituent paths, $\nu _{2}$, from the original
distinguished vertex to $\ell _{oe'}$. Then, complete the new path
set by adding in their given order the constituent paths $\nu _{3},
\ldots , \nu _{N}$ from the original concatenating path set.

It is left as an exercise to verify that the output parameterization for
$K_{e'}$ is not changed when this new data is used as input to the
parametrizing algorithm.

\substep{Case 2}
Suppose that the parameterization of
$K_{e}$ is changed by the action of a given integer pair $N$ as depicted in
\eqref{eq2.12} and \eqref{eq2.13}. In addition, suppose that the new $\mathbb{R}$--lift of the
$\mathbb{R}/(2\pi \mathbb{Z})$ coordinate is related to the old via \eqref{eq2.16}.
Repeating what is said in Part 2 at each run of the parametrizing
subroutine finds that the canonical parameterization of $K_{e'}$ is
also changed by the action of the integer pair $N$ via the $e'$ version of
\eqref{eq2.12} and \eqref{eq2.13}.

\substep{Case 3}
 This is the story in the case that the
distinguished vertex on the $\sigma =\theta _{o}$ boundary of the
parametrizing cylinder is changed. Agree to use the same parameterization
for $K_{e}$ as the original. Also, agree to take the $R$--valued lift of the
$\mathbb{R}/(2\pi \mathbb{Z})$ coordinate for the new vertex to be that given
by the lift of the original distinguished vertex. Take the new concatenating
path set as follows: The new version is $\{{\nu _{1}}', \nu _{2},
\ldots , \nu _{N}\}$ where $\nu _{1}$' is the concatenation that
travels from the new distinguished vertex to the original distinguished
vertex in $\ell _{oe}$, against the given orientation of $\ell _{oe}$,
and then proceeds outward from the old distinguished vertex along $\nu_{1}$.

The parametrizing algorithm finds no difference between the new and old
parametrizations of $K_{e'}$ when using this new input data.

\substep{Case 4} The change in just the $\mathbb{R}$--valued
lift of the $\mathbb{R}/(2\pi \mathbb{Z})$ coordinate of the distinguished
point is given by the $N \in   \mathbb{Z}\cdot Q_{e}$ version of Case 2.
In particular, if $z  \in   \mathbb{Z}$ and if the new lift is related to the
old via the $N = z Q_{e}$ version of \eqref{eq2.16}, then the parameterization of
$K_{e'}$ is changed by the action of $N = z Q_{e}$ via the $e'$ version
of \eqref{eq2.12} and \eqref{eq2.13}.

\substep{Case 5}
As is explained momentarily, the story for
any given case from the final point in \eqreft2{21} can be obtained from that for
the special case in which the new concatenating path set shares its
respective starting and ending arcs with the original set. Meanwhile, the
discussion for the latter case requires the introduction of some new notions
and is deferred for this reason to the upcoming Part 7 of this subsection.

To begin the discussion here, suppose that the original concatenating path
set is changed subject to \eqreft2{18} so as to obtain a new concatenating path
set whose starting arc differs from the original. Let $\{{\nu _{1}}',
{\nu _{2}}', \ldots , {\nu _{N'}}'\}$ denote the new set. The
resulting parametrization of $K_{e'}$ is not changed when the new
concatenating path set has the form $\{{\nu _{1}}'', {\nu _{2}}',
\ldots , {\nu _{N'}}'\}$ where ${\nu _{1}}''$ is obtained from
${\nu _{1}}'$ by adding two new arcs in $\ell _{oe}$ at its start, the
first leading out along the starting arc of $\nu _{1}$, and the second
returning along this same arc to the distinguished vertex. The third arc in
${\nu _{1}}''$ is the first arc in ${\nu _{1}}'$, the fourth arc in
${\nu _{1}}''$ is the second arc in ${\nu _{1}}'$, and so on.

Suppose next that the final arc in the new concatenating path set
$\{{\nu _{1}}', \ldots , {\nu _{N'}}'\}$ differs from that in the
original. In this case, the new concatenating path set can be modified by
increasing the number of its constituent paths by 2 so as not to change the
resulting parametrization of $K_{e'}$ but so as to have the desired
final arc. To elaborate, let $\tau    \in \ell _{oe'}$
denote the original final arc, viewed as a directed arc. Let $e''$ denote the
incident edge to $o$ that labels $\tau $ with $e'$. There are two cases to
consider: In the first, direction along $\tau $ and direction along the
final arc in ${\nu _{N'}}'$ give the same orientation to $\ell
_{oe'}$. In this case, add a path, $\nu _{N' + 1}
\subset   \ell _{oe'}$, that starts with the final arc in ${\nu _{N'}}'$, proceeds along it in the direction used by ${\nu _{N'}}'$, and continues until it first hits a vertex of $\tau $. By
assumption, this must be the starting vertex of $\tau $. This understood,
continue $\nu _{N' + 1}$ by traversing $\tau $ to its end. The
final path, $\nu _{N' + 2}$, consists solely of $\tau $, but viewed
as in $\ell _{oe''}$.

In the second case, the arc $\tau $ and the final arc on ${\nu _{N'}}'$ define opposite orientations for $\ell _{oe'}$. In
this case, $\nu _{N' + 1}$ starts with the final arc in ${\nu _{N'}}'$ and continues along $\ell _{oe'}$ until the it
first hits a vertex on $\tau $. In this case, the vertex here is the end
vertex. Continue $\nu _{N' + 1}$ by traversing $\tau $ backwards to
its start. Take $\nu _{N' + 2}$ to be the path that starts at the
end vertex of $\tau $, traverses $\tau $ to its start and then reverses
direction to retrace $\tau $ to its end again. Label this path with the edge
$e''$. A check of the parametrizing algorithm reveals that the canonical
parameterization of $K_{e'}$ as defined using $\{{\nu _{1}}',
\ldots , {\nu _{N'}}', \nu _{N' + 1}, \nu _{N'
+ 2}\}$ is identical to that obtained using $\{{\nu _{1}}', \ldots ,
{\nu _{N'}}'\}$.

\step{Part 6}
The remainder of the story for the final point in \eqreft2{21} requires the
introduction of a certain `blow up' of the graph $\underline {\Gamma}_{o}$. This blow up is a graph,
${\underline {\Gamma }^{*}}_{o}$,
with labeled and oriented edges that comes equipped with a canonical `blow
down' map to $\underline {\Gamma }_{o}$ that maps certain edges to
vertices of $\underline {\Gamma }_{o}$. Note that Property 4 from Part 3
is used implicitly in the the constructions that appear in the next few
paragraphs.

To start the definition of ${\underline {\Gamma }^{*}}_{o}$, let
$\upsilon $ denote a vertex in $\underline {\Gamma }_{o}$ with non-zero
integer assignment. Then $\upsilon $ labels an oriented, circular subgraph
$\ell ^{*\upsilon }   \subset {\underline {\Gamma}^{*}}_{o}$ whose vertices are in 1--1 correspondence with the incident
half-arcs to $\upsilon $ in $\underline {\Gamma }_{o}$. For the purpose of
describing this correspondence, keep in mind that each vertex in $\underline {\Gamma }_{o}$ has an even number
of incident half-arcs, half oriented
towards the vertex and half oriented in the outward direction. The
correspondence between the vertices of $\ell ^{*\upsilon }$ and the
incident half-arcs to $\upsilon $ is defined so that the following is true:
Let $\gamma $ denote an incident half-arc to $\upsilon $ with edge labels $e$
$ \in E_{ - }$ and $e' \in  E_{ + }$. If $\gamma $ is an inbound arc,
then $\gamma $ labels a vertex on $\ell ^{*\upsilon }$ and the
subsequent vertex in the oriented direction on $\ell ^{*\upsilon }$
is labeled by the arc in $\ell _{oe}$ that follows $\gamma $. On the
other hand, if $\gamma $ is an outbound arc, then the subsequent vertex on
$\ell ^{*\upsilon }$ is labeled by the arc in $\ell _{oe'}$
that precedes $\gamma $.

This correspondence has the following consequence: The vertices that are met
upon a circumnavigation of $\ell ^{*\upsilon }$ correspond
alternately to inbound and outbound incident half-arcs to the vertex
$\upsilon $. Moreover, if a given vertex corresponds to an inbound arc in
$\underline {\Gamma }_{o}$ with the edge pair label $(e, e')$, then the
subsequent vertex is the subsequent arc in $\ell _{oe}$ and the previous
vertex is the subsequent arc in $\ell _{oe'}$. On the other hand,
if the vertex corresponds to an outbound arc with label $(e, e')$, then the
subsequent vertex on $\ell ^{*\upsilon }$ corresponds to the
previous arc on $\ell _{oe'}$ and the previous vertex corresponds
to the previous arc on $\ell _{oe}$.

With the definition set, then each arc in $\ell ^{*\upsilon }$ is
labeled by an ordered pair whose first element is $\upsilon $ and whose
second is an incident edge to $o$. To elaborate, an incident edge $e \in
E_{-}$ labels the arc when its starting vertex (as defined by its given
orientation) corresponds to an inbound arc to $\upsilon $ in $\underline{\Gamma }_{o}$ whose $E_{ - }$ edge
label is $e$. An edge $e' \in  E_{ +}$ labels the arc when its starting vertex corresponds to an outbound arc in
$\underline {\Gamma }_{o}$ whose $E_{ + }$ label is $e'$.

The remaining arcs in ${\underline {\Gamma}^{*}}_{o}$ are labeled
by pairs of incident edges to $o$, and this set enjoys a 1--1 correspondence
with the arcs in $\underline {\Gamma }_{o}$ that respects their incident
edge labels.

All of this is designed so that the map from ${\underline {\Gamma }^{*}}_{o}$ to
$\underline {\Gamma }_{o}$ that collapses the circles
$\{\ell ^{*\upsilon }\}$ has the following properties: First, it
is an orientation preserving map that sends vertices to vertices. Second,
the map is 1--1 on neighborhoods of vertices with zero integer assignment. On
the other hand, the inverse image of any given vertex in $\underline {\Gamma}_{o}$ with non-zero
integer assingment is the circular subgraph that
carries its label. Third, the inverse image of any arc in $\underline {\Gamma }_{o}$ is
an arc that bears its same edge pair label. Finally,
each $\ell _{oe}$ in $\underline {\Gamma }_{o}$ has a canonical inverse
image as an embedded circular subgraph of ${\underline {\Gamma }^{*}}_{o}$. The latter is
obtained from the inverse image of the arcs that
comprise $\ell _{oe}$ by adding the arcs from the collection in $ \cup_{\upsilon }\ell^{ * \upsilon}$ whose
label contains the edge $e$. The inverse image of $\ell _{oe}$ in ${\underline {\Gamma}^{*}}_{o}$
is denoted subsequently as $\ell ^{*}_{oe}$.

With their orientations as specified above, each edge labeled loop ${\ell^{*}}_{o(\cdot )}$ defines a
homology class in $H_{1}({\underline {\Gamma}^{*}}_{o}$; $\mathbb{Z})$ as does each vertex labeled $\ell^{*(\cdot )}$.
In this regard, the collection of these classes,$\{[\ell^{*}_{oe}], [\ell ^{*\upsilon
}]\}$, generate the integral homology of ${\underline {\Gamma}^{*}}_{o}$ subject to the following single constraint:
\begin{equation}\label{eq2.22}
\sum_{e \in E_ + } [\ell ^{*}_{oe}] - \sum _{e \in E_ - }
[\ell ^{*}_{oe}] = \sum _{\upsilon }[\ell ^{*\upsilon }].
\end{equation}
This last identity provides a canonical class in the $\mathbb{Z}  \times\mathbb{Z}$ valued
cohomology of ${\underline {\Gamma}^{*}}_{o}$.
Indeed, this class, $\phi _{o}$, is defined so as to send any given $[\ell
^{*}_{oe}]$ to $Q_{e}$. Meanwhile, $\phi _{o}$ sends the class
$[\ell ^{*\upsilon }]$ to the pair $m_{\upsilon } P_{o}$, where $m_{\upsilon }$ is the integer
weight assigned to the vertex $\upsilon $ and $P_{o}$ denotes here the relatively prime integer pair that defines
$\theta_{o}$ via \eqref{eq1.8}.

Here is one way to view ${\underline {\Gamma }^{*}}_{o}$: The result
of attaching one boundary circle of a closed cylinder to each circle from
the collection $\{\ell ^{*}_{oe}\} \cup \{\ell
^{*\upsilon }\}$ results in a space whose interior is a
multipunctured sphere that is homeomorphic to $\Gamma _{o}$'s component of
every small but positive $\delta $ version of the $| \theta  -\theta _{o}|  < \delta$ portion of $C_{0}$.

\step{Part 7}
This part of the subsection describes how the canonical parametrization of
$K_{e'}$ changes when the given concatenating path is changed with no
change to the first traversed arc and no change either to the last arc or to
the direction of traverse on the last arc.

To begin the story, note first that when $\{\nu _{1}, \ldots , \nu_{N}\}$ is a concatenating path set,
then each if its elements
can be assigned a canonical inverse image in ${\underline {\Gamma }^{*}}_{o}$.
The inverse image of a given $\nu _{k}$ is denoted here as
${\nu ^{*}}_{k}$. The latter path lies in the version of ${\ell^{*}}_{o(\cdot )}$ that shares $\nu _{k}$'s incident edge
label, it starts with the inverse image of $\nu _{k}$'s starting arc and
ends with the inverse image of $\nu _{k}$'s ending arc. The next point to
make is that the union of the collection $\{{\nu ^{*}}_{k}\}_{1 \le k \le N}$ of such lifts defines a single directed
path whose starting point is a vertex on the version of $\ell ^{*(\cdot )}$ that projects to the starting vertex on
$\nu _{1}$ and whose end point is on the version of $\ell ^{*(\cdot )}$ that projects to
the ending vertex of $\nu _{N}$. In this regard, the various subpaths that
comprise this union are traversed in their given order in their given
direction. Let $\nu ^{*}$ denote the latter path.

Now, if $\nu $ and $\nu '$ are concatenating path sets that obey \eqreft2{18}, and
if they share both starting arcs and ending arcs, then the corresponding
lifts, $\nu ^{*}$ and $\nu'^{*}$, share the same starting
vertex and also share the same ending vertex. Thus, a closed loop, $\mu^*$, is defined in
${\underline {\Gamma}^{*}}_{o}$ by
traveling first on $\nu'^{*}$ from start to finish, and then
traveling on $\nu ^{*}$ in the `wrong' direction, thus from its
ending vertex back to its starting vertex.

The following lemma now summarizes the relation between the $\nu $ and $\nu'$
versions of the canonical parameterization for $K_{e'}$:

\begin{lemma}\label{lem:2.3}
The action of the integer pair $-\phi _{o}([\mu ^{*}])  \in   \mathbb{Z}\times   \mathbb{Z}$
on the canonical parameterization defined by $\nu $  gives the canonical parameterization defined by $\nu '$.
\end{lemma}

The remainder of this subsection is occupied with the

\begin{proof}[Proof of \fullref{lem:2.3}]
The proof is given in six steps.

\substep{Step 1}
Write $\nu ' = \{{\nu _{1}}', \ldots, {\nu _{N'}}'\}$ and let $\gamma ' \subset   \ell_{oe'}$
denote the directed path that first crosses the final arc
on ${\nu _{N'}}'$ from start to finish in the direction used by $\nu'$, and then
recrosses this arc in the opposite direction. Label $\gamma '$
with the incident edge $e'$. With $\gamma '$ understood, introduce the
concatenated path set given by the following ordered set of constituent
paths:
\begin{equation}\label{eq2.23}
\bigl\{{\nu _{1}}', \ldots , {\nu _{N'}}', \gamma ', {\nu_{N}}^{- 1}, \ldots , {\nu _{2}}^{- 1}, \nu _{1}
\circ \nu _{1}^{- 1}, \nu _{2}, \ldots ,\nu _{N}\bigr\}.
\end{equation}
Here, $\nu _{1} \circ \nu _{1}^{- 1}$ denotes the path in
$\ell _{oe}$ obtained by first traveling $\nu _{1}$ in reverse from its
end to its start, and then returning on $\nu _{1}$ back to its end vertex.
The canonical parameterization that the path set in \eqref{eq2.23} defines for
$K_{e'}$ is the same as that assigned when using the path $\nu '$.

\substep{Step 2}
Let $\tau $ denote the first arc in $\nu_{1}$, viewed as a directed arc that starts at the distinguished vertex
$\nu $. Let $\hat{e}$ denote the edge that labels the arc $\tau $ with $e$. Now,
let $\gamma $ denote the following directed path in $\ell_{o{\hat{e}}}$: Start at the ending vertex on $\tau $ and
traverse this arc in reverse so as to return to the distinguished vertex.
Then, reverse direction and retrace $\tau $ to return its end vertex. Thus,
$\gamma =\tau  \circ \tau ^{- 1}$. Label $\gamma $ with the edge $\hat{e}$.

This labeled, directed path $\gamma $ appears in the concatenating path set
whose ordered constituent paths are
\begin{equation}\label{eq2.24}
\bigl\{{\nu _{1}}', \ldots , {\nu _{N'}}', \gamma ', {\nu _{N}^{- 1}}, \ldots , {\nu _{2}}^{- 1}, {\nu _{1}}^{ - 1}, \gamma \bigr\}.
\end{equation}
Running the parametrizing algorithm on the set in \eqref{eq2.24} provides as
output a new pair of parameterization of $K_{e}$ and lift to $\mathbb{R}$ of
the $\mathbb{R}/2\pi \mathbb{Z}$ coordinate of the distinguished point on the
$\sigma =\theta _{o}$ circle in the parametrizing cylinder. The
latter pair is obtained from the former by the action of some integer pair
as depicted in \eqref{eq2.12}, \eqref{eq2.13} and \eqref{eq2.16}.

This integer pair is relevant because the conclusions of Step 1 together
with the Case 3 story for \eqreft2{21} from Part 5 imply that the canonical
parametrization of $K_{e'}$ that is obtained using the concatenating
path $\nu '$ is obtained from the canonical parametrization that is obtained
using the concatenating path $\nu $ by the action of this same integer pair
via the $e'$ version of \eqref{eq2.12} and \eqref{eq2.13}.

\substep{Step 3}
Let $\mathcal{C}$ denote the collection of
concatenating path sets that obey the $e' = e$ version of \eqreft2{18} and have
$\tau$ as both first and last arc. For example the
set in \eqref{eq2.24} lies in $\mathcal{C}$. Note that each $\mu    \in   \mathcal{C}$
defines an integer pair as just described in Step 2, this denoted below by
$N_{*}(\mu )$. To elaborate, any given $\mu    \in   \mathcal{C}$ can
be used with the parametrizing algorithm to define a new parameterization
for $K_{e}$ and a corresponding lift of the $\mathbb{R}/(2\pi \mathbb{Z})$
coordinate of the distinguished point on the $\sigma =\theta _{o}$
circle of the parametrizing domain. The latter pair is obtained from the
original by the action of $N_{*}(\mu )$ as depicted in \eqref{eq2.12},
\eqref{eq2.13} and \eqref{eq2.16}.

The assignment $\mu \to  N_{*}(\mu )$ defines a map from $\mathcal{C}$ to $\mathbb{Z}  \times   \mathbb{Z}$ and
the task is to identify this map.
It is useful in this regard to introduce various structures of an algebraic
nature to $\mathcal{C}$ of which the first is the product operation, $\wp:
\mathcal{C}  \times   \mathcal{C}   \to   \mathcal{C}$, given by
\begin{equation}\label{eq2.25}
\wp (\mu ', \mu ) = \bigl\{\mu _{1}', \ldots, \mu _{N'}',
\mu _{1}, \ldots , \mu _{N}\bigr\}.
\end{equation}
The conclusions from Step 2 imply that
\begin{equation}\label{eq2.26}
N_{*}(\wp (\mu , \mu ')) = N_{*}(\mu ) + N_{*}(\mu').
\end{equation}

The set $\mathcal{C}   \to   \mathcal{C}$ also admits a self map, $t$, that changes
the sign of $N_{*}(\cdot)$. To elaborate, let $\gamma _{0}$
denote the directed path that starts at the distinguished vertex, runs along
$\tau $ to its end, and the returns to the distinguished vertex by reversing
direction on $\tau $. Label $\gamma _{0}$ with the edge $e$. Now, let $\mu = \{\mu _{1}, \ldots,\mu _{N}\}\in
\mathcal{C}$ denote a given data set and set
\begin{equation}\label{eq2.27}
t(\mu ) = \bigl\{\gamma _{0}, {\mu _{N}}^{- 1}, {\mu _{N -1}}^{- 1}, \ldots , {\mu _{1}}^{ - 1},
\gamma \bigr\},
\end{equation}
The parametrizing algorithm finds that $N_{*}(t(\mu ))= -N_{*}(\mu )$.

By virtue of \eqref{eq2.26} and \eqref{eq2.27}, the set $\mathcal{C}$ contains elements with
$N_{*}(\cdot ) = (0, 0)$. Perhaps the simplest such element is the
concatenating path set $\{\gamma , \gamma '\}$ with $\gamma$ and
$\gamma '$ defined as in \eqref{eq2.27}.

The next point to make is that the assignment $N_{*}\co  \mathcal{C}   \to
  \mathbb{Z}  \times   \mathbb{Z}$ is insensitive to certain modifications of a
given concatenating path set. To elaborate, suppose again that $\mu =
\{\mu _{1}, \ldots ,\mu _{N}\} \in   \mathcal{C}$. Suppose, in
addition that $j < N$ has been specified. Let $e_{j}$ denote the version of
$\ell _{o(\cdot )}$ that contains $\mu _{j}$ and let $\tau _{1}$
denote the final arc in $\mu _{j}$ viewed here as a directed arc. Use
$\hat{e}_{j}$ to denote the incident edge to $o$ that is partnered with $e_{j}$
in the label of $\tau _{1}$. Let $\{\nu _{1}, \ldots ,\nu_{N'}\}$ now denote a concatenating path set with the property
that the first arc in $\nu _{1}$ has label $\hat{e}_{j}$ and directed first
arc equal to $\tau _{1}$. Let $\tau _{N}$ denote the final arc in $\nu_{N}$, also viewed
as a directed arc. With the concatenations ${\tau_{N}}^{- 1} \circ \tau _{N}$ and $\tau _{1} \circ {\tau_{1}}^{- 1}$
suitably labeled by edges, the set
\begin{equation}\label{eq2.28}
\bigl\{\mu _{1}, ... , \mu _{j}, \nu _{1},... ,\nu _{N},
({\tau _{N}}^{- 1} \circ \tau _{N}),
{\nu _{N}^{- 1}}, ... , \nu _{1}^{- 1}\!, (\tau _{1} \circ {\tau _{1}^{- 1}}), \mu _{j + 1}, ... ,\mu _{N}\bigr\}.
\end{equation}
is an element in $\mathcal{C}$. Running the parametrizing algorithm find that
$N_{*}$ has the same value on $\mu $ as it has on the set in \eqref{eq2.28}.

\substep{Step 4}
Let $\mu _{e}   \subset   \ell_{oe}$ denote the path that starts at the vertex $\upsilon $, runs out
along $\tau $ and continues once around $\ell _{oe}$ so as to end after
crossing $\tau $ a second time. The ordered pair $\{\mu _{e}, \gamma\}$ defines a concatenating path set in
$\mathcal{C}$. The considerations in Case 4 of the preceding part of this subsection imply that $N_{*}$ has
value $\pm Q_{e}$ on this element, with the $-$ sign occurring if and only
if travel along $\tau $ from $\upsilon $ defines the oriented direction in
$\ell_{oe}$.

As will now be explained, every other version of $Q_{(\cdot )}$ whose
label is an incident edge to $o$ is a value of $N_{*}$ in $c$. To
see how this comes about, let $e' \ne e$ denote a given incident edge to $o$
and let $\{\nu _{1}, \ldots ,\nu _{N}\}$ denote a
concatenating path set with the following properties: First, $\nu _{1}$ is
labeled by $e$, it starts at $\upsilon $ and its first arc is $\tau $.
Meanwhile, $\nu _{N}$ is not labeled by $e'$ but its last arc lies in $\ell_{oe'}$
and the direction along $\nu _{N}$ on this last arc
defines the given orientation of $\ell _{oe'}$. Let $\mu_{e'}$ denote the path in $\ell _{oe'}$ that starts with
this last arc in $\nu _{N}$ continues once around $\ell _{oe'}$
to the end of this last arc, and then reverses direction to retrace this
last arc back to its starting vertex. Label $\mu _{e'}$ with the
edge $e'$.

The ordered set $\{\nu _{1}, \ldots ,\nu _{N}, \mu
_{e'}, {\nu _{N}^{- 1}}, \ldots ,{\nu _{1}}^{ - 1},\gamma \}$ is an element in $\mathcal{C}$, and a
run of the parametrizing
algorithm finds that $N_{*}$ has the desired value $-Q_{e'}$ on
this element.

\substep{Step 5}
The results from the preceding steps suggest that the value of $N_{*}$ on any given
$\mu    \in   \mathcal{C}$ depends only on the homology class of the path in $\underline {\Gamma}_{o}$ that
is defined by the union of $\mu $'s constituent paths.
However, this conclusion is wrong when $\underline {\Gamma }_{o}$ contains
vertices that correspond to ends of $C_{0}$.

That such is the case can be seen most readily when the vertex at the end of
$\tau $ is a bivalent vertex. Let $\upsilon '$ denote the latter vertex and
let $\tau _{1}$ denote the arc in $\ell _{oe}$ that shares $\upsilon '$
with $\tau $. Thus, both $\tau $ and $\tau _{1}$ are labeled by the edge
pair that consists of $e$ and $\hat{e}$. This understood, then
\begin{equation}\label{eq2.29}
\bigl\{\tau _{1} \circ \tau , \tau  \circ \tau ^{- 1}
\circ {\tau _{1}^{- 1}}\bigr\}
\end{equation}
is a concatenating path set in $\mathcal{C}$ if the first constituent path is
labeled by $e$ and the second by $\hat{e}$. As is explained next, a run of the
parametrizing algorithm on this path finds that $N_{*}$ assigns it either $m_{\upsilon '} P_{o}$ or
$-m_{\upsilon'} P_{o}$. Here, $m_{\upsilon '}$ is the integer weight
accorded $\upsilon '$ as a vertex in $\underline {\Gamma }_{o}$, and
$P_{o}$ is the relatively prime integer pair that is defined by $\theta_{o}$ via \eqref{eq1.8}.
Meanwhile, the + sign appears if and only if either $e\in  E_{ - }$ and travel from $\upsilon $ along $\tau $ is
in the oriented direction in $\ell _{oe}$, or else $e \in  E_{ + }$ and travel from
$\upsilon $ along $\tau $ goes against the orientation from $\ell _{oe}$.

To see why the parametrizing algorithm must give $\pm m_{\upsilon'}P_{o}$, it
is necessary to go back to \eqreft25 for the cases of $e$
and $\hat{e}$. Let $R$ be some very large number, chosen so that $|s| $
has value $R$ on the end $E\subset C_{0}$ that is labeled by the vertex
$\upsilon '$. Write $P_{o} = (p_{o}, {p_{o}}')$ and keep in mind that the
respective integrals of $\frac{1 }{ {2\pi }} dt$ and $\frac{1}{ {2\pi }}d\varphi $ about the $|s| = R$
slice of $E$ are $|m_{\upsilon '}| p_{o}$ and $| m_{\upsilon'}| p_{o}$ granted that the latter are oriented by the
pull-back of the 1--form $-\alpha $ with $\alpha $ depicted in \eqref{eq1.3}.

As can be seen from \eqref{eq2.4}, there are two points of the $|s| = R$
circle in $E$ that lie on $\Gamma _{o}$, and one will lie in $\tau $ and the
other in $\tau _{1}$ when $R$ is large. Label these points $z$ and $z'$. Given
that $K_{{\hat{e}}}$ is parametrized using the single element
concatenating path set $\{\tau _{1} \circ \tau \}$, it then
follows that the respective values of the $\mathbb{R}$--valued functions $\hat{v}_{e}$
and $\hat {v}_{e'}$ are very close at the respective
points in the $K_{e}$ and $K_{{\hat{e}}}$ parametrizing cylinders that
map to $z$. They also have very similar values at the respective points that
map to $z'$. The values of $w_{e}$ and $w_{e'}$ at the respective
points in the $K_{e}$ and $K_{e'}$ parametrizing cylinders that map
to $z'$ are also nearly equal since these points are used in \eqref{eq2.14} and \eqref{eq2.15}
to define the parameterization of $K_{e'}$.

However, given these near equalities, and given that the pair $(\frac{1}{ {2\pi }} dt, \frac{1 }{ {2\pi }}d\varphi )$
has a non-zero integral in $\mathbb{Z}  \times   \mathbb{Z}$ around the $|s| = R$
circle, the form of \eqreft25 precludes nearly equal values for $w_{e}$ and
$w_{e'}$ at the respective points in their parametrizing cylinders
that map to $z$. Indeed, the difference in their values is exactly accounted
for by the $\mathbb{Z}  \times   \mathbb{Z}$ action of the appropriately signed
version of $m_{\upsilon '} P_{o}$ on the given pair of $K_{e}$
parameterization and lift to $R$ of the $R/(2\pi Z)$ coordinate of $\upsilon$.

The analysis just done for this simple case can be repeated in the case that
the vertex at the end of $\tau $ has some $2(k+1)$ incident half-arcs. To
elaborate, each such arc can be directed and then the set suitably labeled
as $\{\tau , \tau _{1}, \ldots , \tau _{k}\}$ so that
\begin{equation}\label{eq2.30}
\bigl\{\tau _{1} \circ \tau , \tau _{2} \circ {\tau_{1}}^{- 1}, \ldots , \tau _{k} \circ {\tau _{k -1}}^{- 1}, \tau  \circ \tau ^{- 1} \circ
{\tau_{k}}^{- 1}\bigr\}
\end{equation}
defines an element in $\mathcal{C}$. For example, $\tau _{1}$ is the incident
half-arc to $\upsilon '$ in $\ell _{oe}$ that follows $\tau $ when
traveling along $\tau $ in $\ell _{oe}$. Meanwhile, $\tau _{2}$ is the
incident half-arc that follows $\tau _{1}^{ - 1}$ when traveling the
latter in the indicated direction on the version of $\ell _{o(\cdot )}$ whose edge shares with
$e$ the label of $\tau _{1}$. The arc $\tau_{3}$ is defined by the analogous rule when $\tau _{2}$ is used in lieu
of $\tau _{1}$; in general, $\tau _{k}$ for $k > 3$ is defined by
successive applications of this same rule using $\tau _{k - 1}$ in lieu of $\tau _{1}$.

The parametrizing algorithm can be run using \eqref{eq2.30} and then \eqreft25 can be
employed to prove that $N_{*}$ again has value $\pm m_{\upsilon'}P_{o}$ with the $+$ sign appearing under
the same circumstances as in case when $\upsilon '$ is bivalent.

The story is much the same for a null-homotopic path in $\underline {\Gamma}_{o}$ of the following sort: Let
$\upsilon '$ denote any given arc in $\underline {\Gamma }_{o}$ and let $2(k+1)$ denote the number of its
incident half-arcs. Direct and label these as $\{\tau ',\tau _{1},
\ldots , \tau _{k}\}$ by mimicking the scheme used in \eqref{eq2.30} with $\tau'$ replacing $\tau $.
Now, let $\{\nu _{1}, \ldots , \nu _{N}\}$ denote a concatenating path set with the following properties: The path
$\nu_{1}$ starts with $\tau $ and is labeled by $e$. Meanwhile, the final arc on
$\nu _{N}$ is $\tau '$, but the edge label of $\nu _{N}$ is not that of
both $\tau '$ and $\tau _{1}$. Given these properties, the ordered set
\begin{equation}\label{eq2.31}
\bigl\{\nu _{1}, \ldots , \nu _{N}, \tau _{1} \circ \tau ',
\tau _{2} \circ {\tau _{1}}^{- 1}, \ldots , \tau _{k}
\circ {\tau _{k - 1}}^{- 1}, \tau ^{- 1} \circ {\tau_{k}}^{- 1}, {\nu _{N}}^{- 1}, \ldots , {\nu_{1}}^{- 1}, \gamma \bigr\}
\end{equation}
defines an element in $\mathcal{C}$ on which $N_{*}$ has value $\pm m_{\upsilon }P_{o}$ with the
$\pm $ determined as before but for the replacement of $\tau '$ by $\tau $.

\substep{Step 6}
As explained previously, a concatenating
path set defines a canonical path in the blow up graph ${\underline {\Gamma}^{*}}_{o}$. Indeed, each constituent
subpath is specified as a concatenated union of arcs that all lie in a single version of $\ell_{o(\cdot )}$.
Thus, each such path has a canonical lift to $\ell ^{*}_{e}   \subset {\underline {\Gamma }^{*}}_{o}$ to give a
directed path of concatenated arcs. Moreover, consecutive paths from the
given path set lift to paths in ${\underline {\Gamma }^{*}}_{o}$ so
that their union defines a directed path with travel starting on the lift of
first path and continuing on that of the second.

When $\mu    \in   \mathcal{C}$, let $\mu ^{*}$ denote its lift. Let
$\tau ^{*}   \subset {\underline {\Gamma }^{*}}_{o}$
denote the inverse image of the first arc in $\mu $, and thus the first arc
in $\mu ^{*}$. This is also the final arc in $\mu ^{*}$.
Thus, travel on $\mu ^{*}$ up to but not including the final
appearance of $\tau ^{*}$ defines a closed loop, $\hat {\mu },
\subset {\underline {\Gamma }^{*}}_{o}$, that starts and ends at
the starting vertex of $\tau ^{*}$.

By virtue of what has been said in the previous steps, the value of $N_{*}$ on $\mu $ is minus the value
of the homomorphism $\phi _{o}$ on $\hat{\mu }$, and this last conclusion implies the assertion of \fullref{lem:2.3}.
\end{proof}

\setcounter{section}{2}
\setcounter{equation}{0}
\section{The map from $\mathcal{M}$ to $O_{T}$}\label{sec:3}

Suppose for the moment that $\hat{A}$ is an asymptotic data set
with $N_{ -}+\hat {N}+\text{\c{c}}_{ - }+\text{\c{c}}_{ + } = 2$.
Granted that $\hat{A}$ also obeys the conditions in
\eqreft1{16} that guarantee a non-empty version of
$\mathcal{M}_{\hat{A}}$, then \fullref{thm:1.2} asserts a
diffeomorphism between $\mathcal{M}_{\hat{A}}$ and the product of
$\mathbb{R}$ and a space denoted by $\hat{O}^{\hat{A}}/\Aut
^{\hat{A}}$. Here, $\hat{O}^{\hat{A}}$ is the part of
$O^{\hat{A}}$ from \eqref{eq1.21} where $\Aut ^{\hat{A}}$ acts
freely. The purpose of this subsection is to describe one such
map from $\mathcal{M}_{\hat{A}}$ to $\mathbb{R}\times
\hat{O}^{\hat{A}}/\Aut ^{\hat{A}}$. The map given here is the
prototype for a suite of such maps that are used in subsequent
sections to describe other components of the space moduli space
of multiply punctured, pseudoholomorphic spheres.

\subsection{The more general context}\label{sec:3a}

If $\hat{A}$ is an asymptotic data set, reintroduce from
\fullref{sec:1a} the set $\Lambda _{\hat{A}}$ of angles in
$[0, \pi ]$. Unless noted to the contrary, restrict attention in the
remainder of this section to asymptotic data sets with the following
properties:

\qtaubes{3.1}
\textsl{There is a unique element in $\hat{A}$ that supplies the minimal angle in
$\Lambda _{\hat{A}}$, and there is also a unique element in $\hat{A}$
that supplies the maximal angle in $\Lambda _{\hat{A}}$.}
\endqtaubes

To make this explicit, note that $\hat{A}$ has a unique element that gives
$\Lambda _{\hat{A}}$'s minimal angle if and only if the one of
following is true:

\itaubes{3.2}
\textsl{There are no $(1,\ldots)$ elements in $\hat{A}$, the integer
$\text{\rm\c{c}}_{ + }$ is zero, and there is a unique
$(0,-,\ldots)$ element whose integer pair gives the minimal angle
in $\Lambda _{\hat{A}}$ via \eqref{eq1.8}.}

\item
\textsl{There are no $(1,\ldots)$ elements in $\hat{A}$ and $\text{\rm\c{c}}_{ + } > 0$.}

\item
\textsl{There is a unique $(1,\ldots)$ element in $\hat{A}$ and $\text{\rm\c{c}}_{ + } = 0$.}
\end{itemize}

An analogous set of conditions must hold when there is a unique element in
$\hat{A}$ that supplies the maximal angle to $\Lambda _{\hat{A}}$. Note
that an asymptotic data set with $N_{ - }+\hat {N}+\text{\c{c}}_{ - }+\text{\c{c}}_{ + } = 2$
automatically obeys \eqreft31.

As is explained next, the data set $\hat{A}$ can be used to define a
linear graph, $T^{\hat{A}}$, much like that introduced in
\fullref{sec:1b} in the case $N_{ - }+\hat {N}+\text{\c{c}}_{
- }+\text{\c{c}}_{ + } = 2$. As in \fullref{sec:1b}, the graph
$T^{\hat{A}}$ is viewed as a closed subinterval in $[0, \pi]$ whose
vertices include its endpoints; the edges of $T^{\hat{A}}$ are the
closed intervals that run from one vertex to the next. The vertices
of $T^{\hat{A}}$ are in 1--1 correspondence with the angles in
$\Lambda _{\hat{A}}$ and this correspondence is such that the angle
of any vertex in $\Lambda _{\hat{A}}$ is also its angle in $[0,
\pi]$.

The edges of $T^{\hat{A}}$ are labeled by integer pairs that are
determined using the rules that follow. The notation is such that when $e$
denotes an edge, then $Q_{e} = (q_{e}, {q_{e}}')$ denotes its integer pair
label.

\itaubes{3.3}
\textsl{Let $e$ denote the edge that contains the minimum angle vertex.}

\begin{enumerate}\leftskip 25pt
\item[\rm(a)]
\textsl{ If this angle is positive, then $Q_{e}$ is the integer pair from the
 $(0,-,\ldots)$ element in $\hat{A}$ that supplies this minimal angle to
 $\Lambda _{\hat{A}}$.}

\item[\rm(b)]
\textsl{If the miminal angle in $\Lambda _{\hat{A}}$ is zero and $\text{\rm\c{c}}_{ + } > 0$,
then $Q_{e} = (0, -\text{\rm\c{c}}_{ +})$.}

\item[\rm(c)]
\textsl{If the minimal angle in $\Lambda _{\hat{A}}$ is zero and
$\text{\rm\c{c}}_{ + } = 0$, then $Q_{e} = -\varepsilon P$ given that the element from
$\hat{A}$ that supplies this angle to $\Lambda _{\hat{A}}$ has the form
$(1, \varepsilon , P)$.}
\end{enumerate}
\item
\textsl{Let $o$ denote a bivalent vertex in $T^{\hat{A}}$ and let $e$ and $e'$
denote its incident edges with $e$ connecting $o$ to a vertex with angle less
than that of $o$. Then $Q_{e} = Q_{e' } + P_{o}$ where $P_{o}$ is obtained by
subtracting the sum of the integer pairs from the $(0,-,\ldots)$ elements from
$\hat{A}$ that supply $o$'s angle to $\Lambda _{\hat{A}}$ from the some of the
integer pairs from the $(0,+,\ldots)$ elements from $\hat{A}$ that supply $o$'s angle to
$\Lambda _{\hat{A}}$.}
\end{itemize}

In \cite[Theorem~1.3]{T3} asserts that $\mathcal{M}_{\hat{A}}$ is
non-empty if and only if the conditions in \eqreft1{16} hold for
the current version of $T^{\hat{A}}$. Assume for the remainder of
this subsection that $\hat{A}$ obeys \eqreft31 and that the
conditions of \eqreft1{16} hold for $T^{\hat{A}}$. What follows
describes a generalization of the space $O^{\hat{A}}$ that appears
in \eqref{eq1.21}. This space is then used to parametrize a certain
subspace in $\mathcal{M}_{\hat{A}}$.

To begin the description of $O^{\hat{A}}$, introduce $\hat{A}_{*} \subset \hat{A}$
to denote the subset of 4--tuples whose integer pairs
define the non-extremal angles in $\Lambda _{\hat{A}}$. Associate
to each $u  \in \hat{A}_{* }$ a copy, $\mathbb{R}_{u}$ of the affine
line and let
\begin{equation}\label{eq3.4}
\mathbb{R}^{\hat{A}}  \subset  \Maps(\hat{A}_{* }; \mathbb{R})
\end{equation}
denote the subspace of points where distinct $u$ and $u'$ from
$\hat{A}_{* }$ have distinct images in $\mathbb{R}/(2\pi
\mathbb{Z})$ in the case that their integer pairs defined the
same angle via \eqref{eq1.8}. Let $\mathbb{R}_{ - }$ denote an
auxiliary copy of $\mathbb{R}$.

As is explained next, there is an action of $\Maps(\hat{A}_{* }; \mathbb{Z})$ on
\begin{equation}\label{eq3.5}
\mathbb{R}_{ - }  \times \Maps(\hat{A}_{*}; \mathbb{R})
\end{equation}
that preserves $\mathbb{R}_{ - }  \times  \mathbb{R}^{\hat{A}}$.
This action is trivial on $\mathbb{R}_{ - }$. To describe the
action on the second factor, let $u  \in \hat{A}_{* }$ and let
$z_{u}$ denote the corresponding generator of $\Maps(\hat{A}_{*
}; \mathbb{Z})$. If $x  \in \Maps(\hat{A}_{* }; \mathbb{R})$,
then $(z_{u}\cdot x)(\hat{u}) =x(\hat{u})$ if the integer pair
from $\hat{u}$ defines an angle via \eqref{eq1.8} that is less
than that defined by the integer pair from $u$. Such is also the
case when the two angles are equal except if $\hat{u} = u$. If
$\hat{u} = u$, then $(z_{u}\cdot x)(u) = x(u) - 2\pi $. When the
integer pair from $\hat{u}$ defines an angle that is greater than
that of the pair from $u$, then $(z_{u}\cdot x)(\hat{u})$ is
obtained from $x(\hat{u})$ by adding
\begin{equation}\label{eq3.6}
-2\pi   \varepsilon _{u}  \frac{{p_u} ' p_{\hat{u}} - p_u
{p_{\hat{u}}} ' }{q_{\hat{e}} ' p_{\hat{u}} - q_{\hat{e}} {p_{\hat{u}}} ' },
\end{equation}
where the notation is as follows: First, $\hat{e}$ labels the edge
that contains $\hat{u}$ as its largest angle vertex and $(q_{\hat{e}
}, {q_{\hat{e} }}')$ is the integer pair that is associated to
$\hat{e}$. Meanwhile, $(p_{u}, {p_{u}}')$ is the integer pair from $u$
and $(p_{\hat{u} }$, $p_{\hat{u} }')$ is that from $\hat{u}$.
Finally, $\varepsilon _{u} \in  \{\pm \}$ is the second entry in the
4--tuple $u$.

Also needed is the action of $\mathbb{Z}  \times   \mathbb{Z}$ on
the space in \eqref{eq3.5} that is described just as in Step 2 of
Part 1 in \fullref{sec:1a} for the latter's version of
$\mathbb{R}_{ - } \times \mathbb{R}^{\hat{A}}$. This $\mathbb{Z}
\times \mathbb{Z}$ action commutes with that of $\Maps(\hat{A}_{* };
\mathbb{Z})$. This understood, define
\begin{equation}\label{eq3.7}
O^{\hat{A}}\equiv  [\mathbb{R}_{ - } \times
\mathbb{R}^{\hat{A}}]/[(\mathbb{Z}
\times \mathbb{Z})\times \Maps(\hat{A}_{* }; \mathbb{Z})].
\end{equation}
As is explained next, $O^{\hat{A}}$ is a smooth manifold.

To see that $O^{\hat{A}}$ is smooth, it is sufficient to verify
that each point in \eqref{eq3.5} has the same stabilizer under the action of the
group in \eqref{eq3.7}. As is explained below, this stabilizer is a copy of $\mathbb{Z}$
whose generator projects to $\Maps(\hat{A}_{ + }; \mathbb{Z})$ as $-\sum_{u} z_{u}$
and projects to $\mathbb{Z} \times \mathbb{Z}$ as the integer
pair $Q_{e}$ with $e$ here denoting the edge in $T^{\hat{A}}$ that
contains the smallest angle vertex. To prove such is the case, suppose that
$\tau\in (\tau _{ - }, x)$ is a point in the space depicted in
\eqref{eq3.5}, and that $g = (N, z)$ fixes $\tau $. As $\tau _{ - }$ is fixed, the
pair $N$ must have the form $r_{ - }Q_{e}$ where $Q_{e}$ here denotes the
integer pair that is assigned to the edge with the smallest angle vertex and
where $r_{ - }$ is a fraction whose denominator is the greatest common
divisor of these same two integers.

To proceed, suppose next that $\hat{u}  \in \hat{A}_{* }$ supplies
an integer pair that gives the second smallest angle in $\Lambda
_{\hat{A}}$. Since the pair $Q_{\hat{e} } = (q_{\hat{e}
},{q_{\hat{e} }}')$ in the relevant version of \eqref{eq1.19} is the
same as the just defined $Q_{e}$, it follows that $g$ fixes both
$\tau _{ - }$ and $x(\hat{u})$ if and only if $z(\hat{u}) = -r_{ -
}$. Note that this means that $r_{ - }$ is an integer since
$z(\hat{u})$ is an integer.

Suppose next that $\hat{u}  \in \hat{A}_{* }$ supplies an integer pair
that gives the third smallest angle in $\Lambda _{\hat{A}}$. The
corresponding $x(\hat{u})$ is then fixed if and only if
\begin{equation}\label{eq3.8}
z(\hat{u}) + 2\pi\sum _{u} z(u) \varepsilon _{u}
\frac{{p_u} ' p_{\hat{u}} - p_u {p_{\hat{u}}} ' }{{q_{\hat{e}}} ' p_{\hat{u}} -
q_{\hat{e}} p_{\hat{u}} ' } + r_{ - }\frac{q_e ' p_{\hat{u}} - q_e
p_{\hat{u}} ' }{{q_{\hat{e}}} ' p_{\hat{u}} - q_{\hat{e}} p_{\hat{u}} ' } = 0,
\end{equation}
where the sum is over the elements in $\hat{A}_{* }$ whose integer pair
defines the second smallest angle in $\Lambda _{\hat{A}}$. Here, $e$
is the edge that has the smallest angle vertex in $T^{\hat{A}}$ and
$\hat{e}$ is the edge that contains the second and third smallest angle
vertices. To make something of \eqref{eq3.8}, note that each $z(u)$ that appears in
the sum is $-r_{ - }$. In addition, \eqreft33 identifies
$\sum _{u}\varepsilon (p_{u}, {p_{u}}')$ with $Q_{e} - Q_{\hat{e} }$.
These points understood, then \eqref{eq3.8} asserts that $x(\hat{u})$ is fixed by $g$ if
and only if $z(\hat{u}) = -r_{ - }$.

One can now continue in this vein in an inductive fashion through
the elements from $\hat{A}_{* }$ with integer pairs that give
successively larger angles in \eqref{eq1.8}. In particular, an
application of \eqreft33 at each step finds that $g = (N, z)$ fixes
$(\tau _{ - }, x)$ if and only if $N = r_{ -}Q_{e}$ and $z$ sends
each element in $\hat{A}_{* }$ to $-r_{ - }$. This
straightforward task is left to the reader.

Define the group $\Aut ^{\hat{A}}$ to be the group of 1--1 maps of
$\hat{A}_{* }$ to itself that only mix elements with identical 4--tuples.
This group acts smoothly on $O^{\hat{A}}$. Set $\hat{O}^{\hat{A}}\subset O^{\hat{A}}$
to be the set of points where the
action is free. Propositions~\ref{prop:4.4} and~\ref{prop:4.5} say more about
$\hat{O}^{\hat{A}}$.

\fullref{thm:1.2} in \fullref{sec:1b} now has the
following generalization:

\begin{theorem} \label{thm:3.1}

Let $\hat{A}$ denote an asymptotic data set that obeys \eqreft31 and
whose graph $T^{\hat{A}}$ obeys the conditions in \eqreft1{16}.
The subspace of subvarieties in $\mathcal{M}_{\hat{A}}$ whose
graph from \fullref{sec:2a} is linear is a closed
submanifold of $\mathcal{M}_{\hat{A}}$ that is diffeomorphic to
$\mathbb{R}\times \hat{O}^{\hat{A}}/\Aut ^{\hat{A}}$. Moreover,
there is a diffeomorphism that intertwines the $\mathbb{R}$
action on $\mathcal{M}_{\hat{A}}$ as the group of constant
translations along the $\mathbb{R}$ factor of $\mathbb{R}\times
(S^{1}\times S^{2})$ with the $\mathbb{R}$ action on
$\mathbb{R}\times \hat{O}^{\hat{A}}/\Aut ^{\hat{A}}$ as the group
of constant translations along the $\mathbb{R}$ factor in the
latter space.
\end{theorem}

Let $\mathcal{M}\subset\mathcal{M}_{\hat{A}}$ denote the indicated
subspace. The remaining subsections describe a map from $\mathcal{M}$ to
$\mathbb{R}\times O^{\hat{A}}/\Aut ^{\hat{A}}$ that is seen
in the next section to be a diffeomorphism between these two spaces.

\subsection{The map to $\mathbb{R}$}\label{sec:3b}

Fix $C  \in\mathcal{M}$. The image of $C$ in the $\mathbb{R}$ factor of
$\mathbb{R}\times O^{\hat{A}}/\Aut ^{\hat{A}}$ is the simplest part
of the story. Even so, its definition is different in each of the three
cases of \eqreft32. These are treated in turn.

\step{Case 1}
Let $E  \subset C$ denote the convex side end where the $|s|  \to\infty $
limit of $\theta $ gives the minimal angle in $\Lambda
_{\hat{A}}$. Associated to the latter is the real number $b$ that
appears in \eqref{eq2.4}. In this regard, note that the integer $n_{E}$ that appears
here is zero and thus, $b$ must be positive since $\theta _{E}$ is the
infimum of $\theta $ on $C$. This understood, the map to $\mathbb{R}$ sends the
subvariety $C$ to $-\zeta ^{-1} \ln(b)$ where $\zeta\equiv \surd 6 \sin^{2}\theta _{E}
(1+3\cos^{2}\theta_{E})/(1+3cos^{4}\theta _{E})$.

\step{Case 2}
In this case, $C$ intersects the $\theta  = 0$ cylinder at a single point
and $C$'s image in the $\mathbb{R}$ factor of $\mathbb{R}\times O^{\hat{A}}/\Aut ^{\hat{A}}$
is the s coordinate of this point.

\step{Case 3} In this case, $C$ has a single end whose constant
$|s| $ slices limit to the $\theta  = 0$ cylinder as
$|s|\to\infty $. This end defines the positive constant $\hat
{c}$ that appears in \eqref{eq1.9}. Note that the integers $p$
and $p'$ that appear in \eqref{eq1.9} comprise the pair from the
$(1,\ldots)$ element in $\hat{A}$. The image of $C$ in
$\mathbb{R}$ is $-(\sqrt {\frac{3}{2}} +\frac{{p'
}}{p})^{-1} \ln(\hat {c})$.

\subsection{The map to $O^{\hat{A}}/\Aut ^{\hat{A}}$}\label{sec:3c}

The definition of $C$'s assigned point in $O^{\hat{A}}/\Aut
^{\hat{A}}$ requires a preliminary digression to set the stage.
To start, let $o$ denote a bivalent vertex in $T^{\hat{A}}$ and
let $\hat{A}_{o} \subset \hat{A}$ denote the subset of elements
whose integer pair defines $o$'s angle via \eqref{eq1.8}. Use
$n_{o}$ in what follows to denote the number of elements in
$\hat{A}_{o}$.

Introduce $S_{o}$ to denote the subspace of 1--1 maps from $\hat{A}_{o}$ to
$\mathbb{R}/(2\pi \mathbb{Z})$. The inverse image of $S_{o}$ in
$\Maps(\hat{A}_{o}; \mathbb{R})$ is the part of the latter space that
contributes to $\mathbb{R}^{\hat{A}}$. The components of
$S_{o}$ are in 1--1 correspondence with the set of cyclic orderings of the
elements in $\hat{A}_{o}$. If $u_{o}\in \hat{A}_{o}$ is a chosen
distinguished element, then any given component $S  \subset S_{o}$ can be
identifed with
\begin{equation}\label{eq3.9}
\mathbb{R} / (2\pi \mathbb{Z}) \times \Delta _{o},
\end{equation}
where $\Delta _{o}\subset\mathbb{R}^{n_o }$ is the open, $n_{o}-1$
dimensional simplex in the positive quadrant where the sum of the
coordinates is equal to $2\pi $. Here, the identification in \eqref{eq3.9} is
obtained as follows: Identify the cyclic ordering that defines $S$ with the
linear ordering that has $u_{o}$ as the last element. This identification
provides a 1--1 correspondence between $\hat{A}_{o}$ and the set $\{1, \ldots, n_{o}\}$.
Granted this, the identification of $S$ with the space in \eqref{eq3.9}
is given by identifying a given $(\tau , (r_{1},\ldots ))  \in\mathbb{R}/(2\pi \mathbb{Z})
\times \Delta _{o}$ with the map from
$\hat{A}_{o}$ to $\mathbb{R}/(2\pi \mathbb{Z})$ that sends the $k$'th element in
$\hat{A}_{o}$ to the $\mod(2\pi \mathbb{Z})$ reduction of $\tau +r_{1}+\cdots +r_{k}$.

Now suppose that $F$ is a space with an action of $\Maps(\hat{A}_{o}; \mathbb{Z})$.
Let $F_{o}$ denote the associated fiber bundle over $S_{o}$ with fiber
$F$. Thus,
\begin{equation}\label{eq3.10}
F_{o} = (\Maps(\hat{A}_{o}; \mathbb{R})\times  F)/\Maps(\hat{A}_{o}; \mathbb{Z}).
\end{equation}
The identification between $S$ and the space in \eqref{eq3.9} is covered by one
between $F_{o}| _{S}$ and
\begin{equation}\label{eq3.11}
(\mathbb{R}_{o }\times  F)/\mathbb{Z}_{o }\times\Delta _{o},
\end{equation}
where $\mathbb{R}_{o}$ is a copy of $\mathbb{R}$ while $\mathbb{Z}_{o}$ is a copy
of $\mathbb{Z}$ whose generator acts on $\mathbb{R}_{o}$ as the
translation by $-2\pi $ and on $F$ as that of $\sum _{u} z_{u}\in \Maps(\hat{A}_{o};
\mathbb{Z})$.

All of this has the following implications: Fix a distinguished element in
each version of $\hat{A}_{o}$. This done, then any given component of
$O^{\hat{A}}$ is identified with
\begin{equation}\label{eq3.12}
(\times _{o}\Delta _{o})\times  [\mathbb{R}_{ - }\times
(\times _{o}\mathbb{R}_{o})] /[(\mathbb{Z} \times \mathbb{Z})\times
(\times _{o}\mathbb{Z}_{o})],
\end{equation}
where the notation $\times _{o}$ indicates a product that is
labeled by the vertices in $T^{\hat{A}}$, and where the group
actions are as follows: First, $N = (n, n')  \in\mathbb{Z} \times
\mathbb{Z}$ acts on $\mathbb{R}_{ - }$ as before while acting on
any given version of $\mathbb{R}_{\hat{o}}$ as translation by the
element in \eqref{eq1.20}. Meanwhile, any given $\mathbb{Z}_{o}$
acts trivially on $\mathbb{R}_{ - }$ and also trivially on
$\mathbb{R}_{\hat{o}}$ in the case that the angle of $\hat{o}$ is
less than that of $o$. On the other hand,  $1 \in\mathbb{Z}_{o}$
acts on $\mathbb{R}_{o}$ as the translation by $-2\pi $ and it
acts on $\mathbb{R}_{\hat{o}}$ in the case that $\hat{o}$'s angle
is greater than $o$'s angle as the translation by
\begin{equation}\label{eq3.13}
-2\pi \frac{{p_o} ' \hat {p}_{\hat{o}} - p_o {\hat {p}_{\hat{o}}} '
}{{q_{\hat{e}}} ' \hat {p}_{\hat{o}} - q_{\hat{e}} {\hat {p}_{\hat{o}}} ' },
\end{equation}
where the notation is as follows: First, $(\hat {p}_{\hat{o}},
\hat {p}_{\hat{o}}')$ is the relatively prime pair of integers
that defines $\hat{o}$'s angle in \eqref{eq1.8}. Second
$P_{o}\equiv (p_{o}, {p_{o}}')$ is obtained by subtracting the sum
of integer pairs from the $(0,-,\ldots)$ elements in
$\hat{A}_{o}$ from the sum of those from the $(0,+,\ldots)$
elements. Thus, $P_{o}$ is the integer pair from $o$'s version of
the second point in \eqreft33.

Granted what has just been said, if a distinguished element has been chosen
in each version of $\hat{A}_{o}$, then a point for $C$ in
$O^{\hat{A}}/\Aut ^{\hat{A}}$ is obtained from an assigned point in
\begin{equation}\label{eq3.14}
\mathbb{R}_{ - }\times (\times _{o}(\mathbb{R}_{o}\times \Delta _{o}))
\end{equation}
together with an assigned cyclic ordering for each version of $\hat{A}_{o}$.
The story on these assignments appears below in four parts.

With the preceding understood, the digression is now ended.

\step{Part 1}
The assignment to $C$ of a point in the
$\mathbb{R}_{ - }$ factor of \eqref{eq3.14} requires the choice of a parameterization
for the component of $C_{0}-\Gamma $ whose labeling edge has the
minimal angle vertex. With such a parameterization in hand, let $w$ denote the
associated function from \eqreft25. The $\mathbb{R}_{ - }$ assignment of $C$ is the
value of the expression
\begin{equation}\label{eq3.15}
-\frac{1}{{2\pi }}\alpha _{Q}(\sigma )\int_{\mathbb{R} / (2\pi \mathbb{Z})} w(\sigma , v) dv
\end{equation}
as computed using any $\sigma $ that lies between the vertex angles on the
edge in question. Here, $Q$ is the integer pair that labels this same edge.

\step{Part 2}
As has most probably been noted, there is an evident
`forgetful' map from a graph $T$ with a correspondence in $(C_{0}, \phi )$
to the graph $T^{\hat{A}}$ that is obtained by dropping the circular
graph labels from the vertices in $T$. Thus, the map is an isomorphism of
underlying graphs that respects vertex angles and the integer pair labels of
the edges. This forgetful map is used implicitly in what follows to identify
respective vertices and respective edges in the two graphs.

Let $o$ denote a given bivalent vertex in $T$. Fix a 1--1 map from the vertices
in the $T$ version of $\underline {\Gamma }_{o}$ to $\hat{A}_{o}$ with the
following property: A vertex in $\underline {\Gamma }_{o}$ is assigned an
element $u  \in \hat{A}_{o}$ if and only the integer label for the vertex
is equal to $\varepsilon _{u} m_{u}$ where $\varepsilon _{u}$ is the
second entry to $u$ and $m_{u}$ is the greatest common divisor of the integer
pair entry for $u$. Granted this 1--1 correspondence, define a cyclic ordering
of $\hat{A}_{o}$ using the ordering of the vertices on the circular graph
$\underline {\Gamma }_{o}$ as they are met on an oriented
circumnavigation.

Label the arcs in $T$'s version of $\underline {\Gamma }_{o}$ by integers
consecutively from 1 to $n_{o}$ so that the labeling gives the order in
which arcs are crossed when circumnavigating $\underline {\Gamma }_{o}$ in
its oriented direction when starting at the vertex that is mapped to the
distinguished element in $\hat{A}_{o}$.

Having done the above, fix a correspondence of $T$ in $(C_{0}, \phi )$ and
thus define $T_{C}$.

\step{Part 3}
Let $k$ denote an integer that labels a
given arc on $\underline {\Gamma }_{o}$. To obtain the assignment for $C$ in
the $k$'th coordinate factor of the simplex $\Delta _{o}$ in \eqreft3{16}, first
integrate the pull-back of $(1-3\cos^{2}\theta ) d\varphi -\surd
6\cos\theta  dt$ over the $k$'th arc in the corresponding $C_{0}$ version of
the graph $\Gamma _{o}$. Then, divide the result by 2$\pi \alpha
_{Q}(\theta _{o})$ where $\theta _{o}$ is $o$'s angle and where $Q$ is
the integer pair that labels the edge in $T^{\hat{A}}$ with largest
angle $\theta _{o}$.

\step{Part 4}
The assignment to $C$ of a point in the
$\times _{o}\mathbb{R}_{o}$ factors in \eqref{eq3.14} is made in an iterative
fashion that starts with the minimal angle bivalent vertex in $T^{\hat{A}}$
and proceeds from vertex to vertex in their given order along
$T^{\hat{A}}$. Here is the basic iteration step: Let $o$ denote a
bivalent vertex in $T^{\hat{A}}$, let $e$ denote the edge in $T_{C}$
that contains $o$ as its largest angle vertex and let $e'$ denote the edge that
contains $o$ as its smallest angle vertex. Suppose that $e$'s component
$K_{e} \subset C_{0}-\Gamma $ has been assigned a canonical
parameterization. Let $\upsilon $ denote the missing point on the $\sigma = \theta _{0}$
circle in the corresponding parametrizing cylinder that
corresponds to the distinguished element in $\hat{A}_{o}$. Choose a lift to
$\mathbb{R}$ of the $\mathbb{R}/(2\pi \mathbb{Z})$ coordinate of $\upsilon $. Use
this lift as the value that is assigned to $C$ in the factor $\mathbb{R}_{o}$
that appears in \eqref{eq3.14}. To initiate the next iteration round, use this same
lift and the canonical parameterization of $K_{e}$ in the manner described
by Part 2 of \fullref{sec:2c} to define the canonical parametrization of the
component $K_{e' }\subset C_{0}-\Gamma $.

The chosen parameterization from Part 1 above should be used for the
canonical parameterization when starting the iteration at the smallest
angled bivalent vertex.

\subsection{The invariance of the image in $O^{\hat{A}}$}\label{sec:3d}

The point just assigned to $C$ in \eqref{eq3.14} required the following suite of
choices:

\itaubes{3.16}
\textsl{A distinguished element in each version of $\hat{A}_{o}$.}

\item
\textsl{A suitably constrained 1--1 correspondence from each  $o \in T$ version of the vertex
set in $\underline {\Gamma }_{o}$ to the corresponding $\hat{A}_{o}$.}

\item
\textsl{A correspondence of $T$ in $(C_{0}, \phi )$.}

\item
\textsl{A parameterization for the component of $C_{0}-\Gamma $ whose corresponding
edge contains the minimal angled vertex in $T_{C}$.}

\item
\textsl{A lift to $\mathbb{R}$ made for each multivalent vertex in $T_{C}$.
In particular, let $o$ denote such a vertex and let $e$ denote the edge
that contains $o$ as its maximal angled vertex. The lift for $o$
is that of the $\mathbb{R}/(2\pi \mathbb{Z})$ coordinate of the missing point
on the $\theta =\theta _{o}$ boundary of the pararameterizing cylinder for
the canonical parameterization of $K_{e}$ that corresponds to the distinguished
element in $\hat{A}_{o}$}
\end{itemize}
This subsection explains why the image in $O^{\hat{A}}/\Aut ^{\hat{A}}$
of the point given $C$ in \eqref{eq3.14} is insensitive to the choices that
are described in \eqreft3{16}. These choices are considered below in the order 5,
4, 1, 2, 3. Note for future reference that the arguments given below proves
that the image of $C$ in $O^{\hat{A}}$ is already insenstive to changes
that are described points 5, 4 and 1. In any event, the discussion that
follows is in five parts, one for each of the points in \eqreft3{16}.

\step{Part 1}
To start the explanation for the fifth point in \eqreft3{16}, suppose that $o$ is
a multivalent vertex in $T_{C}$. Let $\theta _{o}$ denote $o$'s angle, let $e$
denote the edge that has $o$ as its largest angle vertex and let $e'$ denote the
edge that has $o$ as its smallest angle vertex. Note that any parameterization
of $K_{e}$ has a `distinguished' missing point on the $\sigma =\theta
_{o}$ boundary circle, this the point that corresponds to the chosen
distinguished point in $\hat{A}_{o}$.

Suppose that $\tau _{o}\in\mathbb{R}$ is the original lift of the
$\mathbb{R}/(2\pi \mathbb{Z})$ coordinate of the distinguished point on the
$\sigma =\theta _{o}$ boundary of the parametrizing cylinder for
$K_{e}$. Now change this lift to $\tau _{o}-2\pi $. Such a change has no
affect on the assignment of $C$ in $\mathbb{R}_{ - }$ or in $\mathbb{R}_{\hat{o}}$
if $\hat{o}$ is a bivalent vertex with angle less than $\theta
_{o}$. Of course, it changes the assignment in $\mathbb{R}_{o}$ from $\tau
_{o}$ to $\tau _{o}-2\pi $.

The affect of this change on the remaining $\mathbb{R}_{(\cdot )}$ factors in
\eqref{eq3.14} is examined vertex by vertex in order of increasing angle. To start
this process, invoke the discussion in Case 4 in Part 5 of \fullref{sec:2c}, to
conclude that the change $\tau _{o}\to\tau _{o}-2\pi $ changes
the parameterization of $K_{e' }$ by the action of the integer pair
$Q_{e}$. In this regard, note that the equality $Q_{e} = Q_{e' }+ P_{o}$,
implies that this change is identical to that obtained by the
action of $P_{o}$. The latter view proves the more useful for what follows.

Let $o'$ now denote the vertex with the largest angle on $e'$ and suppose that
 $o'$ is bivalent. Let $\tau ^{\old}$ denote the original assignment to $C$ in
$\mathbb{R}_{\hat{o}}$. As applied now to $o'$, the assertion of \eqref{eq2.16}
and the conclusions of Case 2 of Part 5 in \fullref{sec:2c} imply that there is a
choice for the lift of the $\mathbb{R}/(2\pi \mathbb{Z})$ coordinate of the
distinguished point on the $\sigma =\theta _{o' }$ circle for
the new parametrizing cylinder whose relation to the old is given by the $N= P_{o}$
version of
\begin{equation}\label{eq3.17}
\tau ^{\new}=\tau ^{\old}- 2\pi \frac{{\alpha _N (\theta
_{o' } )}}{{\alpha _{Q_{e' } } (\theta _{o' } )}} -
2\pi  k_{o' }
\end{equation}
with $k_{o' }\in\mathbb{Z}$.

To continue, let $\hat{o}$ denote the vertex that shares an edge with $o'$ but has
greater angle. Suppose, for the sake of argument, that $\hat{o}$ is bivalent.
Let $\hat{e}$ denote the edge between $\hat{o}$ and $o'$. Apply Case 2 of Part 5 in
\fullref{sec:2c} here together with the identity $Q_{e' } = Q_{\hat{e} } + P_{o' }$
to conclude that the canonical
parameterization for $K_{\hat{e} }$ is changed by the action of the
integer pair $P_{o} + k_{o' }P_{o' }$. This then means
that the old and new assignments for $C$ in $\mathbb{R}_{\hat{o}}$ are
related by the $\sigma =\theta _{\hat{o}}$ version of
\begin{equation}\label{eq3.18}
\tau ^{\new}=\tau ^{\old} - 2\pi \frac{{\alpha _{P_o }
(\theta _{\hat{o}} )}}{{\alpha _{Q_{\hat{e}} } (\theta _{\hat{o}} )}} - 2\pi
k_{o' }\frac{{\alpha _{P_{o' } } (\theta _{\hat{o}} )}
}{{\alpha _{Q_{\hat{e}} } (\theta _{\hat{o}} )}} - 2\pi k_{\hat{o}}
\end{equation}
with $k_{\hat{o}}\in\mathbb{Z}$.

Continue in this vein vertex by vertex in order of increasing angle with
applications of Case 2 in Part 5 of \fullref{sec:2c} so as to obtain the
generalization of \eqref{eq3.18} that is summarized as the next lemma.

\begin{lemma}\label{lem:3.2}

The bivalent vertices of $T$ with angles greater than $\theta _{o}$ label a collection,
$\{k_{(\cdot )}\}$, of integers with the following significance: Let $\hat{o}$
denote any bivalent vertex in $T$ with angle greater than $\theta _{o}$, and let $\hat{e}$
denote the edge that has $\hat{o}$ as it maximal angled vertex. Then, the change,
$\tau _{o}\to\tau_{o}-2\pi $ of the assignment to $C$ in $\mathbb{R}_{o}$ changes
the assignment of $C$ in $\mathbb{R}_{\hat{o}}$ by the rule
\begin{equation}\label{eq3.19}
\tau _{\hat{o}}\to\tau _{\hat{o}}- 2\pi
\frac{{\alpha _{P_o } (\theta _{\hat{o}} )}}{{\alpha _{Q_{\hat{e}} }
(\theta _{\hat{o}} )}} - 2\pi\sum '_{o' } k_{o' }\frac{{\alpha _{P_{o' } }
(\theta _{\hat{o}} )}}{{\alpha_{Q_{\hat{e}} } (\theta _{\hat{o}} )}} - 2\pi  k_{\hat{o}},
\end{equation}
where the prime on the summation symbol is meant to indicate that the sum is over
those bivalent vertices whose angles lie between $\theta _{o}$ and $\theta _{\hat{o}}$.
\end{lemma}

Together, \eqref{eq1.8} and \eqref{eq3.19} imply that the action of
$\times _{o}\mathbb{Z}_{o}$ on \eqref{eq3.14} rectifies any change in the
choices that are described in the fifth point of \eqreft3{16}. Thus,
the image of $C$ in $O^{\hat{A}}$ is insensitive to any such
change.

\step{Part 2} Consider now the effect of a change as described by
the fourth point in \eqreft3{16}. The analysis of this change starts
with its affect on the various $\mathbb{R}_{o}$ factors. This can
be analyzed vertex by vertex in order of increasing angle by
successively invoking the observations from Case 2 in Part 5 of
\fullref{sec:2c}. In particular, with the help of
\eqref{eq1.8}, these observations dictate the following: Suppose
that $N = (n, n') \in \mathbb{Z} \times \mathbb{Z}$ and that the
chosen parameterization of the smallest angled component of
$C_{0}-\Gamma $ is changed by the action of $N$ as described in
\eqref{eq2.12} and \eqref{eq2.13}. This changes $C$'s assignment in the $\times
_{o}\mathbb{R}_{o}$ factor of \eqref{eq3.14} by the action on the
original assignment in $\times _{o}\mathbb{R}_{o}$ of an element
of the form $(N,\ldots)\in (\mathbb{Z} \times \mathbb{Z})\times
(\times _{o}\mathbb{Z}_{o})$.

As can be seen directly from \eqref{eq2.13} and \eqref{eq3.15}, the
change induced by $N$ on the parameterization of the $C_{0}-\Gamma $
component with the smallest $\theta $ values changes the assignment
of $C$ in $\mathbb{R}_{ - }$ by subtracting $2\pi (n'q_{e}-
n{q_{e}}')$. Meanwhile, the integer pair
 $N$ acts affinely on $\mathbb{R}_{ - }$ as described in the preceding
subsection by subtracting $2\pi (n'q_{e} - n{q_{e}}')$.

Taken together, the conclusions of these last two paragraphs imply that the
image of $C$ in $O^{\hat{A}}$ is insensitive to any change that is
described by the fourth point in \eqreft3{16}.

\step{Part 3}
This part analyzes the affect on $C$'s assigned point in $O^{\hat{A}}$ of a
change in the chosen distinguished element from any given version
of $\hat{A}_{(\cdot )}$. The subsequent five steps prove that this assignment
is not changed.

\substep{Step 1}
The new and old version choices for
distinguished elements result in respective new and old assignments of a
point for $C$ in $\mathbb{R}_{ - } \times \mathbb{R}^{\hat{A}}$. To
be explicit, a choice of cyclic ordering and distinguished element in each
version of $\hat{A}_{(\cdot )}$ identifies $\times _{\hat{o}}(\mathbb{R}_{\hat{o}}
\times \Delta _{\hat{o}})$ with a
component of $\mathbb{R}^{\hat{A}}$ as follows: The distinguished
element and cyclic ordering of a given $\hat{A}_{o}$ endow its elements in
their cyclic order with a labeling by the integers in the set $\{1, \ldots, n_{o}\}$
so that $n_{o}$ labels the distinguished element. Granted this
linear ordering, let $u_{k}$ denote the $k$'th element in some given
$\hat{A}_{o}$. Then all maps in $\mathbb{R}^{\hat{A}}$ that arise from a
point in $\times _{\hat{o}}(\mathbb{R}_{\hat{o}} \times \Delta _{\hat{o}})$
with a given $(\tau , (r_{1}, \ldots )) \in\mathbb{R}_{o} \times \Delta _{o}$
send $u_{k}$ to $\tau+r_{1}+\cdots +r_{k}-2\pi\in\mathbb{R}$.

This identification of $\times _{\hat{o}}(\mathbb{R}_{\hat{o}} \times \Delta _{\hat{o}})$
with a component of $\mathbb{R}^{\hat{A}}$ results in the identification described earlier
between the space in \eqref{eq3.12} and a component of $O_{\hat{A}}$.

\substep{Step 2}
Let $o$ denote a bivalent vertex in $T_{C}$,
let $u  \in \hat{A}_{o}$ the original choice for a distinguished element,
and $u'$ denote the new choice. As explained in Step 1, these choices result
in respective points for $C$ in $\mathbb{R}_{ - } \times \mathbb{R}^{\hat{A}}$.
Such a change has no affect on the assignment to $\mathbb{R}_{ -}$.
Let $x$ and $x'$ denote the respective original and new assignments of $\mathbb{R}^{\hat{A}}$.

Because the change from $u$ to $u'$ has has no affect on $C$'s assignment in any
$\mathbb{R}_{\hat{o}}$ or $\Delta _{\hat{o}}$ factor in
\eqref{eq3.12} when $\theta _{\hat{o}}<\theta _{o}$, the maps $x$ and
$x'$ send any given $\hat{u}  \in \hat{A}_{* }$ to the same point in
$\mathbb{R}$ if the integer pair component of $\hat{u}$ is less than $\theta_{o}$.

\substep{Step 3}
This step describes the image via $x'$ of an
element whose integer pair gives $\theta _{o}$. This task requires an
analysis of the change to $C$'s assignment to the $\Delta _{o}$ and $\mathbb{R}_{o}$
factors in \eqref{eq3.14}. Start this analysis by labeling the arcs in
$\underline {\Gamma }_{o}$ from 1 through $n_{o}$ with the first arc
starting at $u$'s vertex in $\underline {\Gamma }_{o}$. Suppose that the arc
that ends at the vertex that corresponds to $u'$ is the $k$'th arc. Let
$r =(r_{1}, r_{2}, \ldots )  \in\Delta _{o}$ denote the original
assignment for $C$. Then $r' = (r_{k + 1}, r_{k + 2},\ldots )$ gives the new
assignment.

Consider next the change in the assignment to $\mathbb{R}_{o}$. The
description here is simplest if it is agreed beforehand to keep the original
parameterization of the component of $C_{0}-\Gamma $ labeled by
the smallest angle vertex in $T_{C}$ and also to keep $C$'s assigned point in
each $\mathbb{R}_{\hat{o}}$ in the case that $\theta _{\hat{o}}<\theta _{o}$.
Granted this, let $e$ now denote the edge that has $o$ as
its largest angle vertex. Then the canonical parametrization of the
component $K_{e}\subset C_{0}-\Gamma $ is unchanged. With
this last point understood, take the lift to $\mathbb{R}$ of the
$\mathbb{R}/(2\pi \mathbb{Z})$ coordinate of the missing point on the
$\sigma  = \theta _{o}$ circle that maps to $u'$ to be that obtained from the original
lift by adding $r_{1}+\cdots +r_{k}$.

The preceding conclusions have the following implications for the map $x'$:
Let $u_{j}$ denote the $j$'th element in the original linear ordering of the
set $\hat{A}_{o}$. Then $x(u_{j}) = x'(u_{j})$ in the case that
$j \in\{k+1,\ldots , n_{o}\}$ and $x'(u_{j}) = x(u_{j}) + 2\pi$  in the
case that  $j \in  \{1,\ldots , k\}$.

\substep{Step 4}
This step describes the value of $x'$ on the
elements in $\hat{A}_{* }$ whose integer pair defines an angle that is
greater than $\theta _{o}$. To start the analysis, suppose that $\hat{o}$ is a
bivalent vertex with angle greater that $o$'s angle. There is no change to the
simplex $\Delta _{\hat{o}}$ with the change from $u$ to $u'$. To
consider the affect on the assignment in \eqref{eq3.12}'s factor $\mathbb{R}_{\hat{o}}$,
let $e'$ denote the edge in $T_{C}$ with $o$ as its smallest angle
vertex. According to what is said in Case 3 of Part 5 and \fullref{lem:2.3} of Part
6 from \fullref{sec:2c}, the parameterization for $K_{e'}$ is changed by
the action of the integer pair
\begin{equation}\label{eq3.20}
N = \sum _{1 \le j \le k}P_{u_j }.
\end{equation}
Given the discussion from Part 2 of \fullref{sec:2c} and the fourth point in
\eqreft3{16}, this then has the following consequence: The new assignment for $C$ in
\eqref{eq3.14} can be made consistent with what has been said so far and such that
the new and old assignments to the $\theta _{\hat{o}}>\theta
_{o}$ versions of $\mathbb{R}_{\hat{o}}$ change by the addition of
\begin{equation}\label{eq3.21}
2\pi\sum _{1 \le j \le k}\frac{{p_{u_j }} ' \hat {p}_{\hat{o}} -
p_{u_j } {\hat {p}_{\hat{o}}} ' }{{q_{\hat{e}}} ' \hat {p}_{\hat{o}} - q_{\hat{e}}
{\hat {p}_{\hat{o}}} ' }.
\end{equation}
The implication for the map $x'$ is as follows: Let $\hat{u}  \in
\hat{A}_{* }$ denote an element whose integer pair component
defines via \eqref{eq1.8} an angle that is greater than $\theta
_{o}$. Then $x'(\hat{u})$ is obtained from $x(\hat{u})$ by acting
on the former by adding the term in \eqref{eq3.21}.

\substep{Step 5}
The conclusions of the previous steps
imply that the map $x'$ is obtained from $x$ by acting on the former with the
element $-\sum _{1 \le j \le k} z(u_{j})$ from $\Maps(\hat{A}_{* }; \mathbb{Z})$.

\step{Part 4}
What follows in this part is an explanation of why the image of $C$ in
$O^{\hat{A}}/\Aut ^{\hat{A}}$ is insensitive to the change
that is described by the second point in \eqreft3{16}. To start, suppose that $o$ is
a multivalent vertex in $T_{C}$ and the original correspondence between to
the vertex set of $\underline {\Gamma }_{o}$ and $\hat{A}_{o}$ is changed to
some new assignment. The new assignment is thus obtained by composing the
original with a 1--1 self map of $\hat{A}_{o}$ that only permutes elements with
identical 4--tuples. Let $\iota \co  \hat{A}_{o}\to \hat{A}_{o}$ denote
this permutation.

As explained in Step 1 of the Part 3, the original point for $C$ in
$\mathbb{R}_{ - }\times (\times _{o}\mathbb{R}_{o})$ corresponds to an
assigned point $(\tau _{ - }, x)  \in\mathbb{R}_{ - }\times \mathbb{R}^{\hat{A}}$.
The change induced by $\iota $ in the
identification between $\underline {\Gamma }_{o}$ and $\hat{A}_{o}$ to $C$'s
assigned point in $\mathbb{R}_{ - }\times (\times _{o}\mathbb{R}_{o})$
will change the assigned point in $\mathbb{R}_{ - }\times \mathbb{R}^{\hat{A}}$.
Let $({\tau _{ - }}', x')$ denote this new
point. The task ahead is to prove that $({\tau _{ - }}', x')$ and $(\tau ,x)$
define the same point in $O^{\hat{A}}/\Aut ^{\hat{A}}$.

To start this task, let $\hat{e}$ denote the edge in $T$ that
contains $T$'s minimal angled vertex. Agree to keep the original
parameterization of $K_{\hat{e} }$. This then means that ${\tau _{ - }}' = \tau _{ - }$. One can also arrange that $x'(\hat{u}) =
x(\hat{u})$ if the integer pair from $\hat{u}$ defines an angle
via \eqref{eq1.8} that is less than $\theta _{o}$. This is done
by an iterative scheme that keeps the canonical parameterization
unchanged on each component of $C_{0}-\Gamma $ where $\theta
<\theta _{o}$. Indeed, suppose that $\hat{o}$ is a bivalent
vertex, and suppose that the original parameterization is used
for the component labeled by the edge with $\hat{o}$ as its
largest angle vertex. Then, the original parameterization will
arise on the component labeled by the edge with $\hat{o}$ as its
smallest angle vertex if there is no change made to the lift at
$\hat{o}$ as described in the fourth point of \eqreft3{16}.

To consider the behavior of $x'$ on elements whose integer pair gives an angle
as large as $\theta _{o}$, remark that the new identification between
$\underline {\Gamma }_{o}$ and $\hat{A}_{o}$ has two affects. First, it
changes the distinguished point on the $\sigma =\theta _{o}$ boundary
of the parametrizing cylinder for the component of $C_{0}-\Gamma$
that is labeled by the edge with $o$ as its largest angle vertex. It also
changes the embedding of $\mathbb{R}_{o} \times \Delta _{o}$ into
$\Maps(\hat{A}_{o}; \mathbb{R})$.

Now, the change of the distinguished missing point can be undone
if one first changes the original choice for the distinguished
element in $\hat{A}_{o}$. As explained in Part 3, such a change
modifies $x$ to some $z_{\iota }\cdot x$ where $z_{\iota }\in
\Maps(\hat{A}_{* },\mathbb{Z})$. Note that by virtue of what was
said in the preceding paragraph, $z_{\iota }(\hat{u}) = 0$ in the
case that the integer pair from $\hat{u}$ defines an angle via
\eqref{eq1.8} that is less than $\theta _{o}$.

Meanwhile, the new embedding of $\mathbb{R}_{o} \times \Delta
_{o}$ into $\Maps(\hat{A}_{o}; \mathbb{R})$ only affects the
value of $z_{\iota }\cdot x$ on elements $u$ whose integer pair
defines $\theta _{o}$ via \eqref{eq1.8}. To see how $x'$ differs
from $z_{\iota }\cdot x$ on the latter set, note first that the
new embedding of $\mathbb{R}_{o} \times \Delta _{o}$ to
$\Maps(\hat{A}_{o}; \mathbb{R})$ is obtained from the old by
composing the latter with the action of the permutation $\iota $.
This then means that $x'$ is obtained from $z_{\iota }\cdot x$ by
the action of $\iota \in  \Aut ^{\hat{A}}$.

\step{Part 5}
A change in the choice for the correspondence of $T$ in $(C_{0}, \phi )$
can be rectified by changing the choices for the first and second points in
\eqreft3{16}. This understood, then the image of $C$ is insensitive to the choice of
such a correspondence.

\subsection{The local structure of the map}\label{sec:3e}

The purpose here is to investigate the local structure around any given
subvariety in $\mathcal{M}$ of the map just defined to
$\mathbb{R}\times O^{\hat{A}}/\Aut ^{\hat{A}}$. However, there is one important
point to establish first, this is summarized by:

\begin{proposition}\label{prop:3.3}
The subspace $\mathcal{M} \subset\mathcal{M}_{\hat{A}}$ is a smooth submanifold
whose dimension is $N_{ + }+ N_{ - }+2$.
\end{proposition}

\begin{proof}[Proof of  \fullref{prop:3.3}]
This is a special case in \cite[Proposition~2.12]{T3}.
\end{proof}

The next proposition summarizes some of the salient local features of the
map. In this regard, one should keep in mind that each $C  \in\mathcal{M}$
has an open neighborhood with the following property: The graph $T_{(\cdot)}$
of any subvariety from the neighborhood is canonically isomorphic
to $T_{C}$. Indeed, the ambiguity with the choice of an isomorphism between
$T_{C}$ and some $T_{C' }$ arises when there is no canonical pairing
between the respective sets of ends in $C$ and in $C'$ that define identical
elements in $\hat{A}$. However, if $C'$ and $C$ are close in $\mathcal{M}$, then each
end of $C'$ is very close at sufficiently large $|s| $ to a unique
end of $C$, and vice versa. This geometric fact provides the canonical pairing
of ends and thus the canonical identification between $T_{C}$ and
$T_{C' }$.

Having made this last point, it then follows that the map from $\mathcal{M}$ to
$O^{\hat{A}}/\Aut ^{\hat{A}}$ has a canonical lift to
$O^{\hat{A}}$ on a neighborhood of any given subvariety. The
following proposition describes the nature of this lift.

\begin{proposition}\label{prop:3.4}

Every subvariety in $\mathcal{M}$ has an open neighborhood on which the map to
$O^{\hat{A}}/\Aut ^{\hat{A}}$ lifts as a smooth embedding onto an open set in $O^{\hat{A}}$.
\end{proposition}
\noindent This proposition is also proved momentarily.

What follows is another way to view \fullref{prop:3.4}. To set things up,
remark that all graphs in any given component of $\mathcal{M}$ have isomorphic
versions of $T_{(\cdot )}$. This said, fix a component and fix a graph, $T$,
that is in the corresponding isomorphism class. Let ${\mathcal{M}_{\hat{A},T}}$
denote the corresponding component, and let ${\mathcal{M}_{\hat{A},T}}^{\Lambda }$
denote the set of pairs $(C, T_{C})$ where $C  \in {\mathcal{M}_{\hat{A},T}}$
and $T_{C}$ signifies a chosen correspondence
of $T$ in $(C_{0}, \phi )$. The tautological projection from
${\mathcal{M}_{A,T}}^{\Lambda }$ to ${\mathcal{M}_{\hat{A},T}}$ defines the
former as a covering space and principal $\Aut(T)$ bundle over ${\mathcal{M}_{\hat{A},T}}$.
Given a bivalent vertex $o \in  T$, fix an
admissable identification between the vertices of $T$'s version of
$\underline {\Gamma }_{o}$ and the set $\hat{A}_{o}$. To elaborate, the identification
is admissable if the second component of the assigned 4--tuple gives the sign
of the integer label of the vertex, and if the greatest common divisor of
the integer pair component of the 4--tuple is the absolute value of the
integer label.

With these identification chosen, the map from ${\mathcal{M}_{\hat{A},T}}$ to
$O^{\hat{A}}/\Aut ^{\hat{A}}$ lifts as a map from
${\mathcal{M}_{\hat{A},T}}^{\Lambda }$ to $O^{\hat{A}}$ and
\fullref{prop:3.4} asserts that the lifted map is a local diffeomorphism. Note
that ${\mathcal{M}_{\hat{A},T}}^{\Lambda }$ can be viewed as the moduli
space of pairs that consist of a subvarieties with a distinct labeling of
its ends.

As indicated by the discussion from \fullref{sec:1b} in the
paragraphs after \fullref{thm:1.2}, the various affine
parameters that enter in the definition of $O^{\hat{A}}$ have direct
geometric interpretations. To elaborate, suppose that $T$,
${\mathcal{M}_{\hat{A},T}}$ and ${\mathcal{M}_{\hat{A},T}}^{\Lambda }$
are as just defined. Let $o \in T$ denote a bivalent vertex and let
$e$ denote the edge of $T$ that contains $o$ as its largest angle
vertex. Introduce $(\hat {p}_{o}, {\hat {p}_{o}}')$ to denote the
relatively prime pair of integers that defines $\theta _{o}$ via
\eqref{eq1.8}. Now define a map from $\Maps(\hat{A}_{o};
\mathbb{R})$ to $\Maps(\hat{A}_{o}; \mathbb{R}/(2\pi \mathbb{Z}))$
as follows: First multiply any given $x  \in \Maps(\hat{A};
\mathbb{R})$ by $({q_{e}}'\hat {p}_{o} - q_{e}{\hat{p}_{o}}')$ and
then take the $\mod(2\pi \mathbb{Z})$ reduction of the result. Using
the map for the relevant bivalent vertex, evaluation on any given $u
\in \hat{A}_{* }$ defines a map from $\mathbb{R}^{\hat{A}}$ to
$\mathbb{R}/(2\pi \mathbb{Z})$ that descends to give a map
\begin{equation}\label{eq3.22}
\Psi _{u}\co  O^{\hat{A}}\to\mathbb{R}/2\pi \mathbb{Z}.
\end{equation}
The composition of such a map with the map from ${\mathcal{M}_{\hat{A},T}}^{\Lambda }$
has following geometric interpretation:

\begin{proposition}\label{prop:3.5}

Let $(C, T_{C})\in{\mathcal{M}_{\hat{A},T}}^{\Lambda }$,
let $o$ denote a bivalent vertex in $T$, let $u  \in \hat{A}_{o}$, and let
$E  \subset C$ denote the end that corresponds via $T_{C}$
to $u$. Then the restriction to the end $E$ of $\hat {p}_{o}\varphi -{\hat {p}_{o}}'t$
has a unique $|s|\to \infty $ limit in $\mathbb{R}/(2\pi \mathbb{Z})$ and the latter is
obtained by composing the map $\Psi _{u}$ with the map from
${\mathcal{M}_{\hat{A},T}}^{\Lambda }$ to $O^{\hat{A}}$.
\end{proposition}

The parameter on the line $\mathbb{R}_{ - }$ also has geometric
interpretation. To describe the latter, first introduce m here to denote the
greatest common divisor of the integer pair that is assigned to the edge in
 $T$ with $T$'s smallest angle vertex. Then, the map from $\mathbb{R}_{ - }$ to
$\mathbb{R}/(2\pi \mathbb{Z})$ that takes the $\mod(2\pi \mathbb{Z})$ reduction
of $\frac{1}{m}\tau _{ - }$ descends to $O^{\hat{A}}/\Aut ^{\hat{A}}$
Let $\Psi _{ - }$ denote the latter map.

\begin{proposition}\label{prop:3.6}
The composition of $\Psi _{ - }$ with the map from $\mathcal{M}$ to
$O^{\hat{A}}/\Aut ^{\hat{A}}$ is the following:

\begin{itemize}

\item
When the first point in \eqreft32 holds, let $o$ denote the minimal angle vertex in $T$.
Let $C  \in\mathcal{M}$ and let $E  \subset  C$ denote the end where the
$|s|\to\infty $ limit of $\theta$ is $\theta _{o}$. Then this composition maps $C$
to the $|s|\to\infty $ limit on $E$ of $\hat{p}_{o}\varphi -{\hat {p}_{o}}'t$.

\item
When the second point in \eqreft32 holds, then this composition assigns to $C$
the $t$--coordinate of its intersection point with the $\theta  = 0$ locus.

\item
When the third point in \eqreft32 holds, write the $(1,\ldots)$ element in $\hat{A}$ as\break
$(1,\varepsilon , m(\hat {p}_{o},{\hat {p}_{o}}'))$. Let $E  \subset  C$ denote the end of $C$
where the $|s|\to\infty $ limit of $\theta $ is 0. This composition then assigns to $C$
the $|s|\to\infty $ limit of $-\varepsilon (\hat{p}_{o}\varphi -{\hat {p}_{o}}'t)$.
\end{itemize}
\end{proposition}

Formal proofs of these last two propositions are omitted as both
follow directly from \eqref{eq1.8} and \eqreft25 using the given
definitions in \fullref{sec:3c} above. The remarks that
follow are meant to indicate what is going on. Consider first the
assertion in \fullref{prop:3.5}. The point here is that
if $e$ is the edge that ends at $o$, and if $(\sigma , v)$
parametrize $K_{e}$ via \eqref{eq1.8}, then \eqreft25 finds
\begin{equation}\label{eq3.23}
\hat {p}_{o}\varphi -{\hat {p}_{o}}'t = \bigl(\hat {p}_{o}q' -
{\hat {p}_{o}}'q\bigr) v
+ \hat {p}_{0}\surd 6 \cos(\sigma )-{\hat{p}_{o}}'(1 - 3\cos^{2}\sigma ) w.
\end{equation}
According to \eqref{eq1.8}, the coefficient in front of $w$
vanishes when $\sigma $ is the angle that is assigned to the
vertex o. As a consequence, the $|s|\to\infty $ limit of $\hat
{p}_{o}\varphi -{\hat {p}_{o}}'t$ on the end that is associated to
$o$ is equal to $(\hat{p}_{o}q' - {\hat {p}_{o}}'q) v_{o}$ (modulo
$2\pi \mathbb{Z})$, where $v_{o}$ is the coordinate of the
missing point on the boundary of the parametrizing cylinder that
corresponds to the end in question.

The conclusions of \fullref{prop:3.6} are derived using similar considerations.
For example, in the first case of the proposition, the left most term in
\eqref{eq3.23} is zero on the end in question and the right most term is $\alpha_{Q}(\sigma ) w$.
Thus, the claim follows directly using \eqref{eq3.15}.

\begin{proof}[Proof of  \fullref{prop:3.4}]
The result follows using the previous two propositions in
conjunction with \cite[Proposition~2.13]{T3}. To elaborate, note that the differential
structure on ${\mathcal{M}_{\hat{A},T}}^{\Lambda }$ is that given by
\fullref{thm:1.1}. \cite[Proposition~2.13]{T3} guarantees that
the angles from Propositions~\ref{prop:3.5} and~\ref{prop:3.6}
provide local coordinates.
\end{proof}


\setcounter{section}{3}
\setcounter{equation}{0}

\section{Proving diffeomorphisms}\label{sec:4}

The previous section introduced the submanifold $\mathcal{M}
\subset   \mathcal{M}_{\hat{A}}$ of subvarieties with a
corresponding $T_{(\cdot )}$ graph that is linear, and it
described a map from $\mathcal{M}$ to $\mathbb{R}  \times O^{\hat{A}}/\Aut^{\hat{A}}$.
This map is denoted in
what follows by $\mathfrak{B}$. This section proves that the map $\mathfrak{B}$ is a
dffeomorphism onto $\mathbb{R}  \times \hat{O}^{\hat{A}}/\Aut^{\hat{A}}$. The proof of this
assertion verifies \fullref{thm:3.1}. Meanwhile, the proof introduces
various techniques that are used in the subsequent sections to
analyze the whole of ${\mathcal{M}^{*}}_{\hat{A}}$ and the
other multiply punctured sphere moduli spaces.

The proof that $\mathfrak{B}$ is a diffeomorphism has three parts. The first
part establishes that the map is 1--1 onto its image. The second
proves that the image is in $\mathbb{R}  \times \hat{O}^{\hat{A}}/\Aut^{\hat{A}}$. The third proves
that the map is proper onto $\mathbb{R}  \times \hat{O}^{\hat{A}}/\Aut^{\hat{A}}$. Granted that the map
is 1--1 and that its image is $\mathbb{R}  \times \hat{O}^{\hat{A}}/\Aut^{\hat{A}}$, \fullref{prop:3.4}
establishes that the map is a local diffeomorphism. Granted that
$P$ is a proper map onto $\mathbb{R}  \times \hat{O}^{\hat{A}}/\Aut^{\hat{A}}$, it diffeomorphically
identifies $\mathcal{M}$ with $\mathbb{R}  \times \hat{O}^{\hat{A}}/\Aut^{\hat{A}}$.

The first subsection below introduces the machinery that is used
to prove both that $\mathfrak{B}$ is 1--1 and that its image is in $\mathbb{R} \times \hat{O}^{\hat{A}}/\Aut^{\hat{A}}$. The
former conclusion is established in the subsequent subsection.
\fullref{sec:4c} constitutes a digression to describe
$O^{\hat{A}}- \hat{O}^{\hat{A}}$, and then \fullref{sec:4d} contains the proof that $P$
maps $\mathcal{M}$ onto $\mathbb{R}  \times \hat{O}^{\hat{A}}/\Aut^{\hat{A}}$. The final subsection
explains why $\mathfrak{B}$ is a proper as a map onto $\mathbb{R}  \times \hat{O}^{\hat{A}}/\Aut^{\hat{A}}$.

\subsection{Graphs and subvarieties}\label{sec:4a}

One can ask of any two elements from some version of ${\mathcal{M}^{*}}_{\hat{A}}$ whether one is obtained from the
other by a constant translation along the $\mathbb{R}$ factor of
$\mathbb{R}  \times (S^1\times S^2)$. This section
provides a sufficient condition for such to be the case. This
condition is summarized by \fullref{lem:4.1}, below. The next subsection
considers the consequences of \fullref{lem:4.1} in the case that both
subvarieties come from \fullref{thm:3.1}'s moduli space $\mathcal{M}$.

To set the stage for \fullref{lem:4.1}, suppose that $(C_{0}, \phi )$
and $({C_{0}}', \phi ')$ define points in some version of ${\mathcal{M}^{*}}_{\hat{A}}$. Then the same version of $T$ must
have correspondences in both $(C_{0}, \phi )$ and in $({C_{0}}',\phi ')$. Choose respective correspondences and
denote them as $T_{C}$ and $T_{C'}$.

For each edge $e \subset T$, the parametrizing cylinder for
the $e$--labeled component in $C_{0}- \Gamma$ is the same
as that used for ${C_{0}}'- \Gamma '$. This understood, say
that respective parametrizations of the $C_{0}$ and ${C_{0}}'$
versions of $K_{e }$ are  compatible when the following two conditions are met:

\itaubes{4.1}
\textsl{When $o$  is a multivalent vertex on $e$, then the respective extensions of the two
parametrizations to the $\sigma =\theta _{o}$  circle in
the parametrizing cylinder have identical sets of missing
points, and also identical sets of singular points. Moreover,
these identifications define identical versions of the circular
graph $\ell _{oe}$  in the sense that a given point has
the same integer label whether viewed in the $C_{0}$  or
in the ${C_{0}}'$  version.}

\item
\textsl{This identification between the two versions of $\ell_{oe}$  is the same as that given by the chosen
correspondences $T_{C}$  and $T_{C'}$.}
\end{itemize}

Assume now that the respective parametrizations of the $C_{0}$
and ${C_{0}}'$ versions of each $K_{(\cdot )}$ are compatible.

To continue, suppose again that $e$ is an edge in $T$, and define
functions $(\hat{a}_{e}, \hat{w}_{e})$ on any given $C_{0}$
version of $K_{e}$ as follows: Let $(a_{e}, w_{e})$ denote the
versions of the functions that parametrize $C_{0}$'s version of
$K_{e}$ via \eqreft25, and let $({a_{e}}', {w_{e}}')$ denote those that
appear in the ${C_{0}}'$ version. Both versions are pairs of
functions on the parametrizing cylinder. Define 
$\hat {a}_{e}  = a_{e}- {a_{e}}'$ and 
$\hat {w}_{e}  = w_{e}- {w_{e}}'$.

Now make the following assumption:

\qtaubes{4.2}
{\sl There exists a continuous function on the complement in
$C_{0}$  of the critical point set of $\cos\theta $ whose pull-back from
$C_{0}- \Gamma $  via the various parametrizing
maps gives the collection $\{\hat {w}_{e}\}$.}
\endqtaubes

Let $\hat {w}$ denote this continuous function. As is explained
below, such a function exists if and only if the identification
between the $C_{0}-\Gamma $ and ${C_{0}}'- \Gamma '$ via
the compatible parametrizing maps extends to define a
homeomorphism between $C_{0}$ and ${C_{0}}'$ that identifies
respective $\cos(\theta )$ critical point sets and is
differentiable in their complements.

To finish the stage setting, introduce $G$ to denote the portion of
the $\hat {w} = 0$ locus that lies in the complement of the set
of critical points of $\cos\theta $ on $C_{0}$. The guarantee that
$\phi '$ and $\phi $ are translates of each other involves $G$:

\begin{lemma}\label{lem:4.1}
If $G \ne${\o}, then $\phi '$  is obtained from $\phi $  by composing
the latter with a constant translation along the $\mathbb{R}$ factor of $\mathbb{R}  \times (S^1  \times S^2)$.
\end{lemma}

The proof of this lemma has three fundamental inputs. The first
concerns the edge labeled set $\{\hat{a}_{e}   \equiv a_{e}- {a_{e}}'\}$:

\begin{lemma}\label{lem:4.2}
There is a continous function on $C_{0}$  that pulls back from $C_{0}- \Gamma$
via the parametrizing maps as the collection $\{\hat{a}_{e}\}$.  Moreover, this function is smooth
away from the critical points of the pull-back of $\theta $.
\end{lemma}
\noindent Let $\hat{a}$ denote the function from \fullref{lem:4.2}.

The second input concerns the nature of $G$:

\begin{lemma}\label{lem:4.3}
If $G$  is neither empty nor all of the complement of the $\cos\theta $  critical point
set, then it has the structure of an embedded, real analytic
graph. In addition:
\begin{itemize}
\item
Each edge is an embedded arc whose
interior is oriented by the pull-back of $d\hat{a}$.

\item
The number of incident edges to any given vertex is even
and at least four. Moreover, any circle about a vertex with
sufficiently small and generic radius misses all vertices of
$G$ , and intersects the edges transversely. Furthermore, a
circumnavigation of such a circle alternately meets inward
pointing and outward pointing incident edges.

\item
Any sufficiently small and generic radius circle about a
critical point of $\cos\theta $  misses all vertices of $G$
and intersects the edges transversely. Furthermore, a
circumnavigation of such a circle alternately meets inward
pointing and outward pointing incident edges.

\item
If $R$  is sufficiently large and generic, then the $|s| = R$ locus misses the vertices $G$
and intersects the edges transversely. This locus
intersects an even number of edges and a circumnavigation about
any given component of the $|s| = R$ locus
alternately meets edges where $|s|$  is
respectively increasing and decreasing in the oriented direction.

\item
 Let $E$  denote an end of $C_{0}$. Then, $\hat{a}$  is bounded on $G \cap E$ and it has a
unique $|s|    \to   \infty $  limit on $G \cap  E$.
\end{itemize}
\end{lemma}

The third input, and part of \fullref{lem:4.2}'s proof, is a certain
Cauchy--Riemann equation that is obeyed any given pair $(\hat{a}_{e}, \hat {w}_{e})$:
\begin{equation}\label{eq4.3}
\begin{split}
\alpha _{Q_e } \hat{a}_{e_\sigma }&-\surd 6 \sin\sigma
 (1 + 3 \cos^{2}\sigma ) \big(w_{e} \hat{a}_{e v}+\hat {w}_{e}
{a_{e}}'_{v}\big) = -\frac{{1 + 3\cos ^4\sigma } }{ {\sin
\sigma }}\hat {w}_{e v}
\\
(\alpha _{Q_e} \hat{w}_{e})_{\sigma }&-\surd 6 \sin\sigma  (1
+ 3 \cos^{2}\sigma ) \big(w_{e}  \hat {w}_{e v }+\hat{w}_{e}{w_{e}}'_{v}\big)=\frac{1 }{ {\sin \sigma}}\hat{a}_{ev},
\end{split}
\end{equation}
Indeed, this last equation appears when the $({a_{e}}', {w_{e}}')$
version of \eqref{eq2.6} is subtracted from the $(a_{e}, w_{e})$ version.

If Lemmas \ref{lem:4.2} and \ref{lem:4.3} are taken on faith for the
moment, here is the proof of \fullref{lem:4.1}.

\begin{proof}[Proof of \fullref{lem:4.1}]
If $G$ $ \ne  ${\o}, then there are two possibilities to
consider, that where the closure of $G$ is all of $C_{0}$, and that
where it is not. To analyze the first, appeal to \eqref{eq4.3} to
conclude that each $\hat{a}_{e}$ is constant when the closure of $G$
is $C_{0}$. According to \fullref{lem:4.3}, all of these constants are
the same; $\phi '$ is obtained from $\phi $ by composing with a
constant translation along the $\mathbb{R}$ factor in $\mathbb{R}\times (S^1  \times {\rm g}^{2})$.

As is argued next, the case that $G$'s closure is not $C_{0}$ can
not occur unless $G$ is empty. Here is why: Since $\hat{a}$ is bounded
on $G$, it has some supremum. By virtue of the first point in \fullref{lem:4.3}, this supremum is
not achieved in the interior of any edge.
By virtue of the second point, it is not achieved at any vertex.
Meanwhile, the third point implies that it is not achieved on the
closure of $G$. The fourth and fifth points in \fullref{lem:4.3} imply that
the supremum is not an $|s|    \to   \infty$ limit of $\hat{a}$ on $G$. This exhausts all possibilities and so $G=
${\o}.
\end{proof}

To tie up the first of the loose ends, here is the proof of \fullref{lem:4.2}.

\begin{proof}[Proof of \fullref{lem:4.2}]
Let $e$ $ \in T$ denote a given edge. A diffeomorphism is defined from $C_{0}$'s version of $K_{e}$ to
the $C_{0}'$ version by pairing respective points that have the
same inverse image in the parametrizing cylinder. Let $\psi_{e}$ denote the latter map.

As is explained momentarily, there is a homeomorphism between
$C_{0}$ and ${C_{0}}'$ that is smooth away from the $\theta $
critical set and whose restriction to any given $K_{e}$ is the
corresponding $\psi _{e}$. This homeomorphism is denoted by $\psi$. Then $\hat{a} = \phi *s - \psi *\phi '*s$,
and so $\hat{a}$ is continuous and smooth away from the $\theta $--critical set.

With the preceding understood, consider now the asserted
existence of $\psi$. To start the analysis, focus attention on a
multivalent vertex, $o \in T$. At issue here is whether the
relevant versions of $\psi _{(\cdot )}$ fit together across
the locus $\Gamma _{o}$ in $C_{0}$. For this purpose, let $e$
denote one of $o$'s incident edges and let $\gamma $ denote an arc
in $\ell _{oe}$ with the latter viewed as the $\sigma =\theta_{o}$ circle in the parametrizing cylinder for $K_{e}$. By
virtue of the first point in \eqreft41, any point in the interior of
$\gamma $ simultaneously parametrizes a unique point in
$C_{0}$'s version of $\Gamma _{o}$ and also one in the ${C_{0}}'$
version. This fact gives the map $\psi _{e}$ a canonical
extension as a homeomorphism from a neighorhood of the $\theta =\theta _{o}$ boundary in
$C_{0}$'s version of the closure of $K_{e}$ to a neighorhood of the $\theta =\theta _{o}$ boundary in
the closure of the ${C_{0}}'$ version.

To continue, let $e'$ denote $\gamma$'s other edge label. Just as
with $\psi _{e}$, the extension of $\psi _{e'}$ to $\gamma
$ defines a homeomorphism from its interior to that of an arc in
the ${C_{0}}'$ version of $\Gamma _{o}$. Together, \eqref{eq2.14} and the
second point in \eqreft41 imply that these two homeomorphisms agree.
As this conclusion holds for any arc of any version of $\Gamma_{o}$, it thus follows from
\eqreft41 that the collection $\{\psi_{e}\}$ is the restriction of a homeomorphism from $C_{0}$ to
${C_{0}}'$.

Having established that $\psi $ is at least continuous across
$\Gamma $ in $C_{0}$, the next step is to verify that it is
smooth across this locus. For this purpose, return to the arc
$\gamma $ in $\Gamma _{o}$, and let $e$ and $e'$ again denote its
edge labels. Fix respective lifts, $\hat {v}_{e}$ and $\hat{v}_{e'}$, for the $\mathbb{R}/(2\pi \mathbb{Z})$ coordinate
functions on the $e$ and $e'$ versions of the parametrizing
cylinder. Having done so, an integer pair arises in $C_{0}$'s
version of \eqref{eq2.14} and \eqref{eq2.15}, and a similar pair arises in the
${C_{0}}'$ version. Let $L = (\ell , \ell ')$ denote the
difference between these two integer pairs. The assumptions in
\eqreft42 force $\alpha _{L}(\theta _{o})$ to vanish. This understood,
then the assumption in \eqref{eq4.3} forces $(q_{e}\ell ' -
q_{e'}\ell )$ to vanish as well. These two vanishing
conditions require $L$ to vanish.

Granted that $L= 0$, fix a small radius disk about any given point
in the interior of $\gamma $'s image in the version of $\Gamma_{o}$ from $C_{0}$. Let $D$ denote this disk.
If the radius of $D$ is
very small, then the parametrizing maps for both the $C_{0}$
versions of $K_{e}$ and $K_{e'}$ extend to parametrize D.
There is an analogous $D'$ in ${C_{0}}'$. If the radius of $D$ is
small, both $\psi _{e}$ and $\psi _{e'}$ extend as
diffeomorphisms from $D$ onto an open set in $D'$. Since $\psi $ is
continuous, these extensions agree on the $\theta=\theta _{o}$ locus in $D$. However, since $L= 0$, they agree on the
whole of $D$. Thus, $\psi $ is smooth on $D$.
\end{proof}

\begin{proof}[Proof of \fullref{lem:4.3}]
The analysis of $G$ has four parts. The first studies $G$ in $C_{0}- \Gamma $, the second examines
how $G$ intersects $\Gamma $. Together, these prove that $G$ has the
structure of a graph. The third part establishes the third and
fourth points of the lemma. The final part establishes the final
point about $\hat{a}$.

\step{Part 1}
Fix an edge, $e \subset T$, so as to
consider the part of the $\hat {w} = 0$ locus that lies in
$K_{e}$. In this regard, assume that this locus is not the whole
of $K_{e}$. It is a consequence of \eqref{eq4.3} that the pair
$(\hat{a}_{e}, \hat {w}_{e})$ are real analytic on the interior
of the parametrizing cylinder for $K_{e}$. Thus, the zero locus
of $\hat {w}$ in $K_{e}$ is a 1--dimensional, real analytic
variety. In particular, this gives it the structure of a graph
whose vertices are the points where both $d\hat {w}$ and $\hat{w}$ are zero. An edge of
$G$ is the closure of a component of the
complement in the $\hat {w} = 0$ locus of the set where $d\hat
{w} = 0$. As a consequence of \eqref{eq4.3}, the 1--form $d\hat{a}_{e}$ is
zero at any given point on the $\hat {w} = 0$ locus if and only
if $d\hat {w}_{e}$ is also zero there. It also follows from \eqref{eq4.3}
that $d\hat{a}_{e}$ is not a multiple of $d\hat {w}_{e}$ at any
point on the $\hat {w}_{e} = 0$ locus where both are non-zero.
This implies that $d\hat{a}_{e}$ pulls-back to orient each edge of
$G$. Moreover, the latter orientation is consistent with the one
that comes by viewing the interior of any given edge as a portion
of the boundary of the region where $\hat {w}_{e}   \le 0$.

To continue, remark that \eqref{eq4.3} also implies that $G$ has no
isolated points in $K_{e}$.\\
To see why, first introduce the function
\begin{equation}\label{eq4.4}
\sigma \to k(\sigma ) = \bigg(\frac{1 }{ {(1 + 3\cos
^4\sigma )\alpha _{Q_e } (\sigma )}}\biggr)^{1 / 2 }.
\end{equation}
Next, set $\eta    \equiv \hat{a}_{e}- i k \hat
{w}_{e}$. Now fix any point in the parametrizing cylinder, and
standard arguments find a complex coordinate, $z$, centered on the
point and a disk about that point that makes \eqref{eq4.3} equivalent to
a single complex equation that has the form
\begin{equation}\label{eq4.5}
\bar {\partial }\eta  + u \hat {w}_{e} = 0.
\end{equation}
Here, $u$ is a smooth function on the disk. As $\eta $ is real
analytic, so it has a non-trivial Taylor's expansion about any
given point; and \eqref{eq4.5} implies that this expansion about any
given $\hat {w}_{e} = 0$ point has the form
\begin{equation}\label{eq4.6}
\eta  = r + c z^{p} + o\big(|z| ^{p + 1}\big),
\end{equation}
where $r$ is real and $p$ is a positive integer. This last
observation precludes local extrema for $\hat {w}_{e}$ where
$\hat {w}_{e}$ is zero. Therefore, $G$ has no isolated points in
$K_{e}$.

Equation \eqref{eq4.6} also implies that each vertex of $G$ in $K_{e}$ has
at least four, and an even number of incident edges.

In any event, the second point in \fullref{lem:4.3} for vertices in
$K_{e}$ follows using transversality theory and the fact that the
$d\hat{a}$ orientation on any given edge is consistent with its
orientation as part of the boundary of the region where $\hat {w} \le 0$.

\step{Part 2}
Let $o \in T$ be a bivalent vertex, and let
$\gamma    \subset \Gamma _{o}$ be an arc. This part of
the proof describes $G$'s behavior in a neighborhood of $\gamma $.
To start, focus attention on a small radius disk about a point in
the interior of $\gamma $ whose closure lies some positive
distance from the vertices of $\gamma $. Let $e$ and $e'$ denote the
edges that label $\gamma $. Thus, $\gamma $ is in the
closure of both $K_{e}$ and $K_{e'}$. Because the pairs
$(a_{e}, w_{e})$ and $({a_{e}}', {w_{e}}')$ extend over the
$\sigma =\theta _{o}$ boundary of $K_{e}$'s parametrizing
domain, so the functions $(\hat{a}_{e}, \hat {w}_{e})$ also
extend. By taking the chosen disk to have small radius, the disk
can be assumed to lie in the image of this extention of the
parametrizing domain. Likewise, the pair $(\hat{a}_{e'},
\hat {w}_{e'})$ can be assumed to extend from the
parametrizing cylinder of $K_{e'}$ to a subset in $[0,
\pi ]\times   \mathbb{R}/(2\pi \mathbb{Z})$ that maps to the
chosen disk via an extension of the $e'$ version of \eqreft25.

A comparison of the respective extensions of $(\hat{a}_{e}, \hat
{w}_{e})$ and $(\hat{a}_{e'}, \hat {w}_{e'})$ can
be made using the corresponding $C_{0}$ and ${C_{0}}'$ versions of
\eqref{eq2.14} and \eqref{eq2.15}. For this purpose, fix respective lifts, $\hat{v}_{e}$ and
$\hat {v}_{e'}$ chosen for the $\mathbb{R}/(2\pi \mathbb{Z})$ coordinate functions on the $e$ and $e'$
versions of the parametrizing cylinder. As remarked in the proof
of \fullref{lem:4.2}, this can be done so that respective integer pairs
that appear in the $C_{0}$ and ${C_{0}}'$ versions of \eqref{eq2.14} and
\eqref{eq2.15} are equal. Granted this, \eqref{eq2.15} implies that the
coordinate transformation that is defined by \eqref{eq2.14} on the given
disk identifies $(\hat{a}_{e}, \hat {w}_{e})$ with
$(\hat{a}_{e'}, \hat {w}_{e'})$. As a consequence,
the use on this disk of the extension of $K_{e}$'s parametrizing
cylinder coordinates identifies $\hat {w}^{- 1}(0)$ with the
zero locus of $\hat {w}_{e}$.

With this last identification understood, then $G$'s intersection
with the chosen disk must be a graph that has the structure
described in Part 1 above for $G \cap K_{e}$. Moreover, any
give edge of this intersection is either contained in the $\sigma=\theta _{o}$ locus, or it
intersects this locus in a finite set.
By the way, if $G$'s intersection with the chosen disk has an edge
in the $\sigma =\theta _{o}$ locus, then it contains the whole of
$\gamma $ since this locus and $G$ are real analytic.

\step{Part 3}
The transversality assertions in the third and
fourth points follow from Sard's theorem. The assertion about the
alternating orientation follows from the fact that the $d\hat{a}$
orientation on any given edge is consistent with its orientation
as a part of the boundary of the region where $\hat {w}  \le 0$.

\step{Part 4}
The proof of the fifth point of the lemma
requires some understanding of the behavior of $\hat{a}$ on the ends
of $C_{0}$. For this purpose, suppose first that $E \subset C_{0}$ is an end with a
$(\pm 1,\ldots)$ label
from $\hat{A}$. The identification $T_{C} = T_{C'}$ pairs $E$
with a ${C_{0}}'$ end, $E'$, that has the same label. This
understood, a comparison between the $E$ and $E'$ versions of \eqref{eq1.9}
proves that $\hat{a}_{e}$ is bounded on $E$ and has a unique limit on $E$ as $\theta \to  0$.

Next, suppose that $E \subset C_{0}$ is an end with a
$(0,-,\ldots)$ label from $\hat{A}$ where the constant
$n_{E}$ in \eqref{eq2.4} is zero. Let $E'$ denote the corresponding end in
${C_{0}}'$ as defined by the chosen identification of $T_{C}$ and
$T_{C'}$. Let $b$ and $b'$ denote the respective constants
that appear in the $E$ and $E'$ versions of \eqref{eq2.14}. It then follows
from \eqref{eq2.14} that
\begin{equation}\label{eq4.7}
\hat{a} = \frac{1 }{ r} \ln\bigg(\frac{b }{ {b'
}}\bigg) + o(1) ,
\end{equation}
where $r$ is the constant that appears in \eqref{eq2.4}. Here,
the term designated as $o(1)$ limits to zero as $\sigma $ limits to
the angle $\theta _{E}$, thus as $|s|    \to
 \infty $ on $E$.

Finally, suppose that $E \subset C_{0}$ is an end with a $(0,\cdots)$ label from $\hat{A}$ where
the integer $n_{E}$ in \eqref{eq2.4} is non-zero. It then follows from \eqref{eq2.4} that $E$
corresponds to some multivalent vertex, $o \in T$. Now,
equation \eqref{eq2.4} is one of a pair of equations that describe $E$ at
large $|s|$. To write this second equation, note that
$(1-3\cos^{2}\theta _{o}) d\varphi -\surd 6 \cos\theta _{o}dt$
is exact on $E$. This understood, in \cite[Equation (2.13)]{T4} with the
analysis from of \cite[Section 2]{T4} can be used to prove that any
anti-derivative of this 1--form appears as follows: Let $f$ denote
the antiderivative, and let $(\rho, \tau )$ denote the
coordinates that appear in \eqref{eq2.4}. Then
\begin{equation}\label{eq4.8}
f(\rho , \tau )=f_{E} + e^{ - r\rho }(\kappa b
\sin(n_{E}(\tau +\sigma )) + \hat{o}) ,
\end{equation}
where $n_{E}$, $r$ and $b$ are the same constants that
appear in $E$'s version \eqref{eq2.4}. Meanwhile, $\kappa $ is a constant
that depends only $n$ and on $E$'s label in $\hat{A}$. Finally, $f_{E}$
is a constant and $\hat{o}$ is a term that limits to zero as $\rho  \to   \infty $.

To exploit \eqref{eq2.4} and \eqref{eq4.8} for the present purposes, keep in mind
that the coordinate $\rho $ is a constant, positive multiple of
the function $s$. In this regard, it is important to realize that
this multiplier depends only on $E$'s label in $\hat{A}$.

To see how \eqref{eq2.4} and \eqref{eq4.8} describe the behavior of $\hat{a}$ on $G\cap E$,
suppose that $e$ is an edge of $T$ that is incident to $o$,
and suppose that the closure of $K_{e}$ has non-compact
intersection with $E$. This being the case, the end $E$ then
corresponds to a vertex on $\ell _{oe}$ and thus a missing point
on the $\sigma =\theta _{o}$ circle in $C_{0}$'s version of the
parametrizing cylinder for $K_{e}$. Fix an $\mathbb{R}$--valued
lift, $\hat {v}$, of the coordinate $v$ of this parametrizing
cylinder that is defined in a contractible neighborhood of $E$'s
missing point. To save on notation, assume that $\hat {v}$ has
value zero at this missing point. It then follows from \eqref{eq2.4} that
\begin{equation}\label{eq4.9}
f \equiv \alpha  _{Q_e } (\theta _{o})  \hat {v}+\surd
6 \Big[\big(1-3\cos^{2}\theta _{o}\big) \cos\theta  - \cos\theta _{o}
\big(1-3\cos^{2}\theta \big)\Big] w_{e}.
\end{equation}
is a pull-back of an antiderivative of
$(1-3\cos^{2}\theta _{o})dt - \surd 6 \cos\theta _{o} d\varphi$ to a neighborhood of $E$'s
missing point in the $C_{0}$ version of the parameter cylinder for $K_{e}$.

The respective identifications of $T$ with both $T_{C}$ and
$T_{C'}$ pairs $E$ with an end, $E'  \subset {C_{0}}'$ that
also corresponds to a vertex on $\underline {\Gamma }_{o}$. In
this regard, the vertex is the same as the one given $E$, and this
implies that the $E'$ versions of \eqref{eq2.4} and \eqref{eq4.8} use the same
integer $n_{E}$. Even so, the $E'$ version can employ a different
positive constant, $b'$.

There is also a primed version of \eqref{eq4.9}, this denoted by $f'$.
Note, in particular, that $f- f'$ is zero on the portion of $G \cap E$ in the closure of $K_{e}$.
This being the case, \eqref{eq2.4}
and \eqref{eq4.8} imply that $\hat{a}$ is bounded on this part of $G \cap E$, and that it has a
unique $|s|    \to\infty $ limit here which is given by \eqref{eq4.7}. Since the pair $b$ and
$b'$ are determined by $E$ and $E'$ without regard for the edge $e$, so
the function $\hat{a}$ has a unique $|s| \to
 \infty $ limit on $G \cap E$.
\end{proof}

\subsection{Why the map from $\mathcal{M}$ to $\mathbb{R}  \times O^{\hat{A}}{/\Aut}^{\hat{A}}$
is 1--1}\label{sec:4b}

The previous section defined a continuous map, $P$, from $\mathcal{M}$
to $\mathbb{R}  \times O^{\hat{A}}/\Aut^{\hat{A}}$ and the purpose of this
subsection is to explain why the latter map is 1--1. The
explanation that follows is in three parts.

\step{Part 1}
To start, suppose that $C$ and $C'$ have the same
image. As is explained in a moment, this requires that the same
version of $T$ have correspondences in both $(C_{0}, \phi )$ and
in $(C_{0}, \phi ')$. Granted that such is the case, choose
such correspondences. Such choices are implicit in Parts 2 and 3
that follow.

Let $T_{C}$ denote a graph with a correspondence in $(C_{0}, \phi )$ and let $T_{C'}$ denote
one with a correspondence
in $(C_{0}, \phi ')$. To explain why $T_{C}$ must be isomorphic
to $T_{C'}$, remark first that the respective vertex sets
enjoy a 1--1 correspondence that preserves angles, and that the
latter correspondence induces one between the respective edge
sets that preserves the integer pair labels. Thus, the issue here
is whether the circular graphs that label the respective equal
angle bivalent vertices in $T_{C}$ and $T_{C'}$ are
isomorphic.

To see that such is the case, let $o$ denote a given bivalent
vertex in $T_{C}$ and let $o'$ denote its equal angle partner in
$T_{C'}$. Choose admissable identifications between
$\underline {\Gamma }_{o}$ and $\hat{A}_{o}$, and likewise
between $\underline {\Gamma }_{o'}$ and $\hat{A}_{o}$.
Here as before, an identification is deemed to be admissable when
the second component of a 4--tuple gives the sign of the integer
label of the corresponding vertex while the greatest common
divisor of the integer pair component gives the absolute value of
the integer label. These identifications induce respective cyclic
orderings of $\hat{A}_{o}$. These cyclic orderings are relevant
because $\underline {\Gamma }_{o}$ is isomorphic to $\underline {\Gamma }_{o'}$ as a labeled graph if and only if one
cyclic ordering is obtained from the other by composing with a
1--1 self map of $\hat{A}_{o}$ that only permutes identical 4--tuples.

To see that the cyclic orderings have the desired relation, take
a lift of the common image point to $\mathbb{R}_{ - }  \times
 \mathbb{R}^{\hat{A}}$. The image of $\hat{A}_{o}$ by the
$\mathbb{R}^{\hat{A}}$ component of this lift defines a set of
$n_{o}$ distinct points in $\mathbb{R}/(2\pi \mathbb{Z})$ and thus a
cyclic ordering for $\hat{A}_{o}$. As dictated by the construction
in \fullref{sec:3c}, this cyclic ordering is obtained from those
induced by the ordering of the vertices on the respective $T_{C}$
and $T_{C'}$ versions of $\underline {\Gamma }_{o}$ by
composing the latter with an appropriate permutation of
$\hat{A}_{o}$ that mixes only elements with identical 4--tuples.
Thus, the induced cyclic orderings on $\hat{A}_{o}$ from the
$T_{C}$ and $T_{C'}$ versions of $\underline {\Gamma
}_{o}$ are related in the desired fashion.

\step{Part 2}
The respective images of $C$ and $C'$ in $O^{\hat{A}}/\Aut^{\hat{A}}$ are defined by first
assigning these subvarieties points in the space depicted in
\eqref{eq3.14}. As the images of $C$ and $C'$ agree in
$O^{\hat{A}}/\Aut^{\hat{A}}$, so their images agree in
\eqref{eq3.12}; and this means that the choices that are used to define
the respective assignments in \eqref{eq3.14} can made so that these
assignments agree. This has the following consequences: First,
the conditions in \eqreft41 are met at each multivalent vertex. To
see why this is so, let $o \in T$ denote such a vertex and let
 $e$ denote the incident edge with maximal angle $\theta _{o}$.
Since $C$ and $C'$ have the same image in $\Delta _{o}$, the spacing
of the missing points on the $\sigma =\theta _{o}$ circle in the
parametrizing cylinder for $K_{e}$ agree. Moreover, as the $C$ and
$C'$ versions of the assigned point in $\mathbb{R}_{o}$ agree, so the
$\mathbb{R}/(2\pi \mathbb{Z})$ coordinates of the respective
distingushed missing points on the $\sigma =\theta _{o}$ circle
must agree. This then implies that the respective $C$ and $C'$
versions of the missing point set on the $\sigma =\theta _{o}$
circle agree. Given the definitions of the respective $T_{C}$ and
$T_{C'}$ versions of $\Gamma _{o}$, this agreement of
missing point sets induces the isomorphism between the two
versions of $\ell _{oe}$ that comes from the given identification
between $T_{C}$ and $T_{C'}$ as the graph $T$.

The fact that $C$ and $C'$ define the same points in \eqref{eq3.14} has
additional consequences. To describe the first, let $e$ denote for
the moment the edge of $T$ that contains the minimal angle vertex
of $T$. Since $C$ and $C'$ have the same image in $\mathbb{R}_{ - }$,
there are respective parametrizations of $K_{e}$ that make their
versions of \eqref{eq3.15} agree. Choose such parametrizations. Now, let
 $o$ denote the vertex on $e$ with the larger angle and, if $o$ is
bivalent, let $e'$ denote the second of $o$'s incident edges. The
point assigned both $C$ and $C'$ in the line $\mathbb{R}_{o}$ endows
$K_{e'}$ with its `canonical' parameterization. This
parameterization is such that the integer pair $(n, n')$ that
appears in \eqref{eq2.15} is zero. As such is the case for both the $C$ and
$C'$ versions, so the functions $\hat {w}_{e}$ and $\hat
{w}_{e'}$ agree along the $\theta =\theta _{o}$ locus in $C$.

\step{Part 3}
Suppose that $o$ is a bivalent vertex. Let $e$ and
 $e'$ denote the incident edges to $o$ with $e$ the edge where the
maximum of $\theta $ is $\theta _{o}$. Suppose that both the $C$
and $C'$ versions of $K_{e}$ are given their `canonical'
parameterization as defined inductively using the following data:
First, the already chosen parametrizations for the respective $C$
and $C'$ versions of the component of $C_{0}- \Gamma $ whose
labeling edge has the minimal angle vertex in $T$. Second, the
assigned point in each $\mathbb{R}_{(,)}$ factor from \eqref{eq3.14} whose
label is a multivalent vertices with angle less than $\theta
_{o}$. The canonical parametrizations of the two versions of
$K_{e}$ and the assigned point in $\mathbb{R}_{o}$ endows both the $C$
and $C'$ versions of $K_{e'}$ with a `canonical'
parameterization. The parametrizations of the $C$ and $C'$ versions
of $K_{e}$ give the function $\hat {w}_{e}$ on $C$'s version of
$K_{e}$, while that of the two versions of $K_{e'}$ give
$\hat {w}_{e'}$ on the $C'$ version of $K_{e'}$. This
understood, the argument used at the end of the previous
paragraph finds that $\hat {w}_{e}=\hat {w}_{e'}$ along
the $\theta =\theta _{o}$ locus in $C$.

Granted the preceding, it then follows that \eqreft42 holds when the
components of the $C$ and $C'$ versions $C_{0}- \Gamma $ are
given the parametrizations just described. Since $C$ and $C'$ have
the same image in the $\mathbb{R}$ factor of $\mathbb{R}  \times O_{T}/\Aut(T)$,
the desired identity $C = C'$ follows from \fullref{lem:4.2} if the graph $G \subset  C$
is non-empty. That such is the case follows from the identity between the assigned values of $C$
and $C'$ in \eqref{eq3.14}'s factor $\mathbb{R}_{ - }$. Indeed, subtracting
the $C'$ version of \eqref{eq3.15} from the $C$ version finds that $\hat {w}$
has average 0 over any constant $\theta $ slice of the component
of $C$'s version $C_{0}- \Gamma$ whose labeling edge
contains the minimal angle vertex in $T$. As such there must be a
zero of $\hat {w}$ on each such slice.

\subsection{The complement of $\hat{O}^{\hat{A}}$  in $O^{\hat{A}}$}\label{sec:4c}

Before proving that the map from $\mathcal{M}$ to
$O^{\hat{A}}/\Aut^{\hat{A}}$ lands in
$\hat{O}^{\hat{A}}/\Aut^{\hat{A}}$, it is worthwhile to
describe the complement of $\hat{O}^{\hat{A}}$ in
$O^{\hat{A}}$, this the subspace of points where
$\Aut^{\hat{A}}$ has non-trivial stabilizer.

To start, let $o$ denote a multivalent vertex in $T^{\hat{A}}$
and let $\Cyc_{o}$ denote the set of cyclic orderings of
$\hat{A}_{o}$. The components of $O^{\hat{A}}$ are in 1--1
correspondence with $\times _{o}\Cyc_{o}$. When $v \in
\times _{o}\Cyc_{o}$, let ${O^{\hat{A}}}_{v}$ denote the
corresponding component. If $x \in  {O^{\hat{A}}}_{v}$ is
fixed by some $\iota \in  \Aut^{\hat{A}}$, then
$\iota $ must preserve the cyclic orderings that are
determined by $v$.

To see the implications of this last observation, let $o$ denote a
bivalent vertex and let $\Aut_{o,v}$ denote the group of
permutations of $\hat{A}_{o}$ that preserve the cyclic order
defined by $v$ while permuting only elements with identical
4--tuples. This is a cyclic group whose order is denoted in what
follows by $m_{o}$. The group $\times _{o}\Aut_{o,v}$ is the
subgroup of $\Aut^{\hat{A}}$ that preserves ${O^{\hat{A}}}_{v}$.

Set $k_{v}$ to denote the greatest common divisor of the integers
in the set that consists of $m_{ - }$ and the collection
$\{m_{o}\}$. Thus, each $\Aut_{o,v}$ has a unique $\mathbb{Z}/(k_{v}\mathbb{Z})$ subgroup.
Note that in the case that $k_{v}> 1$, each such subgroup has a canonical generator. To explain,
let $\iota _{o}$ denote the generator of $\Aut_{o,v}$ that moves
elements the minimal amount in the direction that increases the
numerical order in the following sense: Fix a distinguished
element in $\hat{A}_{o}$ so as to turn $v$ into a linear ordering
with the distinguished element last. Then $\iota _{o}$ moves the
distinguished element to the position numbered by
$n_{o}/m_{o}$. The canonical generator of the $\mathbb{Z}/(k_{v}\mathbb{Z})$ subgroup moves the
distinguished element to the position numbered $n_{o}/k_{v}$.

Granted the preceding, a canonical $\mathbb{Z}/(k_{v}\mathbb{Z})$
subgroup of $\Aut^{\hat{A}}$ is defined by the requirement
that its generator project to any given $\Aut_{o,v}$ so as to give
the canonical generator of the latter's $\mathbb{Z}/(k_{v}\mathbb{Z})$ subgroup.

With this understood, consider:

\begin{proposition}\label{prop:4.4}
The $\Aut^{\hat{A}}$ stabilizer of any given point in ${O^{\hat{A}}}_{v}$
 is a subgroup of the canonical $\mathbb{Z}/k_{v}\mathbb{Z}$
 subgroup.  Conversely, any subgroup of the latter
has a non-empty set of fixed points in ${O^{\hat{A}}}_{v}$.
\end{proposition}
\noindent This proposition is proved momentarily.

To view this fixed point set in a different light, fix a
distinguished element in each $\hat{A}_{(\cdot )}\subset \hat{A}_{*}$ so as to use the
description in \eqref{eq3.12} for ${O^{\hat{A}}}_{v}$. In particular, the space in \eqref{eq3.12} has
the evident projection to $\times _{o}  \Delta _{o}$ and
this projection is equivariant with respect to the action of
$\times _{o}\Aut_{o,v}$ on ${O^{\hat{A}}}_{v}$ and an action
on $\times _{o}  \Delta _{o}$. In this regard, the action on
$\times _{o}  \Delta _{o}$ is the product of the action of
each version of $\Aut_{o,v}$ on the corresponding version of
$\Delta _{o}$ that cyclically permutes the coordinates in the
manner dictated by a given element in $\Aut_{o,v}$. Here, the
latter action is obtained by labeling the Euclidean coordinates
of points in $\Delta _{o}$ by the elements in $\hat{A}_{o}$ so that
the $j$'th coordinate corresponds to the $j$'th element in $\hat{A}_{o}$
using the linear ordering that gives $v$'s cyclic ordering and has
the distinguished element last.

Granted all of this, consider:

\begin{proposition}\label{prop:4.5}
Let $G$  be a subgroup of the canonical $\mathbb{Z}/k_{v}\mathbb{Z}$  subgroup in $\Aut^{\hat{A}}$.
Then the set of points in \eqref{eq3.12} with $\Aut^{\hat{A}}$  stabilizer $G$
 is the restriction of the fiber bundle projection from $O^{\hat{A}}$  to $\times _{o}  \Delta _{o}$
 over the set of points in $\times _{o}  \Delta_{o}$  with stabilizer $G$.
\end{proposition}
\noindent The remainder of this subsection contains the following proofs.

\begin{proof}[Proof of Propositions \ref{prop:4.4} and \ref{prop:4.5}]
The proof is given in six steps.

\substep{Step 1}
Let $\Auttt $  denote the semi
direct product of the groups $\Aut^{\hat{A}}$ and $(\mathbb{Z}
 \times   \mathbb{Z})\times  \Maps(\hat{A}_{*}; \mathbb{Z})$. This group acts on
 $\mathbb{R}_{ - }  \times\mathbb{R}^{\hat{A}}$. The stabilizers of points in
$O^{\hat{A}}$ can be determined by studying the stabilizers
in $\Auttt ^{\hat{A}}$ of points in $\mathbb{R}_{ - }
\times   \mathbb{R}^{\hat{A}}$ since the image in
$O^{\hat{A}}$ of $(\tau _{ - }, x) \in   \mathbb{R}_{ -}  \times   \mathbb{R}^{\hat{A}}$ is fixed by
$g \in  \Aut^{\hat{A}}$ if and only if $(\tau _{ - }, x)$ is fixed
by an element $g\in \Auttt$ that has the form of $(g, (N, z))$
with $N \in   \mathbb{Z}  \times   \mathbb{Z}$ and $z
\in  \Maps(\hat{A}_{*}; \mathbb{Z})$. This understood,
suppose now that the point $(\tau _{ - }, x)$ projects to
${O^{\hat{A}}}_{v}$ and that $g \in   \times _{o}\Aut_{o,v}$ fixes its image.

\substep{Step 2}
To start the analysis of $g$, note that $\tau _{ - }$ is fixed if and only if $N = r_{ - }$
$Q_{e}$ where $e$ here denotes the edge in $T^{\hat{A}}$ with
the smallest angle vertex in $T^{\hat{A}}$ and where $r_{ -}$ is a fraction with
$m_{ - } r_{ - }   \in   \mathbb{Z}$.

\substep{Step 3}
Let $o$ now denote the vertex in
$T^{\hat{A}}$ with the second smallest angle. Assuming $o$ is
bivalent, g fixes both $\tau _{ - }$ and the restriction of $x$ to
$\hat{A}_{o}$ if and only if
\begin{equation}\label{eq4.10}
x(u) = x(g_{o}u) - 2\pi (z(u) + r_{ - })
\end{equation}
for all $u \in \hat{A}_{o}$. Here, $g_{o}$ is the
component of $g$ in $\Aut_{o,v}$. Let $k_{o}$ denote the order of
$g_{o}$. Using \eqref{eq4.10} some $k_{o}$ times in succession finds that
$g$ fixes both $\tau _{ - }$ and the restriction of $x$ to
$\hat{A}_{o}$ if and only if $m_{ - }r_{ - }\in \mathbb{Z}$ and
\begin{equation}\label{eq4.11}
\sum _{0 \le j< k_o } z(g_{o}^{j}u) + k_{o}r_{ - } = 0
\end{equation}
for all $u \in \hat{A}_{o}$. Note that this last
condition can be satisfied if and only if $k_{o} r_{ - }
\in   \mathbb{Z}$. Note as well that the case $r_{ - }
\in   \mathbb{Z}$ is precluded unless $g_{o}$ is trivial since
 $x$ is assumed to lie in $\mathbb{R}^{\hat{A}}$.

\substep{Step 4}
If $T^{\hat{A}}$ has a second bivalent vertex, let $\hat{o}$ denote the one that shares an
edge with $o$. Thus, $\hat{o}$ has the third smallest vertex angle.
Then $g$ fixes the restriction of $x$ to $\hat{A}_{{\hat{o}}}$ if
and only if
\begin{equation}\label{eq4.12}
x(\hat{u}) = x(g_{{\hat{o}}}\hat{u}) - 2\pi  \bigg(z(\hat{u}) + \sum
_{u \in \hat{A}_o } z(u) \varepsilon _{u}  \frac{{p_u}' p_{\hat{u}} - p_u {p_{\hat{u}}}' }{{q_{\hat{e}}}' p_{\hat{u}} -
q_{\hat{e}} {p_{\hat{u}}}'}  + r_{ - }
\frac{{q_e}' p_{\hat{u}} - q_e {p_{\hat{u}}}' }{{q_{\hat{e}}}' p_{\hat{u}} - q_{\hat{e}}
{p_{\hat{u}}}' }\bigg)
\end{equation}
for all $\hat{u} \in \hat{A}_{{\hat{o}}}$. Here,
$g_{{\hat{o}}}$ denotes $g$'s component in $\Aut_{o,v}$ and $\hat{e}$
denotes the edge in $T^{\hat{A}}$ that contains both $o$ and
$\hat{o}$.

To make sense of this last condition, use \eqref{eq4.11} as applied to
the various $g_{o}$ orbits in $\hat{A}_{o}$ to identify the sum in
\eqref{eq4.12} with
\begin{equation}\label{eq4.13}
-r_{ - }  \frac{{p_o}' p_{\hat{u}} - p_o {p_{\hat{u}}}'}{{q_{\hat{e}}}' p_{\hat{u}} - q_{\hat{e}} {p_{\hat{u}}}' }.
\end{equation}
Then, use \eqreft33 to write $P_{o}$ as $Q_{e}-Q_{{\hat{e}}}$
and thus see that the equality in \eqref{eq4.12} is exactly the equation
that results from \eqref{eq4.10} by replacing $o$ with $\hat{o}$ and $u$ with
$\hat{u}$.

This last point leads to the following conclusion: The element $g$
fixes $\tau _{ - }$ and the restriction of $x$ to both $\hat{A}_{o}$
and $\hat{A}_{{\hat{o}}}$ if and only if $m_{ - }r_{ - }
  \in   \mathbb{Z}$ and both the $(o, u)$ and $(\hat{o}, \hat{u})$
versions of \eqref{eq4.11} hold.

\substep{Step 5}
With the help of \eqreft33, essentially the same analysis can be continued in an inductive
fashion through the bivalent vertices with ever larger angle so
as to draw the following conclusions:
\begin{itemize}
\item
 \textsl{The element $g$  fixes $(\tau _{ - }, x)$ if and only if both $m_{ - }r_{ - }   \in
\mathbb{Z}$  and \eqref{eq4.10} holds for each bivalent vertex $o \in  T^{\hat{A}}$}.

\item
\textsl{These conditions are satisfiable if and only if \eqref{eq4.11}
holds for each bivalent vertex, and the latter can be satisfied
if and only if $k_{o}r_{ - }   \in   \mathbb{Z}$
 for each bivalent vertex $o \in  T^{\hat{A}}$}.
\end{itemize}

\substep{Step 6}
The `only if' direction of the
two conclusions from Step 5 have two relevant consequences, and
here is the first: Because $x \in   \mathbb{R}^{\hat{A}}$, the various versions of $k_{o}$ must agree.
Indeed, if not, let $k$ denote the smallest and let $o$ denote a
vertex where $k_{o} = k$. Now, repeat the analysis with $g$ replaced
by $g^{k}$. The corresponding $g^{k}$ versions of \eqref{eq4.11} has
$kr_{ - }$ replacing $r_{ - }$ and $o$'s version would require
$kr_{ - }   \in   \mathbb{Z}$. But then a version where
$k_{(\cdot )}\ne  k$ would need $x$ to lie outside of $\mathbb{R}^{\hat{A}}$.

Here is the second consequence: Each $k_{o}$ is a divisor of the
corresponding $m_{o}$, this the order of $\Aut_{o,v}$. As they are
all equal to the same $k\in   \mathbb{Z}$, so $k$ divides each
$m_{o}$ and so $k$ is a multiple of the greatest common divisor of
the collection $\{m_{o}\}$. In addition as $kr_{ - }
\in   \mathbb{Z}$, so $k$ must also be a divisor of $m_{ - }$.
Thus, $k$ is a multiple of $k_{v}$ and $g$ is an element in some
subgroup of the canonical $\mathbb{Z}/k_{v}\mathbb{Z}$ subgroup of
$\times _{o}\Aut_{o,v}$.

Meanwhile, the `if' directions of Step 5's observations directly
imply that any subgroup of the canonical $\mathbb{Z}/k_{v}\mathbb{Z}$
subgroup of $\times _{o}\Aut_{o,v}$ has a non-empty set of
fixed points, and that the fixed point set is described by
\fullref{prop:4.5}.
\end{proof}

\subsection{Why the map from $\mathcal{M}$ lands in
$\mathbb{R}\times \hat{O}^{\hat{A}}/\Aut^{\hat{A}}$}\label{sec:4d}

As indicated by the heading, the purpose of this subsection is to explain
why the image of the map $P$ from $\mathcal{M}$ to $\mathbb{R}  \times O^{\hat{A}}/\Aut^{\hat{A}}$
lies in $\mathbb{R}  \times \hat{O}^{\hat{A}}/\Aut^{\hat{A}}$. To this end, suppose that $C
\in   \mathcal{M}$. The image of $C$ in $O^{\hat{A}}/\Aut^{\hat{A}}$
has a lift to some component ${O^{\hat{A}}}_{v}   \subset O^{\hat{A}}$, and suppose that this lift
is fixed by some $g \in \times _{o} \Aut_{o,v}$. This being the case, $g$ is some multiple of the
generator of the canonical $\mathbb{Z}/k_{v}\mathbb{Z}$ subgroup. Let $k \in  \{0, \ldots , k_{v}-1\}$
denote this multiple. The proof that $k = 0$ has four parts.

\step{Part 1}
Fix a distinguished element in each version of $\hat{A}_{o}$ so as to write
${O^{\hat{A}}}_{v}$ as in \eqref{eq3.12}. Using this view of
${O^{\hat{A}}}_{v}$, lift $C$'s image as a point in \eqref{eq3.14} by choosing an
admissable identification between the vertex set of each $T_{C}$ version of
$\underline {\Gamma }_{o}$ and the corresponding $\hat{A}_{o}$. Denote the
resulting point in \eqref{eq3.14} as $(\tau _{ - }, (\tau _{(\cdot )},
r_{(\cdot )}))$ where $\tau _{ - }   \in   \mathbb{R}_{ - }$, each $\tau
_{o}$ is a point in the corresponding version of $\mathbb{R}_{o}$, and each
$r_{o}$ is a point in the corresponding version of $\Delta _{o}$.

Let $e$ denote for the moment the edge in $T^{\hat{A}}$ that contains the
minimal angle vertex and choose a parameterization of $K_{e}$ that makes
\eqref{eq3.15} equal to $\tau _{ - }$. Use this parameterization with the data
$(\tau _{(,)})$ to inductively define canonical parametrizations for the
remaining $C_{0}- \Gamma $. This is done with the following
induction step: Suppose that $\hat{o}$ is a bivalent vertex and that the
component of $C_{0}- \Gamma $ whose label is the edge with largest
angle $\theta _{{\hat{o}}}$ has its canonical parameterization. To
obtain one for the component that has $\theta _{{\hat{o}}}$ as its
smallest angle, first introduce the arc in $\underline {\Gamma
}_{{\hat{o}}}$ starting at the vertex that corresponds to the
distinguished element in $\hat{A}_{{\hat{o}}}$. This arc constitutes a one
element concatenating path set. The latter with the lift $\tau
_{{\hat{o}}}$ defines the canonical parameterization for the component
of $C_{0}- \Gamma $ whose label is the edge that has $\theta
_{{\hat{o}}}$ as its smallest angle, this according to the rules laid
out in Part 4 of \fullref{sec:2c}.

The next step is to change each of these parameterization. To start, let $e$
denote an edge of $T^{\hat{A}}$ and let $\phi _{e}$ denote the
parametrizing map from the relevant parametrizing cylinder to $K_{e}$. The
new parameterization of $K_{e}$ is defined by composing $\phi _{e}$ with
the diffeomorphism of $[0, \pi ]\times   \mathbb{R}/(2\pi \mathbb{Z})$
that pulls back the coordinate functions as
\begin{equation}\label{eq4.14}
(\sigma, v) \to (\sigma , v - 2\pi k/k_{v}).
\end{equation}
This new parameterization is of the sort that is described in \fullref{sec:2b}
because $k_{v}$ evenly divides both integers from the pair associated to any
edge of $T^{\hat{A}}$. This divisibility arises for the following
reasons: First, $k_{v}$ divides both integers from the edge with the
smallest angle vertex. Second, it divides the order of $\Aut_{o,v}$ for each
bivalent vertex $o \in  T^{\hat{A}}$ and the order of $\Aut_{o,v
}$ divides the corresponding version of $P_{o}$ that appears in \eqreft33.

For reference below, note that if $(a_{e}, w_{e})$ denotes that original
versions of the functions $(a, w)$ that appear in \eqreft25, then the new
versions, $({a_{e}}', {w_{e}}')$ are given by
\begin{equation}\label{eq4.15}
{a_{e}}'(\sigma, v) = a_{e}(\sigma ,v - 2\pi k/k_{v}) \quad\text{and} \quad
{w_{e}}'(\sigma, v) = w_{e}(\sigma ,v - 2\pi k/k_{v}).
\end{equation}

\step{Part 2}
Agree to implicitly view $g$ as an element in $\Aut(T_{C})$ via the chosen
identification between any given $T_{C}$ version of $\underline {\Gamma
}_{o}$ and the corresponding $\hat{A}_{o}$. Granted this view of $g$, let $o$
now denote a bivalent vertex in $T^{\hat{A}}$. According to \fullref{prop:4.5},
$g$'s action on $\Delta _{o}$ must fix the point $r_{o}   \in \Delta _{o}$ and this has
the following consequence: Let $\upsilon $
denote a given vertex in $\underline {\Gamma }_{o}$. Now sum the
coordinates of $r_{o}$ that correspond to the arcs in $\underline {\Gamma
}_{o}$ that are met on the oriented path that starts at $g^{ -
1}\upsilon $ and ends at $\upsilon $. This sum is $2\pi k/k_{v}$.

This last observation implies that the diffeomorphism in \eqref{eq4.15} restricts to
the $\sigma =\theta _{o}$ circle of the parametrizing cylinder so as
to map missing points to missing points. Moreover, the labels granted these
missing points as vertices in $\ell _{oe}$ are preserved by \eqref{eq4.14}. Thus,
both of the conditions in \eqreft41 hold for the new and original
parametrizations using $C' = C$ and using $g$ as the element in $\Aut(T_{C})$
for the isomorphism between $T_{C}$ and itself.

\step{Part 3}
The condition in \eqreft42 also holds in this context. To see this, let $o$
denote a bivalent vertex in $T^{\hat{A}}$, let $e$ denote the edge that
contains $o$ as its maximal angle vertex, and let $e'$ denote the edge that
contains $o$ as its minimal angle vertex. Let $\gamma \subset\underline{\Gamma}_{o}$ denote
the arc that starts at the vertex that corresponds
to the distinguished element in $\hat{A}_{o}$. This arc corresponds to
respective arcs on the $\sigma =\theta _{o}$ boundary circle for the
parametrizing domain of both $K_{e}$ and $K_{e'}$. The $\mathbb{R}/(2\pi \mathbb{Z})$
coordinates on these arcs lift to $\mathbb{R}$ so as to
make the $N= 0$ version of \eqref{eq2.14} hold. These lifts identify the two boundary
circle arcs. With this identification understood, the $N= 0$ version of
\eqref{eq2.15} describes the relationship between $w_{e}$ and $w_{e'}$ on the
interior of $\gamma $ when the latter is viewed as an arc in the $\theta  =
\theta _{o}$ locus in $C$'s model curve.

On this same arc, the pair ${w_{e}}'$ and ${w_{e'}}'$ are related by
another version of \eqref{eq2.15}; the version that applies to the arc that is
mapped to $\gamma $ by the inverse of the map in \eqref{eq4.14}. However, using
\fullref{lem:2.3} and the fact that $g$ acts on $\underline {\Gamma }_{o}$ as an
isomorphism, the relationship between ${w_{e}}'$ and ${w_{e'}}'$ on
$\gamma $'s image in $C_{0}$ is also given by the $N = 0$ version of \eqref{eq2.15}.
Indeed, \fullref{lem:2.3} finds that
\begin{equation}\label{eq4.16}
\begin{split}
w_{e}\big(\sigma , \hat {v}-2\pi k/k_{v}\big) &= w_{e'}\big(\sigma
,\hat {v}-2\pi k/k_{v}\big)\\
&\quad+\frac{1 }{ {\alpha _{Q_e } (\sigma
)}}\big({q_{e}}'q_{e'} - q_{e}{q_{e'}}'\big)\big(\hat {v}-2\pi
k/k_{v}\big)\\
&\quad+\frac{{2\pi } }{ {\alpha _{Q_e } (\sigma
)}}\big({q_{e}}'j - q_{e}j'\big),
\end{split}
\end{equation}
where $(j, j')$ is proportional to the relatively prime integer pair that
defines $\theta _{o}$ via \eqref{eq1.8}. In particular, the proportionality factor
here is minus the sum of the integer weights that are assigned to all
vertices but $\upsilon $ that lie on the oriented path in $\underline {\Gamma }_{o}$ from
$g^{- 1}(\upsilon )$ to $\upsilon $. Because $g$ acts as an isometry, this sum must equal
$(k/k_{v})P_{o}$ with $P_{o}$ as
in \eqreft33. Thus, $(j, j')$ can be written as $(k/k_{b})(Q_{e}-
Q_{e'})$. With $(j, j')$ so identified, the equality in \eqref{eq4.16} finds
that ${w_{e}}'$ and ${w_{e'}}'$ do indeed obey the predicted $N = 0$
version of \eqref{eq2.15} on $\gamma $.

As both $(w_{e}, w_{e'})$ and $({w_{e}}', {w_{e'}}')$ obey the
same $N = 0$ version of \eqref{eq2.15} on $\gamma $, it then follows that $\hat
{w}_{e}=\hat {w}_{e'}$ along $\gamma $. The fact that this
equality holds along the whole of $\Gamma _{o}$ in $C_{0}$ can be seen
with the help of \eqreft25. To see this, let $\gamma '$ denote another arc in
$\underline {\Gamma }_{o}$. The relationship between $w_{e}$ and
$w_{e'}$ on $\gamma '$ can be obtained with the use of the $e$ and $e'$
versions of \eqreft25; it is given by the version of \eqref{eq2.15} where $(j, j')$ is a
certain multiple of the relatively prime pair that defines $\theta _{o}$
via \eqref{eq1.8}. To be precise, this multiple is obtained by summing with an
appropriate sign the weights of the vertices that lie on the oriented path
in $\underline {\Gamma }_{o}$ that starts at the endpoint of $\gamma $ and
ends at the starting point of $\gamma '$. Since the same weights appear on
the analogous path defined by $g^{- 1}(\gamma )$ and $g^{- 1}(\gamma ')$,
the relation between ${w_{e}}'$ and ${w_{e'}}'$ on
$\gamma '$ is the same $(j, j')$ version of \eqref{eq2.15}. Thus, $\hat {w}_{e} =
\hat {w}_{e'}$ along $\gamma '$.

\step{Part 4}
Let $G\subset C_{0}$ denote the part of the zero locus of $\hat {w}$
that lies in the complement of the critical point set of the pull-back of
$\cos(\theta )$. This set is described by \fullref{lem:4.3}. Moreover, $G\ne  ${\o}
since the $w_{e}$ and ${w_{e}}'$ versions of the integral in \eqref{eq3.15} are
identical in the case that $e$ is any edge in $T^{\hat{A}}$.

To see the implications of \fullref{lem:4.1} in this context, return momentarily to
the proof of \fullref{lem:4.2}. The latter explains how the various versions of
\eqref{eq4.14} fit together across $\Gamma\subset C_{0}$ so as to define a
diffeomorphism, $\psi \co  C_{0}   \to C_{0}$. With $\psi $ in hand,
introduce the tautological map, $\phi \co  C_{0}   \to   \mathbb{R}  \times (S^1  \times  S^2)$, onto $C$.
Keep in mind that $\phi $ is almost everywhere 1--1. Let $\phi '  \equiv   \phi  \circ \psi $. In the
present context, \fullref{lem:4.1} asserts that $\phi'(C_{0})$ is obtained from $C$
by a constant translation along the $\mathbb{R}$ factor in $\mathbb{R}  \times (S^1  \times  S^2)$.
However, this factor must be zero since the average of any given ${a_{e}}'$ around a constant $\sigma $
circle is equal to that of $a_{e}$ around the same circle. Thus, $\phi '$ also maps $C_{0}$
onto $C$. Granted this, then \eqref{eq4.14} implies that $\psi $ is 1--1 if and only if
$k$ is zero.

\subsection{Why the map from $\mathcal{M}$ is proper}\label{sec:4e}

The map $P$ from $\mathcal{M}   \to   \mathbb{R}  \times \hat{O}^{\hat{A}}/\Aut^{\hat{A}}$
is proper if and only if all
sequences in $\mathcal{M}$ with convergent image in $\mathbb{R}  \times
\hat{O}^{\hat{A}}/\Aut^{\hat{A}}$ have convergent subsequences.
The proof that such is the case is given here in three parts.

\step{Part 1}
Let $\{C_{j}\}$$_{j = 1,2,\ldots }$ denote a sequence with convergent
image in $\mathbb{R}  \times  \hat{O}^{\hat{A}}/\Aut^{\hat{A}}$.
The desired convergent subsequence in $\mathcal{M}$ is found by first invoking
 \cite[Proposition~3.7]{T3} to describe the $j \to   \infty $ behavior of a
subsequence from $\{C_{j}\}$ in terms of a limiting data set, $\Xi $.
Here, $\Xi $ is a finite set of pairs where each has the form $(S, n)$ with $S$
being a pseudoholomorphic, multiply punctured sphere in $\mathbb{R}  \times (S^1  \times  S^2)$
and with $n$ being a positive integer.

This first part of the subsection establishes that the conclusions
of \cite[Proposition 3.7]{T3} hold in the case that $K = \mathbb{R} \times (S^1 \times  S^2)$.
The precise version needed here is stated formally as follows.

\begin{proposition}\label{prop:4.6}

Let $\{C_{j}\}_{j = 1,2,\ldots }   \subset \mathcal{M}$  denote an infinite sequence
with convergent image in $\mathbb{R}  \times  \hat{O}^{\hat{A}}/\Aut^{\hat{A}}$.
 There exists a subsequence, hence renumbered consecutively from 1,
 and a finite set, $\Xi $, of pairs of the form $(S, n)$  where $n$  is a positive integer and $S$
 is an irreducible, pseudoholomorphic, multiply punctured sphere; and these have the following
 properties:

\begin{itemize}

\item
$\lim_{j \to \infty }  \int _{C_j } \varpi =  \sum_{(S,n) \in \Xi }^{ } $n$ \int _{S}  \varpi $
for each compactly supported 2--form $\varpi $.

\item
The following limit exists and is zero:
\begin{equation}\label{eq4.17}
\lim_{j \to \infty } \Bigl(\sup_{z \in C_j } \dist(z,  \cup _{\Xi } S) +
\sup_{z \in \cup _{(S,n) \in \Xi } S}  \dist(C_{j}, z)\Bigr)
\end{equation}

\end{itemize}
\end{proposition}

The proof of this proposition appears momentarily. It is employed in the
following manner to prove that the map $P$ from $\mathcal{M}$ to
$\mathbb{R}  \times  \hat{O}^{\hat{A}}/\Aut^{\hat{A}}$ is proper: Part 2 of this
subsection uses \fullref{prop:4.6} to prove that $\Xi $ contains but a single
element, this denoted as $(S, n)$. Part 3 of the subsection proves that $n = 1$
and that $S$ is a subvariety from $\mathcal{M}$. Granted that such is the case,
then \eqref{eq4.17} asserts that $S$ is in fact the limit of $\{C_{j}\}$ with
respect to the topology on $\mathcal{M}$ as defined in \eqref{eq1.13}. This last
conclusion establishes that $P$ is proper.

\begin{proof}[Proof of \fullref{prop:4.6}]
Except for the second point in \eqref{eq4.17},
this proposition restates conclusions from \cite[Proposition~3.7]{T3}. The proof
of the second point in \eqref{eq4.17}, assumes it false so as to derive some
nonsense. For this purpose, note that \cite[Proposition~3.7]{T3} asserts that
the second point in \eqref{eq4.17} holds if the supremums that appear are restricted
to those points z that lie in any given compact subset of $\mathbb{R}  \times (S^1  \times  S^2)$.

The derivation of the required nonsense starts with the following lemma.

\begin{lemma}\label{lem:4.7}

Assume that the  second point in \eqref{eq4.17} does not hold for the given infinite sequence,
$\{C_{j}\}$, of multiply punctured, pseudoholomorphic spheres. Then, there exists an
$\mathbb{R}$--invariant, pseudoholomorphic cylinder, $S_{*}   \subset   \mathbb{R}$;
an infinite subsequence of $\{C_{j}\}$  (hence renumbered consecutively from 1);
and given $\varepsilon  > 0$  but small, there is a real number $s_{0}$,  a sequence
$\{s_{j - }\}_{j =1,2\ldots }   \subset  (-\infty , s_{0}]$  and a corresponding sequence
$\{s_{j + }\}_{j =1,2\ldots }   \subset  [s_{0}, \infty )$ ; all with the following significance:

\begin{itemize}

\item
Both $s_{j - }$  and $s_{j + }$  are regular values of $s$ on $C_{j}$.

\item
If either $\{s_{j - }\}$  and $\{s_{j + }\}$  is bounded, then it is convergent;
but at least one of the two is unbounded.

\item
The $s \in [s_{j - }, s_{j + }]$  portion of $C_{j}$'s intersection with the radius
$\varepsilon $  tubular neighborhood of $S_{*}$  has a connected component, $C_{j*}$,
whose points have distance $\frac{1 }{ 2}\varepsilon $ or less from $S_{*}$.

\item
Both the $s = s_{j - }$  and $s = s_{j + }$  slices of $C_{j*}$  are non-empty and both
contain points with distance $\frac{1 }{2}\varepsilon $  from $S_{*}$.

\item
Let $\delta _{j}$  denote the maximum of the distances from the point on the
$s = \frac{1 }{ 2}(s_{j - }+s_{j + })$  locus in $C_{j*}$  to $S_{*}$.
The corresponding sequence $\{\delta _{j}\}$  is then decreasing with limit zero.
\end{itemize}
\end{lemma}

\fullref{lem:4.7} is proved momentarily. To see where this lemma leads, suppose
first that $\theta $ is neither 0 nor $\pi $ on $S_{*}$. In this case,
the conclusions in \cite[Lemma~3.9]{T3} hold. But topological considerations
find that \cite[Lemma 3.9]{T3} and \fullref{lem:4.7} lead to nonsense.

To elaborate, first let $\theta _{*}$ denote the constant value for
$\theta $ on $S_{*}$ and assume that $\theta _{*}$ is neither
0 nor $\pi $. Next, fix $\varepsilon $ small and then $j$ large. By
virtue of what is said in \cite[Lemma~3.9]{T3}, there exists some very small
and $j$--independent constants $\delta _{\pm }$, either both positive and
negative and with the following properties: First, neither $\theta _{*}+\delta _{ + }$
or $\theta _{*}+\delta _{ - }$ is an
$|s|    \to \infty $ limit of $\theta $ on $C_{j}$. Second,
when $j$ is large, $\theta _{*}+\delta _{ + }$ is a value of
$\theta $ on the $s=s_{j + }$ boundary of $C_{j*}$ and $\theta
_{*}+\delta _{ - }$ is a value of $\theta $ on the $s=s_{j -}$ boundary.
Let $z_{ + }$ and $z_{ - }$ denote respective points on the $s=s_{j + }$ and $s=s_{j - }$
boundaries of $C_{j*}$ where $\theta $
has the indicated value. Since both $\delta _{ - }$ and $\delta _{ + }$
have small absolute value, the $\theta =\theta _{*}+\delta_{ - }$ and
$\theta =\theta _{*}+\delta _{ + }$ loci are
circles in the same component of $C_{j}$'s version of $C_{0}- \Gamma $.
Thus, there is a path in this component that lies in the set
where $| \theta -\theta _{*}|    \ge  \min(|\delta _{ - }| , | \delta _{ + }| )$ and runs from
$z_{ - }$ to $z_{ + }$. Denote the latter by $\gamma $. When $j$ is large, the
last point in \fullref{lem:4.7} forbids an intersection between $\gamma $ and the
$s=\frac{1 }{ 2}(s_{ - j}+s_{ + j})$ locus. On the other hand,
there is a path in $C_{j*}$ that runs from $z_{ + }$ to $z_{ - }$ and
does intersect this locus. The concatenation of the latter path with $\gamma$
defines a closed loop in $C_{j}$ that has non-zero intersection number
with the constant s slices of $C_{j*}$. But no such loop exists since
$C_{j}$ has genus zero.

Suppose next that that $\theta  = 0$ or $\pi $ on $S_{*}$. In this
case, the maximum value of $\theta $ on both the $s=s_{j - }$ and $s=s_{j+ }$
slices of $C_{j*}$ is bounded away from zero by a uniform
multiple of $\varepsilon ^{2}$ when $\varepsilon $ is small. Meanwhile all
values of $\theta $ on the $s=\frac{1 }{ 2}(s_{ - j}+s_{ +
j})$ slice of $C_{j*}$ are bounded by $o({\delta _{j}}^{2})$. This
being the case, the mountain pass lemma demands a non-extremal critical
point of $\theta $ on $C_{j}$. Note in this regard that there are at most a
finite number of intersections between any two distinct, irreducible,
pseudoholomorphic subvarieties.

Thus, in all cases, \fullref{lem:4.7} leads to nonsense. Granted the lemma is
correct, then its conclusions are false and so the second point in \eqref{eq4.17} must
hold.\end{proof}

\begin{proof}[Proof of \fullref{lem:4.7}]
Let $\Xi $ denote the data set that is provided
by  \cite[Proposition~3.7]{T3}. Since the second point in \eqref{eq4.17} is supposed to
fail, there exits some subvariety, $S$, from $\Xi $; an end $E \subset  S$; a
constant $R_{0}   \ge 1$; and, given $\varepsilon  > 0$, a divergent
sequence $\{R_{j}\} \subset  [R_{0}, \infty )$; all with the
following properties:

\itaubes{4.18}
 \textsl{Each sufficiently large $j$ version of $C_{j}$  intersects the $|s| \in [R_{0}, R_{j}]$
 portion of the radius $\varepsilon $
 tubular neighborhood of $E$ where the distance to $E$ is no greater than $\frac{1 }{ 2}\varepsilon $.}

\item
 \textsl{Meanwhile, each such $C_{j}$  has a point on the $|s|  = R_{j}$  slice of this neighborhood
 where the distance to $E$ is equal to $\frac{1 }{2}\varepsilon $.}
\eit

Of course, there may be more than one such $S$ and more than one such end in $S$
to which \eqreft4{18} applies. \fullref{lem:4.7} holds if there exists a pair $(S, E)$ as in
\eqreft4{18} where $S$ is not an $\mathbb{R}$ invariant cylinder. \fullref{lem:4.7} also holds
if there exists a pair $(S, E)$ as in \eqreft4{18} where $S$ is an $\mathbb{R}$--invariant
cylinder that intersects a subvariety from some other pair in $\Xi $.
Indeed, such is the case because each $C_{j}$ is irreducible. In fact, for
this very reason, \fullref{lem:4.7} holds if $\Xi $ contains any pair whose
subvariety is a not an $\mathbb{R}$--invariant cylinder.

To finish the argument for \fullref{lem:4.7}, consider now the case when all
subvarieties from $\Xi $ are $\mathbb{R}$--invariant cylinders. Even so, \fullref{lem:4.7}
must hold unless one of the following is true:

\itaubes{4.19}
 \textsl{Let $S$  denote any cylinder from $\Xi $. Given $\varepsilon  > 0$,  there exists a divergent
 sequence $\{R_{j}\} \subset  [0, \infty )$  such that each sufficiently large $j$
 version of $C_{j}$  intersects the $s  \in  (-\infty, R_{j}]$  portion of the radius
 $\varepsilon $ tubular neighborhood of $S$ where the distance to $S$ is no greater than
$\frac{1 }{ 2}\varepsilon $,  and it intersects the $s = R_{j}$  slice at a point with distance to
$S$ equal to $\frac{1 }{2}\varepsilon $}

\item
\textsl{Let $S$  denote any cylinder from $\Xi $. Given $\varepsilon  > 0$, there exists a divergent
sequence $\{R_{j}\} \subset  [0, \infty )$  such that each sufficiently large $j$
 version of $C_{j}$  intersects the $s  \in  [-R_{j}, \infty )$  portion of the radius
 $\varepsilon $  tubular neighborhood of $S$ where the distance to $S$ is no greater than
$\frac{1 }{ 2}\varepsilon $,  and it intersects the $s = -R_{j}$  slice at a point with distance to
$S$ equal to $\frac{1 }{2}\varepsilon $.}
\eit

As is explained next, both possibilities violate the assumed convergence of
the images of $\{C_{j}\}$ in the $\mathbb{R}$ factor $\mathbb{R}  \times
\hat{O}^{\hat{A}}/\Aut^{\hat{A}}$.

In the case that the top point in \eqreft32 holds, this can be seen by making a
new sequence whose $j$'th subvariety is a $j$--dependent, constant translation of
$C_{j}$ along the $\mathbb{R}$ factor of $\mathbb{R}  \times (S^1  \times
S^2)$. Were \eqreft4{19} to hold, a new sequence of this sort could be found
whose limit data set in \cite[Proposition~3.7]{T3} has an irreducible
subvariety with the following incompatible properties: It is not an $\mathbb{R}$--invariant cylinder,
it has the same $\theta $ infimum as each
$\{C_{j}\}$, and there is no non-zero constant $b$ that makes \eqref{eq2.4} hold
for the corresponding end.

Were the second point in \eqreft32 to hold, then one of the limit cylinders from
$\Xi $ would be the $\theta  = 0$ cylinder and then the convergence dictated
by \eqreft4{19} would demand a $(1,\ldots)$ element in $\hat{A}$.

An argument much like that just used to rule out the first point of \eqreft32
rules out the third point as well. In this case, the derived nonsense from
the translated sequence is an irreducible subvariety from the data set of
the corresponding version of \cite[Proposition~3.7]{T3} with the following
mutually incompatible properties: First, it is not the $\theta  = 0$
cylinder but contains an end where the $|s| \to \infty $
limit of $\theta $ is zero. Second, the respective integrals of
$\frac{1 }{ {2\pi }} dt$ and $\frac{1 }{ {2\pi }}d\varphi $
about the constant $|s| $ slices of this end have the form
$\frac{1 }{ m}p$ and $\frac{1 }{ m}p'$ where $(p, p')$ is the
pair from the $(1,\ldots)$ element in $\hat{A}$ and where $m\ge 1$ is some common divisor of
this same pair. Finally, there would be no
non-zero version of the constant $\hat {c}$ that would make \eqref{eq1.9} hold. Note
that the argument for the second property uses the fact that there are no
critical points of $\theta $ with $\theta $ value in $(0, \pi )$ on any
$C_{j}$. Indeed, the lack of critical points implies that the component of
the $C_{j}$ version of $C_{0}- \Gamma $ where the $|s| \to \infty $ limit of
$\theta $ is zero is parametrized by a
$j$--dependent map from a $j$--independent cylinder. Of course, this same cylinder
parametrizes the corresponding component in any translation of $C_{j}$.
Granted this, apply  of \cite[Proposition~3.7]{T3} to the translated sequence to
find the asserted form for the integrals of $\frac{1 }{ {2\pi }} dt$
and $\frac{1 }{ {2\pi }}d\varphi $.
\end{proof}

\step{Part 2}
This part of the argument proves that $\Xi $ contains but a single
element. The proof assumes the converse so as to derive an absurd
conclusion. For this purpose, introduce $\mathcal{S}    \subset   \mathbb{R}\times (S^1  \times  S^2)$
to denote the union of the
subvarieties from $S$ that are not $\mathbb{R}$--invariant cylinders. Then set $Y$
to denote the subset of $\mathcal{S} $ that contains the images of critical
points of $\theta $ from the model curves of the irreducible components of
$\mathcal{S} $, the singular points of $ \cup _{(S,n) \in \Xi }S$, and the
points on $\mathcal{S} $ where $\theta    \in  \{0, \pi \}$. This set $Y$ is
finite.

The first point to make is that $\Xi $ contains at least one element whose
subvariety is not an $\mathbb{R}$--invariant cylinder. Indeed, this follows from
an appeal to \fullref{prop:4.6}. However, $\Xi $ cannot have two elements whose
subvarieties are not $\mathbb{R}$--invariant. To explain, suppose first that
there are two such subvarieties from $\Xi $ and a value of $\theta $ that is
taken simultaneously on both. Since $Y$ is a finite set, this value can be
taken so as to be disjoint from any value of $\theta $ on $Y$, and also
disjoint from any value of $\theta $ from the set $\Lambda
_{\hat{A}}$. Let $\theta _{*}$ denote such a value. Then, the
$\theta =\theta _{*}$ locus in $\mathcal{S} $ is at least two
disjoint circles.

To see where this leads, introduce the notion a `$\theta $--preserving
preimage' in \cite[Step 4 of Section~3.D]{T3}. Note here that a compact
submanifold in $\mathcal{S} - Y$ has a well defined $\theta $--preserving
preimage even if $C_{j}$ has nearby immersion points. This said, each circle
from the $\theta =\theta _{*}$ locus in $\mathcal{S} $ has $\theta$--preserving
preimages in each sufficiently large $j$ version of $C_{j}$. The
set of these preimages gives a set of disjoint, embedded circles on which
$\theta =\theta _{*}$. However, there can be at most one such
circle.

Thus, if $\Xi $ has two subvarieties that are not $\mathbb{R}$--invariant, then
the supremum of $\theta $ on one must be the infimum of $\theta $ on the
other. To rule out this possibility, let $S$ and $S'$ denote the subvarieties
that are involved, and suppose that $\theta _{*}$ is the supremum of
$\theta $ on $S$ and the infimum of $\theta $ on $S'$. Note that $\theta_{*}$ must be the value
of $\theta $ on some bivalent vertex
in $T^{\hat{A}}$. Let $e$ denote the edge in $T^{\hat{A}}$ where the
maximum of $\theta $ is $\theta _{*}$ and let $e'$ denote that where
the minimum of $\theta $ is $\theta _{*}$. Let $P$ denote
the relatively prime integer pair that defines $\theta _{*}$ via
\eqref{eq1.8}. As is explained next, both $Q_{e}$ and $Q_{e'}$ must be
non-zero multiples of $P$, and this is impossible for the following reason:
The $\theta =\theta _{*}$ locus in each $C_{j}$ is a non-empty
union of arcs on which both the $Q = Q_{e}$ and $Q =Q_{e'}$ versions
of \eqref{eq2.2}'s function $\alpha _{Q}$ are positive. Thus, the $Q =P$ version of
$\alpha _{Q}$ is positive at $\theta =\theta _{*}$, and this
contradicts \eqref{eq1.8}.

To see why $Q_{e}$ is proportional to P, let $\delta  > 0$ be very small.
\fullref{prop:4.6} guarantees that the integrals of $\frac{1 }{ {2\pi}} dt$ and
$\frac{1 }{ {2\pi }}d\varphi $ around the $\theta  =\theta _{o}-\delta $ circle in any
large $j$ version of $C_{j}$ is
proportional to its integral around the analogous circle in $S$, and the
latter integral is proportional to $P_{o}$. The analogous argument using the
$\theta =\theta _{*}+\delta $ circles proves the claim about $Q_{e'}$.

The proof that there are no $\mathbb{R}$--invariant cylinders from $\Xi $ is
given next. For this purpose, assume to the contrary that one subvariety
from $\Xi $ is an $\mathbb{R}$--invariant cylinder. Let $S_{*}$ denote
this cylinder. The second point in \eqref{eq4.17} requires that $\theta $'s value on
$S_{*}$ is its value at some bivalent vertex on $T$. Let $o$ denote the
vertex involved. In what follows, $(p, p')$ are the relatively prime integers
that define $\theta _{o}$ via \eqref{eq1.8}.

Given $\varepsilon  > 0$, let $C_{j*}$ denote the subset of $C_{j}$
whose points have distance less than $\varepsilon $ from $S_{*}$. If $j$
is large, the second point in \eqref{eq4.17} requires that $C_{j*}$ have
both a concave side end and a convex side end of $C_{j}$. The next lemma
asserts that there is a path in any large $j$ version of $C_{j*}$ from
the convex side end in $C_{j*}$ to the concave side one.

\begin{lemma}\label{lem:4.8}
If $\varepsilon  > 0$  is small, then all sufficiently large $j$ versions of $C_{j*}$
contain an embedding of $\mathbb{R}$  on which $s$  is neither bounded from above nor below.
Moreover, at large $|s| $  this embedding sits in the $\theta =\theta _{o}$  locus in $C_{j*}$.

\end{lemma}

Accept this lemma for the moment to see where it leads. For this purpose,
note that the 1--form $pd\varphi  - p'dt$ is zero on $S_{*}$ and thus
exact on the radius $\varepsilon $ tubular neighborhood of $S_{*}$. It
can therefore be written as $df$ on this neighborhood where $f$ is a smooth
function that vanishes on $S_{*}$. Now let $\gamma _{j} \subset C_{j*}$ denote a given large
$j$ version of the line from \fullref{lem:4.8}.
Because $f$ vanishes on $S_{*}$, the line integral of pd$\varphi -p'dt$ along $\gamma _{j}$
has small absolute value. In fact, if the
$j$--version of this absolute value is denoted by $\nu _{j}$, then the second
point in \eqref{eq4.17} demands that $\lim_{j \to \infty }  \nu _{j} = 0$. This
last conclusion is nonsense as can be seen using \eqreft25 to reinterpret the
integral of $pd\varphi  - p'dt$ along $\gamma _{j}$ as a $j$--independent
multiple of the sum of one or more of the coordinates of the image of
$C_{j}$ in the simplex $\Delta _{o}$. To explain, let $e$ denote the
incident edge in $T$ that contains $o$ as its largest angle vertex and fix any
parameterization for $K_{e}$. Since the large $|s| $ part of
$\gamma _{j}$ lies in the $\theta =\theta _{o}$ locus, referral to
\eqreft25 finds that $\nu _{j}$ is the absolute value of the integral of
$(p{q_{e}}' - p'q_{e}) dv$ between two missing points on the $\sigma =\theta _{o}$
circle of parametrizing cylinder for $K_{e}$. According to
Part 3 in \fullref{sec:3}c, the latter integral is a fixed multiple of a sum of
one or more of the coordinates from $C_{j}$'s image in the $\Delta _{o}$
factor that is used in \eqref{eq3.12} to define $C_{j}$'s image in
$\hat{O}^{\hat{A}}/\Aut^{\hat{A}}$. In particular, such a sum must
have a positive, $j$--independent lower bound if the image of $\{C_{j}\}$
is to converge in $\hat{O}^{\hat{A}}/\Aut^{\hat{A}}$.

\begin{proof}[Proof of \fullref{lem:4.8}]
There are three steps to the proof.

\substep{Step 1}
Any given $(\rho , \chi )   \in   \mathbb{R}  \times  S^1$ defines a diffeomorphism of
$\mathbb{R}  \times (S^1  \times  S^2)$ by sending $(s, t, \theta, \varphi )$ to
$(s+\rho, t+p\chi, \theta ,\varphi +p'\chi )$. All such
diffeomorphisms act transitively on $S_{*}$ with trivial stabilizer
and so parametrize $S_{*}$ once a point in $S_{*}$ is chosen as
the point where $\rho  = 0$ and $\chi  = 0$. Such a parameterization of
$S_{*}$ is used implicitly in what follows.

Fix a very small radius pseudoholomorphic disk in $\mathbb{R}  \times (S^1  \times  S^2)$
whose closure intersects $S_{*}$ at a
single point, this its center. For example, a disk of this sort can be found
in one of the pseudoholomorphic cylinders where either t or $\varphi $ is
constant (see, eg \cite[Subsection 4(a)]{T4}.) Let $D$ denote such a disk. The
translates of $D$ by the just described $\mathbb{R}  \times  S^1$ group of
diffeomorphisms defines an embedding of $S_{*}  \times D$ into $\mathbb{R} \times(S^1 \times S^2)$
as a tubular neighborhood of $S_{*}$.

\substep{Step 2}
Let $S$ denote the subvariety from $\Xi $
that is not $\mathbb{R}$--invariant. If $S$ intersects $S_{*}$, it does so
at a finite set of points. The subvariety $S$ can also approach $S_{*}$
at large $|s| $ if $S$ has ends whose constant $|s| $
slices converge as $|s| \to \infty $ as a multiple cover
of the Reeb orbit that defines $S_{*}$. In fact such ends must exist
if $S \cap S_{*} = ${\o} since each $C_{j}$ is connected and since
\eqref{eq4.17} holds. By the same token, if there are no such ends, then $S$ must
intersect $S_{*}$.

Given $\varepsilon  > 0$ and small, let $D_{\varepsilon }   \subset D$ denote the radius
$\varepsilon $ subdisk about the origin. If $\varepsilon $
is small enough, then the points where $S$ intersects $S_{*}  \times D_{\varepsilon }$
are of two sorts. The first lie in a small radius ball
about the points where $S$ intersects $S_{*}$. As $\varepsilon    \to 0$, the radii of
these balls can be taken to zero. The second sort are points
where the absolute value of the parameter $\rho $ on $S_{*}$ is
uniformly large. Of course, points of the latter sort exist if and only if $S$
has an end whose constant $|s|$ slices converge as $|s| \to \infty $ to a multiple cover
of the Reeb orbit that
defines $S_{*}$. Assuming that such is the case, then each small
$\varepsilon $ has an assigned positive and large number, $\rho_{\varepsilon }$,
this a lower bound for the absolute value of the
parameter $\rho $ on $S_{*}$ at the points in $S \cap (S_{*}  \times  D_{\varepsilon })$
with distance 1 or more from $S \cap S_{\varepsilon }$. In this regard, the assignment
$\varepsilon \to \rho _{\varepsilon }$ can be made in a continuous fashion. It is perhaps
needless to add that there is no finite $\varepsilon    \to  0$ limit of
the collection $\{\rho _{\varepsilon }\}_{\varepsilon > 0}$.

In the case where $\rho _{\varepsilon }$ is defined, the analysis
 in \cite[Section~2]{T4} can be used to find an even larger
 $\rho _{\varepsilon}'$ together with non-negative integers $n_{ + }$ and $n_{ - }$ with the
following significance: There are precisely $n_{ + }$ intersections between
 $S$ and the disk fiber of $S_{*}\times  D_{\varepsilon }$
over any point where $\rho    \ge   {\rho _{\varepsilon }}'$. Here, all
have intersection number 1 and all occur in the subdisk $D_{\varepsilon /2}$.
There are also precisely $n_{ - }$ intersections between $S$ and the
fiber of $S_{*}  \times D_{\varepsilon }$ over any point where
$\rho    \le -{\rho _{\varepsilon }}'$; all of these have intersection
number 1 and lie in $D_{\varepsilon / 2}$.

\substep{Step 3}
Fix $\varepsilon  > 0$ but very small, and
let $C_{j\varepsilon }   \subset C_{j}$ denote the intersection of
$C_{j}$ with $S_{*}  \times D_{\varepsilon }$. Let $n_{*}$
denote the integer that is paired with $S_{*}$ in $\Xi $. If $j$ is
large, then $C_{j\varepsilon }$ intersects each fiber disk in $S_{*}\times  D_{\varepsilon }$
where $| \rho | \le \rho_{\varepsilon }$ in $n_{*}$ points counted with multiplicity. In
this regard, each such point has positive local intersection number.
Moreover, these intersection points vary continuously with the chosen base
point in $S_{*}$. This is to say that the assignment of the $n_{*}$ intersection points
to the base point defines a continuous map from the
$| \rho | \le |\rho _{\varepsilon }| $
portion of $S_{*}$ to $\Sym^{n_ * }(D_{\varepsilon })$. As a
consequence, each component of the $\rho =\rho _{\varepsilon }$ slice
of any large $j$ version of $C_{j\varepsilon }$ is connected in the
$|\rho | \le \rho _{\varepsilon }$ part of $C_{j\varepsilon} \cap (S_{*} \times D_{\varepsilon })$
to at least one component in the $\rho = -\rho _{\varepsilon }$ slice, and vice versa.

Thus, if $\varepsilon $ is small, then there is no obstruction to choosing a
path between $\rho =\rho _{\varepsilon }$ and $\rho $ = -$\rho
_{\varepsilon }$ slices of any sufficiently large $j$ version of
$C_{j\varepsilon }$.

The argument just used can be repeated to extend the path just chosen, first
as a path from the $\rho ={\rho _{\varepsilon }}'$  slice of
$C_{j\varepsilon }$ to the latter's $\rho $ = -${\rho _{\varepsilon }}'$
slice, and then to arbitrarily large values of $| \rho | $. For
example, to extend the path to where $\rho _{\varepsilon }=
\rho    \le   {\rho _{\varepsilon }}'$, note first that the number of
intersections counted with multiplicity between $C_{j\varepsilon }$ and the
fiber disks over this part of $S_{*}$ may change. Even so, when $j$ is
very large, all such intersections lie either very close to $S$ or very close
to $S_{*}$. In particular, there are precisely $n_{*}$ such
intersections in each disk that lie very much closer to $S_{*}$ than
to $S$. Granted this, the preceding argument implies that any component of the
$\rho =\rho _{\varepsilon }$ slice of $C_{j\varepsilon }$ is connected
to one or more components of the $\rho ={\rho _{\varepsilon }}'$ slice of $C_{j\varepsilon }$.

To explain why this last extension can be continued to arbitrarily large
values of $\rho $, note first that when $j$ is large, then $C_{j*}$ has
precisely $n_{*} + n_{ + }n$ intersections with any given fiber disk
over the $\rho\ge {\rho _{\varepsilon }}'$ portion of $S_{*}$ when counted with multiplicity.
Here, $n$ is the integer that is paired
with $S$ in $\Xi $. Note that all intersections are again positive, and that
these disk intersections now define a continous map from the $\rho    \ge {\rho _{\varepsilon }}'$
portion of $S_{*}$ to $\Sym^{n_ * + n_ +n}(D)$. This then means that any given
component of the $\rho =\rho_{\varepsilon }'$ slice of $C_{j\varepsilon }$ is part of a slice of a
component of the $\rho    \ge   {\rho _{\varepsilon }}'$ part of
$C_{j\varepsilon }$ where $\rho $ has no finite upper bound.
\end{proof}

\step{Part 3}
Let $S$ denote the single element from $\Xi $ and let $S_{0}$ denote the
model curve for S. An argument from Part 2 can be used to prove that the
pull-back of $\theta $ to $S_{0}$ has no non-extremal critical points.
Indeed, were there such a point, then there would be an open interval of
disconnected, compact $\theta $ level sets in $S_{0}$. Most of these level
sets would avoid the set Y, and any of the latter would have $\theta$--preserving
preimages in any sufficiently large $j$ version of $C_{j}$. Of
course, no such thing is possible since $C_{j}$ lacks disconnected, compact
$\theta $ level sets.

A very similar argument implies the following: Let $E \subset S$ denote
any end where the $|s| \to \infty $ limit of $\theta $ is
neither 0 nor $\pi $. Then $E$'s version of \eqref{eq2.4} has integer $n_{E} = 1$.

These last points have the following consequence: Let $T_{S}$ denote a graph
with a correspondence in a pair whose first component is the model curve for
$S$. Then $T_{S}$ is necessarily linear. Moreover, by virtue of \fullref{prop:4.6},
the vertex set of $T_{S}$ enjoys an angle preserving, 1--1
correspondence with that of $T^{\hat{A}}$. Meanwhile, any two $C =C_{j}$ and $C = C_{j'}$
versions of the graph $T_{C}$ must be
isomorphic when $j$ and $j'$ are large because the image of $\{C_{j}\}$
converges in $\hat{O}^{\hat{A}}/\Aut^{\hat{A}}$. Let $T$ denote a
graph in the isomorphism class. The remainder of this Part 3 explains why
the graphs $T_{S}$ and $T$ are isomorphic. The explanation is given in eight
steps.

\substep{Step 1}
Each edge in $T$ has its evident analog in
$T_{S}$ since the vertex sets share the same angle assignments. If $e\in T$
is a given edge, then the corresponding $T$ and $T_{S}$
versions of the integer pair $Q_{e}$ are related as follows: The $T$ version
is $n$ times the $T_{S}$ version where $n$ here is the integer that pairs with $S$
in $\Xi $. This then means that $n = 1$ in the case that the integers that
comprise the suite of $T$ versions have 1 as their greatest common divisor. In
particular if $n$ is larger than 1, then it must divide both integers that
comprise the version of $Q_{(\cdot )}$ whose labeling edge has the
smallest angle vertex in $T$.

\substep{Step 2}
Let $o$ denote a bivalent vertex in $T_{S}$
and let $E \subset S$ denote an end where the $|s| \to \infty $ limit of
$\theta $ is $\theta _{o}$. Since $E$'s version of \eqref{eq2.4}
uses the integer $n_{E} = 1$, the $\theta =\theta _{o}$ locus in the
model curve for $S$ intersects the very large $|s| $ portion of $E$
as a pair of open arcs, one oriented in the increasing $|s| $
direction and the other oriented with $|s| $ decreasing.
Moreover, given a positive and very large number, $R$, there is an arc in the
$\theta    \le   \theta _{o}$ portion of $E$ with the following
properties:

\itaubes{4.20}
\textsl{$|s| \ge R$ on the arc.}

\item
\textsl{The arc starts on one component of the $\theta =\theta _{o}$  locus in the $|s|  \ge R$ part of
$E$ and ends on the other.}

\item
 \textsl{ The function $\theta $  restricts to this arc so as to have but a single critical point, this its minimum.}
\eit

Let $\nu $ denote such an arc. Note that $\nu $ has an analog that shares
its endpoint and sits where $\theta    \ge   \theta _{o}$. The latter is
denoted below by $\nu'$.

\substep{Step 3}
The arc $\nu $ has $\theta $--preserving
preimages in every large $j$ version of $C_{j}$; there are $n$ such preimages,
each with its endpoints on the the $\theta =\theta _{o}$ locus in
$C_{j}$. To see how these appear, let $e$ denote the edge in $T$ whose largest
angle is $\theta _{o}$. Fix a parameterization of the component $K_{e}$ in
$C_{j}$'s version of $C_{0}- \Gamma $. Each $\theta $ preserving
preimage of $\nu $ appears in the corresponding parametrizing cylinder as a
closed arc with both endpoints on the $\sigma =\theta _{o}$ circle.
Denote these arcs as $\{\nu _{j,k}\}_{1 \le k \le n}$.

To say more about these preimages, remember that the constant $|s| $ slices of
$E$ converge as $|s|    \to   \infty $ as a
multiple cover of a $\theta =\theta _{o}$ closed Reeb orbit in $S^1\times  S^2$.
The latter has a tubular neighborhood with a function,
$f$, with the following two properties: First, it vanishes on the Reeb orbit
in question. Second, $df=\frac{1 }{ {2\pi }}(pd\varphi  -p'dt)$ where $p$ and $p'$
are the relatively prime integers that define $\theta_{o}$ via \eqref{eq1.8}.
Because $f$ has limit 0 as $|s|    \to \infty $ on $E$, the integral over $\nu $ of $df$
is small. Moreover, given
$\varepsilon  > 0$, there exists $R_{\varepsilon }$ such that the integral
of $df$ over any $R \ge R_{\varepsilon }$ version of $\nu $ from \eqreft4{20}
has absolute value less than $\varepsilon $.

Granted this, it then follows that the same is true for the integral of $df$
over any $R \ge R_{\varepsilon }$ and sufficiently large $j$ version
of any of the $\theta $--preserving preimages of $\nu $. This
understood, use the parametrization of $K_{e}$ to identify the $\mod(2\pi\mathbb{Z})$
image of the latter integral as that of $(p{q_{e}}' - p'q_{e})dv$ between the two endpoints
of the arc $\nu _{j,k}$. In particular, this
means that when $R$ is large and $j$ is large, the two endpoints of $\nu_{jk}$
on the $\sigma =\theta _{o}$ circle of the parametrizing
cylinder are very close to each other. They divide the circle into one very
short arc and one arc on which the integral of dv is almost $2\pi $.

\substep{Step 4}
Because the images of $\{C_{j}\}$
converge in $\hat{O}^{\hat{A}}/\Aut^{\hat{A}}$, their images
converge in the $\Delta _{o}$ factor of \eqref{eq3.12}. This fact and the final
conclusion from Step 3 have the following consequence: The very short arc in
the $\sigma =\theta _{o}$ parametrizing cylinder circle between the
endpoints of any large $R$ and correspondingly large $j$ version of $\nu_{j,k}$
contains at most one missing point in its interior.

As is explained next, such a $\sigma =\theta _{o}$ arc must contain
precisely one missing point. For this purpose, let $\theta _{*}$
denote the minimum of $\theta $ on $\nu $. Were there no missing point in
the indicated short arc, then one of the $\theta $--preserving preimages of
$\nu $ in $C_{j}$ would be homotopic rel boundary in the
$[\theta , \theta_{*}]$ portion of $C_{j}$ to an arc lying entirely in the
$\theta= \theta _{o}$ locus. Such a homotopy could then be projected back to $S$
to give a homology rel boundary in the model curve for $S$ between the arc
$\nu $ and the disjoint union of an arc in the $\theta =\theta _{o}$
locus and a union of very small radius circles, each surrounding some point
that maps to one of the immersion singularities in S. No such homology is
possible because any constant $|s| $ slice of $E$ generates a
non-trivial homology class in $S_{0}$.

\substep{Step 5}
As a consequence of the result from Step
4, each end of $S$ where $\lim_{| s| \to \infty }  \theta $ is neither
0 nor $\pi $ labels $n$ ends of any sufficiently large $j$ version of $C_{j}$.
These $n$ ends have the same $|s|    \to   \infty $ limit of
$\theta $ as their namesake in S; and by virtue of \fullref{prop:4.6}, they are
all convex side if their namesake is a convex side end. Otherwise, they are
all concave side ends. Here is one way to view this correspondence: Let $E\subset S$
denote an end as in the previous steps. Fix some large $R$, and
the concatenation of the arc $\nu $ with its $\theta    \ge   \theta_{o}$
cousin $\nu '$ defines a closed loop in $E$ that is homologous to a
constant $|s| $ slice. Such a loop has $\theta $--preserving
preimages in each sufficiently large $j$ version of $C_{j}$. Any such preimage
must be a union of one $\theta $--preimage of $\nu $ and one of $\nu '$.
Indeed, it must, in any event, contain the same number of $\nu $ preimages
as $\nu '$ preimages; and said number must be 1 because of the very small
length of one of the arcs between the endpoints of each $\nu _{j,k}$ in
the $\sigma =\theta _{o}$ circle of the relevant parametrizing
cylinder.

Granted the preceding the concatenation of $\nu $ with $\nu' $ has $n$
distinct $\theta $ preserving preimages in all large $j$ versions of $C_{j}$.
Since the aforementioned short arc between the endpoints of each $\nu_{j,k}$
contains one and only one missing point on the $\sigma =\theta_{o}$
circle, each preimage of the $\nu -\nu' $ concatenation is
homologous in $C_{j}$ to the constant $|s| $ slice in some end of
$C_{j}$ where the $|s|    \to   \infty $ limit of $\theta $ is
$\theta _{o}$. Moreover, these preimages account for $n$ distinct ends in
$C_{j}$. Finally, an appeal to \fullref{prop:4.6} finds that distinct
$\lim_{| s| \to \infty }  \theta =\theta _{o}$ ends of $S$ label
disjoint, n-element subsets of such ends in $C_{j}$.

\substep{Step 6}
An end $E \subset S$ corresponds to a
vertex in the $S$ version of the circular graph $\underline {\Gamma }_{o}$.
As such, it comes with an integer weight. Meanwhile, the corresponding $n$
ends in $C_{j}$ correspond to $n$ vertices on the $C_{j}$ version of the graph
$\underline {\Gamma }_{o}$. As is explained here, each of the latter
vertices have the same integer weight as $E$'s vertex.

The conclusions of Step 5 guarantee that the $n+1$ weights involved all have
the same sign. Here is how to compute their magnitudes: The integral of the
form $\frac{1 }{ {2\pi }}(pdt+p'd\varphi )$ over any constant
$|s| $ slice of a relevant end has the form $m (p^{2}+p'^{2})$
where $m$ is the desired magnitude.

Apply this last observation first to the concatenation of $\nu $ and its
$\theta    \ge   \theta _{o}$ cousin $\nu '$ in a given end $E \subset S$.
Since this concatenation is homologous to a constant $|s| $
slice, the form $\frac{1 }{ {2\pi }}(pdt+p'd\varphi )$ integrates
over this concatenation to give $m_{E} (p^{2}+p'^{2})$ where $m_{E}$ is
the absolute value of the integer weight for $E$'s vertex in the $S$ version of
$\underline {\Gamma }_{o}$. Next, apply the observation to any one of the
$\theta $--preserving preimages of the $\nu -\nu '$ concatenation in each
very large $j$ version of $C_{j}$. As there are $n$ of these, the integral over
any one is $m_{E} (p^{2}+p'^{2})$ as well. The desired equality follows
because one of these preimages is homologous to the constant $|s|$
slices any given E-labeled end in $C_{j}$.

\substep{Step 7}
The results of the previous steps imply
that $T_{S}$ is isomorphic to $T$ in the case that $n = 1$. Such an isomorphism
is obtained via the correspondence given in Step 5 between the ends of $S$ and
those of $C_{j}$. In the hypothetical $n> 1$ case, it implies that each $T$
version of the group $\Aut_{o}$ has a $\mathbb{Z}/(n\mathbb{Z})$ subgroup. More
to the point, the following is true: Let ${O^{\hat{A}}}_{v } \subset O^{\hat{A}}$
denote a component whose $\Aut^{\hat{A}}$ orbit
contains the the limit point of the image of $\{C_{j}\}$. Then
$\Aut^{\hat{A}}_{v}$ has a canonical $\mathbb{Z}/(n\mathbb{Z})$ subgroup
since $n$ also divides the integer pair that is associated to the edge in $T$
with the smallest angle vertex.

This step explains an observation that is used in the subsequent step in two
ways: It is used to prove that the image of $S$ in $\mathbb{R}  \times
\hat{O}^{\hat{A}}/\Aut^{\hat{A}}$ is the limit of the images of
$\{C_{j}\}$ in the case that $n= 1$, and it is used to preclude the case
that $n> 1$.

To start, let $E$ and $E'$ denote ends of $S$ where $\lim_{| s| \to \infty} \theta =\theta _{o}$
and such that travel in the oriented
direction along a component of the $\theta =\theta _{o}$ locus in the
model curve of $S$ proceeds from large $|s| $ on $E$ to large $|s| $ on $E'$. The case that
$E = E'$ is allowed here. In any event, let
$\gamma $ denote the component of the $\theta =\theta _{o}$ locus in
question and let $r_{\gamma }$ denote the integral along $\gamma $ of the
1--form $(1-3\cos^{2}\theta _{o}) d\varphi -\surd 6\cos\theta_{o} dt$.

Now, fix $R$ very large and let $z$ denote the endpoint on $\gamma $ of $E$'s
version of the arc $\nu $. Meanwhile, let $z'$ denote the endpoint on $\gamma$ of the $E'$
version of this arc in the case that $E'  \ne  E$. If $E' = E$,
take $z'$ to be the second of the two endpoints of the arc $\nu $. Deform
$\gamma $ slightly if it passes through a point in $Y$ so that the result
misses $Y$, lives where $\theta    \le   \theta _{o}$, and agrees with
$\gamma $ on $E$ and $E'$. Let $\gamma _{R}$ denote the portion of such a
deformation that runs between $z$ and $z'$. If $R$ is large, then $r_{\gamma}$ is
very nearly the integral along $\gamma _{R}$ of the pull-back
of the 1--form $(1-3\cos^{2}\theta _{o}) d\varphi -\surd 6\cos\theta _{o} dt$.
Use $r_{\gamma ,R}$ to denote the latter integral.
Thus, the $R \to \infty $ limit of $\{r_{\gamma ,R}\}$ is $r_{\gamma}$.

With $R$ fixed and then $j$ very large, the arc $\gamma _{R}$ has $n$ disjoint,
$\theta $ preserving preimages in any parametrizing cylinder for the
$C_{j}$ version of the component $K_{e}$. Any such preimage runs from very
close to one of $E$'s missing points on the $\sigma =\theta _{o}$ circle
in the oriented direction to a point that is very close to the subsequent
missing point, this one labeled by $E'$. This understood, it then follows from
\fullref{prop:4.6} that the integral of the $Q =Q_{e}$ version of $\alpha_{Q}(\theta _{o})dv$
between these two missing point is very nearly $r_{\gamma }$.

Now, this arc between the two missing points labels one of the coordinates
for $C_{j}$ in the symplex $\Delta _{o}$ that appears in \eqref{eq3.12}, and it
follows from what has just been said that the value of this coordinate is
very nearly $r_{\gamma }$.

\substep{Step 8}
In the case that $n = 1$, the results of the
preceding step assert that the assigned point to any large $j$ version of
$C_{j}$ in any given $\Delta _{o}$ is very near that assigned to $S$. The
implication is that the $j \to \infty $ limit of these images is that of
$S$.

In the hypothetical $n> 1$ case, the results from the preceding step imply
that the image of any large $j$ version of $C_{j}$ in any given $\Delta_{o}$ is very close
to the set of points that are fixed by the $\mathbb{Z}/(n\mathbb{Z})$ subgroup version of
$\Aut_{o,v}$. In fact, the results from
the preceding step imply that the image of $C_{j}$ in the space $\times_{o}  \Delta _{o}$
is very near the fixed set of the canonical $\mathbb{Z}/(n\mathbb{Z})$ subgroup of
$\times _{o} \Aut_{o,v}   \subset \Aut^{\hat{A}}$. Granted the observations
from \fullref{prop:4.6}, this
then implies that the image in $\hat{O}^{\hat{A}}/\Aut^{\hat{A}}$
of any large $j$ version of $C_{j}$ is very close in
$O^{\hat{A}}/\Aut^{\hat{A}}$ to $(O^{\hat{A}}- \hat{O}^{\hat{A}})/\Aut^{\hat{A}}$.

This last conclusion rules out the $n> 1$ case because it is incompatible
with the initial assumption of the convergence in
$\hat{O}^{\hat{A}}/\Aut^{\hat{A}}$ of the image of $\{C_{j}\}$.

\setcounter{section}{4}
\setcounter{equation}{0}
\section{The first chapter of story when $N_{ - }+\hat{N}+\text{\c{c}}_{ - }+
\text{\c{c}}_{ + }$ is greater than 2}\label{sec:5}

Introduce as in \fullref{sec:1c}, the larger space
${\mathcal{M}^{* }}_{\hat{A}}$ whose elements consist of honest
subvarieties in the sense of \eqreft15 along with `multiple
covers' of honest subvarieties. \fullref{sec:1c} also
introduced a stratification of ${\mathcal{M}^{* }}_{\hat{A}}$. The
first subsection below summarizes results about the local
structure of ${\mathcal{M}^{* }}_{\hat{A}}$ and its
stratification. The remaining subsections contain the proofs of
these results.

\subsection{The local structure of ${\mathcal{M}^{* }}_{\hat{A}}$ and its stratification}\label{sec:5a}

As outlined in the first section, the space ${\mathcal{M}^{*
}}_{\hat{A}}$ consists of equivalence classes of pairs $(C_{0},
\phi )$ where $C_{0}$ is a complex curve homeomorphic to an $N_{
+ }+N_{ - }+\hat {N}$ times punctured sphere and $\phi $ is a
proper, pseudoholomorphic map from $C_{0}$ into
$\mathbb{R}\times  (S^1 \times S^{2})$ whose image is a
subvariety as defined in \eqreft15. Moreover, the pair
$(C_{0}, \phi )$ must be compatible with the data set $\hat{A}$
in the following sense: First, there is a 1--1 correspondence
between the ends of $C_{0}$ and the 4--tuples in $\hat{A}$; and
this correspondence must pair an end and a 4--tuple if and only if
the 4--tuple comes from the end as described in
\fullref{sec:1a}. Second, the integers $\text{\c{c}}_{ +
}$ and $\text{\c{c}}_{ - }$ are the respective intersection
numbers between $C_{0}$ and the $\theta  = 0$ and $\theta =\pi $
cylinders. To be precise here, note that there is a finite number
of $\theta  = 0$ and $\theta =\pi $ points in $C_{0}$. This
understood, a sufficiently generic but compactly supported
perturbation of $\phi $ gives an immersion of $C_{0}$ into
$\mathbb{R}\times  (S^1\times  S^2)$ that has transversal
intersections with the $\theta  = 0$ and $\theta =\pi $ loci. The
latter has a well defined intersection number with both loci, and
these are respectively $\text{\c{c}}_{ + }$ and $\text{\c{c}}_{-
}$. As noted in \fullref{sec:1c}, pairs $(C_{0}, \phi )$
and $(C_{0}, \phi')$ define the same point in ${\mathcal{M}^{*
}}_{\hat{A}}$ if $\phi '$ is obtained from $\phi $ by composing
with a holomorphic diffeomorphism of $C_{0}$.

A base for the topology on ${\mathcal{M}^{* }}_{\hat{A}}$ is
depicted in \eqref{eq1.24}. \fullref{thm:1.3} asserts that
${\mathcal{M}^{* }}_{\hat{A}}$ is a smooth orbifold whose
singular points consist of those pairs $(C_{0}, \phi )$ where
there is a holomorphic diffeomorphism that fixes $\phi $.
\fullref{thm:1.3} also asserts that the inclusion of
$\mathcal{M}_{\hat{A}}$ in ${\mathcal{M}^{* }}_{\hat{A}}$ is a
smooth embedding onto an open subset. \fullref{thm:1.3} is
proved below so grant it for the time being.

As noted in \fullref{sec:1c}, the strata of ${\mathcal{M}^{* }}_{\hat{A}}$
are indexed by ordered triples of the form $(B, c, \mathfrak{d})$
where $B  \subset \hat{A}$ is a
set of $(0,-,\ldots)$ elements, $c$ is a
non-negative integer no greater than $N_{ + }+N_{ - }+\hat {N}+\text{\c{c}}_{ - }+
\text{\c{c}}_{ + }-2-|B| $ and $\mathfrak{d}$ is a partition of the integer
$d \equiv N_{ + }+|B| +c$ as a sum of positive integers. By
way of a reminder, the stratum
$\mathcal{S} _{B,c,\mathfrak{d}}\subset{\mathcal{M}^{* }}_{\hat{A}}$ lies in the subset,
$\mathcal{S}_{B,c}$ that consists of the equivalence classes of pairs
$(C_{0},\phi )$ that have two properties stated next. In these
statements and subsequently, functions on $\mathbb{R}\times (S^1\times  S^{2})$
and their pull-backs via $\phi $ are not distinguished by
notation except in special circumstances. Here is the first property: The
curve $C_{0}$ has precisely $c$ critical points of $\theta $ where the value
of this function is neither 0 nor $\pi $. Here is the second: The ends that
correspond to elements in $B$ are the sole convex side ends of $C_{0}$ where
the $|s|\to\infty $ limit of $\theta $ is neither 0 nor
$\pi $ and whose version of \eqref{eq2.4} has a strictly positive integer $n_{E}$.
To define $\mathcal{S} _{B,c,\mathfrak{d}}$, introduce the map $I_{d}$ to denote the space
of unordered $d$--tuples of points in $(0, \pi )$ and the map $f\co  \mathcal{S}_{B,c}\to I_{d}$
that sends a given $(C_{0}, \phi )$ to the $d$--tuple that consists of the critical
values in $(0, \pi )$ of $\theta $'s
pull-back to $C_{0}$ and the $|s|\to\infty $ limits in
$(0, \pi )$ of $\theta $ on the concave side ends of $C_{0}$ and on the ends
that correspond to the 4--tuples in $B$. The stratum $\mathcal{S} _{B,c,\mathfrak{d}}$ is the
inverse image via $f$ of the stratum in $I_{d}$ that is labeled by the
partition $\mathfrak{d}$.

The next proposition describes the local structure of
$\mathcal{S} _{B,c,\mathfrak{d}}$. The proposition speaks of a
locally constant function on $\mathcal{S} _{B,c,\mathfrak{d}}$
whose value at a given point defined by some pair $(C_{0}, \phi
)$ is the number of distinct critical values of $\theta $ in the
subset of the critical values in $(0, \pi )$ that do not arise
via \eqref{eq1.8} from an integer pair component of either a
$(0,+,\ldots)$ 4--tuple in $\hat{A}$ or a 4--tuple in $B$. Let $m$
denote this locally constant function.

\begin{proposition}\label{prop:5.1}

If not empty, then the stratum $\mathcal{S} _{B,c,\mathfrak{d}}$ is a smooth orbifold in
$\mathcal{M}^{* }$ that intersects $\mathcal{M}^{* }-\mathbb{R}$ as a smooth manifold.
In this regard, a given component has dimension $N_{ + }+|B| +c+m+2$.

\end{proposition}
\noindent A description of the components of any given version of
$\mathcal{S} _{B,c,\mathfrak{d}}$ is provided in Sections~\ref{sec:6} and~\ref{sec:8}.

The proof of \fullref{thm:1.3} has five parts and these are presented the next two
subsections. The final subsection contains the proof of \fullref{prop:5.1}.

\subsection{Parts 1--4 of the proof of Theorem \ref{thm:1.3}}\label{sec:5b}

The proof is very much like that in \cite[Proposition~2.9]{T3}. The five parts
that follow focus on the points that differ.

\step{Part 1}
Suppose that $(C_{0}, \phi )$ defines a point in ${\mathcal{M}^{* }}_{\hat{A}}$.
There is, in all cases, a fixed radius ball
subbundle $B_{1}\subset\phi ^*T_{1,0}(\mathbb{R}\times  (S^1\times  S^{2}))$
and an exponential map, $e$, that maps $B_{1}$ into
$\mathbb{R}\times  (S^1\times  S^{2})$ so as to embed each fiber
and send the zero section to $C_{0}$. As noted in  \cite[Section~2.D]{T3} for
pairs that map to $\mathcal{M}_{\hat{A}}$, the bundle $\phi^*T_{1,0}(\mathbb{R}\times
(S^1\times S^{2}))$ splits
as a direct sum, $W  \oplus  N$, of complex line bundles such that the
differential, $\phi _{* }$, of $\phi $ maps $T_{1,0}C_{0}$ into $W$,
and $N$ restricts to the points where $\phi _{* }\ne 0$ as the
pull-back normal bundle. In this regard, $e$ can be chosen so as to embed the
fibers of $B_{1}\cap N$ and of $B_{1}\cap  W$ as
pseudoholomorphic disks. Note that in case where $\phi $ is not almost
everywhere 1--1, there is a complex curve, $C_{1}$, with an attending, almost
everywhere 1--1, pseudoholomorphic map, $\phi _{1}$, to $\mathbb{R}\times
(S^1\times S^2)$ whose image is $C \equiv\phi (C_{0})$. In
this case, $\phi $ factors as $\phi _{1} \circ \psi $ where $\psi $ is
a holomorphic, branched covering map to $C_{1}$. The
$\phi_{1}^*T_{1,0}(\mathbb{R}\times  (S^1\times  S^2))$
decomposition as $W  \oplus N$ then pulls back by $\psi $ to give the $W \oplus N$
decomposition for $\phi ^*T_{1,0}(\mathbb{R}\times  (S^1\times  S^2))$.
Because of this factoring property, the map $e$ can be
chosen to be invariant under the action on $\phi ^*T_{1,0}(\mathbb{R}\times(S^1\times  S^2))$
of the group of holomorphic diffeomorphisms of $C_{0}$ that fix $\phi $.
Such a choice is assumed in what follows.

\step{Part 2}
\cite[Section~2.D]{T3} described an operator, $D_{C}$, whose domain is a
certain Hilbert space of sections of $\phi ^*T_{1,0}(\mathbb{R}\times(S^1\times  S^2))$
and whose range is a Hilbert space of sections
of $\phi ^*T_{1,0}(\mathbb{R}\times  (S^1\times  S^2))
\otimes T^{0,1}C_{0}$. The discussion in \cite[Section~2.D]{T3} involves
only pairs $(C_{0}, \phi )$ that define points in $\mathcal{M}_{\hat{A}}$,
but the story generalizes in an almost verbatim fashion to define
$D_{C}$ for any pair that defines a point in ${\mathcal{M}^{* }}_{\hat{A}}$.

By way of a reminder, $D_{C}$ is defined from an operator, $\underline {D}$,
whose kernel is the space of first order deformations of $\phi $ that result
in maps that are pseudoholomorphic with respect to the given complex
structure on $C_{0}$ and the almost complex structure $J$. The salient
features of $\underline {D}$ are as follows: First, $\underline {D}$ is a first
order differential operator that maps sections of
$\phi ^*T_{1,0}(\mathbb{R}\times  (S^1\times  S^2))$ to sections of
$\phi^*T_{1,0}(\mathbb{R}\times  (S^1\times  S^2))$ so as to map
sections of $W$ to those of $W \otimes T^{0,1}C_{0}$. In
particular, if $v$ is a section of $T_{1,0}C_{0}$, then $\phi _{* }v$
is a section of $W$ and $\underline {D}\phi _{* }v = \phi _{* }(\bar {\partial }v)$.
Second, composing \underline {D} with orthogonal
projection onto the $N$ summand in $\phi ^*T_{1,0}(\mathbb{R}\times  (S^1\times  S^2))$
defines an $\mathbb{R}$--linear operator
that sends a section, $\eta $, of $N$ to a section of the form
\begin{equation}\label{eq5.1}
\bar {\partial }\eta+ \nu \eta+ \mu
\bar {\eta };
\end{equation}
here $\nu$ is a fixed section of $T^{1,0}C_{0}$ and $\mu $ is a
fixed section of $N^{2 } \otimes T^{1,0}C_{0}$.

The operator $\underline {D}$ has an extension as a Fredholm
operator that maps a certain weighted Hilbert space completion of
its range to that of its domain. The inner products that define the
range and domain Hilbert spaces for sections of the $N$ and $N
\otimes T^{1,0}C_{0}$ summands are as depicted in \cite[Equation (2.7)]{T3}. Similar
weighted norms are used for the respective $W$
and $W  \otimes T^{1,0}C_{0}$ summands, but the latter insure that
all sections are square integrable. In this regard, some care must
be taken when there are holomorphic diffeomorphisms of $C_{0}$ that
preserve $\phi $. To elaborate, note first that $\phi $ is finite to
one, and as a consequence, the set of diffeomorphisms of $C_{0}$
that preserve $\phi $ defines a finite group. Let $G_{C}$ denote the
latter. The norms used for these Hilbert spaces can and should be
taken to be $G_{C}$ invariant.

\step{Part 3}
In the case that $C_{0}$ is a disk or cylinder, the operator $D_{C}$ is
$\underline {D}$ in the just described Fredholm context. When $C_{0}$
has negative Euler characteristic, $D_{C}$ is obtained from this Fredholm
$\underline {D}$ by composing with an orthogonal projection on the latter's
range. The definition of this projection requires the choice of some
$3(N_{+ }+N_{ - }+\hat {N}-1)$ dimensional, $G_{C}$--invariant vector space of
sections of $T_{1,0}C_{0}\otimes T^{1,0}C_{0}$ that projects
isomorphically to the cokernel of $\bar {\partial }$. Let $V$ denote the
latter choice and let $\prod $ denote the orthogonal projection onto $\phi_{* }$V.
Then $D_{C} = (1-\prod )\underline {D}$.

The following proposition describes the important facts about the kernel and
cokernel of $D_{C}$. The proof uses verbatim arguments from the proof of
\cite[Propositions~2.9]{T3} and so is omitted.

\begin{proposition}\label{prop:5.2}

Suppose that $(C_{0}, \phi )$ defines a point in ${\mathcal{M}^{* }}_{\hat{A}}$.
Then the operator $D_{C}$ has index $N_{ + }+2(N_{ - }+\hat {N}+\text{\rm\c{c}}_{ + }
+\text{\rm\c{c}}_{ -}-1)$. Moreover, this is its kernel dimension as its cokernel is trivial.

\end{proposition}

\step{Part 4}
Here is the significance of $D_{C}$: A small ball in the vector space $V$
parametrizes the complex structures on $C_{0}$ that are near to the given
one. This understood, an element in the kernel of $D_{C}$ gives a
deformation of $\phi $ that is pseudoholomorphic to first order with respect
to a complex structure that is parametrized by a point in $V$. More to the
point, the implicit function theorem can be employed in a relatively
standard manner to obtain the following description of a neighborhood of the
point defined by $(C_{0}, \phi )$ in ${\mathcal{M}^{* }}_{\hat{A}}$:
There is a ball, $B  \subset  \kernel(D_{C})$, a smooth function, $f$,
from $B$ to $\cokernel(D_{C})$ that vanishes with its differential at zero, and
a homeomorphism between a neighborhood of $(C_{0}, \phi )$'s point in
${\mathcal{M}^{* }}_{\hat{A}}$ and the quotient of $f^{-1}(0)$
by the action of the group $G_{C}$ on the kernel of $D_{C}$. To
elaborate, this homeomorphism comes from a $G_{C}$--equivariant map, $F$, from
$B$ to the domain space of $D_{C}$ that maps the origin to 0 with differential
at 0 the identity on $\kernel(D_{C})$. The homeomorphism is obtained by
restricting the composition $e \circ F$ to $f^{-1}(0)$.

It is worth a moment now to say something about why these local
charts map onto a neighborhood of $(C_{0}, \phi )$ as defined by
\eqref{eq1.24}. The point here is that if $({C_{0}}', \phi ')$ is
close to $(C_{0}, \phi )$ in the sense of \eqref{eq1.24}, then
the map $\phi ' \circ \psi $ can be obtained from $\phi $ by
composing the exponential map from $\phi
^*T_{1,0}(\mathbb{R}\times (S^1\times  S^2))$ with a small normed
section. A change in the diffeomorphism changes the section, and
 \cite[Proposition~2.2]{T4} can be used to find diffeomorphisms that
give small normed section in the domain of $D_{C}$. Granted this,
the implicit function theorem asserts that there is a unique such
small normed section from the image of $F$.

Note that in the case that $\cokernel(D_{C}) = \{0\}$, then
$\kernel(D_{C})/G_{C}$ is a local Euclidean orbifold chart for a
neighborhood of $(C_{0}, \phi )$'s point in ${\mathcal{M}^{* }}_{\hat{A}}$.
As is usually the case with implicit function theorem
applications of the sort just described, these charts fit together to give a
smooth orbifold structure to the set of points in ${\mathcal{M}^{* }}_{\hat{A}}$
that are defined by pairs $(C_{0}, \phi )$ with
trivial $D_{C} \cokernel$. Granted this, the assertion in \fullref{thm:1.3} about
the local structure of ${\mathcal{M}^{* }}_{\hat{A}}$ follows
directly from \fullref{prop:5.2}.

\subsection{Part 5 of the proof of Theorem \ref{thm:1.3}}\label{sec:5c}

This part explains why the set inclusion of $\mathcal{M}_{\hat{A}}$
into ${\mathcal{M}^{* }}_{\hat{A}}$ is a topological embedding.
Note that the equivalence of the two topologies is, in fact, implied by the
statement of \cite[Proposition~3.2]{T4}. However, the arguments in \cite{T4} for
this proposition focused for the most part on issues that are not present in
the analogous compact symplectic manifold assertion and so left the proof of
the equivalence to the reader. In its deference to \cite[Proposition~3.2]{T4},
the proof of \cite[Proposition~2.9]{T3} does not address the implied
equivalence between the two topologies on $\mathcal{M}_{\hat{A}}$. The
explanation that follows has nine steps.

\step{Step 1} The inclusion
$\mathcal{M}_{\hat{A}}\to{\mathcal{M}^{* }}_{\hat{A}}$ is
continuous since the condition for closeness given by
\eqref{eq1.24} implies that given in \eqref{eq1.13}. Thus, to
prove it an embedding, it is enough to prove that the condition
for closeness in \eqref{eq1.13} implies that in \eqref{eq1.24}.
This understood, the task is as follows: Fix $C
\in\mathcal{M}_{\hat{A}}$ and suppose that some positive, but
small $\kappa $ is given. Find some positive $\kappa '$ such that
when $C' \in \mathcal{M}_{\hat{A}}$ obeys the $\kappa '$ version
of \eqref{eq1.13}, then there is a diffeomorphism between $C$'s
model curve and that of $C'$ that makes the $\kappa $ version of
\eqref{eq1.24} hold. In what follows, $C_{0}$ and ${C_{0}}'$ are
the respective model curves for $C$ and $C'$ while $\phi $ and
$\phi '$ are their respective pseudoholomorphic maps to
$\mathbb{R}\times (S^1\times S^2)$.

To construct the required diffeomorphism, let $\vartheta\subset C_{0}$
denote the set of points that are mapped by $\phi $ to singular
points of $C$. The bundle $N$ restricts to $C_{0}-\vartheta $ as the
normal bundle to the embedding. Fix $\varepsilon >0$ but very much
less than one, and $\delta\in (0, \varepsilon ^{4})$. Now let $U$
denote the union of the radius $\delta $ disks about the points in
$\vartheta $ and the $|s|\ge  1+ 1/\varepsilon $ portions
of $C_{0}$. In this regard, choose $\varepsilon $ so that the
$|s|\ge  1+1/\varepsilon $ part of $U$ is far out on the ends of $C_{0}$.

With $\varepsilon $ and $\delta $ chosen, there is an exponential map that
is defined on a small, constant radius disk bundle in $N$ over
$C_{0}-U$ so as to embed this disk bundle as a tubular neighborhood of
$\phi (C_{0}-U)$ and to embed each fiber as a pseudoholomorphic disk.
Note that there is quite a bit of freedom here with the choice for this
exponential map and this freedom is used in what follows to fine tune things
near the boundary of $C_{0}-U$. In any event, suppose that the disk
bundle and the exponential map have been fixed. Let $N_{1}\subset N$
denote the disk bundle and $e\co  N_{1}\to\mathbb{R}\times  (S^1\times S^2)$ the exponential map.

Now if $C'$ is very close to $C$ in the sense of \eqref{eq1.13},
then $C'$ must intersect the image of each fiber of $N_{1}$ over
$C_{0}-U$ in precisely one point with multiplicity one. Indeed,
because $C$ and $C'$ come from the same version of
$\mathcal{M}_{\hat{A}}$, the net intersection number with any
given fiber must be one. Meanwhile, all such intersection points
count with positive weight by virtue of the fact that the fibers
are embedded as pseudoholomorphic disks.

Because $C'$ intersects the image of each fiber of $N_{1}$ over $C_{0}-U$ just
once, it can be written in the tubular neighborhood of $\phi (C_{0}-U)$
as the image of $e \circ \eta $ where $\eta $ is a very small normed
section of $N_{1}$.

The characterization of $C'$ as the image of $e \circ \eta $
defines a diffeomorphism, $\psi $, between $C_{0}-U$ and a part
${C_{0}}'$ by demanding that $\phi ' \circ \psi  = e \circ \eta
$. This diffeomorphism is such as to make $\dist(\phi , \phi '
\circ \psi )$ and the ratio $r(\psi )$ from \eqref{eq1.24} both
very small on the whole of $C_{0}-U$ in the case that $C'$ is
very close to $C$ in the sense of \eqref{eq1.13}. The argument as
to why $r(\psi )$ is small is deferred to Step 3.

\step{Step 2} This step constitutes a digression make four points
about 4--dimensional pseudoholomorphic geometry. To set the stage,
let $X$ denote the 4--manifold and $J$ an almost complex structure on
$X$. The relevant case is that where $X =\mathbb{R}\times  (S^1
\times S^2)$ and $J$ is the almost complex structure that is
described in \fullref{sec:1}. Let $D$ denote a standard disk
in $\mathbb{C}$, and suppose that an embedding of $D\times D$ into a
$X$ has been specified with the following specific property: The
image of $D\times  0$ and the image of each $\{(z \times  D)\}_{z
\in D}$ disk is pseudoholomorphic.

\begin{point}
There is a ball about the image of $(0, 0)$ in $X$ with complex coordinates $(x, y)$
that have three properties: First, $y = 0$ is in the image of the disk $D\times  0$.
Second, each constant $x$ disk lies in the image of some disk from the collection
$\{z\times  D\}_{z \in D}$. Finally, $T^{1,0}X$ is spanned over this coordinate chart
by the 1--forms
\begin{equation}\label{eq5.2}
\nu \equiv dx + \sigma  d\bar {x} \qquad\text{and}\qquad \nu' \equiv dy+ \sigma 'd\bar {x},
\end{equation}
where $\sigma $ and $\sigma '$ vanish both at the origin and along the whole $y = 0$ locus.
\end{point}
\noindent The proof that such coordinates exist is straightforward and left to the reader.

To make the remaining points, let $B$ denote the coordinate chart just
described, let $\Omega\subset\mathbb{C}$ denote a disk and let $w\co \Omega\to B$
denote a proper, pseudoholomorphic map.

\begin{point}
The pull-back of $x$ to $\Omega $ obeys
$\bar {\partial }x +\sigma \bar {\partial }\bar {x} = 0$.
Indeed, this follows by virtue of the fact that $w^*\nu$ is a section of $T^{1,0}\Omega $.
As a consequence, $| \bar {\partial }x|  \ll | \partial x| $ when
$|y| $ is small on the image of $\Omega $. Note that this implies that the critical
points of the pull-back of $x$ are the zeros of $\partial x$.
\end{point}

\begin{point}
Let $z  \in\Omega $ denote a zero of $\partial x$.
There is holomorphic coordinate, $w$, for a neighborhood of $z$ such that
$\partial x = w^{q}+{\mathcal{O}}(|w| ^{q + 1})$ where $q$ is a positive integer.
\end{point}
\noindent Indeed, this can be seen from the following considerations: The holomorphic
derivative of the equation from Point 2 gives one for $\partial x$ that has
the form $\bar{\partial}(\partial x)+\gamma\partial x+\hat{\gamma}\overline{\partial x}= 0$,
where $\gamma $ and $\hat {\gamma}$ are smooth functions on $\Omega $.
This last equation implies that
$\partial x$ vanishes near $z$ to leading order as a holomorphic function.

The fourth point is an immediate consequence of the latter two:

\begin{point}
Viewed as mapping $w^{-1}(B)$ to $\mathbb{C}$, the function $x$
looks locally like a ramified covering map onto its image.
\end{point}

\step{Step 3} This step explains why $r(\psi )$ is small at all
points in $C_{0}-U$ when $C'$ obeys a sufficiently small
$\kappa'$ version of \eqref{eq1.13}. To start, remark that by
virtue of what is said in Step 2, any given point in $\phi
(C_{0}-U)$ has local complex coordinates $(x, y)$ with the
following three properties: First, each $x = \constant$ disk is
the image of a fiber of $N_{1}$. Second, the disk where $y = 0$
is in $C$. Finally, $T^{1,0}(\mathbb{R}\times  (S^1\times  S^2))$
is spanned by the 1--forms $\nu $ and $\nu '$ as in \eqref{eq5.2}.

Let $B$ denote the domain in $\mathbb{R}\times  (S^1\times S^2)$
of these coordinates. The map $\psi ^{ - 1 }$ on the $\phi '$--inverse image
of $B$ is the composition of $\phi '$ with the projection to the $y = 0$ locus.
This being the case, the fact that $\phi '$ is pseudoholomorphic implies
that $\nu $ must pull-back via $\phi '$ to ${C_{0}}'$ as a form of type $(1,0)$,
and this implies that $r_{z}(\psi )=| \bar {\partial }x|/| \partial x| $
is the value of $| \sigma | $ at $(e \circ \eta )(z)$.

Granted the preceding, there exists $(\varepsilon
,\delta)$--dependent constants $\kappa _{0} > 0$ and $c_{0}$ with
the following significance: If $\kappa ' < \kappa _{0}$ and if
$C'$ obeys the $\kappa' $ version of \eqref{eq1.13}, then $\psi $
is well defined on $C_{0}-U$. Moreover, both $\dist(\phi , \phi '
\circ \psi )$ and $r_{(\cdot )}(\psi )$ are bounded by an
expression of the form $c_{0}\kappa' $ at all points in $C_{0}-U$.

\step{Step 4} To extend $\psi $, so as to make \eqref{eq1.24}
hold for all $z$ and very small $\kappa $, note first that
topological considerations imply that the complement $o$f $\psi
(C_{0}-U)$ in ${C_{0}}'$ must be diffeomorphic to $U$, thus a
union of some number of cylinders and some number of disks.
Moreover, each cylinder must bound one of the $|s|\sim
1/\varepsilon $ circles in the boundary of $\psi (C_{0}-U)$ and
each disk must bound one of the radius ${\mathcal{O}}(\varepsilon
)$ circles in the boundary.

The cylinder story is simpler than that for the disks, so it is treated
first. For this purpose, let $S \subset  U$ denote one of the cylinder
components, and let $\gamma $ denote its boundary circle. Let $\gamma ' \subset {C_{0}}'$
denote $\psi (\gamma )$ and let $S'  \subset {C_{0}}'$
denote the component of the complement of $\psi (C_{0}-U)$ whose
boundary circle is $\gamma '$. When $\varepsilon $ is large, both $S$ and
$S'$ are very close to an $\mathbb{R}$--invariant, pseudoholomorphic cylinder,
$S_{0}$. In this regard, take $S_{0}$ so that a multiple cover of its
defining Reeb orbit is the $|s|\to\infty $ limit of the
constant $|s| $ slices of $S$. Fix a point $x  \in S_{0}$ and a
pseudoholomorphic disk, $D$ with center at $x$ that is normal to $S_{0}$. Since
$TS_{0}$ is an orbit of the product of $\mathbb{R}$ with a 1--parameter subgroup
from the group $\mathbb{T}$ generated by $\partial _{t}$ and $\partial_{\varphi }$,
the corresponding $\mathbb{R}\times  S^1$ group can be
used to translate a small subdisk in $D$ centered at $x$ to each point in $S$; and
these translates foliate a tubular neighborhood of $S_{0}$ in
$\mathbb{R}\times  (S^1\times S^2)$ by pseudoholomorphic disks. The
exponential map on $N_{1}$ from Step 1 can and should be chosen so as to map
each fiber of $N_{1}$ near $\gamma $ into one of these $\mathbb{R}\times S^1$ translates of $D$.

Near $\gamma $, both $S$ and $S'$ intersect each fiber the same number of times
and in distinct points. Let $m$ denote this number. Let $\pi $ denote the
projection from the tubular neighborhood of $S_{0}$ to $S_{0}$ that moves
any given point to the center point of its particular translate of $D$. As
will now be explained, the restriction of $\pi $ to either $S$ or $S'$ defines a
degree $m$, unramified covering of $\pi (S)$. To see why, use the first point
in Step 2 to put coordinates $(x, y)$ on a neighborhood of any given point in
$S_{0}$ where the $x = \constant$ slices are the translates of $D$, and where the
$y = 0$ locus corresponds to $S_{0}$. Moreover, the 1--forms $\nu $ and $\nu '$
from \eqref{eq5.2} span $T_{1,0}(\mathbb{R}\times  (S^1\times  S^2))$ on
this neighborhood.

As set up, the projection $\pi $ on the parts of $S$ and $S'$ in this
neighborhood is the function $x$. This understood, the fourth point of Step 2
implies that $\pi $ restricts to either $S$ or $S'$ as a degree $m$, ramified
cover over $\pi (S)$. As such, its critical points are isolated, and each
counts positively to a ramification number. As both $S$ and $S'$ are cylinders,
the ramification number must be zero and so $\pi $ maps $S$ and $S'$ to $\pi
(S)$ as honest degree $m$ covers.

Now introduce the fibered product $S\times _{\pi } S'$, this the
subspace in $S\times S'$ of pairs with the same image via $\pi $
in $S_{0}$. The latter is a smooth manifold with projections to
$S$ and to $S'$. In fact, because $\pi $ is non-singular on both
$S$ and $S'$, these two projections are covering maps. Moreover,
each is trivial because $\pi $ has the same degree on $S$ as it
has on $S'$.  Thus, both such projections have sections. In
particular, there is a unique section over $S$ whose restriction
to $\gamma $ composes with the projection to $S'$ as the map
$\psi $. The latter section thus composes with projection to $S'$
so as to extend $\psi $ as a diffeomorphism from $S$ to $S'$.
This extension obeys any given small $\kappa $ version of
\eqref{eq1.24} over $S$ if $\varepsilon $ is small and then
$\kappa'$ very small.

\step{Step 5}
Suppose now that $\gamma $ is a circle in $\partial U$ that lies very
close to some point $z \in\Xi $ and let $\gamma '$ denote its $\psi $
image in ${C_{0}}'$. Let $D$ denote the disk in $C_{0}$ that $\gamma $ bounds
and let $D'  \subset {C_{0}}'$ denote the disk that $\gamma '$ bounds.
Since the whole of $D'$ is not in $\psi (C_{0}-U)$, its $\phi '$
image must lie very close to $\phi (z)$ and thus very close to $\phi (D)$.
In fact, the distance between any point of $\phi '(D')$ and any point of
$\phi(D)$ will be ${\mathcal{O}}(\varepsilon )$ when $\varepsilon $ is small.
As is argued in the subsequent steps, there are diffeomorphism between $D$ and
 $D'$ that extend $\psi $ with small $r(\psi )$.

To see why $\psi $ can be extended to map from $D$ to $D'$ with small
$r(\psi)$, remark that when $\varepsilon $ is small, then the results from Step 2
can be used to find a holomorphic coordinate, $u$, that is defined on the
radius $4\varepsilon $ disk centered at $z$, and complex coordinates $(x, y)$
centered at $\phi (z)$ with the following five properties: First, $\phi $ on
$D$ has the form
\begin{equation}\label{eq5.3}
\phi (u) = \bigl(u^{p + 1}, 0\bigr) + {\mathcal{O}}\bigl(|u| ^{p + 2}\bigr),
\end{equation}
where $p$ is a non-negative integer. Moreover,
\begin{equation}\label{eq5.4}
\phi^* dx = (p+1)u^{p} du + {\mathcal{O}}\bigl(|u| ^{p + 1}\bigr) \qquad \text{and}\qquad
\phi^*dy = {\mathcal{O}}\bigl(|u| ^{p + 1}\bigr).
\end{equation}
Second, the constant $x$ disks and the $y = 0$ disk are pseudoholomorphic.
Third, the forms $\nu $ and $\nu '$ from \eqref{eq5.2} span
$T^{1,0}(\mathbb{R}\times  (S^1\times  S^2))$ over the domain of these
coordinates. Finally, the $x = \constant$ disks where $|u|\ge 2\delta $
contain the image of the fiber disks in the bundle $N_{1}$ from
Step 1. Use $B$ in what follows to denote the domain of the $(x, y)$
coordinates.

Consider first the pull-back of $x$ to $\phi ^{-1}(B)$. Take
$\varepsilon $ small, and granted that $x = u^{p + 1}+{\mathcal{O}}(|u|^{p + 2})$,
if $\varepsilon $ is small, then the restriction of
$\partial x$ to the $u$ coordinate chart is non-zero away from the origin.
Moreover, with $W \subset C_{0}$ denoting the inverse image via $x$ of the
radius $\varepsilon ^{p + 1}$ disk about 0 in $\mathbb{C}$, the map $x$ sends
$W$ to the radius $\varepsilon ^{p + 1}$ disk in $\mathbb{C}$ as a degree
$p+1$ ramified cover with a single ramification point where $\partial x$
vanishes with degree $p$. Note that $W$ is a disk.

Let $W'  \subset {C_{0}}'$ denote the inverse image via $x$ of the
same radius $\varepsilon ^{p + 1}$ disk in $\mathbb{C}$. When $\varepsilon $
and $\kappa'$ are small, then $W'$ is also a disk since its boundary is a
small radius, embedded circle in $D'$. Since $x$ is pulled up from $W$ near the
boundary of $W'$, it has degree $p+1$ there as a map to the radius $\varepsilon
^{p + 1}$ circle in $\mathbb{C}$. Moreover, an appeal to the second point in
Step 2 finds that
\begin{equation}\label{eq5.5}
|\bar {\partial }x|  < c_{\varepsilon }\kappa '| \partial x|
\end{equation}
on the whole of $W'$ where $c_{\varepsilon }$ is determine once and for all by
$\varepsilon $. This last equation implies that all zeros count with
positive multiplicity. Note that all occur in the complement of the $\psi $
image of the $|u|  > 2\delta $ portion of $W$.

Here is the final remark for this step: According to the fourth point, $x$
maps $W'$ to $\mathbb{C}$ as a degree p+1, ramified cover. As it turns out, the
sum of the orders of vanishing of $\partial x$ at its zeros in $W'$ is equal
to $p$, but this fact is not proved directly.

\step{Step 6}
The argument for a small $r(\psi )$ extension of $\psi $ over $W'$ is
simplest in the case that $\partial x$ on $W'$ is zero at a single point and
this is also the only zero of $\partial \psi $ in $W'$. In the latter case,
define first a $C^{1}$ extension as follows: The function $x$ on both $W$ and $W'$
has a $p$'th root, this denoted by $x^{1 / p}$. On both $W$ and $W'$ this $p$'th
root provides a $C^{1}$ homeomorphism onto the radius $\varepsilon $ disk in
$\mathbb{C}$ centered at 0. This function is smooth and, in both cases, maps
the complement of $x^{-1}(0)$ diffeomorphically to the complement of 0
in the centered, radius $\varepsilon $ disk. The composition of the map
$x^{1 / p}$ from $W$ with its inverse to $W'$ thus defines a $C^{1}$
homeomorphism between $W$ and $W'$ that is smooth except at a single point. Let
$\psi _{0}$ denote the latter. Then $| \bar {\partial }\psi_{0}|  \ll |\partial \psi _{0}|$
when $\varepsilon $ is small by appeal to \eqref{eq5.5}. Moreover, $| \partial
\psi _{0}| $ is uniformly positive while $\bar {\partial }\psi
_{0}$ is zero at the one non-smooth point. This understood, a suitable
perturbation of $\psi _{0}$ then gives a smooth diffeomorphism, $\psi $,
with small $r(\psi )$.

\step{Step 7}
To proceed with the general case, consider that $x$, when viewed as a map
from $W'$ to $\mathbb{C}$, pulls back the complex structure from $\mathbb{C}$ on the
complement of the points where $\partial x$ vanishes. Indeed, this follows
from \eqref{eq5.5}. This pull-back complex structure extends over the zeros of
$\partial x$ to define a complex structure on $W'$ that makes $x$ a holomorphic map.

As will now be explained, there is a holomorphic coordinate for this new
complex structure on $W'$ that makes $x$ out to be a polynomial. To find such a
coordinate, remark that near the boundary of $W'$, $x$ is pulled up from $W$. In
particular, when $\varepsilon $ is small, $x$ has a $(p+1)$'st root near the
boundary of $W'$ that maps the boundary of $W'$ in a 1--1 fashion to the radius
$\varepsilon $ disk in $\mathbb{C}$. This understood, let $D_{\infty }$ denote
the complement in $\mathbb{C}\cup\infty $ of the radius $\varepsilon $
disk about the origin, and let $\tau $ denote a holomorphic coordinate on
$D_{\infty }$ that vanishes at 0 and has constant norm $1/\varepsilon $ on
the boundary of $D_{\infty }$. Now let $M$ denote the complex curve obtained
from the disjoint union of $D_{\infty }$ and $W'$ by identifying the boundary
of $D_{\infty }$ with the boundary of $W'$ by pairing points with $1/\tau  =x^{1 / (p + 1)}$.
Of course, $M$ is $\mathbb{CP}^{1}$ with strangely
presented holomorphic coodinate patches. The point of this
construction is that $x$ extends to the whole of $M$ as a degree $p+1$ holomorphic
map from $\mathbb{CP}^{1 }$ as $M$ to $\mathbb{CP}^{1}$ as $\mathbb{C}
\cup\infty $. Moreover, this extension has the property that the inverse
image by $x$ of the point $\infty $ is a single point, this the origin in
$D_{\infty }$. Thus, with the complement of the origin in $D_{\infty }$
viewed as $\mathbb{C}\subset M$, this extension of $x$ is a polynomial of
degree $p+1$ when written with the standard holomorphic coordinate on $\mathbb{C}$.

\step{Step 8}
Granted that $\varepsilon $ and $\kappa'$ are small, the next step
provides a holomorphic coordinate, $\tau $, on $\mathbb{C}$ such that
\begin{equation}\label{eq5.6}
|x(\tau )-\tau ^{p + 1}|\le\varepsilon ^{p +
2} \qquad\text{and}\qquad | x'(\tau )- (p+1)\tau ^{p}|\le\varepsilon
^{p + 1}
\end{equation}
at all points where $|x|\ge (\frac{1}{2}\varepsilon )^{p + 1}$.
Here, $x'$ denotes the $\tau $--derivative of $x$.

To see what this brings, identify $W'$ with its image in $\mathbb{C}$ via
$M$. On the $|x|\ge (\frac{1}{2}\varepsilon)^{p + 1}$ portion of $W'$,
the function $x$ is pulled up from $W$ via $\psi^{-1}$. On $W$,
$|x -u^{p + 1}|\le  c\cdot \varepsilon ^{p + 2}$ where $c$ is a fixed constant.
Thus the fact that $|x - \tau ^{p + 1}| <\varepsilon ^{p + 2}$ where
$|x|\ge (\frac{1}{2}\varepsilon )^{p + 1}$ implies that $\tau $ can be chosen so that
\begin{equation}\label{eq5.7}
|u - \psi ^*\tau |\le\varepsilon ^{2} \qquad\text{and}\qquad |du -\psi^ *d\tau
|\le\varepsilon
\end{equation}
where $|x|\ge (\frac{1}{2}\varepsilon )^{p+ 1}$ on $W$.

With this in mind, fix a favorite smooth function $\beta \co [0, \infty )\to  [0, 1]$
with value 1 on $[0, \frac{5}{8}]$, value 0 on
$[\frac{7}{8}, \infty )$, and with $|d\beta |  <8$. With $\beta $ chosen, define
$\lambda \co  W  \to W'$ by setting
\begin{equation}\label{eq5.8}
\lambda \ast \tau= (1-\beta (|u| /\varepsilon ))\psi ^*\tau +\beta (|u| /\varepsilon ) u.
\end{equation}
Thus, $\lambda $ extends $\psi $. Moreover, if $\varepsilon $ is small, then
the inequalities in \eqref{eq5.7} guarantee that $\lambda $ is a diffeomorphism.

It remains now to explain why $r(\lambda )$ is very small if both
$\varepsilon $ and $\kappa'$ are small. For this purpose, note that
$r(\lambda )$ is uniformly ${\mathcal{O}}(\varepsilon )$ where
$|x|\ge(\frac{1}{2}\varepsilon )^{p + 1}$ since the
differentials of $\lambda $ and $\psi $ differ there by ${\mathcal{O}}(\varepsilon )$.
On this rest of $W$, the function $r(\lambda )$ is the
ratio of $| \bar {\partial }\tau | $ to $| \partial\tau | $ where
$\bar {\partial }$ and $\partial $ are defined by the
restriction to $W'$ of the almost complex structure from $\mathbb{R}\times
(S^1\times S^2)$ and $\tau $ is considered here a function on $W'$.
Indeed, such is the case since the pull-back via $\lambda $ of $\tau $ is a
holomorphic function on $W$. To compute this ratio, remember that $x$ on $W'$ is a
holomorphic function of $\tau $ so $\bar {\partial }x = x'\bar {\partial}\tau $
and $\partial x = x'\partial \tau $. Thus, $| \bar{\partial }\tau | /| \partial \tau |
=|\bar {\partial }x| /| \partial x| $ and this is very small if $\varepsilon $
and $\kappa'$ are small.

\step{Step 9}
The claim that \eqref{eq5.7} holds when $\varepsilon $ is small is an immediate
consequence of the following lemma:

\begin{lemma}\label{lem:5.3}

Fix an integer $p \ge 0$, and $\varepsilon $, $\varepsilon ' > 0$.
There exists $\rho\in (0, \varepsilon )$ with the following property: Let
$f$ denote a non-trivial polynomial of degree $p+1$ on $\mathbb{C}$ such that the
locus where $|f| =\rho $ is a simple closed curve. Then there is a holomorphic
coordinate $\tau $ on $\mathbb{C}$ such that
$| f-\tau ^{p +1}|\le\varepsilon '\varepsilon ^{p + 1}$ and
$| f' -(p+1)\tau ^{p}|\le\varepsilon '\varepsilon^{p}$ where
$| f| >(\frac{1}{2}\varepsilon )^{p + 1}$.
\end{lemma}

\begin{proof}[Proof of  \fullref{lem:5.3}]
A degree $p+1$ polynomial determines, up to a
$(p+1)'$st root of unity, a holomorphic coordinate, $\tau $, for $\mathbb{C}$ and
a set, $\Lambda $, of $p+1$ not necessarily distinct complex numbers such that
\begin{equation}\label{eq5.9}
f(\tau )=\prod _{b \in \Lambda }(\tau  - b) \text{ and such that }
\sum _{b \in \Lambda }b = 0.
\end{equation}
It follows from this representation of $f$ that there exists a constant,
$c_{0}$, such that no point in $\Lambda $ has absolute value greater than
$d \equiv c_{0}\rho ^{1 / (p + 1)}$ if the $|f| =\rho $
locus is connected. Here $c_{0}$, and constants $\{c_{j}\}_{1 \le j
\le 5}$ that follow are independent of $f$.

Meanwhile, $| f(\tau )|\le c_{1} d^{p + 1}$ where
$|\tau |\le  2d$. Thus,the locus where $| f(\tau)|\ge (\frac{1}{2}\varepsilon )^{p + 1}$ must
occur where $| \tau |  > 2d$ in the case that $\rho \leq c_{2}\varepsilon ^{p + 1}$.
However, $| f(\tau )-\tau^{p + 1}|\le c_{3} d | \tau | ^{p}$ at the
points $| \tau |\ge  2d$. Thus, $| f-\tau ^{p +1}|\le c_{4}\rho ^{1 / (p + 1)}\varepsilon ^{p}$
where $| f|\ge (\frac{1}{2}\varepsilon)^{p + 1}$. This understood, there is a
fifth $f$--independent constant,
$c_{5}\in  (0, 1/c_{4})$ such that $\rho\le (c_{5}\varepsilon '\varepsilon )^{p + 1}$
makes the lemma's claim true.
\end{proof}

\subsection{The proof of Proposition \ref{prop:5.1}}\label{sec:5d}

The proof of this proposition is almost verbatim that of \cite[Proposition~2.13]{T3}.
The following three parts of the proof focus on the salient differences.

\step{Part 1}
Suppose that $(C_{0}, \phi )$ defines a point in $\mathcal{S} _{B,c,\mathfrak{d}}$.
This first part of the proof describes what turns out to be the part of a
complete set of local coordinates for a neighborhood of this point in
$\mathcal{S} _{B,c,\mathfrak{d}}$. To start, let $\Crit(C)  \subset C_{0}$ denote the
set of size $c$ whose elements are the critical points of $\theta $ where
$\theta\in  (0, \pi )$. If $({C_{0}}', \phi ')  \in\mathcal{S}_{B,c,\mathfrak{d}}$
defines a point near to that of $(C_{0}, \phi )$, then
$({C_{0}}', \phi '$) is represented by a small normed section in the
$(C_{0}, \phi )$ version of $\kernel(D_{C})$. Such a representation allows
the $c$ critical points of $\theta $ on ${C_{0}}'$ to be partnered with the
points in $\Crit(C)$ so that a ${C_{0}}'$ critical point maps very close to the
image of its partner in $\mathbb{R}\times  (S^1\times S^2)$.
This pairing of critical points also preserves the conditions that are
defined by the partition $d$. In addition, the degree of vanishing of $d\theta$
at any given critical point in ${C_{0}}'$ is the same as that of its partner
in $C_{0}$.

Keeping these facts in mind, let $z \in  \Crit(C)$. Fix a small ball, $B$,
centered on $z$'s image in $\mathbb{R}\times  (S^1\times S^2)$
whose closure excludes the images of any other point in $\Crit(C)$. Introduce
the function, $r$, on $B$ as defined in \eqref{eq2.9}. If $z'$ is the critical point in
${C_{0}}'$ that is paired with $z$, then the pair $(\theta (z'), r(z'))$ is well
defined and in this way, some $2c$ functions, $\{(\theta _{z},r_{z})\}_{z \in \Crit(C)}$,
are defined on a neighborhood of $(C_{0},\phi _{0})$'s point in
$\mathcal{S} _{B,c,\mathfrak{d}}$, at least in the case where
the group $G_{C}$ is trivial. In the case $G_{C}\ne  \{1\}$, then
these functions are distinguishable only modulo the action of $G_{C}$.

Note that the collection $\{\theta _{z}\}_{z \in \Crit(C)}$ defines
at most $m$ independent functions on a neighborhood of $(C_{0}, \phi )$'s
point in $\mathcal{S} _{B,c,\mathfrak{d}}$.

\step{Part 2}
This part of the proof supplies additional coordinates for a neighborhood
in $\mathcal{S} _{B,c,\mathfrak{d}}$ of the point defined by $(C_{0}, \phi )$. To start,
let $E  \subset C_{0}$ denote either one of the $N_{ + }$ convex side
ends where $\lim_{| s| \to \infty }\theta $ is neither 0 nor $\pi$,
or one of the ends that correspond to a 4--tuple in $B$. In any case, let
$(p, p')$ denote the integer pair from the corresponding 4--tuple. Then the
$\mathbb{R}/(2\pi \mathbb{Z})$ valued function $p\varphi  - p't$ has an
$\mathbb{R}$ valued lift on $E$ with a well defined $|s|\to\infty $
limit. Now, as in Part 1, if $({C_{0}}', \phi ')$ defines a point in
$\mathcal{S} _{B,c,\mathfrak{d}}$ near to that of $(C_{0}, \phi )$, then the ends of
${C_{0}}'$ can be partnered with those of $C_{0}$ so that partners share the
same $\hat{A}$ 4--tuple and map very near each other in
$\mathbb{R}\times  (S^1\times S^2)$. This understood, the function $p\varphi -p't$
has an $\mathbb{R}$--valued lift on $E$'s partner in ${C_{0}}'$ with a well defined
$|s|\to\infty $ limit that is very close to the
corresponding limit on $E$.

The assignment of these limits to the points near the image of $(C_{0}, \phi )$
define a collection of $N_{ + }+|B| $ functions
(modulo the action of $G_{C})$ on a neighborhood of $(C_{0}, \phi )$'s
point in $\mathcal{S} _{B,c,\mathfrak{d}}$. Let $\{\varpi _{ + 1}, \ldots \}$ denote
the $N_{ + }$ functions so defined (modulo the $G_{C}$ action) from the
convex side ends, and let $\{\varpi _{ - 1}, \ldots \}$ denote the
corresponding collection of $|B| $ functions that come from the
4--tuples in $B$.

\step{Part 3}
Two more functions are defined here for a neighborhood of $(C_{0}, \phi )$'s
point in $\mathcal{S} _{B,c,\mathfrak{d}}$. For this purpose, choose either an
$(1,\ldots)$ element from $\hat{A}$, or a $(0,-,\ldots)$ 4--tuple that is not from $B$,
or a point $z \in C_{0}$ where $\theta $ is zero. In all three cases, a complex valued
function, $\varpi'$, is defined modulo the $G_{C}$ action on a neighborhood of
$(C_{0}, \phi )$'s point in $\mathcal{S} _{B,c,\mathfrak{d}}$ as in the statement
of \cite[Proposition~2.13]{T3}.

To summarize the story, the function is defined on the point
defined by some pair $({C_{0}}', \phi ')$ by first identifying
the latter with an element near zero in $\kernel(D_{C})$. This
done, then if $\varpi '$ is defined from an end of $C_{0}$, there
is a partnered end, $E'  \subset {C_{0}}'$. If, as before, $(p,
p')$ denotes the integer pair from the corresponding element in
$\hat{A}$, then the phase of the complex number $\varpi '$ is
proportional to the $|s|\to\infty $ limit of $p'\varphi  - pt$ on
$E'$. The absolute value of the complex number is proportional to
the logarithm of the constant $b$ that appears in the $E'$
version of \eqref{eq2.4} in the case that $E'$ corresponds to a
$(0,-,\ldots)$ element. In the case that $E'$ corresponds to a
$(1,\ldots)$ element form $\hat{A}$, the absolute value is
proportional to the logarithm of the constant $\hat {c}$ that
appears in the $E'$ version of \eqref{eq1.9}.

When $\varpi '$ is defined from a $\theta  = 0$ point $z \in C_{0}$, the
chosen $\kernel(D_{C})$ element for $(C_{0},\phi )$ partners the
$\theta  = 0$ points in ${C_{0}}'$ with those in $C_{0}$ so that the partner
of $z$ is mapped very close to $z$'s image in $\mathbb{R}\times  (S^1\times S^2)$
and makes the same local contribution as does $z$ to the
count for $\text{\c{c}}_{ + }$. This understood, $\varpi '$ is assigned the value
of a holomorphic coordinate for a small disk in the $\theta  = 0$ cylinder
that is centered on $z$'s image.

\step{Part 4}
Granted the preceding, introduce the vector space $K^{* }\subset \kernel(D_{C})$
as defined in \cite[(2.23)]{T3}. The asserted structure of
$\mathcal{S} _{B,c,\mathfrak{d}}$ near the point defined by $(C_{0}, \phi )$ follows via
the implicit function theorem with a proof that $K^{* } = \{0\}$.
The argument for this conclusion, and the argument that proves $K^{* }= \{0\}$
are essentially verbatim copies of the arguments given
for \cite[Proposition~2.13]{T3}. Nothing new is needed to treat the case where
$\phi $ is not almost everywhere 1--1.

Note that the proof that $K^{* } = \{0\}$ proves more, for it
establishes the following:

\begin{lemma}\label{lem:5.4}

If $G_{C}$ is trivial, then the collection $\{\varpi _{ + 1},\ldots \}$,
$\{\varpi _{ - 1},\ldots \}, \varpi '$,\linebreak $\{r_{z}\}_{z \in \Crit(C)}$
and a certain subset of $m$ functions from the collection
$\{\theta _{z}\}_{z \in \Crit(C)}$ define coordinates for
$\mathcal{S} _{B,c,\mathfrak{d}}$ near the point defined by $(C_{0}, \phi )$.
Here, the $m$ functions from the collection $\{\theta _{z}\}_{z \in \Crit(C)}$
are chosen so that the values of the chosen versions of
$\theta_{z}$ at $(C_{0}, \phi )$ account for the critical values of $\theta $
that avoid the $|s|\to\infty $ limits of $\theta$ on the concave side ends in
$C_{0}$ and on the ends that correspond to 4--tuples from $B$.
In the case that $G_{C}\ne  \{1\}$, then the analogous collection of $N_{ + }+|B| +c+m$
functions give smooth orbifold coordinates on a neighborhood in
$\mathcal{S} _{B,c,\mathfrak{d}}$ of the point defined by $(C_{0}, \phi )$.

\end{lemma}


\setcounter{section}{5}
\setcounter{equation}{0}
\section{Slicing the strata}\label{sec:6}

This section describes the structure of each component of any given
stratum from \fullref{prop:5.1}. To explain the point of
view here, let $\mathcal{S}_{b,c,d }$ denote a given stratum and let
$\mathcal{S} $ denote a given component. The component $\mathcal{S} $ is then
mapped to the $m$'th symmetric product of $(0, \pi )$ using the
critical values of $\theta $ that do not coincide with angles from
$\Lambda _{\hat{A}}$. As is explained in \fullref{sec:8}, this
map fibers $\mathcal{S} $ over a certain $m$--dimensional simplex and so the
structure of $\mathcal{S} $ is determined by that of a typical fiber. The
subsections that follow focus on the structure of such a fiber. The
principle results are in Theorems~\ref{thm:6.2} and \ref{thm:6.3}
and in Propositions~\ref{prop:6.4} and~\ref{prop:6.7}. Theorems~\ref{thm:6.2} and
\ref{thm:6.3} are proved in \fullref{sec:7}. Sections~\ref{sec:8} and~\ref{sec:9}
use the results from this section to fully paint the picture
of $\mathcal{S} $ as a fiber bundle.
%

\subsection{Graphs for the stratification}\label{sec:6a}

The constructions from \fullref{sec:2a} and Part 3 of \fullref{sec:2c} associate a
graph, $T_{(\cdot )}$, to each pair $(C_{0}, \phi )$ from ${\mathcal{M}^{*}}_{\hat{A}}$.
This graph has labeled edges and vertices that
reflect the structure of the level sets of the function $\theta $ on the
subvariety. As it turns out, certain aspects of these graphs are constant on
any given component of any given strata in ${\mathcal{M}^{*}}_{\hat{A}}$ and
serve to classify these components. This
subsection describes in more detail the graphs that are involved and the
manner in which they classify components of the stratification.

To start, remember that a graph, $T$, of the sort under consideration is
contractible and has labeled vertices and labeled edges. What follows
summarizes what is involved.

\substep{The vertex labels}
Each vertex in $T$ is labeled in part by an angle in $[0, \pi ]$ subject to various
constraints, the first of which are as follows:

\itaubes{6.1}
\textsl{The two vertices on any given edge have distinct angles.}

\item
\textsl{No multivalent vertex angle is extremal in the set of the angles of the vertices on
the union of its incident edges.}
\eit

The vertices have additional labels. To elaborate, a subset of the angle 0
vertices are labeled via a 1--1 correspondence with the set of $(1,\ldots)$ elements in
$\hat{A}$. The remaining angle 0 vertices are
labeled by positive integers that sum to $\text{\c{c}}_{ + }$. Likewise, a subset
of angle $\pi $ vertices are labeled via a 1--1 correspondence with the set
of $(-1,\ldots)$ elements in $\hat{A}$; and the remainder are
labeled by negative integers that sum to $-\text{\c{c}}_{ - }$.

Each vertex with angle in $(0, \pi )$ is labeled jointly by a subset of the
$(0,\ldots)$ elements in $\hat{A}$ and a certain sort of graph.
To describe these labels, remark first that distinct vertices are assigned
disjoint subsets, and that the union of these subsets is the whole set of
$(0,\ldots)$ elements in $\hat{A}$. The empty subset can only be
assigned to vertices with three or more incident edges. Meanwhile, a
monovalent vertex must get a singleton set with a $(0,-,\ldots)$ element.
Finally, the integer pair component of any element from
an assigned subset defines the corresponding vertex angle via \eqref{eq1.8}. The
subset of $\hat{A}$ that is assigned to a vertex $o$ is denoted in what follows by
$\hat{A}_{o}$.

The graph that is assigned to the vertex $o$ is denoted by $\underline {\Gamma}_{o}$.
The latter graph is connected, has labeled vertices and oriented,
labeled arcs. Here, the edges of $\underline {\Gamma }_{o}$ are called
`arcs' to avoid confusing them with the edges in $T$ that are incident to $o$.
The first Betti number of $\underline {\Gamma }_{o}$ must be one less than
the number of incident edges in $T$ to $o$. In particular, this means that
$\underline {\Gamma }_{o}$ is a single point when $o$ is monovalent. In
addition, each vertex in $\underline {\Gamma }_{o}$ has an even number of
incident half-arcs with half oriented so as to point towards the vertex and
half are oriented to point away.

The labeling of the vertices in $\underline {\Gamma }_{o}$ is as follows:
Each vertex is labeled with an integer, subject to two constraints. First,
the number of elements in $\hat{A}_{o}$ that are identical to a given 4--tuple
is equal to the number of vertices in $\underline {\Gamma }_{o}$ where the
sign of the label is the second component of the 4--tuple and where the
absolute value of the label is the greatest commons divisor of the integer
pair from the 4--tuple. Here is the second constraint: Vertices with label 0
must have four or more incident half-arcs.

Each arc in $\underline {\Gamma }_{o}$ is labeled by a pair of $o$'s
incident edges subject to constraints that will now be described. For this
purpose, partition the incident edge set to $o$ as $E_{ - } \cup E_{ +}$
where $E_{ - }$ contains the edges on which $o$ is the largest angle vertex
and $E_{ + }$ those on which $o$ is the smallest angle vertex. The label of
any given arc must contain one edge from $E_{ - }$ and one from $E_{ + }$.
To say more, let $e$ denote an incident edge to $o$. The collection of arcs
whose label contains $e$ concatenate to define an oriented, immersed loop in
$\underline {\Gamma }_{o}$ such that a traverse of this loop crosses no
arc more than once. Thus, this loop is the image of an abstract, oriented,
circular graph, $\ell _{oe}$, via an immersion that maps vertices to
vertices and is 1--1 and orientation preserving on the edges. The collection
$\{\ell _{oe}\}_{e \text{ is incident to } o}$ is then constrained by the
rules in \eqref{eq2.17}. Here, and in what follows, $\ell _{oe}$ is used to denote
both the abstract circular graph and its image in $\underline {\Gamma}_{o}$
since the former can be recovered from the latter. For reference
below: Properties 1--4 from Part 3 of \fullref{sec:2c} are valid here; Properties 1
and 2 are assumed, while Properties 3 and 4 then follow automatically.

The label of any given vertex in $T$ determines an ordered pair of
integers. When $o$ denotes a vertex, then $P_{o}$ or $(p_{o},{p_{o}}')$ is used to denote
the associated integer pair. Here are the rules for this assignment:

\itaubes{6.2}
\textsl{If $o$ is a monovalent vertex with an assigned integer $m$,  then $P_{o} = (0, -m)$.}

\item
\textsl{If $o$  is a monovalent vertex with label $(\pm 1,\ldots)$  or $(0,-,\ldots)$,  then $P_{o}$
is the label's integer pair component.}

\item
\textsl{If $o$  is a multivalent vertex with angle in $(0, \pi )$,  then $P_{o}$  is obtained
by subtracting the sum of the integer pair components of the $(0,-,\ldots)$  elements in
$\hat{A}_{o}$  from the sum of the integer pair components of the $(0,+,\ldots)$  elements in
$\hat{A}_{o}$.}
\eit

\substep{The edge labels}
Each edge in $T$ is labeled by an non-zero ordered pair of integers. When $e$ denotes an edge, its
corresponding pair is denoted as $Q_{e}$ or as $(q_{e}, {q_{e}}')$. These
labels are determined by the branching of $T$ and the data from $\hat{A}$
according to the rules that follow. Note that these rules determine the
labels via an induction that starts with the edges whose minimal angle
vertices are monovalent. Note as well that this induction uses the fact that
 $T$ is a tree. Here are the rules:

\itaubes{6.3}
\textsl{If $e$  is incident to a monovalent vertex, $o$, then $Q_{e}=\pm P_{o}$  where the $+$
sign is used in the following cases:}

\begin{enumerate}
\leftskip 25pt
\item[\rm(a)]
\textsl{The vertex label consists of either a $(1,-,\ldots)$  element from $\hat{A}$, or a positive
integer, or a $(-1,+,\ldots)$  element from $\hat{A}$}

\item[\rm(b)]
\textsl{The vertex angle is in $(0,\pi )$  and it is the lesser of the two vertex angles from $e$.}
\end{enumerate}

\item
\textsl{If $o$  is a multivalent vertex, then $\sum _{e \in E_-} Q_{e}-\sum _{e \in E_ + }Q_{e} =P_{o}$.}
\eit

There is one additional constraint:

\qtaubes{6.4}
\textsl{Let $e$  denote a given edge, and let $\theta _{ - }< \theta _{ + }$  denote the angles of
the vertices on $e$.  The function
$\alpha _{Q_e } (\theta ) = (1-3\cos^{2}\theta){q_{e}}'-\surd 6 \cos\theta q_{e}$
is positive on $(\theta _{ - }, \theta_{ + })$  and zero at an endpoint if and only if
the angle is in $(0, \pi )$  and the corresponding vertex on $e$  is monovalent.}
\endqtaubes

As done for example in \cite[Section 5]{T3}, this last condition can be
rephrased as inequalities that reference only the integer pairs from the
vertex angles.

\substep{Graph isomorphisms}
Let $T$ and $T'$ denote graphs of the sort just described. An isomorphism, $\iota $,
from $T$ to $T'$ consists of a 1--1 and onto map from $T$'s vertex set
to the vertex set of $T'$ along with a compatible map from $T$'s edge
set to the set of $T'$ edges. Both maps must preserve all labels.
Thus, the respective integer pairs of an edge in $T$ and its $\iota$
image in $T'$ agree; the respective angles of a vertex in $T$ and
its $\iota $ image in $T'$ agree; and the labels of a vertex in $T$
and its $\iota $ image in $T'$ are themselves isomorphic in the
following sense: First, $P_{\iota (o)}= P_{o}$ for all
vertices $o\in T$. Second, there is an associated collection of
isomorphisms that are labeled by $T$'s multivalent vertices where the
version, $\hat {\iota }_{o}$, labeled by vertex $o$ is an
isomorphism from the graph $\underline {\Gamma }_{o}$ to $\underline
{\Gamma }_{\iota (o)}$. Thus, $\hat {\iota }_{o}$ consists of a 1--1
and onto map from the set of vertices in $\underline {\Gamma }_{o}$
to the set of vertices in $\underline {\Gamma }_{\iota (o)}$ along
with a compatible 1--1 and onto map from the set of oriented arcs in
$\underline {\Gamma }_{o}$ to the corresponding set in
$\underline {\Gamma }_{\iota (o)}$. Both of these set maps must also
preserve labels. Thus,

\bit

\item
\textsl{the vertex set map must preserve the integer labels of the vertices,}

\item
\textsl{the arc set map sends an arc with label $(e, e')$  to one with label $(\iota (e), \iota (e'))$.}
\begin{equation}\label{eq6.5}
\end{equation}

\eit

An isomorphism from $T$ to itself is deemed an automorphism of $T$. The set of
such automorphisms is designated as $\Aut(T)$. Composition makes $\Aut(T)$ into a
group and it is viewed as such in what follows.

\substep{Graph homotopy}
It proves useful to introduce an
equivalence between graphs that is weaker than graph isomorphism. For this
purpose, graphs $T$ and $T'$ are said to be `homotopic' when there is a one
parameter family, $\{T_{\tau }\}_{\tau \in [0,1]}$ of graphs with
the following two properties: First, $T_{0}= T$ and $T_{1} = T'$. Second,
the various $T_{\tau }$ differ one from another only in their vertex angles,
and these angles change as continuous functions of $\tau $ with the
constraint that the number of distinct, multivalent vertex angles is
independent of $\tau $. In any event, only the vertex angles can change;
neither edge labeled integer pairs nor vertex labeled graphs are modified in
any way.

\fullref{sec:2a} and Part 3 of \fullref{sec:2c} explain how a graph such as $T$ can be
assigned to any given pair $(C_{0}, \phi )$ from any given element in
${\mathcal{M}^{*}}_{\hat{A}}$. As remarked at the end of \fullref{sec:2a},
the isomorphism type of the assigned graph depends only on the given
element in ${\mathcal{M}^{*}}_{\hat{A}}$. The following is a consequence:

\begin{proposition} \label{prop:6.1}

Two elements in any given component of any given stratum have homotopic graphs, and
elements in either distinct strata or in distinct components of the same stratum have
graphs that are not homotopic.

\end{proposition}

This proposition is proved in \fullref{sec:8a}.

\substep{The subspace ${\mathcal{M}^{*}}_{\hat{A},T}$}
Suppose that $T$ is a graph of the sort just described and let ${\mathcal{M}^{*}}_{\hat{A},T}$ denote the
subset of elements in ${\mathcal{M}^{*}}_{\hat{A}}$ that are defined
by a pair whose graph is isomorphic to $T$. Here is why such a space
is relevant: Let $\mathcal{S} $ denote a component of some stratum
$\mathcal{S} _{B,c,\mathfrak{d}}$ from \fullref{prop:5.1}. A map from
$\mathcal{S} $ to the $m$'th symmetric power of the interval $(0, \pi )$ is
defined by taking the distinct critical values of $\theta $ in $(0,\pi )$
that do not arise via \eqref{eq1.8} from the integer pair of
either a $(0,+,\ldots)$ element from $\hat{A}$ or an
element from $B$. The typical fiber of this map is a version of
${\mathcal{M}_{\hat{A},T}}$ for a graph $T$.

The rest of this section contains a description of ${\mathcal{M}^{*}}_{\hat{A},T}$.

\subsection{The structure of ${\mathcal{M}^{*}}_{\hat{A},T}$}\label{sec:6b}

By way of a preview for what is to come, \fullref{thm:6.2} below
asserts that ${\mathcal{M}^{*}}_{\hat{A},T}$, when non-empty, is
diffeomorphic as an orbifold to a space of the form $\mathbb{R}
\times  O_{T}/\Aut(T)$ where $O_{T}$ is a version of \eqref{eq3.12}
where $\Aut(T)$ acts. As $O_{T}$ is connected, this implies that
${\mathcal{M}^{*}}_{\hat{A},T}$ is a connected component of a stratum
of ${\mathcal{M}^{*}}_{\hat{A}}$.

The detailed description of $O_{T}$ requires some preliminary stage setting,
and this is done next in four parts. The length of this stage setting
preamble is due for the most part to the subtle nature of the $\Aut(T)$ action.

\step{Part 1}
This part of the digression describes the relevant version of \eqref{eq3.12}
leaving aside the $\Aut(T)$ action. To start the story, let $o$ denote a
multivalent vertex in $T$ and let $\theta _{o}$ denote its assigned angle.
Meanwhile, let $E_{o}$ denote the set of incident edges to $o$, and let
$\Arc(\underline {\Gamma }_{o})$ denote the set whose elements are the arcs
in $\underline {\Gamma }_{o}$. Now define $\Delta _{o}$ to be the set of
maps $r\co  \Arc(\underline {\Gamma }_{o})   \to (0, \infty )$ that obey
\begin{equation}\label{eq6.6}
\sum_{\gamma \in \ell _{o,e} } r(\gamma ) = 2\pi   \alpha _{Q_e }(\theta _{o})
\end{equation}
for all multivalent vertices $o \in T$ and edges $e \in  E_{o}$. Note
that $\Delta _{o}$ is a simplex whose dimension is one less then the
number of vertices on $\underline {\Gamma }_{o}$.

Each multivalent vertex in $T$ is also assigned a real line. If $o$ designates
such a vertex, then $\mathbb{R}_{o}$ is used to denote its associated line.
Let $\mathbb{R}_{ - }$ denote an auxiliary copy of $\mathbb{R}$. Then
$\mathbb{R}_{ - }  \times _{o}(\mathbb{R}_{o}  \times   \Delta _{o})$
appears below in the role that \eqref{eq3.14} plays in \eqref{eq3.12}. Here, and in what
follows, the symbol $\times _{o}$ appears with no accompanying definition
signifies a product that is indexed by the set of multivalent vertices in $T$.

To define the group actions in the case at hand, it is necessary to first
choose an $\Aut(T)$ invariant vertex in $T$. To obtain such a vertex, note that
any connected, non-empty, $\Aut(T)$ invariant subgraph with the least number of
vertices amongst all subgraphs of this sort is necessarily a 1--vertex graph.
Here is why: Were such a subgraph to have two or more vertices, there would
be one that was both monovalent as a subgraph vertex and not $\Aut(T)$
invariant. Removing this vertex, the interior of its incident edge, and
their orbits under $\Aut(T)$ would result in a non-empty, proper, $\Aut(T)$
invariant subgraph of the original.

Fix a smallest angle, $\Aut(T)$--invariant vertex and denote it as
$\diamondsuit $. This done, it proves convenient to introduce
$\mathcal{V} $ to denote the set of multivalent vertices in $T-\diamondsuit $.
Each $o \in \mathcal{V} $ labels a copy of
$\mathbb{Z}$ and these groups are going to act in a mutually
commuting fashion on $\mathbb{R}_{ - }  \times  (\times_{\hat{o}}\, \mathbb{R}_{\hat{o}})$.
To describe these
actions, note first that when $o$ is a multivalent vertex in $T-\diamondsuit $,
then $T-o $ has a unique component that contains the
vertex $\diamondsuit $. Let $T_{o}$ denote the closure in $T$ of the
complement of this component.

Granted this notation, then $1  \in   \mathbb{Z}_{o}$ is defined to act
trivially on $\mathbb{R}_{ - }$ and on $\mathbb{R}_{\hat{o}}$ in the case
that $\hat{o}  \notin T_{o}$. In the case that $\hat{o}  \in T_{o}$, then
$1 \in   \mathbb{Z}_{o}$ acts on $\mathbb{R}_{\hat{o}}$ as the
translation
\begin{equation}\label{eq6.7}
- 2\pi   \frac{\alpha_{Q_e}(\theta _{\hat{o}} )}
{\alpha _{Q_{\hat{e} }}(\theta _{\hat{o}} )},
\end{equation}
where $e$ and $\hat{e}$ designate the respective edges that connect $o$ to $T- T_{o}$ and $\hat{o}$
to $T- T_{\hat{o}}$.

Meanwhile, define an action of the group $\mathbb{Z}  \times   \mathbb{Z}$ on
$\times _{\hat{o} \in \mathcal{V} }  \mathbb{R}_{\hat{o}}$ by
having an integer pair $N= (n, n')$ act as the translation by
\begin{equation}\label{eq6.8}
- 2\pi   \frac{\alpha _N (\theta _{\hat{o}})}{\alpha _{Q_{\hat{e}}} (\theta_{\hat{o}} )}.
\end{equation}
This action commutes with that just defined $\times _{o \in \mathcal{V} }\mathbb{Z}_{o}$.

The extension of the $\mathbb{Z}\times \mathbb{Z}$
action to $\mathbb{R}_{ - }$ and $\mathbb{R}_{\diamondsuit }$ requires an additional choice,
this the choice of a distinguished
$\Aut(T)$ orbit in the set of incident edges to $\diamondsuit $. Let
$\hat{E}$ denote this distinguished orbit, let $m_{\hat{E}}$
denote the number of edges in $\hat{E}$, and let $Q_{\hat{E}}= (q_{\hat{E}}, {q_{\hat{E}}}')$
denote the integer
pair that is associated to the edges that comprise $\hat{E}$. The
action of $N$ on $\mathbb{R}_{\diamondsuit }$ is then defined by the
version of \eqref{eq6.8} that uses $\diamondsuit $ for $\hat{o}$ and
$Q_{\hat{E}}$ for $Q_{{\hat{e}}}$. Meanwhile, the action
of $N$ on $\mathbb{R}_{ - }$ is defined so that $(n, n')$ acts as the
translation by $-2\pi m_{\hat{E}} (n'q_{\hat{E}} - n{q_{\hat{E}}}')$.

Granted all of the above, set
\begin{equation}\label{eq6.9}
O_{T}   \equiv  (\times _{o}\Delta _{o})  \times
\bigl[\mathbb{R}_{ - }  \times (\times _{o}
\mathbb{R}_{o})\bigr]/\bigl[(\mathbb{Z}  \times   \mathbb{Z})  \times
(\times _{\hat{o} \in \mathcal{V}}\mathbb{Z}_{\hat{o}})\bigr].
\end{equation}
By way of reminder, $\times _{o}$ here designates a product that is
indexed by the full set of multivalent vertices in $T$, while $\times_{\hat{o} \in \mathcal{V} }$
designates a product that is indexed by
the set of vertices in $T- \diamondsuit $. It is left to the reader
to verify that $O_{T}$ is a smooth manifold. This $O_{T}$ is the desired
version of \eqref{eq3.12}.

\step{Part 2}
This part of the story constitutes a digression to provide a sort of
inductive description of $\Aut(T)$ using as components a subgroup of
automorphisms of $\underline {\Gamma }_{\diamondsuit }$ and a collection of
cyclic groups that are labeled by the multivalent vertices in $T- \diamondsuit $.
To start this description, let $o$ denote a vertex in
$\mathcal{V} $. Then $T- o$ has a one component that contains
$\diamondsuit $. Let $T_{o}$ denote the closure of the remaining components.
This is a connected, contractible subgraph of $T$. The vertex $o$ also has a
distinguished incident edge, this the edge $e \equiv e(o)$ that
connects $o$ to $T- T_{o}$. Let $\ell _{o}$ to denote the loop $\ell_{oe(o)}$.

Now, let $\Aut(T_{o})   \subset  \Aut(T)$ denote the subgroup that acts
trivially on the data from $T- T_{o}$. This group fixes $o$ and so it
has a canonical homomorphism in the group of automorphisms of the pair
$\{\underline {\Gamma }_{o}, E_{o}\}$, where $E_{o}$ here designates
the set of incident edges to $o$. Let $\Aut_{o}$ denote the image of
$\Aut(T_{o})$ in this last group. As $\Aut_{o}$ fixes the incident edge $e(o)$,
so it preserves the loop $\ell _{o}$ and thus has an image in the cyclic
group of automorphisms of the abstract version of $\ell _{o}$. In fact,
$\Aut_{o}$ is isomorphic to its image in this cyclic group. Such is the case
by virtue of the following fact:
\begin{equation}\text{\textsl{An automorphism of $\underline {\Gamma }_{o}$  that fixes any arc must be trivial.}}
\label{eq6.10}
\end{equation}
Indeed, suppose that $\gamma $ is fixed and that $e'$ is one of $\gamma $'s
labeling edges. Then the automorphism must act trivially on $\ell_{oe'}$
and so fixes all of its arcs. It then follows from Property
3 of Part 3 in \fullref{sec:2c} that the maximal collection of fixed arcs must
constitute the whole of $\underline {\Gamma }_{o}$.

Let $n_{o}$ denote the order of the cyclic group $\Aut_{o}$. Since the
automorphism group of $\ell _{o}$ has a canonical generator, so does
$\Aut_{o}$, this smallest power of the generator of $\Aut(\ell _{o})$ that
resides in $\Aut_{o}$. This generator provides a canonical isomorphism
between $\Aut_{o}$ and $\mathbb{Z}/(n_{o}\mathbb{Z})$.

The preceding implies an exact sequence
\begin{equation}\label{eq6.11}
1 \to   \times _{o'} \Aut(T_{o'})   \to
\Aut(T_{o})   \to   \mathbb{Z}/(n_{o}\mathbb{Z})   \to 1,
\end{equation}
where the product in \eqref{eq6.11} is indexed by the vertices in $T_{o}$ that share
some edge with $o$. There is a similar exact sequence in the case that $o =\diamondsuit $
but where $\mathbb{Z}/(n_{o}\mathbb{Z})$ is replaced by
$\Aut_{\diamondsuit }$. Of course, if $\Aut _{\diamondsuit }$ fixes an edge in
$E_{\diamondsuit }$, then $\Aut _{\diamondsuit }=\mathbb{Z}/(n_{\diamondsuit }\mathbb{Z})$.

With regards to the various versions of $\Aut _{o}$, keep in mind the
following:
\begin{equation}\label{eq6.12}
\text{\textsl{An element in $\Aut_{o}$  is trivial if it fixes more than two edges in $E_{o}$.}}\end{equation}
To see why such is the case, first construct a closed, 2--dimensional cell
complex by taking the 1--skeleton to be $\underline {\Gamma }_{o }$ and then
attaching a disk to each $\ell _{o(\cdot )}$. An Euler class computation
finds that this complex is a 2--sphere. The $\Aut _{o}$ action on $\underline {\Gamma }_{o}$
extends to the action on the sphere if it is agreed that
its action on the 2--cells is obtained by a linear extension. This extended
action is then a piecewise linear, orientation preserving action. In
particular, if an edge is fixed by some element, the action on the
corresponding disk is a rotation through a rational fraction of $2\pi $ and
has the origin as its fixed point. Since $\Aut(T)$ is finite, such an action
can have at most two fixed points unless it is trivial.

The last point to make in this part of the subsection is that the various
versions of \eqref{eq6.11} all split and these splitting can be used to write $\Aut(T)$
as the iterated semi-direct product
\begin{equation}\label{eq6.13}
\Aut(T)  \approx  \Aut _{\diamondsuit }  \times  \left[\times _{o}
\left[\mathbb{Z}/(n_{o}\mathbb{Z})  \times  \bigl[\times _{o'}
[\mathbb{Z}/(n_{o'}\mathbb{Z})  \times \cdots]\cdots\bigr]\right]\right].
\end{equation}
Here, the left most subscripted product is indexed by the vertices
that share edges with $\diamondsuit $; meanwhile, any subscripted
product, $\times _{\hat{o}'}$, that appears as
$\mathbb{Z}/(n_{\hat{o}}\mathbb{Z})  \times  [\times _{\hat{o}'}\cdots]$ is indexed by the vertices in
$T_{\hat{o}}$ that share an edge with $\hat{o}$. Moreover, such
an isomorphism identifies any given $o \in  T- \diamondsuit $
version of $\Aut(T_{o})$ with the semi-direct product
\begin{equation}\label{eq6.14}
\mathbb{Z} / (n_{o}\mathbb{Z})  \times  \bigl[\times _{o'}
[\mathbb{Z}/(n_{o'}\mathbb{Z})  \times \cdots]\cdots\bigr].
\end{equation}
It is important to keep in mind that the isomorphism in \eqref{eq6.13} is not
canonical. Rather, it requires the choice of a `distinguished' vertex in
each $o \ne \diamondsuit $ version of $\ell _{o}$ subject to the
following constraint: Let $\upsilon _{o}$ denote the distinguished vertex
for $o$ and let $\iota \in \Aut(T)$. Then $\upsilon _{o}$ and $\upsilon_{\iota \cdot o}$
define the same $\Aut(T)$ orbit in the set of vertices
in $ \cup _{\hat{o} \ne \diamondsuit }  \ell _{\hat{o}}$.
Assume in what follows that such choices have been made.

To explain how \eqref{eq6.13} arises, note that the effect of an automorphism on any
$o\ne  \diamondsuit $ version of $\underline {\Gamma }_{o}$ is
determined completely by its affect on a single vertex in $\ell _{o}$,
thus on $\upsilon _{o}$. This follows by virtue of the fact that any given
$\iota    \in  \Aut(T)$ must map $\ell _{o}$ to $\ell _{\iota (o)}$ so
as to preserve the cyclic order of the vertices. Granted this, a lift to
$\Aut(T)$ of any element in any given version of $\Aut _{o}$ can be defined in an
iterated fashion along the following lines: Let $\iota $ denote the given
element. The lift is defined to live in $\Aut(T_{o})$. The first step to
defining this lift specifies its action on $ \cup _{o'}\underline {\Gamma }_{o'}$,
where the union is labeled by the vertices in
$T_{o}- o$ that share edges with $o$. Let $o'$ denote such a vertex.
Then $\iota (o')$ is determined apriori by the action of $\iota $ on $o$'s
incident edge set. The corresponding isomorphism from $\underline {\Gamma}_{o'}$
to $\underline {\Gamma }_{\iota (o')}$ is defined
by requiring that $\upsilon _{o'}$ go to $\upsilon _{\iota(o')}$.
Note that this now determines how $\iota $ acts on the set
$\cup _{o'}E_{o'}$ whose elements are the incident edges
to the various vertices in $T_{o}- o$ that share an edge with $o$.
To continue, suppose that $o'' \in T_{o'}- o'$ shares an
edge with $o'$. Then $\iota (o'')$ is determined apriori by the just defined
action of $\iota $ on $ \cup _{o'}E_{o'}$. The
isomorphism between $\underline {\Gamma }_{o''}$ and
$\underline {\Gamma }_{\iota (o'')}$ is again determined by
the requirement that $\iota (\upsilon _{o''})=\upsilon
_{\iota (o'')}$. Continuing in this vein from a given vertex
$\hat{o}  \in T_{o}$ to those in $T_{\hat{o}}- \hat{o}$ that
share edges with $\hat{o}$ ends with an unambiguous definition for the lift of
$\iota $. In this regard, note that the lifts of elements $\iota $ and
$\iota '$ from a given $\Aut _{o}$ must multiply to give the lift of $\iota\cdot \iota '$.
Indeed, this is guaranteed by the afore-mentioned fact
that the affect of an automorphism on any given $o\ne \diamondsuit $
version of $\underline {\Gamma }_{o}$ is determined by its affect on
$\upsilon _{o}$.

\step{Part 3}
The desired action of $\Aut(T)$ on $O_{T}$ is described by first introducing
a certain central extension of $\Aut(T)$ by $\mathbb{Z}  \times   \mathbb{Z}$ and
describing an action of this extension on
\begin{equation}\label{eq6.15}
(\times _{o}  \Delta _{o})  \times  (\mathbb{R}_{ - }  \times
\mathbb{R}_{\diamondsuit })  \times  [\times _{\hat{o} \in
\mathcal{V} }  \mathbb{R}_{\hat{o}}]/[\times
_{\hat{o}' \in \mathcal{V} }  \mathbb{Z}_{\hat{o}'}].
\end{equation}
Here, the symbol $\times _{o}  \Delta _{o}$ designates the product
of the simplices that are indexed by the multivalent vertices in
$T$. The desired extension of $\Aut(T)$ is denoted in what follows by
$\Auth(T)$. It is induced from a $\mathbb{Z}  \times
\mathbb{Z}$ central extension of $\Aut _{\diamondsuit }$ in the
following manner: Let $\Auth_{\diamondsuit }$ denote the
extension of $\Aut _{\diamondsuit }$. Then $\Auth(T)$ is the set
of pairs $(\iota, g)\in \Auth_{\diamondsuit }  \times \Aut(T)$ that have the same image in
$\Aut _{\diamondsuit }$. Note that
the isomorphism given in \eqref{eq6.13} is covered by an analogous
$\Auth(T)$ version that has $\Auth(T)$ on the left hand side
and $\Auth_{\diamondsuit }$ instead of $\Aut _{\diamondsuit }$
on the right.

To define $\Auth_{\diamondsuit }$, it is necessary to
reintroduce the `blow up' from Part 6 of \fullref{sec:2c} of
$\underline {\Gamma }_{\diamondsuit }$. For this purpose, note that
the any given ${\underline {\Gamma }^{*}}_{o}$ is available if
Properties 1--4 from Part 3 of \fullref{sec:2c} are satisfied. As
observed in \fullref{sec:6a}, these properties are present; thus
any such ${\underline {\Gamma }^{*}}_{o}$ is well defined.

Granted the preceding, introduce ${\underline {\Gamma}^{*}}_{\diamondsuit }$.
Likewise, reintroduce the cohomology class
$\phi _{\diamondsuit }   \in  H^{1}({\underline {\Gamma}^{*}}_{\diamondsuit };
\mathbb{Z}  \times   \mathbb{Z})$ from
this same Part 6 of \fullref{sec:2c}. Use $\phi _{\diamondsuit}$
to define a $\mathbb{Z}  \times \mathbb{Z}$ covering space over
${\underline {\Gamma }^{*}}_{\diamondsuit }$, this denoted in what
follows by $\bar {\Gamma }^{*}$. The group of deck transformations
of $\bar {\Gamma }^{*}$ is the group $\mathbb{Z}  \times\mathbb{Z}$.
This understood, the group $\Aut _{\diamondsuit }$ has a
central, $\mathbb{Z}  \times   \mathbb{Z}$ extension that acts as a
group of automorphisms of the graph $\bar {\Gamma }^{*}$. The latter
group is $\Auth_{\diamondsuit }$.

The action of $\Auth(T)$ on the space in \eqref{eq6.15} is defined in the five steps
that follow. In this regard, the first three steps reinterpret the various
factors in \eqref{eq6.15}.

\substep{Step 1}
This first step reinterprets the factor
$\mathbb{R}_{\diamondsuit }  \times   \Delta _{\diamondsuit }$. For this
purpose, let $\Vertt_{\hat{E}}$ denote the set of vertices in $\bar{\Gamma }^{*}$
that project to some $e \in  \hat{E}$ version of $\ell_{\diamondsuit e}$.
The plan is to interpret the product $\mathbb{R}_{\diamondsuit }  \times \Delta _{\diamondsuit }$
as a fiber bundle over $\Delta _{\diamondsuit }$, this denoted by $\mathbb{R}^{\Delta }$.
In particular, $\mathbb{R}^{\Delta }$ is the linear subspace
in $\Maps(\Vertt_{\hat{E}}; \mathbb{R})  \times   \Delta _{\diamondsuit }$ where the pair
$(\tau , r)$ obeys the following constraint: Let
$\upsilon $ and $\upsilon'$ denote a pair of vertices from
$\Vertt_{\hat{E}}$ and let $N \in   \mathbb{Z}  \times   \mathbb{Z}$ denote
any element that maps $\upsilon $ to the component of $\bar {\Gamma}^{*}$ that contains $\upsilon'$.
Then
\begin{equation}\label{eq6.16}
\tau (\upsilon ') = \tau (\upsilon ) + 2\pi   \frac{1}{\alpha
_{Q_{\hat{E}} } (\theta _\diamondsuit )} \biggl(\sum _{\gamma }  \pm \hat
{r}(\gamma ) - 2\pi   \alpha _{N}(\theta _{\diamondsuit })\biggr),
\end{equation}
where the notation is as follows: First, the sum is over the arcs in $\bar
{\Gamma }^{*}$ whose concatenated union defines a path that takes
$N\upsilon $ to $\upsilon'$. Second, $\hat {r}(\gamma ) = r(\gamma )$
if $\gamma $ projects to an arc in $\underline {\Gamma }_{\diamondsuit }$;
if not, then $\hat {r}(\gamma ) = 0$. Third, the $+$ sign appears when the
arc is traversed in its oriented direction and the minus sign is used when
the arc is traversed opposite to its orientation. Note that \eqref{eq6.6} guarantees
that the expression on the right hand side of \eqref{eq6.16} is independent of the
precise choice for $N$ or for the path in question granted the given
constraints on both.

\substep{Step 2}
This step reinterprets the
$\mathbb{R}_{ - }$ factor that appears in \eqref{eq6.15}. For this
purpose, suppose that $e$ is an incident edge of $\diamondsuit $,
and then reintroduce the lift ${\ell ^{*}}_{\diamondsuit e}\subset{\underline {\Gamma}^{*}}_{\diamondsuit }$
of the loop $\ell_{\diamondsuit ,e}$. Let $L_{e}$ denote the set of components of the
inverse image of ${\ell ^{*}}_{\diamondsuit e}$ in $\bar {\Gamma}^{*}$. In this regard,
each component of this inverse image is a
linear subgraph of $\bar {\Gamma }^{*}$. Note that the deck
transformation group $\mathbb{Z}  \times \mathbb{Z}$ has a
transitive action on $L_{e}$ whereby the stabilizer of any given
element is $\mathbb{Z}\cdot Q_{e}$. Let $\Lambda    \equiv \times_{e \in \hat{E}}L_{e}$.

The plan is to identify $\mathbb{R}_{ - }$ as a certain subspace in
$\Maps(\Lambda ; \mathbb{R})$. In particular, a map, $\tau _{ - }$, is in
$\mathbb{R}_{ - }$ when the following is true: Given $L \in  \Lambda $,
and then $N= (n, n')\in \mathbb{Z}  \times   \mathbb{Z}$ and $e \in \hat{E}$, let $L'$
denote the element in $\Lambda $ that is obtained from $L$ by
acting on $e$'s entry by $N$. Then
\begin{equation}\label{eq6.17}
\tau _{ - }(L') = \tau _{ - }(L) - 2\pi \bigl(n'q_{\hat{E}} - n {q_{\hat{E}}}'\bigr).
\end{equation}
The conditions depicted here define a 1--dimensional affine line in
$\Maps(\Lambda ; \mathbb{R})$. The latter is is denoted in what follows as
$\mathbb{R}^{ - }$.

\substep{Step 3}
This step reinterprets the factor $[\times_{o \in \mathcal{V} }
\mathbb{R}_{o}]/[\times _{o \in \mathcal{V} }\mathbb{Z}_{o}]$ in \eqref{eq6.15}.
For this purpose, suppose that $e$ is incident to
$\diamondsuit $ and introduce $\mathcal{V} (e)$ to denote the set of vertices
in the component of $T- \diamondsuit $ that contains $e- \diamondsuit $. Set
$W_{e} \equiv  [\times _{\hat{o} \in \mathcal{V} (e)}  \mathbb{R}_{\hat{o}}]/[\times _{\hat{o} \in
\mathcal{V} (e)}  \mathbb{Z}_{\hat{o}}]$ and so write
\begin{equation}\label{eq6.18}
[\times _{o \in \mathcal{V} }  \mathbb{R}_{o}]/[\times _{o \in
\mathcal{V} }  \mathbb{Z}_{o}] = \times _{e \in E_\diamondsuit }  W_{e}.
\end{equation}
Now let $\hat{U}_{e}   \subset \Maps(L_{e}; W_{e})$ denote the subspace
of $\maps x\co  L_{e}   \to W_{e}$ that have the following property: If
$\ell    \in  L_{e}$ and $N \in   \mathbb{Z}  \times   \mathbb{Z}$, then
$x(N\cdot \ell )$ and $x(\ell )$ have respective lifts to
$\times_{\hat{o} \in \mathcal{V} (e)}  \mathbb{R}_{\hat{o}}$ whose
coordinates obey
\begin{equation}\label{eq6.19}
x(N\cdot \ell )_{\hat{o}} = x(\ell )_{\hat{o}} - 2\pi
  \frac{\alpha _N (\theta _{\hat{o}} )}{\alpha _{Q_{\hat{e}} } (\theta _{\hat{o}} )}
\end{equation}
for each $\hat{o} \in \mathcal{V} (e)$.

Note that this condition is well defined even though the integer multiples
of $Q_{e}$ act trivially on $L_{e}$. Indeed, such is the case because $\mathbb{Z}\cdot Q_{e}$
also acts trivially on $W_{e}$. Since the action on
$L_{e}$ of $\mathbb{Z}  \times   \mathbb{Z}$ is transitive, the space
$\hat{U}_{e}$ is diffeomorphic to $W_{e}$.

\substep{Step 4}
The space depicted in \eqref{eq6.15} is diffeomorphic to
\begin{equation}\label{eq6.20}
\mathbb{R}^{ - }  \times \mathbb{R}^{\Delta }  \times \Bigl[\times _{e
\in E_\diamondsuit } \bigl(\hat{U}_{e}  \times  \bigl(\times
_{\hat{o} \in \mathcal{V} (e)}\Delta
_{\hat{o}}\bigr)\bigr)\Bigr].
\end{equation}
The group $\Auth(T)$ will act on this version of \eqref{eq6.15}. This step describes
the action of the subgroups $\{\mathbb{Z}/(n_{o}\mathbb{Z})\}_{o \in\mathcal{V} }$
that appear in the semi-direct decomposition as depicted in the
$\Auth(T)$ version of \eqref{eq6.13}.

The definition of these actions requires an additional set of choices to be
made for each vertex in $\mathcal{V} $: A `distinguished' edge must be
designated from each non-trivial $\mathbb{Z}/(n_{o}\mathbb{Z})$ orbit in
$E_{o}$ for each $o \in   \mathcal{V} $. To explain, note first that \eqref{eq6.12}
has the following consequence: The $\mathbb{Z}/(n_{o}\mathbb{Z})$ action on an
orbit in $E_{o}$ is either free or trivial. Thus, any non-trivial orbit has
$n_{o}$ elements and a canonical cyclic ordering. The choice of a
distinguished element provides a compatible linear ordering with the
distinguished element at the end. These distinguished edges should be chosen
in a compatible fashion with the $\Aut(T)$ action. This is to say that when $o\in \mathcal{V} $
and $e' \in E_{o}$ is the distinguished edge
in its $\Aut _{o}$ orbit, then the following is true: Let $\iota \in \Aut(T)$
denote an automorphism that sends the distinguished vertex
in $\underline {\Gamma }_{o}$ to that in $\underline {\Gamma}_{\iota (o)}$.
Then $\iota (e')$ is the distinguished edge in its  $\Aut _{\iota (o)}$ orbit.

Let $\iota    \in \Auth(T)$ now denote the generator of the
$\mathbb{Z}/(n_{o}\mathbb{Z})$ subgroup. This element acts
trivially on $\mathbb{R}^{ - }$, $\mathbb{R}^{\Delta }$. Its action
is also trivial on $\hat{U}_{e}\times (\times _{\hat{o} \in \mathcal{V} (e)}  \Delta _{\hat{o}})$
unless $\mathcal{V} (e)$ contains $o$.

The affect of $\iota $ on the relevant version of
$\hat{U}_{e}\times (\times _{\hat{o} \in \mathcal{V} (e)}  \Delta_{\hat{o}})$
is defined from a certain action of $\iota $ on
$W_{e}  \times (\times _{\hat{o} \in \mathcal{V} (e)} \Delta_{\hat{o}})$.
In this regard, note that $\iota _{o}$ has the
action on $\times _{\hat{o} \in \mathcal{V} (e)} \Delta_{\hat{o}}$
that sends a map $r\co  \Arc(\underline {\Gamma}_{\hat{o}})   \to (0, \infty )$ to the map
$\iota \cdot r\co  \Arc(\underline {\Gamma }_{\iota (\hat{o})}) \to (0,\infty )$ that obeys
\begin{equation}\label{eq6.21}
(\iota \cdot r)(\iota (\gamma )) = r(\gamma )
\text{ for all }
\gamma \in \Arc(\underline {\Gamma }_{\hat{o}}).
\end{equation}
The action that is described momentarily on
$W_{e}  \times (\times_{\hat{o} \in \mathcal{V} (e)}  \Delta _{\hat{o}})$ is
intertwined by the projection to $\times _{\hat{o} \in\mathcal{V} (e)}  \Delta _{\hat{o}}$
with the action that is
depicted by \eqref{eq6.21}. In any event, an action of
$\mathbb{Z}/(n_{o}\mathbb{Z})$ on $W_{e}  \times (\times_{\hat{o} \in \mathcal{V} (e)}  \Delta _{\hat{o}})$
induces one on $\Maps(L_{e}; W_{e})  \times (\times_{\hat{o} \in \mathcal{V} (e)}  \Delta _{\hat{o}})$
by composition, thus $\iota $ sends a pair $(x, r)$ with $x\co  L_{e} \to W_{e}$ and
$r \in (\times _{\hat{o} \in \mathcal{V} (e)} \Delta _{\hat{o}})$ to the pair whose second component is
$\iota \cdot r$ and whose first component is obtained from $x$ by
composing with $\iota $'s affect on the $W_{e}$ factor in
$W_{e}\times (\times _{\hat{o} \in \mathcal{V} (e)}  \Delta_{\hat{o}})$.
As it turns out, this action of $\iota $
preserves the relation in \eqref{eq6.19} and so the induced action
on $\Maps(L_{e}; W_{e})\times (\times _{\hat{o} \in\mathcal{V} (e)}  \Delta _{\hat{o}})$
induces the required
action on $\hat{U}_{e}  \times (\times _{\hat{o} \in \mathcal{V} (e)}  \Delta _{\hat{o}})$.

To define the action of $\iota $ on $W_{e}\times(\times_{\hat{o} \in\mathcal{V} (e)} \Delta _{\hat{o}})$
consider its affect on the image of some given point $(\tau , r)$
where $\tau  \in   \times _{\hat{o} \in \mathcal{V} (e)}\mathbb{R}_{\hat{o}}$.
For this purpose, use $\tau_{\hat{o}}$ to denote the $\mathbb{R}_{\hat{o}}$
coordinate of $\tau $, and use $r_{\hat{o}}$ to denote the
$\Delta _{\hat{o}}$ coordinate. Also, use $[\tau , r]$ to
denote the image point of $(\tau, r)$ in
$W_{e}\times (\times_{\hat{o} \in \mathcal{V} (e)}  \Delta _{\hat{o}})$. The
point $\iota \cdot [\tau , r]$ has a lift, $(\tau ', \iota \cdot r)$ with
${\tau '}_{\hat{o}}=\tau _{\hat{o}}$ unless
$\hat{o}  \in T_{o}$. Meanwhile,
\begin{equation}\label{eq6.22}
\tau '_{o}=\tau _{o} - 2\pi   \frac{1 }{ {\alpha
_{Q_{e(o)} } (\theta _o )}}\sum _{\gamma }r(\gamma ),
\end{equation}
where $e(o)$ here designates the edge that connects $o$ to $T- T_{o}$ and
where the sum is over the set of arcs on the oriented path in $\underline {\Gamma }_{o}$
between $\iota ^{- 1}(\upsilon _{o})$ and $\upsilon _{o}$.

Consider next the case that $\hat{o} \in T_{o}- o$ and that $\hat{o}$'s
component of $T_{o}- o$ is fixed by $\iota _{o}$. This is to say
that $\hat{o}$'s component is connected to $o$ by an edge that is fixed by the
$\mathbb{Z}/(n_{o}\mathbb{Z})$ action on $o$'s incident edge set. Let $e'(o)$
denote the latter edge. Then
\begin{equation}\label{eq6.23}
{\tau '}_{\hat{o}}=\tau _{\hat{o}} - 2\pi   \frac{1
}{ {n_o }}  \frac{\alpha _{Q_{e(o)} - Q_{e' (o)} } (\theta _{\hat{o}}
)}{\alpha _{Q_{\hat{e}} } (\theta _{\hat{o}})} ,
\end{equation}
where $\hat{e}$ here denotes the edge that connects $\hat{o}$ to $T- T_{\hat{o}}$.

Finally, suppose that $\hat{o}$ is in a component of $T_{o}- o$ that
is not fixed by $\iota $. In this case,
\begin{equation}\label{eq6.24}
\tau '_{\iota (\hat{o})}=\tau _{\hat{o}}-\varepsilon
_{\hat{o}} 2\pi   \frac{\alpha _{Q_{e(o)} } (\theta _{\hat{o}}
)}{\alpha _{Q_{\hat{e}} } (\theta _{\hat{o}} )} ,
\end{equation}
where $\varepsilon _{\hat{o}} = 0$ unless the edge that connects
$\hat{o}$'s component of $T_{o}- o$ to $o$ is a `distinguished' edge
in its $\mathbb{Z}/(n_{o}\mathbb{Z})$ orbit; in the latter case, $\varepsilon_{\hat{o}} = 1$.

It is left for the reader to verify that action just defined for
$\iota $ on $W_{e}  \times (\times _{\hat{o} \in \mathcal{V}(e)}  \Delta _{\hat{o}})$ has
$\iota ^{n_o }$ acting as the
identity. The reader is also asked to verify that the suite of these
actions as $o$ varies through $\mathcal{V} $ defines compatible
actions of the various versions of \eqref{eq6.14} on the space in \eqref{eq6.20}.

\substep{Step 5}
This step explains how the
$\Auth_{\diamondsuit }$ subgroup of $\Auth(T)$ acts on
the space in \eqref{eq6.20}. The explanation starts by describing
the action on the factor $\mathbb{R}^{\Delta }$. In this regard,
keep in mind that $\mathbb{R}^{\Delta }$ is a real line bundle over
$\Delta _{\diamondsuit }$ and that there is an $\Aut _{\diamondsuit }$
action on $\Delta _{\diamondsuit }$ that has any given $\iota    \in  \Aut _{\diamondsuit }$
acting to send a map $r\co  \Arc(\underline{\Gamma }_{\diamondsuit }) \to (0, \infty )$
to the map $\iota \cdot r$ that is given by the $\hat{o} = \diamondsuit $ version of
\eqref{eq6.21}. The projection map from $\mathbb{R}^{\Delta }$ to
$\Delta _{\diamondsuit }$ will intertwine the desired action of
$\Auth_{\diamondsuit }$ with that of $\Aut _{\diamondsuit }$ on
$\Delta _{\diamondsuit }$. In any event, the action on
$\mathbb{R}^{\Delta }$ is induced by the action on
$\Maps(\Vertt_{\hat{E}}; \mathbb{R})  \times \Delta_{\diamondsuit }$ that has
$\iota \in \Auth_{\diamondsuit}$ sending a pair $(\tau , r)$ to the pair
$(\iota \cdot \tau ,\iota \cdot r)$ where $\iota \cdot \tau $ is defined by setting
$(\iota \cdot \tau )(\iota (\upsilon )) \equiv   \tau (\upsilon)$ for all
$\upsilon \in  \Vertt_{\hat{E}}$. This action
preserves the relation in \eqref{eq6.16} so restricts to define an
action of $\Auth_{\diamondsuit }$ on $\mathbb{R}^{\Delta }$.

To continue, consider next the action of $\Auth_{\diamondsuit }$ on $\mathbb{R}^{ - }$.
The action in this case is induced from the action on
$\Maps(\Lambda , \mathbb{R})$ that has $\iota \in \Auth_{\diamondsuit}$ sending a given map
$\tau _{ - }$ to the map $\iota \cdot \tau _{- }$ whose value on any given
$\iota (L)$ is that of $\tau _{ - }$ on $L$.

The final point is that of the $\Auth_{\diamondsuit }$ action on the
bracketed factor that appears in \eqref{eq6.20}. To set the stage for the
discussion, note that $\Aut _{\diamondsuit }$ acts on the simplex product.
Here, the affect of $\iota  \in \Aut _{\diamondsuit }$ on any
given $\hat{o} \in   \mathcal{V} $ version of $\Delta _{o}$ sends $r \in \Delta _{o}$ to the map
$\iota \cdot r  \in   \Delta _{\iota (o)}$
whose values are given by the rule in \eqref{eq6.21}. The projection to
$\times_{\hat{o} \in \mathcal{V} }  \Delta _{\hat{o}}$ from the bracketed factor in \eqref{eq6.20}
will intertwine the desired
$\Auth_{\diamondsuit }$ action with that just described on
$\times_{\hat{o} \in \mathcal{V} }  \Delta _{\hat{o}}$. The action of
$\Auth_{\diamondsuit }$ on the bracketed factor in \eqref{eq6.20} is induced by an
action on
\begin{equation}\label{eq6.25}
\times _{e \in E_\diamondsuit } \Bigl(\Maps(L_{e}; W_{e})  \times
\bigl(\times _{\hat{o} \in \mathcal{V} (e)}^{ }\Delta_{\hat{o}}\bigr)\Bigr).
\end{equation}
To describe the latter action, let $w$ denote a point in
\eqref{eq6.25}, thus a tuple whose coordinates are indexed by
$\diamondsuit $'s incident edges. The coordinate with label
$e \in  E_{\diamondsuit }$ consists of a map, $x_{e}$, from $L_{e}$ to
$W_{e}$ together with a maps,
$\{r_{\hat{o}}\co \Arc(\underline {\Gamma }_{\hat{o}})\to (0,\infty)\}_{\hat{o} \in \mathcal{V} (e)}$.
Here, $x_{e}$ has lifts
as a map of $L_{e}$ to $\times _{\hat{o} \in \mathcal{V} (e)}\mathbb{R}_{\hat{o}}$
that assigns a real number to each
$(\ell , \hat{o})  \in L_{e}  \times   \mathcal{V} (e)$. Fix such
a lift and use $\tau _{\hat{o}}(\ell )   \in   \mathbb{R}$ in
what follows to denote the lift's assignment to a given $(\ell ,\hat{o})$.

With the preceding understood, a given $\iota    \in \Auth_{\diamondsuit }$ sends $w$ to
the tuple $\iota \cdot \mathfrak{m}$ whose coordinate with label $e$ is denoted by ${x'}_{e}$.
The collection
$\{{x'}_{(\cdot )}\}$ thus defines $\iota \cdot \mathfrak{m}$ given the
intertwining requirement vis a vis the action of $\Aut _{\diamondsuit }$ on
$\times _{\hat{o}}  \Delta _{\hat{o}}$. This understood, the
various versions of ${x'}_{(\cdot )}$ are defined by requiring that any given
version of ${x'}_{\iota (e)}$ lift so as to assign the real number
$\tau_{\hat{o}}(\ell )$ to the pair $(\iota (\ell ), \iota (\hat{o}))$
for all pairs $(\ell, \hat{o}) \in  L_{e}  \times   \mathcal{V} (e)$.

It is left as an exercise with the definitions to verify that the rule just
given specifies an action of $\Auth_{\diamondsuit }$ on \eqref{eq6.25} that
preserves the subspace from the bracketed term in \eqref{eq6.20}.

Granted that $\Auth_{\diamondsuit }$ acts as just described on the space
in \eqref{eq6.20}, it is a straighforward task to verify that the action is
compatible vis-\`{a}-vis the $\Auth(T)$ version of \eqref{eq6.13} with those defined
in Step 4 of the various versions of \eqref{eq6.14}. Such being the case, then these
various actions define an action of $\Auth(T)$ on the space in \eqref{eq6.20}. This
is the desired action.

\step{Part 4}
Introduce as notation ${O^{*}}_{T}$ to denote the space depicted in
\eqref{eq6.20}. Then the space $O_{T}$ is diffeomorphic to the
quotient of ${O^{*}}_{T}$ by the $\mathbb{Z}  \times   \mathbb{Z}$
subgroup in $\Auth(T)$ that maps to the identity in $\Aut(T)$. The
group $\Aut(T)$ now acts on $O_{T}$ with the latter now viewed as
${O^{*}}_{T}/(\mathbb{Z}  \times  \mathbb{Z})$. Let $\hat{O}_{T}   \subset O_{T}$ denote the set of points
where the $\Aut(T)$ stabilizer is the identity. Thus, $\hat{O}_{T}$ is the
$\mathbb{Z}  \times   \mathbb{Z}$ quotient of the points in \eqref{eq6.20} with trivial
$\Auth(T)$ stabilizer. In any event, $\hat{O}_{T}/\Aut(T)$ is a smooth manifold
whose dimension is equal to one more than the number of vertices in
$ \cup_{o}\underline {\Gamma }_{o}$; the union here is indexed by
multivalent vertices in $T$. Meanwhile, $O_{T}/\Aut(T)$ is a smooth orbifold.

With the preceding understood, consider:

\begin{theorem}\label{thm:6.2}

If non-empty, ${\mathcal{M}^{*}}_{\hat{A},T}$
is diffeomorphic as an orbifold to $\mathbb{R}$\linebreak
$\times O_{T}/\Aut(T)$. In particular the complement in
${\mathcal{M}^{*}}_{\hat{A},T}$ of $\mathcal{R}$
 is diffeomorphic to $\mathbb{R}\times  \hat{O}_{T}/\Aut(T)$.

\end{theorem}
\noindent
There is an analogy in this case with the story told in
\fullref{prop:3.6} as there are orbifold diffeomorphisms
from ${\mathcal{M}^{*}}_{\hat{A},T}$ that give a direct geometric
interpretation to the various factors that enter the definition of
$\mathbb{R}  \times O_{T}/\Aut(T)$. An elaboration requires a short
digression for two new notions.

The first notion is that of the space ${{\mathcal{M}^{*}}_{\hat{A},T}}^{\Lambda }$
whose elements consist of equivalence
classes of 4--tuples $(C_{0}, \phi , T_{C})$ where $(C_{0}, \phi )$
defines a point in ${\mathcal{M}^{*}}_{\hat{A},T}$ while $T_{C}$ is
a fixed correspondence in $(C_{0}, \phi )$. The equivalence relation
equates $(C_{0}, \phi ', {T_{C}}')$ with $(C_{0}, \phi , T_{C})$ in
the case that there is a holomorphic diffeomorphism $\psi \co  C_{0}   \to C_{0}$ such that
$\phi ' = \phi  \circ \psi $, and such that ${T_{C}}'$ is obtained from $T_{C}$ as follows:
If $T_{C}$ identifies a given component
$K \subset C_{0}- \Gamma $ with an edge $e \in T$, then
${T_{C}}'$ identifies $\psi ^{- 1}(K)$ with $e$; and if $T_{C}$
identifies any given arc $\gamma    \subset   \Gamma $ with an arc,
$\gamma _{*}$, in a version of $\Gamma _{(\cdot )}$ from $T$, then
${T_{C}}'$ identifies $\psi ^{- 1}(\gamma )$ with $\gamma _{*}$.

Note that ${{\mathcal{M}^{*}}_{\hat{A},T}}^{\Lambda }$ is a smooth
manifold since the group $G_{C}$ of holomorphic diffeomorphisms of $C_{0}$
that fix $\phi $ acts freely on the set of corrrespondences of $T$ in
$(C_{0}, \phi )$. Note in addition that the tautological map from
${{\mathcal{M}^{*}}_{\hat{A},T}}^{\Lambda }$ to ${\mathcal{M}^{*}}_{\hat{A},T}$
restricts over ${\mathcal{M}_{\hat{A},T}}$
as a covering space map with fiber $\Aut(T)$.

To describe the second notion, fix a multivalent vertex $o \in T$
and a vertex $\upsilon    \in \underline {\Gamma }_{o}$ with
non-zero integer label. In what follows, use $(\hat {p}_{o}, {\hat{p}_{o}}')$
to denote the relatively prime integer pair that defines
the angle $\theta _{o}$ via \eqref{eq1.8}. Define a map,
$\Psi_{\upsilon }\co  O_{T}   \to \mathbb{R}/(2\pi \mathbb{Z})$ as
follows: $\Map \mathbb{R}_{o}  \times \Delta _{o}$ to
$\mathbb{R}/2\pi \mathbb{Z}$ by the rule that sends $v \in  \mathbb{R}_{o}$ and
$r \in   \Delta _{o}$ to
\begin{equation}\label{eq6.26}
\bigl(\hat {p}_{o}{q_{e}}' - {\hat {p}_{o}}'q_{e}\bigr) v + \frac{({\hat {p}_o} ^2
+ {\hat {{p}_{o }}} ^{'2}\sin ^2(\theta _o ))^{1 / 2}}{(1 + 3\cos ^4(\theta
_o ))^{1 / 2}}  \sum _{\gamma }  \pm r(\gamma )  \mod(2\pi \mathbb{Z}).
\end{equation}
Here, the notation is as follows: First, $e$ is the distinguished edge for $o$.
Second, the sum is indexed by the arcs in any ordered set of arcs from
$\underline {\Gamma }_{o}$ that concatenate end to end so as to define a
path that starts at $\underline {\Gamma }_{o}$'s distinguished vertex and
ends at $\upsilon $. Finally, the $+$ sign in this sum is taken if and only
if the indicated arc is traversed in its oriented direction on the path to
$\upsilon $. By virtue of \eqref{eq6.6}, this map is insensitive to the choice for
the concatenating set of arcs. As the map in \eqref{eq6.26} is also insensitive to
the group actions that define $O_{T}$, so it descends as a map from $O_{T}$
to $\mathbb{R}/2\pi \mathbb{Z}$. The latter is the map $\Psi _{\upsilon }$.

\begin{theorem}\label{thm:6.3}

There are orbifold diffeomorphisms from
${\mathcal{M}^{*}}_{\hat{A},T}$  to $\mathbb{R}\times O_{T}/\Aut(T)$  with the following properties:

\bit

\item
The projection to the $\mathbb{R}$  factor intertwines $\mathbb{R}$'s  action on
${\mathcal{M}^{*}}_{\hat{A},T}$  with its translation action on $\mathbb{R}$.

\item
The orbifold diffeomorphism is covered by a diffeomorphism from ${\mathcal{M}_{\hat{A},T}}^{\Lambda }$  to
$\mathbb{R}  \times O_{T}$  with the following property:  Let $o \in T$
be a multivalent vertex and let $\upsilon $  be a vertex
in $\underline {\Gamma }_{o}$  with weight $m_{o}   \ne  0$.  Let
 $(C_{0}, \phi , T_{C})   \in {{\mathcal{M}^{*}}_{\hat{A},T}}^{\Lambda }$,  and let $E$ denote the end
in $C_{0}$  that corresponds via $T_{C}$  to $\upsilon $.  Then
the $\mathbb{R}/(2\pi \mathbb{Z})$  valued function $\hat{p}_{o}\varphi -{{\hat{p}_{o}}}'t$  on $E$
has a unique limit as $|s| \to \infty $,  and this limit is obtained by composing the map
$\Psi _{u}$  with the map from ${{\mathcal{M}^{*}}_{\hat{A},T}}^{\Lambda }$  to $O_{T}$.

\eit
\end{theorem}

As should be clear from the definitions in the next subsection, the other
parameters that enter $O_{T}$'s definition can also be given direct
geometric meaning.

Theorems~\ref{thm:6.2} and \ref{thm:6.3} are proved in \fullref{sec:7}.

The following proposition says something about $O_{T}- \hat{O}_{T}$
and the corresponding stabilizers in $\Aut(T)$.

\begin{proposition}\label{prop:6.4}

The $\Aut(T)$ stabilizer of any point in $O_{T}$  projects isomorphically onto a subgroup
of $\Aut_{\diamondsuit }$.  Moreover, two stabilizer subgroups are conjugate in $\Aut(T)$
if and only if they have conjugate images in $\Aut_{\diamondsuit}$.

\end{proposition}

The next subsection defines a version of \fullref{thm:6.3}'s map and the
subsequent subsection proves that this map is continuous. The final
subsection provides a simpler picture of the $\Aut(T)$ action on $O_{T}$ when
 $\Aut _{\diamondsuit }$ fixes some incident edge to $\diamondsuit $. For
example, this is the case when $T$ is a linear graph; and the results from
this subsection can be used to deduce \fullref{thm:3.1} from \fullref{thm:6.2}. The
final subsection also contains the proof of \fullref{prop:6.4}.

\subsection{The map from ${\mathcal{M}^{*}}_{\hat{A},T}$ to $\mathbb{R}  \times  O_{T}/\Aut(T)$}
\label{sec:6c}

Suppose that $(C_{0}, \phi )$ gives rise to an element in
${\mathcal{M}^{*}}_{\hat{A},T}$. The image of this element in
$\mathbb{R}  \times O_{T}/\Aut(T)$ is obtained from a point that is
assigned to $(C_{0}, \phi )$ in
$\mathbb{R}_{ - }  \times (\times _{o}(\mathbb{R}_{o}  \times  \Delta _{o}))$.
That latter assignment requires extra choices, some
made just once for all $(C_{0}, \phi )$ and others that are
contingent on the particular pair. As is explained below, the
contingent choices are not visible in $\mathbb{R}  \times O_{T}/\Aut(T)$.
The story on the map to $\mathbb{R}  \times O_{T}/\Aut(T)$ is told in four parts.

\step{Part 1}
This part and Part 2 describe the choices that do not depend on the
given point in ${\mathcal{M}^{*}}_{\hat{A},T}$ nor on its
representative pair $(C_{0}, \phi )$.

The first choice in this regard is an $\Aut(T)$ orbit, $\hat{E}'$, of edges that
end in vertices of $T$ that are labeled by the smallest angle. This angle is
denoted in what follows by $\theta _{ - }$.

The next series of choices involve the vertex $\diamondsuit $. The first of
these involves the already chosen $\Aut(T)$ orbit $\hat{E}$ in the set of incident
edges to $\diamondsuit $. Choose a `distinguished' edge in $\hat{E}$, and a
corresponding `distinguished' vertex on the latter's version of the loop
$\ell _{\diamondsuit (\cdot )}$. Let $\ell _{\diamondsuit }$ denote the
corresponding distinguished version of $\ell _{\diamondsuit (\cdot )}$ and
let $\upsilon _{\diamondsuit }   \in   \ell _{\diamondsuit }$ denote
the distinguished vertex.

To continue this series of choices at $\diamondsuit $, select a
concatenating path set that connects the vertex $\upsilon _{\diamondsuit}$
to each version of $\ell _{\diamondsuit (\cdot )}^{ } \ne   \ell_{\diamondsuit }$.
The definition of such a path set is provided in
\fullref{def:2.2}. However, the selection must be constrained by the two
conditions in the version of \eqreft2{18} where the vertex is $\diamondsuit $ and
 $e$ is the distinguished edge.

An analogous, but more constrained set of choices must be made for each
multivalent vertex $o \in   \mathcal{V} $. To elaborate, let $e$ denote now
the edge that connects $o$ to $T- T_{o}$ and let $e' \ne e$ denote
another incident edge to $o$ that is fixed by the $\mathbb{Z}/(n_{o}\mathbb{Z})$
action. Choose a concatenating path set to connect $\upsilon _{o}$ to
$\ell _{oe'}$ subject to the two conditions in \eqreft2{18}. Such a
concatenating path set must also be chosen when $e'$ is not fixed by
$\mathbb{Z}/(n_{o}\mathbb{Z})$, but additional care must be taken. To describe what
is involved here, let $E$ denote a non-trivial $\mathbb{Z}/(n_{o}\mathbb{Z})$
orbit in $o$'s incident edge set and let $e' \in E$ denote the distinguished
edge. Choose a concatenating path set from $\upsilon _{o}$ to $\ell_{oe'}$
subject to the constraints in \eqreft2{18}. Let $\{\nu _{1},\ldots , \nu _{N}\}$
denote this chosen set. Now, let $\iota  \in   \mathbb{Z}/(n_{o}\mathbb{Z})$ and let
$\nu ^{\iota }\subset   \ell _{o}$ denote the path that starts at $\upsilon _{o}$ and
proceeds opposite the oriented direction along $\ell _{o}$ to its end at
$\iota (\upsilon _{o})$. The set of paths $\{\iota (\nu _{1})\circ \nu ^{\iota }, \iota (\nu _{2}),
\ldots , \iota (\nu_{N})\}$ constitutes a concatenating path set that starts at
$\upsilon_{o}$ and ends at $\ell _{o\iota (e')}$. Here, $\iota (\nu_{1}) \circ \nu ^{\iota }$
denotes the concatenation of the two
paths. This new set obeys the $\iota (e')$ version of \eqreft2{18}. Use the
various versions of this set for the required concatenating paths for the
elements in $E- e'$.

The preceding choices of concatenating path sets at the vertices in
$\mathcal{V} $ must be made so as to be compatible with the $\Aut(T)$ action in
the following sense: Let $o \in   \mathcal{V} $, let $e' \in E_{o}$, and
let $\{\nu _{1}, \ldots , \nu _{N}\}$ denote the chosen
concatenating path set that starts at the vertex $\upsilon _{o}$ and ends
on $\ell _{oe'}$. If $\iota  \in \Aut(T)$ maps the distinguished vertex on
$\ell _{o}$ to that on $\ell _{\iota (o)}$, then
$\{\iota (\nu _{1}), \ldots, \iota (\nu _{N})\}$ is the
chosen concatenating path set that starts at $\upsilon _{\iota (o)}$ and
ends on $\ell _{\iota (o)\iota (e')}$.

\step{Part 2}
Choose a vertex, $\hat {\upsilon }$, in $\bar {\Gamma }^{*}$ that
projects to $\upsilon _{\diamondsuit }$. The choice for this vertex has
three consequences. First, it trivializes the fiber bundle $\mathbb{R}^{\Delta }$
in the following manner: Let $(\tau,r)\in\Maps(\Vertt_{\hat{E}};\mathbb{R})\times\Delta _{\diamondsuit }$
denote any given pair in $\mathbb{R}^{\Delta }$. The trivializing map then
sends $(\tau, r)$ to $(\tau (\hat {\upsilon }), r)$. With this
trivialization understood, write $\mathbb{R}^{\Delta }$ as
$\mathbb{R}_{\diamondsuit }  \times   \Delta _{\diamondsuit }$ where
$\mathbb{R}_{\diamondsuit }$ is a copy of $\mathbb{R}$.

The selection of $\hat {\upsilon }$ identifies $\mathbb{R}^{ - }$
with a fixed copy of $\mathbb{R}$, this labeled as $\mathbb{R}_{ -}$
in what follows. To describe how the identification comes about,
let $e$ denote the chosen distinguished edge in $\hat{E}$. Then $\hat{\upsilon }$
sits on a unique line, $\ell $, in $L_{e}$. Now, let
$e'$ denote some other edge in $\hat{E}$. The $e'$ version of the
concatenating path set defines a path in $\underline {\Gamma}_{\diamondsuit }$
that starts at $\upsilon _{o}$ and ends on $\ell_{\diamondsuit e'}$.
This path has a canonical lift to $\underline{\Gamma }^{*}$ and the latter has a
unique lift to a path in $\bar{\Gamma }^{*}$ that starts at $\hat {\upsilon }$ and ends on a
particular line in $L_{e'}$. The collection consisting of $\ell $
and these other lines specifies a point $L \in   \times_{{\hat{e}} \in \hat{E}}L_{{\hat{e}}}$. As a
map, $\tau _{ - }   \in   \mathbb{R}^{ - }$ is determined by its
value on $L$, so the assignment $\tau _{ - }   \to   \tau _{ - }(L) \in   \mathbb{R}$
identifies $\mathbb{R}^{ - }$ with a fixed line.

The selection of $\hat {\upsilon }$ also identifies each $\hat{U}_{e}$ with
the corresponding space
$W_{e}   \equiv  [\times _{\hat{o} \in \mathcal{V} (e)}  \mathbb{R}_{\hat{o}}]/[\times _{\hat{o} \in
\mathcal{V} (e)}  \mathbb{Z}_{\hat{o}}]$ that appears in \eqref{eq6.18}. Here is
how: If $e$ is the distinguished edge from $\hat{E}$,  then the lift
$\hat{\upsilon }$ sits on a unique $\ell    \in  L_{e}$. Then, the assignment
to a given $x \in  \Maps(L_{e}; W_{e})$ of $x(\ell ) \in W_{e}$ identifies
$\hat{U}_{e}$ with $W_{e}$. If $e'$ is any other edge from
$E_{\diamondsuit }$, the choice of $\hat {\upsilon }$ and the $e'$ version of
the concatenating path set determines a path from $\hat {\upsilon }$ to a
line in $L_{e'}$. The values of $x \in  \Maps(L_{e'},W_{e'})$ on the latter identify
$\hat{U}_{e'}$ with $W_{e'}$.

\step{Part 3}
Granted this identification between \eqref{eq6.20} and \eqref{eq6.15}, a point will be
assigned to the given pair $(C_{0}, \phi )$ in $\mathbb{R}$, in each copy of
$\Delta _{o}$, in each copy of $\mathbb{R}_{o}$, and in $\mathbb{R}_{ - }$.
These assignments all require additional choices that must be made
separately for each pair. These choices are listed below:

\step{Choice 1}
\textsl{A choice of a correspondence of $T$  in $(C_{0}, \phi )$.}

This chosen correspondence is used implicitly in the subsequent choices.

\step{Choice 2}
\textsl{A parameterization for the component in $C_{0}- \Gamma $  that corresponds to the
distinguished edge in $\hat{E}$.}

\noindent The next set of choices are labeled by the multivalent vertices in $T$. The
choice at any given vertex $o$ is contingent on an apriori specification of a
canonical parameterization of the component of $C_{0}- \Gamma $
that corresponds to a certain edge in $E_{o}$. In the case that the $o=\diamondsuit $,
the incident edge is the distinguished edge in $\hat{E}$ and
the parameterization is that provided by Choice 2. In the case that $o \in   \mathcal{V} $,
the edge in question connects $o$ to $T- T_{o}$.

\step{Choice 3} (at $o$)\qua
\textsl{A lift to $\mathbb{R}$  for the $\mathbb{R}/(2\pi \mathbb{Z})$  coordinate of the point on the boundary
of the associated parametrizing cylinder that corresponds to $\upsilon _{o}$.}

The canonical parameterization that is used to make any given $o \in \mathcal{V} $
version of this last choice is determined in an inductive
fashion by the choices for those $\hat{o}  \in T$ where $T_{\hat{o}}$
contains $o$. This induction works as follows: Suppose that $o$ is a multivalent
vertex in $T$ and that a canonical parameterization has been given to the
component of $C_{0}- \Gamma $ whose labeling edge supplies the loop
$\ell _{o}$ in $\underline {\Gamma }_{o}$. Let $e'$ denote any other
incident edge to $o$. Chosen already is a concatenating path set that obeys
the corresponding version of the constraints in \eqreft2{18}. The parametrizing
algorithm from Part 4 of \fullref{sec:2c} uses this path set with $o$'s version of
Choice 3 to specify the canonical parametrization for the $e'$ component of
$C_{0}- \Gamma $.

\step{Part 4}
There are three cases to distinuish in order to describe the data for the
value of the $\mathbb{R}$ factor in $\mathbb{R}  \times O_{T}/\Aut(T)$ on a point that comes from
${\mathcal{M}^{*}}_{\hat{A},T}$.

\substep{Case 1}
In this case $\theta _{ - }   \ne 0$. The orbit $\hat{E}'$ corresponds to a set of convex side ends in
$C$ where the $|s|    \to   \infty $ limit of $\theta $ is $\theta _{- }$. Use $\hat{E}'$
to also denote this set of ends. Associated to
each end $E \in \hat{E}'$ is the real number $b \equiv  b(E)$
that appears in \eqreft23. Here, $b(E)$ must be positive since
$\theta _{ - }$ is the infimum of $\theta $ on $C$. The map to
$\mathbb{R}$ assigns to $(C_{0}, \phi )$ the real number
$-\zeta^{- 1 }\sum _{E \in \hat{E}'} \ln(b(E))$ where
$\zeta\equiv   \surd 6 \sin^{2}\theta _{E} (1+3\cos^{2}\theta _{E})/(1+3\cos^{4}\theta _{E})$.

\substep{Case 2}
In this case $\theta _{ - } = 0$ and
there are no $(1,\ldots)$ elements in $\hat{A}$. Here, $\hat{E}'$
corresponds to a set of disjoint disks in $C_{0}$ whose centers are $\theta = 0$ points.
This understood, then the image of $(C_{0}, \phi )$ in the
$\mathbb{R}$ factor of $\mathbb{R}  \times O_{T}/\Aut(T)$ is the
sum of the $s$--coordinates of the centers of these disks.

\substep{Case 3}
In this case $\theta _{ - } = 0$ and
$\hat{E}'$ corresponds to a set of ends in $C_{0}$ whose constant $|s| $ slices limit to the
$\theta  = 0$ cylinder as $|s|\to   \infty $. Let $\hat{E}'$ also denote this set of ends. Each
$E \in\hat{E}'$ defines the positive constant $\hat {c}   \equiv   \hat {c}(E)$
that appears in \eqref{eq1.9}. Note that the integer $p$ and $p'$ that appear in \eqref{eq1.9}
comprise the pair from the corresponding $(1,\ldots)$
element in $\hat{A}$. The image of $C$ in $\mathbb{R}$ is
$-(\sqrt {\frac{3}{ 2}} +\frac{{p' } }{ p})^{- 1}  \sum _{E \in\hat{E}'} \ln(\hat {c}(E))$.

The point assigned $(C_{0}, \phi )$ in any given version of $\Delta
_{o}$ is defined as follows: As a point in $\Delta _{o}$ is a map
from $\Arc(\underline {\Gamma }_{o})$ to $(0, \infty )$, it is
sufficient to provide a positive number to any given arc subject to
the constraints in \eqref{eq6.6}. For this purpose, let
$\gamma\subset \underline {\Gamma }_{o}$ denote an arc. Then $\gamma $
corresponds via $T_{C}$ to a component of the locus $\Gamma $ in
$C_{0}$. The integral over this component of the pull-back of
$(1-3\cos^{2}\theta ) d\varphi -\surd 6 \cos\theta  dt$ is the
value on $\gamma $ of $C_{0}$'s assigned point in $\Delta _{o}$.

The point assigned $C_{0}$ in any given version of $\mathbb{R}_{o}$ is the
chosen lift from $o$'s version of Choice 3 above.

The point assigned to $C_{0}$ in $\mathbb{R}_{ - }$ is obtained as
follows: Each $e' \in  \hat{E}$ labels a component of $C_{0}-\Gamma $;
keep in mind that each such component has been given a
parametrization. Let $w_{e'}$ denote the function $w$ that appears
in the corresponding version of \eqreft25. Now let $\sigma$ be
any value that is taken by $\theta $ on each $e' \in\hat{E}$ labeled component of
$C_{0}- \Gamma $. With $\sigma $
understood, then $C_{0}$'s assigned value in $R_{ - }$ is
\begin{equation}\label{eq6.27}
-\frac{1 }{ {2\pi }}\alpha _{Q_{\hat{E}} } (\sigma )  \sum
_{e' \in \hat{E}}  {\textstyle\int}  w_{e'}(\sigma , v) dv.
\end{equation}

\subsection{Continuity for the map to $O_{T}/\Aut(T)$}\label{sec:6d}

This subsection makes two key assertions. Here is the first
assertion: The assignment to $(C_{0}, \phi )$ of its point in
$\mathbb{R}  \times O_{T}/\Aut(T)$ is independent of the choices
that are made in Part 3 of the preceding subsection.

The second assertion is that the image of $(C_{0}, \phi )$ in
$\mathbb{R}\times O_{T}/\Aut(T)$ factors through
${\mathcal{M}^{*}}_{\hat{A},T}$. Granted that such is the case, the
constructions in the preceding subsection define a map from
${\mathcal{M}^{*}}_{\hat{A},T}$ to $\mathbb{R}  \times O_{T}/\Aut(T)$
and it follows from the first assertion using \fullref{lem:5.4}
that this is a continuous map.

The remainder of this subsection has nine parts that justify these two
assertions.

\step{Part 1}
The first observation concerns Choice 3 from the preceding subsection. A
referral to the conclusions of Cases 2 and 4 of Part 5 from \fullref{sec:2c} finds
that changes in any $o \in   \mathcal{V} $ version of Choice
3 are already invisible in the space that is depicted in \eqref{eq6.15}. A referral
to these same cases in Part 5 of \fullref{sec:2c} finds that a change in
$\diamondsuit $'s version of Choice 3 or a change in Choice 2 moves
$(C_{0}, \phi )$'s assigned point in \eqref{eq6.15} along its orbit under the
$\mathbb{Z}  \times   \mathbb{Z}$ action that defines the quotient in \eqref{eq6.9}.
Note in this regard that any change in $\diamondsuit $'s version of Choice 3
will move $C$'s point in \eqref{eq6.15} along the $\mathbb{Z}\cdot Q_{\hat{E}}$
subgroup in $\mathbb{Z}  \times   \mathbb{Z}$. In any event, a change in Choice
2 or any version of Choice 3 is invisible in $O_{T}$. The following lemma
states this conclusion in a formal fashion:

\begin{lemma}\label{lem:6.5}

The constructions in the preceding subsection
defines a continuous map from ${{\mathcal{M}^{*}}_{\hat{A},T}}^{\Lambda }$ to
$\mathbb{R}  \times O_{T}$.

\end{lemma}

\begin{proof}[Proof of \fullref{lem:6.5}]
The conclusions made just
prior to the lemma assert that the constructions of the preceding
subsection assign a unique point in $\mathbb{R}  \times O_{T}$ to
each triple of the form $(C_{0}, \phi , T_{C})$ where $(C_{0},\phi )$
defines an equivalence class in
${\mathcal{M}^{*}}_{\hat{A},T}$ and $T_{C}$ is a correspondence in
$(C_{0}, \phi )$ for $T$. To prove the lemma, it is enough to prove
that the point assigned this triple depends only on its equivalence
class in ${{\mathcal{M}^{*}}_{\hat{A},T}}^{\Lambda }$.

For this purpose, suppose that $\psi $ is a holomorphic
diffeomorphism of $C_{0}$ and let $\phi ' = \phi  \circ \psi $.
Let ${T_{C}}'$ denote the correspondence that is obtained from $T_{C}$
as follows: When $e \in T$ is an edge and $K_{e}$ its
corresponding component in $C_{0}- \Gamma $ as given by $T_{C}$,
then ${T_{C}}'$ makes $e$ correspond to $\psi ^{- 1}(K_{e})$.
Likewise, if $\gamma    \subset \underline {\Gamma }_{o}$ is an arc,
then it corresponds via $T_{C}$ to an arc, $\gamma _{C}$, in
$C_{0}$. The arc $\gamma $ corresponds via ${T_{C}}'$ to $\psi ^{ -1}(\gamma _{C})$.
This is designed so that the $\phi $ image of any
$T_{C}$ labeled subset of $C_{0}$ is identical to the $\phi '$ image
of the corresponding ${T_{C}}'$ labeled subset.

Granted the preceding conclusion, then $(C_{0}, \phi , T_{C})$
and $(C_{0}, \phi ', {T_{C}}')$ have the same image in the
$\mathbb{R}$ factor of $\mathbb{R}  \times O_{T}/\Aut(T)$ since the
respective images in this factor are defined from the $\phi $ and
$\phi'$ images of respective components with the same edge labels
in the $\phi ^*\theta $ and $\phi'^*\theta $ versions of $C_{0}-\Gamma $.
A similar line of reasoning as applied to the arcs in any
given version $\underline {\Gamma }_{o}$ explains why the images of
$(C_{0}, \phi , T_{C})$ and $(C_{0}, \phi ', {T_{C}}')$ agree
in the corresponding $\Delta _{o}$ factor of $O_{T}$.

To continue, let $e$ now denote the distinguished edge in the $\Aut(T)$
orbit $\hat{E}$ from $\diamondsuit $'s incident edge set. The chosen
canonical parametrization for $K_{e}$ can be pulled back via $\psi $
and this pull-back can be used for the canonical parametrization of
$\psi ^{- 1}(K_{e})$. If this is done, then the other choices from
Part 3 of \fullref{sec:6c} can be made so that the canonical
parametrization for any given $\psi ^{- 1}(K_{{\hat{e}}})$ is
the pull-back via $\psi $ of that for $K_{{\hat{e}}}$. Indeed,
since the $\phi $ image of $K_{{\hat{e}}}$ is the same as the
$\phi '$ image of $\psi ^{- 1}(K_{{\hat{e}}})$, the lifts to
$\mathbb{R}$ that are used to inductively define these
parametrizations can be taken to agree at each stage of the
induction. In particular, doing so guarantees that the respective
assignments to $(C_{0}, \phi , T_{C})$ and $(C_{0}, \phi ',{T_{C}}')$ in $\mathbb{R}_{ - }$
and in each $\mathbb{R}_{o}$ factor
of $\mathbb{R}_{ - }  \times (\times _{o}  \mathbb{R}_{o})$ also agree.

\step{Part 2}
Left yet to discuss is the affect in a change for the correspondence of $T$
in $(C_{0}, \phi )$. Now any change from the original correspondence can
be viewed as the result of using the original correspondence after acting on
 $T$ by an element $\iota    \in  \Aut(T)$. Such is the view taken here. With
this understood, consider:
\end{proof}

\begin{lemma}\label{lem:6.6}

The change in the original correspondence by
the action of $\iota    \in  \Aut(T)$  changes the assigned point in
$\mathbb{R}  \times O_{T}$  by the $\Aut(T)$  action of this same
$\iota $.

\end{lemma}

Note that the first of the two assertion made at the outset of this
subsection is an immediate consequence of \fullref{lem:6.6}. Meanwhile, the second
of these assertions is a direct corollary to Lemmas~\ref{lem:6.5} and~\ref{lem:6.6}.

The remainder of this part and the subsequent parts of this subsection
contain the folowing proof.

\begin{proof}[Proof of \fullref{lem:6.6}]
Consider first the case where $o$ is a vertex in
$\mathcal{V} $ and the automorphism $\iota $ generates the corresponding
$\mathbb{Z}/(n_{o}\mathbb{Z})$ subgroup. The composition of the original
correspondence with $\iota $ does not change the identification between
\eqref{eq6.20} and \eqref{eq6.15}. This said, the new correspondence does not change the
assignment to $(C_{0}, \phi )$ in either $\mathbb{R}_{ - }$ nor in any
$\hat{o} \in  T- T_{o}$ version of $\mathbb{R}_{\hat{o}}\times   \Delta _{\hat{o}}$.

To analyze the change of $(C_{0}, \phi )$'s assigned point in the
remaining factors $\mathbb{R}_{\hat{o}}  \times   \Delta _{\hat{o}}$,
let $\hat{o}$ denote a vertex in $T_{o}$ and let $\gamma $
denote an arc in $\underline {\Gamma }_{\hat{o}}$. The original
correspondence identified $(\hat{o}, \gamma )$ with a component of the
complement of the critical point set in the $\theta =\theta_{\hat{o}}$
locus in $C_{0}$. The new one identifies the latter
component now with the arc $\iota (\gamma )$ in $\underline {\Gamma}_{\iota (\hat{o})}$.
Thus, the new map for $C_{0}$ in $\Delta_{\iota (\hat{o})}$ gives the value on
$\iota (\gamma )$ that the original map for $C$ in $\Delta _{\hat{o}}$ gives $\gamma $. This is
what \eqref{eq6.21} asserts.
\end{proof}

\step{Part 3}
Consider now the change in the assignment to $\mathbb{R}_{o}$. Let $e$ denote
the edge that connects $o$ to $T- T_{o}$. Choices 2 and 3 from Part 3
of the preceding subsection can be made for the new correspondence so that
the original parametrization of $K_{e}$ is unchanged. The original
correspondence identifies the distinguished vertex $\upsilon _{o}$ with a
particular missing or singular point on the $\sigma =\theta _{o}$
circle of the parametrizing cylinder for $K_{e}$. Let $\upsilon $ denote the
latter point and let $\upsilon'$ denote the missing or singular point on
this same circle that is identified by the original correspondence with
$\iota ^{- 1}(\upsilon _{o})$. The new correspondence identifes
$\upsilon '$ with $\upsilon _{o}$. As a consequence, the $o$ version of
Choice 3 for the new correspondence can be taken to be a lift of the
$\mathbb{R}/(2\pi \mathbb{Z})$ coordinate of $\upsilon'$. This is what \eqref{eq6.22}
asserts.

\step{Part 4}
Turn next to the change in some $\mathbb{R}_{\hat{o}}$ in the case
that $\hat{o}$ is a vertex in $T_{o}-o$. Here, there are two cases to
consider. The first case is that where $\hat{o}$ is in a component of
$T_{o}-o$ that is fixed by $\iota $. Thus, the component connects to
 $o$ through an $\iota $--invariant edge in $E_{o}$. Let $e'$ denote the latter.
The change in the assignment to $\mathbb{R}_{\hat{o}}$ is determined up
to Choice 3 modifications by the change in the canonical parametrization of
$K_{e'}$. Indeed, if the parametrization changes due to the action of
some integer pair $N$, then the conclusions in Part 5 of \fullref{sec:2c} imply that
the new assignments to the versions of $\mathbb{R}_{\hat{o}}$ that
connect to $o$ through $e'$ can be made so that each differs from the original
by the addition of the number depicted in \eqref{eq6.8}.

Granted the preceding, let $\{\nu _{1}, \ldots, \nu _{N}\}$
denote the concatenating path set that connects $\upsilon _{o}$ to a
vertex in $\ell _{oe'}$. The new parametrization can be viewed as
one that is obtained via the parametrizing algorithm of
\fullref{sec:2c} by using the original parametrization of
$K_{e}$, the original $\mathbb{R}$ lift of the $\mathbb{R}/(2\pi\mathbb{Z})$
coordinate of the point $\upsilon $ on the $\sigma=\theta _{o}$
circle in the parametrizing domain, but a different
concatenating path set from $\upsilon _{o}$ to a vertex on $\ell_{oe'}$. To explain,
let $\nu ^{\iota }$ denote the path in $\ell_{o}$ that starts at $\upsilon _{o}$ and runs opposite the
orientation to $\iota ^{- 1}(\upsilon _{o})$. The new concatenating
path set is $\{\iota ^{- 1}(\nu _{1}) \circ \nu ^{\iota }, \iota^{- 1}(\nu _{2}),
\ldots,\iota ^{- 1}(\nu _{N})\}$. Let $\gamma_{N}   \subset   \nu _{N}$ denote the final arc.
Thus, $\gamma _{N}\subset   \ell _{oe'}$. Let $\nu '  \subset   \ell _{oe'}$ denote
the oriented path in $\ell _{oe'}$ that starts with $\iota ^{-1}(\gamma _{N})$
and then runs from its ending vertex in the
oriented direction on $\ell _{oe'}$ to the end vertex of $\gamma_{N}$,
then finishes by traversing $\gamma _{N}$ against its
orientation to its start vertex. The ordered set $\{\iota ^{- 1}(\nu_{1}) \circ \nu ^{\iota },
\iota ^{- 1}(\nu _{2}), \ldots, \iota ^{- 1}(\nu _{N}), \nu ', {\nu _{N}^{- 1}}, \ldots, {\nu_{1}}^{ - 1}\}$ is a concatenating path set that starts at $\upsilon_{o}$
and has final arc ending at $\upsilon _{o}$. Let $\mu $ denote
the oriented loop in $\underline {\Gamma }_{o}$ that is obtained by
traversing the constituent paths from left most to right most in
this ordered set, taking into account that the the final arc of any
one is the initial arc of the next. This loop has a canonical lift,
$\mu ^{*}$, to a loop in ${\underline {\Gamma }^{*}}_{o}$. According
to \fullref{lem:2.3}, the new parametrization of $K_{e'}$ is
obtained from the old by the action of the integer pair
$-\phi_{o}([\mu ^{*}])$ where $[\mu ^{*}]$ denotes the homology class of
the oriented loop $\mu ^{*}$ and $\phi _{o}$ denotes the canonically
defined class in $H^{1}({\underline {\Gamma }^{*}}_{o}; \mathbb{Z}
\times   \mathbb{Z})$.

The computation of $\phi _{o}([\mu _{*}])$ relies on the fact
that $\phi _{o}$ is $\mathbb{Z}/(n_{o}\mathbb{Z})$ invariant. Thus, as
$\iota $ is the generator,
\begin{equation}\label{eq6.28}
\phi _{o}([\mu _{*}]) = \frac{1 }{ {n_o }}\phi
_{o}\Bigl([\mu _{*}] + [(\iota (\mu ))_{*}] + \cdots + \bigl[\bigl(\iota ^{n_o - 1}(\mu )\bigr)_{*}\bigr]\Bigr).
\end{equation}
This understood, the sum of homology classes that appears here is the class
that is obtained by traversing $\ell _{oe'}$ once in its oriented
direction while traversing $\ell _{oe}$ once in the direction that is
opposite to its orientation. Thus,
\begin{equation}\label{eq6.29}
\phi _{o}([\mu _{*}]) = \frac{1 }{ {n_o }}(Q_{e'} - Q_{e}) .
\end{equation}
With the preceding understood, the resulting change in the assignment to
$(C_{0}, \phi )$ in $\mathbb{R}_{\hat{o}}$ is that given by \eqref{eq6.23}.

\step{Part 5}
This part considers the case that $\hat{o} \in T_{o}-o$ is a
vertex that lies in a component that is not fixed by the
$\mathbb{Z}/(n_{o}\mathbb{Z})$ action. Let $e' \in  E_{o}$
again denote the edge that connects $\hat{o}$'s component of $T_{o}-o$ to $o$.
Let $E \subset E_{o}$ denote the orbit of $e'$ for the
$\mathbb{Z}/(n_{o}\mathbb{Z})$ action. Let $\hat{e}$ denote the edge
in $T$ that connects $\hat{o}$ to $T- T_{\hat{o}}$. The
component of $C_{0}- \Gamma $ that is labeled by the new
correspondence by $\iota (\hat{e})$ is the component that the old
one assigned to $\hat{e}$. The parametrization of this component will
change by the action of an integer pair, $N$. This understood, the new
assignment in $\mathbb{R}_{\iota (\hat{o})}$ is thus obtained
from the old assignment in $\mathbb{R}_{\hat{o}}$ by adding
the term in \eqref{eq6.8}.

The key point now is that the integer pair $N$ is the same for all vertices in
$\hat{o}$'s component of $T_{o}-o$. In particular, the integer pair $N$ is
the pair that changes the parameterization of the component that is labeled
by $e'$ using the old parametrization. Let $K'$ denote the latter. To explore
the parametrization change, let $\{\nu _{1}, \ldots, \nu _{N}\}$
denote the chosen concatenating path set that is used originally to give the
canonical parametrization to $K'$. Now, there are two cases to consider. In
the first, $e'$ is the distinguished edge in its orbit. In this case, $\iota(e')$
is the first edge and the concatenating path that gives its new
orientation is $\{\nu _{1} \circ \nu , \nu _{2}, \ldots, \nu_{N}\}$ where
$\nu $ is the path that circumnavigates $\ell _{o}$ in
the direction opposite to its orientation. According to the conclusions from
Part 5 of \fullref{sec:2c}, this then implies that $N= Q_{e}$. In the second
case, $e'$ is not the distinguished edge. In this case, the same concatenating
path set as the original gives the new parametrization of $K'$ and so $N = 0$.
Thus, the affect of $\iota $ on the $\{\mathbb{R}_{\hat{o}}\}$
assignments are as depicted by \eqref{eq6.24}.

\step{Part 6}
This and Parts 7--9 concern the effect on $(C_{0}, \phi )$'s assigned
point in\break $O_{T}/\Aut(T)$ when the change in the correspondence is obtained
using an element in \eqref{eq6.13}'s $\Aut _{\diamondsuit }$ subgroup. Let $\iota $
denote such an element. The first point to consider is the affect of $\iota$
on the assignment to the simplex $\Delta _{\diamondsuit }$. Let $r$ denote
the original assignment and $r'$ the new one. If the original correspondence
identifies a given component of the $T_{C}$ version of $\Gamma_{\diamondsuit }$
with a given arc $\gamma $ in the $T$ version of
$\underline {\Gamma }_{\diamondsuit }$, then the new correspondence
identifies this same component with $\iota (\gamma )$. But this means that
$r'$ has the same value on $\iota (\gamma )$ as $r$ does on $\gamma $. Thus,
$r' = \iota (r)$.

This line of reasoning finds the analogous formula for the change in
the assignment to $\times _{o \in \mathcal{V} }  \Delta _{o}$. To be
precise, use $r_{o}$ to denote the original assignment to $\Delta_{o}$ and
$r'_{o}$ the new one. Then $r'_{\iota (o)}(\iota (\gamma)) = r_{o}(\gamma )$ for all
$\gamma \in \underline {\Gamma}_{o}$.

\step{Part 7}
The story on the other aspects of $(C_{0}, \phi )$'s assignment in
$O_{T}/\Aut(T)$ starts here by considering the new assignment for
$(C_{0}, \phi )$ in the space depicted in \eqref{eq6.15}. This assigned point has a partner
in the space depicted in \eqref{eq6.20}. The original assignment for $(C_{0}, \phi)$
also provides a point for $(C_{0}, \phi )$ in the space
depicted in \eqref{eq6.20}. These two points are compared in what follows.

The new assignment in \eqref{eq6.15} requires new versions of Choice 2 and Choice 3
from Part 3 of the preceding subsection. This is because $\iota $ can move
the distinguished edge in $\hat{E}$. To make these new choices, let $e$ denote the
original distinguished edge in $\hat{E}$. The component of $C_{0}- \Gamma $
that corresponds via the new correspondence to $e$ is the
component, denoted here by $K$, that corresponds via the original to $\iota^{- 1}(e)$.
Choice 2 requires a parametrization of $K$.  A convenient
choice is the canonical parametrization as defined by the original choices.

The $\diamondsuit $ version of Choice 3 requires a lift to $\mathbb{R}$ of the
point on the $\sigma =\theta _{\diamondsuit }$ circle of $K$'s
parametrizing cylinder that corresponds via the new correspondence to the
distinguished vertex $\upsilon _{\diamondsuit }$. This is the point that
corresponds to $\iota ^{- 1}(\upsilon _{\diamondsuit })$ via the
original correspondence. A convenient choice is obtained as follows: Let
$\{\nu _{1}, \ldots, \nu _{N}\}$ denote the concatenating path
set that is used to define the original parametrization for $K$.  Note that the
ending vertex of $\nu _{N}$ corresponds via the original correspondence to
a certain point on the $\sigma =\theta _{\diamondsuit }$ circle in the
original parametrizing cylinder of $K$.  Use $z$ to denote this point. The
parametrizing algorithm from \fullref{sec:2c} provides an $\mathbb{R}$ lift of the
$\mathbb{R}/(2\pi \mathbb{Z})$ coordinate of $z$. Subtract from this the
integral of $dv$ in the oriented direction on the $\sigma =\theta_{\diamondsuit }$
circle along the segment that starts at $\iota ^{- 1}(\upsilon _{\diamondsuit })$
and ends at $z$. Let $\tau _{\new}$
denote the resulting real number. Use $\tau _{new}$ for the new version of
$\diamondsuit $'s Choice 3.

\step{Part 8}
This same $\tau _{\new}$ is the new assignment for $(C_{0}, \phi )$ in
the $\mathbb{R}_{\diamondsuit }$ factor in \eqref{eq6.15}. The story on the new
assignments to the other factors requires the preliminary digression that
follows. To start the digression, let $e$ be an incident edge to $\diamondsuit
$, and let $\{{\nu^{e}}_{1}, \ldots, {\nu^{e}}_{N(e)}\}$ denote
the chosen concatenating path set that connects $\upsilon _{\diamondsuit}$
to a vertex on $\ell _{oe}$. Let $\hat{e}$ denote the edge that is mapped
by $\iota $ to the distinguished vertex in $\hat{E}$. The $\hat{e}$ version of this
concatenating path set is denoted at times as in Part 7, thus by
$\{\nu_{1}, \ldots, \nu _{N}\}$.

To continue, let $\nu $ denote the path in $\ell _{o{\hat{e}}}$ that
starts at the beginning of the final arc in $\nu _{N}$ and, after reaching
the end of this arc, then proceeds against the orientation to the point
$\iota ^{- 1}(\upsilon _{\diamondsuit })$. Then both
$\{{\nu^{e}}_{1}, \ldots, {\nu^{e}}_{N(e)}\}$ and
$\{\nu _{1}, \ldots, \nu _{N}, \iota ^{- 1}({\nu^{\iota (e)}}_{1})\circ \nu,
\iota ^{- 1}({\nu ^{\iota (e)}}_{2}), \ldots, \iota ^{- 1}({\nu ^{\iota (e)}}_{N(\iota (e))})\}$ define
concatenating path sets that run from $\upsilon _{\diamondsuit }$ to $\ell_{\diamondsuit e}$.
Let $\nu' $ denote the path in $\ell_{\diamondsuit e}$ that starts with the last arc in
$\iota ^{- 1}({\nu ^{\iota (e)}}_{N(\iota (e))})$, then continues in the oriented
direction from the latter's end to the start of the final arc in $\nu^{e}_{N(e)}$
and then traverses this last arc in reverse. So defined, the set\vrule
width 0pt depth 10pt
\begin{equation}\label{eq6.30}
\bigl\{\nu _{1}, ... , \nu _{N}, \iota ^{- 1}({\nu^{\iota
(e)}}_{1}) \circ \nu , \iota ^{- 1}({\nu ^{\iota (e)}}_{2}), ...
, \iota ^{- 1}({\nu ^{\iota (e)}}_{N(\iota (e))}), \nu ', {{\nu
^{e}}_{N(e)}}^{-1}\!\!, ... ,{{\nu^{e}}_{1}}^{- 1}\bigr\}
\end{equation}
concatenates by gluing the final arc of each constituent path to the first
arc of the subsequent one so has to give a loop that starts and ends at
$\upsilon _{\diamondsuit }$. Let $\mu _{e}$ denote this loop.

As constructed, the loop $\mu _{e}$ as a canonical lift as a
loop in ${\underline {\Gamma}^{*}}_{\diamondsuit }$, this denoted by
${\mu _{e}}^{*}$. Let $N_{e} = (n_{e}, {n_{e}}')$ denote the value
of $-\phi _{\diamondsuit }$ on $[{\mu _{e}}^{*}]$. According to
\fullref{lem:2.3}, the parameterization of the component of
$C_{0}- \Gamma $ that originally corresponded to $e$ is changed with
the new choices by the action of the integer $N_{e}$.

As a consequence of this last conclusion, the new assignment to $(C_{0}, \phi )$ in the
$\mathbb{R}_{ - }$ factor of \eqref{eq6.15} is obtained from the old
by adding $-2\pi   \sum _{e \in \hat{E}} ({n_{e}}'q_{\hat{E}}- n_{e}{q_{\hat{E}}}')$.

Here is a second consequence: The new assignment to $(C_{0}, \phi )$ in
the $\times _{o \in \mathcal{V} } \mathbb{R}_{o}$ factor can be made
so that if $\tau _{o}$ denotes the original assignment in $\mathbb{R}_{o}$
and ${\tau _{o}}'$ the new one, then
\begin{equation}\label{eq6.31}
{\tau _{\iota (o)}}'
=
\tau _{o}- 2\pi
\frac{\alpha_{N_e } (\theta _{\hat{o}} )}{\alpha _{Q_{\hat{e}} } (\theta _{\hat{o}} )} ,
\end{equation}
where $e$ denotes the edge that connects $o$'s component of $T- \diamondsuit $ to $\diamondsuit $.

Note for the future that there is one more path in $\underline{\Gamma }_{\diamondsuit }$
that plays an important role in the
story, this the path that is constructed from the paths in the set
$\{\nu _{1}, \ldots, \nu _{N}, \nu \}$ by identifying the final
arc in any $\nu _{j}$ with the starting arc in the subsequent path.
This path is denoted below as $\mu $. It starts at $\upsilon_{\diamondsuit }$
and ends at $\iota ^{- 1}(\upsilon _{o})$. The
path $\mu$ has a canonical preimage in ${\underline {\Gamma}^{*}}_{\diamondsuit }$.
The latter is denoted in what follows by $\mu ^{*}$.

\step{Part 9}
It remains now to compare the new and old assignments in \eqref{eq6.15} by viewing
the latter as points in \eqref{eq6.20}. There are three steps to this task.

\substep{Step 1}
For this purpose, use $(\tau, r) \in \Maps(\Vertt_{\hat{E}})  \times \Delta _{\diamondsuit }$
to denote $(C_{0}, \phi )$'s original assignment in $\mathbb{R}^{\Delta }$. Use
$(\tau ', r')$ to denote the new assignment. As noted already, $r' = \iota (r)$.
To see about $\tau '$, let $\hat {\upsilon }$ denote the chosen point
in $\bar {\Gamma }^{*}$ over $\upsilon _{\diamondsuit }$. The path
$\mu ^{*}$ has a lift to $\bar {\Gamma }^{*}$ that starts at
$\hat {\upsilon }$ and ends at some vertex $\hat {\upsilon }'$, a vertex
that maps to $\iota ^{- 1}(\upsilon _{\diamondsuit })$. According
to \eqref{eq6.16}, the value of $\tau $ on $\hat {\upsilon }'$ is $\tau _{new}$.
Meanwhile, $\tau _{\new}$ is also equal to the value of $\tau '$ on $\hat{\upsilon }$.
Indeed, such is the case because Part 7 found $\tau _{\new}$
to be $(C_{0}, \phi )$'s new $\mathbb{R}_{\diamondsuit }$ assignment in
\eqref{eq6.15}.

Now, $\iota $ has a unique lift in $\Auth_{\diamondsuit }$ that sends
$\hat {\upsilon }'$ to $\hat {\upsilon }$. Use $\hat {\iota }$ to
denote the latter. The conclusions of the last paragraph assert that
$\tau'(\hat {\upsilon })=\tau _{new}=\tau (\hat {\iota }^{- 1}\hat {\upsilon })$.
Thus, $\tau ' = \hat {\iota }(\tau )$.

\substep{Step 2}
To study the $\mathbb{R}^{ - }$
factors, let $\tau _{ - }$ denote the original assignment and
${\tau_{ - }}'$ the new one. The values for $\tau _{ - }$ and $\tau _{ -}'$
were defined by specifying them on a particular element in the
set $\Lambda $. Denote this element by $\mathcal{L}$. To elaborate, the
component of $L$ with label any given $e \in  \hat{E}$ was obtained
as follows: Construct the path $\nu ^{e}$ in $\underline {\Gamma}_{o}$
from $\{{\nu^{e}}_{1}, \ldots, {\nu^{e}}_{N(e)}\}$ by gluing
the end arc of any $j< N(e)$ version to the starting arc of the
subsequent version. This path has a canonical lift to ${\underline{\Gamma }^{*}}_{o}$
and thus a canonical lift to $\bar {\Gamma}^{*}$ as a path that starts at $\hat {\upsilon }$. The end
vertex of this lift lies on an inverse image of ${\ell^{*}}_{\diamondsuit e}$.
This is the component of $L$ that is labeled by $e$.

The conclusions in Part 7 about the $\mathbb{R}_{ - }$ factors can now be
summarized by the relation $\tau '(L) = \tau (L) - 2\pi   \sum _{e\in \hat{E}} ({n_{e}}'q_{\hat{E}}-
n_{e}{q_{\hat{E}}}')$. As will now be explained, this is also the
value of $\tau $ on $\hat {\iota }^{- 1}(L)$. To see why, let $e \in \hat{E}$,
let $\ell $ denote $e$'s component of $L$ and let $\ell '$ denote the
component for $\iota (e)$. Thus, $\hat {\iota }^{- 1}(\ell ')$ is
in $L_{e}$ and this is $e$'s component of $\hat {\iota }^{- 1}(L)$. As
such, it is obtained from $\ell $ by the action of some element in $\mathbb{Z}\times \mathbb{Z}$.
To find this element, note that as $\ell $ was defined
by the set $\{{\nu^{e}}_{1}, \ldots, {\nu^{e}}_{N(e)}\}$, so
$\hat {\iota }^{- 1}(\ell ')$ is defined by the ordered set
$\{\nu _{1}, \ldots, \nu _{N}, \iota ^{- 1}({\nu^{\iota (e)}}_{1}) \circ \nu ,
\iota ^{- 1}({\nu ^{\iota(e)}}_{2}), \ldots, \iota ^{- 1}({\nu ^{\iota (e)}}_{N(\iota (e))})\}$. This then implies that $\hat {\iota }^{- 1}(\ell ')$
can be obtained from $\ell $ by the action of $N_{e}$.

Granted that such is the case for all $e \in \hat{E}$,  then \eqref{eq6.17} gives
$\tau (\hat {\iota }^{- 1}(L))$ the desired value.

\substep{Step 3}
The final concern is that of the
assignments to any given version of $\hat{U}_{(\cdot )}$. For this purpose,
suppose that $e$ is an incident edge to $\diamondsuit $. The old and new
correspondences each assign $(C_{0}, \phi )$ a map from $L_{e}$ to
$W_{e}= (\times _{o \in \mathcal{V} (e)}  \mathbb{R}_{o})/(\times _{e \in
\mathcal{V} (e)}  \mathbb{Z}_{o})$. These maps are defined by lifts to
$\times_{o \in \mathcal{V} (e)}\mathbb{R}_{o}$. In what follows,
$\tau_{o}(\ell )$ and ${\tau _{o}}'(\ell )$ are used to denote the
respective old and new values on a given $\ell    \in L_{e}$ for the
factor labeled by a given $o \in   \mathcal{V} (e)$.

Now, both maps to $W_{e}$ are determined by their values on a particular
element in $L_{e}$. The latter, $\ell ^{e}$, is determined as follows: As
described previously, the paths that comprise the set
$\{{\nu^{e}}_{1}, \ldots, {\nu^{e}}_{N(e)}\}$ glue together in a
sequential fashion so as to define a path from $\upsilon _{\diamondsuit }$
to a vertex on $\ell _{oe}$. This path then lifts to a path that starts at
$\hat {\upsilon }$ and ends at a vertex on a unique preimage of
$\ell_{oe}$ in $\bar {\Gamma }^{*}$. This preimage is $\ell ^{e}$.

The components of $\Gamma $ that are assigned by the new correspondence to
vertices in $\mathcal{V} (\iota (e))$ are assigned to vertices in $\mathcal{V}(e)$
by the old correspondence. This understood, let $o \in   \mathcal{V}(e)$.
Equation \eqref{eq6.31} describes the relationship between $\tau _{\iota(o)}'(\ell ^{\iota (e)})$
and $\tau _{o}(\ell )$.

As is explained next, \eqref{eq6.31} is also the formula for $\hat {\iota }(x)$ on
$\ell ^{\iota (e)}$. Indeed, by definition,
$\hat {\iota }(x)(\hat{\iota }(\ell ^{e})) = x(\ell ^{e})$. To see what this means, note
that $\hat {\iota }(\ell ^{e})   \in  L_{\iota (e)}$ and so
$\ell ^{\iota (e)}$ is obtained from $\hat {\iota }(\ell ^{e})$ by the
action of some integer pair. As argued at the end of Step 2 in the case when
 $e \in \hat{E}$,  this integer pair is $N_{e}$. This being the case, then
\eqref{eq6.19} implies that $x'$ equal $\hat {\iota }(x)$.

\subsection{More about $\Aut(T)$}\label{sec:6e}

The story told by Theorems~\ref{thm:6.2} and \ref{thm:6.3} simplifies to some extend when
 $\Aut _{\diamondsuit }$ fixes one of $\diamondsuit $'s incident edges. As is
explained below, there is no need in this case to use the space in \eqref{eq6.20}
because the $\Aut(T)$ action is readily visible on the space in \eqref{eq6.9}.

To begin the story in this case, let $e$ denote an incident edge to
$\diamondsuit $ that is fixed by $\Aut(T)$. Agree to use $e$ for the
distinguished incident edge orbit. Since $\Aut _{\diamondsuit }$ fixes $e$, it
must act as a subgroup of the group of automorphisms of the labeled graph
$\ell _{\diamondsuit e}   \equiv   \ell _{\diamondsuit }$. Thus,
 $\Aut _{\diamondsuit }$ is a cyclic group whose order is denoted by
$n_{\diamondsuit }$. This implies that $\Aut(T)$ is isomorphic to the $o=\diamondsuit $
version of \eqref{eq6.14}. Granted these remarks, view
 $\Aut _{\diamondsuit }=\mathbb{Z}/(n_{\diamondsuit }\mathbb{Z})$ as a
subgroup of $\Aut(T)$ using this same version of \eqref{eq6.14}.

Define the action of $\Aut _{\diamondsuit }$ on the space in \eqref{eq6.9} via an
action of $\Auth_{\diamondsuit }$ on the space in \eqref{eq6.15}. In this regard,
note that the group $\Auth_{\diamondsuit }$ is isomorphic now to
\begin{equation}\label{eq6.32}
\biggl[\biggl(\frac{1 }{ {n_\diamondsuit }}\mathbb{Z}\biggr)  \times (\mathbb{Z}  \times   \mathbb{Z})\biggr]/\mathbb{Z},
\end{equation}
where the notation is as follows: The first factor in the brackets arises as
the $\mathbb{Z}$ extension
\begin{equation}\label{eq6.33}
1  \to   \mathbb{Z}\to \frac{1 }{{n_\diamondsuit }}\mathbb{Z}\to
\mathbb{Z}/(n_{\diamondsuit }\mathbb{Z})\to 1,
\end{equation}
where the $\mathbb{Z}$ action is that of the subgroup that covers the identity
in $\mathbb{Z}/(n_{\diamondsuit }\mathbb{Z})$. Meanwhile, the $\mathbb{Z}$ action
on $\mathbb{Z}  \times   \mathbb{Z}$ is the subgroup $\mathbb{Z}\cdot Q_{e}$
whereby  $1 \in   \mathbb{Z}$ acts as $-Q_{e}$. This understood, the
equivalence class in \eqref{eq6.32} of a pair $(z, N= (n, n'))$ with
$z\cdot n_{\diamondsuit }   \in   \mathbb{Z}$ and $N \in   \mathbb{Z}  \times \mathbb{Z}$
acts on the factor $\mathbb{R}_{ - }$ in \eqref{eq6.15} as the translation
by $-2\pi (n'q_{e}- n{q_{e}}')$. Meanwhile, it acts on $\mathbb{R}_{\diamondsuit }$ as the
translation by the $\hat{o} = \diamondsuit $ version of
\begin{equation}\label{eq6.34}
-2\pi \frac{\alpha _{N + zQ_e } (\theta _{\hat{o}} )}{\alpha _{Q_{\hat{e}} } (\theta _{\hat{o}} )}.
\end{equation}
The action of $\iota  = [z, N]$ on the factors $\times _{o}  \Delta_{o}$
is such that if $r \in   \Delta _{o}$ and $\gamma $ is an arc in
$\underline {\Gamma }_{o}$, then the resulting $\iota (r)(\iota(\gamma ))$ equals
$r(\gamma )$; here, $\iota (\gamma )$ is the image of
$\gamma $ under the induced map from the set of arcs in $\underline {\Gamma}_{o}$
to the set in $\underline {\Gamma }_{\iota (o)}$.

The description of the action on $(\times _{o \in \mathcal{V} }\mathbb{R}_{o})/
(\times _{o \in \mathcal{V} }  \mathbb{Z}_{o})$
requires distinguishing two separate cases. For this purpose, keep
in mind the following: When $E \subset E_{\diamondsuit }$ is an
 $\Aut _{\diamondsuit }$ orbit, then $\Auth_{\diamondsuit }$
preserves the factor $\times _{e' \in E}[(\times _{o \in
\mathcal{V} (e')}  \mathbb{R}_{o})/(\times _{o \in \mathcal{V} (e')}\mathbb{Z}_{o})]$.
Up for discussion first is the case where $E$ is
a single edge. In this regard, note that $\mathcal{V} (e) = ${\o} so
there is at most one such $E$ unless $\Aut _{\diamondsuit }$ is
trivial. Let $e'$ denote the fixed edge. Then the action of $(z, N)
\in \Auth_{\diamondsuit }$ comes from the following action on
$\times _{o \in \mathcal{V} (e')}  \mathbb{R}_{o}$: Let $\tau $ denote
a point in this product and let $\hat{o} \in \mathcal{V} (e')$. Then
the value of $\iota (\tau )$ in the factor labeled by any
given $\iota (\hat{o})$ is obtained by subtracting
\begin{equation}\label{eq6.35}
\frac{2\pi }{\alpha _{Q_{\hat{e}} } (\theta _{\hat{o}} )}\biggl[\frac{1 }{
{n_\diamondsuit }}  \alpha _{Q_e - Q_{e' } } (\theta
_{\hat{o}})+\alpha _{N + zQ_e } (\theta _{\hat{o}})\biggr]
\end{equation}
from the value of $\tau $ in the factor labeled by $\hat{o}$. Here, $\hat{e}$ is the
edge that connects $\hat{o}$ to $T- T_{\hat{o}}$.

If $E \subset E_{\diamondsuit }$ is a non-trivial orbit of
 $\Aut _{\diamondsuit }$, then $E$ has $n_{\diamondsuit }$ edges and
they have a canonical cyclic ordering. The definition of the
$\Auth_{\diamondsuit }$ action on $\times _{e' \in E}[(\times _{o \in \mathcal{V} (e')}
\mathbb{R}_{o})/$\linebreak
$(\times _{o \in \mathcal{V} (e')}\mathbb{Z}_{o})]$ requires the choice of a
distinguished edge in $E$ and the definition of a compatible linear
ordering with the distinguished edge last. Granted such a choice,
define the action of $(z, N) \in \Auth_{\diamondsuit }$ from
the following action on $\times _{o \in \mathcal{V} (e')}\mathbb{R}_{o}$:
Let $\tau $ denote a point in this product and let
$\hat{o} \in   \mathcal{V} (e')$. Then the value of $\iota (\tau )$ in the factor labeled
by any given $\iota (\hat{o})$ is obtained by subtracting
\begin{equation}\label{eq6.36}
\frac{2\pi }{\alpha _{Q_{\hat{e}} } (\theta _{\hat{o}} )}\biggl[\varepsilon
_{\hat{o}}  \alpha _{Q_e } (\theta _{\hat{o}})+\alpha
_{N + zQ_e } (\theta _{\hat{o}})\biggr]
\end{equation}
from the value of $\tau $ on $\hat{o}$. Here, $\varepsilon _{\hat{o}} = 0$
unless the edge that connects $\hat{o}$'s component of $T_{\diamondsuit}- \diamondsuit $ to
$\diamondsuit $ is the distinguished edge in
$E$. In this case, $\varepsilon _{\hat{o}} = 1$.

The identification given above between \eqref{eq6.15} and \eqref{eq6.20} intertwines these
$\Auth(T)$ actions if the various concatenating path sets are chosen in an
appropriate fashion. This understood, the version given here of the $\Aut(T)$
action on $O_{T}$ can be used for $\hat{O}_{T}/\Aut(T)$ in the statements of
Theorems~\ref{thm:6.2} and \ref{thm:6.3}.

With regards to $O_{T}-\hat{O}_{T}$, more can be said here than is
stated in \fullref{prop:6.4}. For this purpose, various
notions must be introduced. The first is the integer, $m_{ - }$,
this the least common divisor of the integers that comprise the pair
$Q_{e}$. Next, let $T^{F}   \subset T$ denote the subgraph on which
$\Aut(T)$ acts trivially. Thus, the edge $e$ is in $T^{F}$, but $T^{F}$
may well be bigger. In any event, $T^{F}$ is connected. Let $k_{T}$
denote the greatest common divisor of the integers in the set that
consists of $m_{ - }$ and the versions of $n_{o}$ for vertices $o \in  T^{F}$.

The next notion is that of a `canonical diagonal subgroup' of
$\Aut(T)$. To define this notion, note that if $o \in  T^{F}$, then
$\mathbb{Z}/(n_{o}\mathbb{Z})$ has a unique
$\mathbb{Z}/(k_{T}\mathbb{Z})$ subgroup and the latter has a
canonical generator; this is the element that is the smallest
multiple of the generator of $\mathbb{Z}/(n_{o}\mathbb{Z})$.
Granted the preceding, a subgroup of $\Aut(T)$ is a canonical diagonal
subgroup when it has the following two properties: First, it is
cyclic of order $k_{T}$. Second, it has a generator that maps to the
canonical generator of each $\mathbb{Z}/(k_{T}\mathbb{Z})$
subgroup of each $o \in  T^{F}$ version of
$\mathbb{Z}/(n_{o}\mathbb{Z})$. According to
\fullref{prop:6.4}, any two canonical diagonal subgroups are
conjugate in $\Aut(T)$.

With the introduction now over, consider:

\begin{proposition} \label{prop:6.7}

The stabilizer in $\Aut(T)$  of any given point in $O_{T}$  is a subgroup of some
canonical diagonal subgroup. Conversely, if $G$  is a subgroup of a canonical diagonal
subgroup of $\Aut(T)$,  then the fixed point set of $G$  is the product of its
corresponding fixed point set in $\times _{o}  \Delta_{o }$ and a product of circles,
one that corresponds to the factor $\mathbb{R}_{ - }$  and the rest labeled in a
canonical fashion by the various orbits of $G$  in $T$'s  multivalent vertex set.

\end{proposition}

The proof of this proposition uses arguments from the proof of \fullref{prop:6.4}
and so the latter proof is offered first.

\begin{proof}[Proof of \fullref{prop:6.4}]
To prove the first assertion, suppose for
argument's sake that an element $\iota    \in  \Aut(T)$ maps to the identity
in $\Aut _{\diamondsuit }$ yet fixes a given point in $O_{T}$. This means that
$\iota $ is determined by its components in the versions of $\Aut(T_{(\cdot)})$
that are labeled by the vertices that share an edge with $\diamondsuit$.
Thus, it is sufficient to consider the case where $\iota    \in\Aut(T_{o})$
for some $o \in  T- \diamondsuit $. As such, it is
permissable to view $O_{T}$ as in \eqref{eq6.9}. Granted this, then the arguments
given in \fullref{sec:3c} to prove Propositions~\ref{prop:4.4} and~\ref{prop:4.5}
can be borrowed almost verbatim to prove that $\iota $ is the identity element.

The proof of the last part of the proposition uses an induction argument
that moves from any given $o \in T$ to the vertices in $T_{o}- o$ that share its incident edges.
The phrasing of the induction step
uses the notion of the `generation number' of a multivalent vertex in $T$.
Here, $\diamondsuit $ is the only generation 0 vertex, and a vertex $o$ has
generation $k > 0$ when it shares an edge with a generation $k-1$ vertex in
$T- T_{o}$. Use $\Vertt(k)$ to denote the set of generation $k$ vertices.
\end{proof}
The lemma that follows facilitates the induction.

\begin{lemma}\label{lem:6.8}
Suppose that $k > 0$  and that $\iota $  and
$\iota '$  both stabilize points in $O_{T}$  and have the same
image in $\Aut(T)/(\times _{o \in \Vertt(k)}$ $\Aut(T_{o}))$.
Then, there exists some\break $j \in  \times _{o \in \Vertt(k + 1)}
\Aut(T_{o})$  such that $j\iota j^{- 1}$  and $\iota' $
have the same image in\break $\Aut(T)/(\times _{o \in \Vertt(k + 1)}  \Aut(T_{o}))$.

\end{lemma}

This lemma is proved below. The actual induction step is made using the
following generalization:

\begin{lemma}\label{lem:6.9}

Suppose that $k > 0$  and suppose that $G$  and
$G'$  are subgroups of $\Aut(T)$  that stabilize points in $O_{T}$  and
have the same image in $\Aut(T)/(\times _{o \in \Vertt(k)}$ $\Aut(T_{o}))$.
Then, there exists some $j  \in   \times _{o \in\Vertt(k + 1)} \Aut(T_{o})$  such that
$jGj^{- 1}$  and $G'$  have the
same image in $\Aut(T)/(\times _{o \in \Vertt(k + 1)}\Aut(T_{o}))$.
\end{lemma}

\begin{proof}[Proof of \fullref{lem:6.9}]
By appeal to \fullref{lem:6.8} there are non-trivial subgroups in $G$ and in $G'$
that have the same image in $\Aut(T)/(\times _{o \in \Vertt(k +1)\_}\Aut(T_{o}))$.
Take $H \subset G$ and $H'  \subset  G'$ to be
a maximal group of this sort in the following sense: There is no
subgroup of $G$ that properly contains $H$ and has a partner in $G'$
with the same image in $\Aut(T)/(\times _{o \in \Vertt(k +1)\_}\Aut(T_{o}))$.
The argument that follows deriving nonsense if $H  \ne G$.

To start this derivation, suppose that $\iota    \in  G- H$ and let
$\iota '  \in  G'$ denote the element that shares $\iota $'s image
in $\Aut(T)/(\times _{o \in \Vertt(k)} \Aut(T_{o}))$. A second appeal
to \fullref{lem:6.8} finds some $j \in   \times _{o \in \Vertt(k)} \Aut(T_{o})$ such that
$j\iota j^{- 1}=\iota '$. Of
concern is whether the commutators of $j$ with the elements in $H$ are
all in $\times _{o \in \Vertt(k + 1)}\Aut(T_{o})$. If such is the
case, then the group generated by $H$ and $j\iota j^{- 1}$ contains $H$
as a proper subgroup and has the same image in $\Aut(T)/(\times _{o\in \Vertt(k + 1)}\Aut(T_{o}))$
as the group generated by $H$ and
$\iota '$. Of course, this contradicts the assumed maximality of $H$.
In any event, $j$ acts in $\mathbb{Z}/(n_{o}\mathbb{Z})$ as the
translation by some integer, here denoted by $j_{o}$. Thus, $j_{o}\in  \{0, \ldots , n_{o}-1\}$.
Note that the commutators of $j$
with the elements of $H$ all lie in $\times _{o \in \Vertt(k + 1)}\Aut(T_{o})$
if and only if the assignment $o \to j_{o}$ is
constant on $H$ orbits in $\Vertt(k)$.

To see that such must be the case, note first that $\iota $ and $\iota '$
have identical actions on $\Vertt(k)$ since their actions here are determined by
their actions on the set of edges that are incident to the vertices in
$\Vertt(k-1)$. Thus if $\iota $ fixes $o \in  \Vertt(k)$, then so does $\iota '$.
In this case, there is nothing lost by taking $j_{o} = 0$. Suppose next that
$\iota $ does not fix $o$. Now, $\iota $ maps $\ell _{o}$ to $\ell _{\iota(o)}$ as does
$\iota'$, and the affect of the latter is that of $j_{\iota(o)}\iota {j_{o}}^{- 1}$.

Keep this last observation on hold for the moment for the following key
oberservations: First, if $u \in \Aut(T)$, then $n_{o}= n_{u(o)}$.
Second, $u$ induces a map from the vertex set on $\ell _{o}$ to that on
$\ell _{u(o)}$ that respects the cyclic ordering. Third, this map
intertwines the action of $i \in   \mathbb{Z}/(n_{o}\mathbb{Z})$ with its
corresponding action in $\mathbb{Z}/(n_{u(o)}\mathbb{Z})$. Written
prosaically, $u\cdot i = i\cdot u$.

Here now are the salient implications of these last observations:
Suppose that $o \in  \Vertt(k)$ and that $h^{- 1}\iota (o) = o$ for
some  $h \in  H$. By appeal to \fullref{lem:6.8}, there exists an
element $i_{o}   \in \Aut(T_{o})$ such that $i_{o}(h^{- 1}\iota)i_{o}^{- 1} = h^{- 1}\iota '$ in
$\mathbb{Z}/(n_{o}\mathbb{Z})$ and so an appeal to the conclusions
of the preceding paragraph finds that $h^{- 1}\iota  = h^{- 1}\iota'$ in
$\mathbb{Z}/(n_{o}\mathbb{Z})$. Meanwhile, the conclusions
from the preceding two paragraphs imply that $h^{- 1}\iota ' =(j_{h(o)}{j_{o}}^{- 1}) h^{- 1}\iota $ in
$\mathbb{Z}/(n_{o}\mathbb{Z})$. Thus, $j_{h(o) } = j_{o}$ and so
the assignment $o \to j_{o}$ is constant on the $H$--orbits in
$\Vertt(k)$. This then means that $jhj^{- 1}h^{- 1}   \subset   \times
_{o \in \Vertt(k + 1)}\Aut(T_{o})$ for all  $h \in  H$.

The proof of \fullref{lem:6.8} requires some knowledge of the conditions that allow
a given $\iota    \in  \Aut(T)$ to have a fixed point in $O_{T}$. What
follows is a digression in six parts to describe both necessary and
sufficient conditions on $\iota $.
\end{proof}

\step{Part 1}
The automorphism $\iota $ has a lift in $\Auth(T)$ with a fixed point in
the space that is depicted in \eqref{eq6.20}. Let $\iota $ also
denote this lift and let $b$ denote a fixed point for this lifted
version of $\iota $, thus a point in \eqref{eq6.20}. To determine
necessary and sufficient conditions for $\iota $ to fix $b$, it proves
useful to make the choices that are described in Parts 1 and 2 of
\fullref{sec:6c} so as to view $b$ as a point in the space that
is depicted in \eqref{eq6.15}. In this incarnation, $b$ is a tuple
$(\tau _{ - }, (\tau _{\diamondsuit }, r_{\diamondsuit }),
(r_{o}, \tau _{o})_{o \in \mathcal{V} })$ where $\tau _{ - }$ and
$\tau _{\diamondsuit }$ are real numbers, each $r_{o}$ is in the
corresponding simplex $\Delta _{o}$ and $(\tau _{o})_{o \in
\mathcal{V} }   \in (\times _{o}  \mathbb{R}_{o})/(\times _{o}
\mathbb{Z}_{o})$.

What follows summarizes how the choices from Parts 1 and 2 of
\fullref{sec:6c} identify \eqref{eq6.20} with \eqref{eq6.15}. To
start, note that the $\mathbb{R}^{\Delta }$ factor of \eqref{eq6.20}
is a pair whose second component is $r_{\diamondsuit }$ and whose
first is a map from $\Vertt_{\hat{E}}$ to $\mathbb{R}$. The
value of this map on the vertex $\hat {\upsilon }$ from Part 2 in
\fullref{sec:6c} gives the $\mathbb{R}_{\diamondsuit }$
factor in \eqref{eq6.15}. The correspondence between the remaining
factors in \eqref{eq6.20} and \eqref{eq6.15} uses an assigned
component of the inverse image in $\bar {\Gamma }^{*}$ of each ${\ell^{*}}_{\diamondsuit (\cdot )}$.
Use $\ell ^{e}$ to denote the
component that is assigned to an edge $e \in E_{\diamondsuit }$.
The value on $\ell ^{e}$ of the $U_{e}$ factor
in \eqref{eq6.20} gives the factor $\times _{o \in \mathcal{V} (e)}
\mathbb{R}_{o}$ in \eqref{eq6.15}. Meanwhile, those $\{\ell
^{e}\}_{e \in \hat{E}}$ that are indexed by the edges in
$\hat{E}$ provide a canonical element in $\Lambda $ and the values of
\eqref{eq6.20}'s $\mathbb{R}^{ - }$ factor on the latter provides
the $\mathbb{R}_{ - }$ factor in \eqref{eq6.15}.

The definition of $\ell ^{e}$ is that used in Part 8 of \fullref{sec:6d}. In
brief, the concatenating path from $\upsilon _{\diamondsuit }$ to $\ell_{\diamondsuit e}$
provides a unique path in $\bar {\Gamma }^{*}$
from $\hat {\upsilon }$ to a vertex that projects to ${\ell ^{*}}_{\diamondsuit e}$.
The latter vertex sits on $\ell ^{e}$.

\step{Part 2}
The most straightforward aspect of the fixed point condition concerns the
collection $(r_{o})_{o \in \mathcal{V} \cup \diamondsuit }$. In
particular, if $\iota $ fixes $b$, then
\begin{equation}\label{eq6.37}
r_{\iota (o)}(\iota (\gamma )) = r_{o}(\gamma )
\end{equation}
for each $o \in   \mathcal{V}    \cup   \diamondsuit $
and for each arc $\gamma    \subset \underline {\Gamma }_{o}$. This
is to say that $\iota $ fixes $b$'s image in $\times _{o}  \Delta _{o}$.

\step{Part 3}
Consider next the condition that $\iota (\tau _{\diamondsuit }) =\tau _{\diamondsuit }$.
For this purpose, construct a path in
$\underline {\Gamma }_{\diamondsuit }$ that starts at $\upsilon_{\diamondsuit }$ and ends
at $\iota ^{- 1}(\upsilon_{\diamondsuit })$ as follows: Let $e$ denote the distinguished
vertex in $E_{\diamondsuit }$ and let $\hat{e} \equiv   \iota ^{ -1}(e)$.
Introduce $\{\nu _{1}, \ldots, \nu _{N}\}$ to denote the
chosen concatenating path set for the edge $\hat{e}$. Thus, the last
vertex on $\nu _{N}$ lies on $\ell _{\diamondsuit {\hat{e}}}$.
Let $\nu $ denote the path that starts with the final arc in $\nu_{N}$
as it is traversed while traveling $\nu _{N}$; after running
along this arc, $\nu $ then proceeds in the oriented
direction along $\ell _{\diamondsuit {\hat{e}}}$ to $\iota ^{-1}(\upsilon _{\diamondsuit })$.
The concatenating path set $\{\nu_{1}, \ldots, \nu _{N}, \nu \}$ defines a unique path in
$\bar{\Gamma }^{*}$ that starts at $\hat {\upsilon }$ and ends at a
vertex on $\ell ^{{\hat{e}}}$ that projects to $\iota ^{-1}(\upsilon _{\diamondsuit })$.
Let $\hat {\upsilon }_{\iota }$
denote this ending vertex. Meanwhile, let $\mu    \subset \underline{\Gamma }_{\diamondsuit }$
denote the path that is obtained from
$\{\nu _{1}, \ldots, \nu _{N}, \nu \}$ by identifying the final
arc in each $\nu _{k}$ with the initial arc in the subsequent path.
Note that $\mu $ is the projection to $\underline {\Gamma}_{\diamondsuit }$ of the
path in $\bar {\Gamma }^{*}$ that was used
to define $\hat {\upsilon }_{\iota }$.

The vertex $\iota ^{- 1}(\hat {\upsilon })$ is some $\mathbb{Z}\times   \mathbb{Z}$
translate of $\hat {\upsilon }_{\iota }$, thus, $N\hat{\upsilon }_{\iota }$ with
$N \in   \mathbb{Z}  \times   \mathbb{Z}$.
Granted all of this, the condition for $\iota $ to fix $\tau_{\diamondsuit }$ is:
\begin{equation}\label{eq6.38}
\sum _{\gamma \subset \mu  }\pm r_{\diamondsuit }(\gamma ) =
\alpha _{N}(\theta _{\diamondsuit }),
\end{equation}
where the sum is over the arcs in $\mu $, and where the $+$
sign is used if and only if the arc is crossed in its oriented
direction. Note that this condition concerns only $r_{\diamondsuit}$ and
$\iota $'s image in $\Auth_{\diamondsuit }$. In
particular, it says nothing about the value of $\tau _{\diamondsuit}$.

\step{Part 4}
The conditions that $\iota $ must satisfy to fix the rest of $b$ involve an
integer pair that is defined for each of $\diamondsuit $'s incident
edges. This integer pair is defined modulo $\mathbb{Z}\cdot Q_{e}$
as follows: If $e \in  E_{\diamondsuit }$, then $\iota(\ell ^{e})$ is a component
of the inverse image of $\ell^{*}_{\diamondsuit \iota (e)}$ and so is equal to a $\mathbb{Z}
\times   \mathbb{Z}$ translate of $\ell ^{\iota (e)}$. Let $R_{e} =(r_{e}, r_{e}')$ denote such a
pair. The condition $\iota (\tau_{ - })=\tau _{ - }$ involves the collection $\{R_{e}\}_{e \in\hat{E}}$:
\begin{equation}\label{eq6.39}
\sum _{e \in \hat{E}} \bigl(r_{e}'q_{\hat{E}} -r_{e}{q_{\hat{E}}}'\bigr) = 0.
\end{equation}
Thus, this condition concerns only $\iota $'s image in
$\Auth_{\diamondsuit }$; it says nothing about $\tau _{ - }$.

By the way, a particular version of $R_{e}$ is $-(N_{e} + N)$,
where $N$ is the integer pair defined in Part 3 and where $N_{e}$ is
defined as in Part 8 of \fullref{sec:6d}.

\step{Part 5}
Granted that $\iota $ fixes $\tau _{ - }$, $\tau _{\diamondsuit }$ and
$(r_{o})$, consider next the conditions for $\iota $'s fixing of
$(\tau _{o})_{o \in \mathcal{V} }$. In this regard, there are two
cases to consider, the first where $\iota $ fixes a given vertex
$\hat{o} \in \mathcal{V} $ and the second where $\iota (\hat{o}) \ne \hat{o}$.
This part tells the story in the case that $\iota $
fixes $\hat{o}$. For this purpose, keep in mind the following feature
of $T$: If $\hat{o}$ is fixed by $\iota $, then so is $o$ if $\hat{o} \in T_{o}$.
With the preceding, let $\iota _{\hat{o}}   \in \{0, \ldots , n_{\hat{o}}-1\} =
\mathbb{Z}/(n_{o}\mathbb{Z})$ denote the image of $\iota $. Then
$\hat{o}$'s version of \eqref{eq6.37} requires that $\iota $ acts so that
\begin{equation}\label{eq6.40}
\tau _{\hat{o}}   \to \tau _{\hat{o}} -
\frac{2\pi }{\alpha _{Q_{e(\hat{o})} } (\theta _{\hat{o}} )} \biggl(\frac{{\iota
_{\hat{o}} } }{ {n_{\hat{o}} }}\alpha _{Q_{e(\hat{o})} } (\theta
_{\hat{o}})+\sum _{o}  \frac{{\iota _o } }{ {n_o
}}\alpha _{Q_{e(o)} - Q_{e' (o)} } (\theta _{\hat{o}}) +
\alpha _{R_e } (\theta _{\hat{o}})\biggr),
\end{equation}
where the vertex labeled sum involves only vertices $o \in   \mathcal{V}- \hat{o}$ where $T_{o}$
contains $\hat{o}$. To explain the notation, $e(o)$
and $e'(o)$ are both incident edges to $o$, the former connecting $o$ to $T- T_{o}$
and the latter connecting $\hat{o}$'s component of $T_{o}- 0$ to $o$. Finally, $e$ is the incident
edge to $\diamondsuit $ that
connects $\hat{o}$'s component of $T- \diamondsuit $ to $\diamondsuit $.

There is a convenient way to rewrite \eqref{eq6.40} that uses the following
observation about $T$: Each vertex in $T- \diamondsuit $ is a
monovalent vertex in a unique subgraph of $T$ whose second monovalent vertex
is $\diamondsuit $. Moreover, if the given vertex is in generation $k$, then
there are $k$ vertices on this graph and their generation numbers increase by
1 as they are successively passed as the graph is traversed from
$\diamondsuit $. The interior vertices in the $\hat{o}$ version of this graph
are precisely the vertices that appear in the sum on the right hand side of
\eqref{eq6.40}. In this regard, if these interior vertices are labeled as
$\{o_{1}, \ldots , o_{k} = \hat{o}\}$ by their generation number, then
$e(o_{j}) = e'(o_{j - 1})$. This understood, then \eqref{eq6.40} implies that
$\iota $ acts on $\tau _{\hat{o}}$ by adding
\begin{equation}\label{eq6.41}
\begin{split}
&- \frac{2\pi }{\alpha _{Q_{e(\hat{o})} } (\theta _{\hat{o}} )}
\biggl(\frac{{\iota _{o_k } } }{ {n_{o_k } }}-\frac{{\iota _{o_{k
- 1} } } }{ {n_{o_{k - 1} } }}\biggr)  \alpha _{Q_{e(o_k )} } (\theta
_{\hat{o}})+\cdots \\
&+ \biggl(\frac{{\iota _{o_2 }
} }{ {n_{o_2 } }}-\frac{{\iota _{o_1 } } }{ {n_{o_1 } }}\biggr)
\alpha _{Q_{e(o_2 )} } (\theta _{\hat{o}}) +
\biggl(\frac{{\iota _{o_1 } } }{ {n_{o_1 } }}\alpha _{Q_e } (\theta
_{\hat{o}})+\alpha _{R_e } (\theta _{\hat{o}})\biggr).
\end{split}
\end{equation}
This last formula has the following consequence when applied to
$\hat{o}$ and to each vertex from the collection $\{o_{j}\}_{1 \le j < k}$: Each
$o \in  \{o_{1}, \ldots , o_{k - 1}, o_{k}= \hat{o}\}$ version of
$\tau _{o}$ is fixed by $\iota $ if and only
if there exists a collection, $\{c_{1}, \ldots , c_{k}\}$, of
integers such that\vrule width 0pt depth 15pt
\begin{equation}\label{eq6.42}
\begin{gathered}
\biggl(\frac{{\iota _{o_1 } } }{ {n_{o_1 } }} - c_{1}\biggr)  \alpha _{Q_e} \bigl(\theta _{o_1 } \bigr)+
\alpha _{R_e } \bigl(\theta _{o_1 } \bigr) = 0.
\\
\biggl(\frac{{\iota _{o_2 } } }{ {n_{o_2 } }}-\frac{{\iota _{o_1 }
} }{ {n_{o_1 } }} - c_{2}\biggr)  \alpha _{Q_{e(o_2 )} } \bigl(\theta _{o_2
} \bigr) + \biggl(\frac{{\iota _{o_1 } } }{ {n_{o_1 } }} - c_{1}\biggr)  \alpha
_{Q_e } \bigl(\theta _{o_2 } \bigr)+\alpha _{R_e } \bigl(\theta _{o_2 } \bigr) = 0,
\\[2pt]
 \text{and so on through}
\\[2pt]
\biggl(\frac{{\iota _{o_k } } }{ {n_{o_k } }}-\frac{{\iota _{o_{k
- 1} } } }{ {n_{o_{k - 1} } }} - c_{k}\biggr)  \alpha _{Q_{e(o_k )} }
\bigl(\theta _{o_k } \bigr)+\cdots + \biggl(\frac{{\iota
_{o_2 } } }{ {n_{o_2 } }}-\frac{{\iota _{o_1 } } }{ {n_{o_1 }
}} - c_{2}\biggr)  \alpha _{Q_{e(o_2 )} } \bigl(\theta _{o_k } \bigr) +
\\
\biggl(\frac{{\iota _{o_1 } } }{ {n_{o_1 } }} - c_{1}\biggr)  \alpha _{Q_e
} \bigl(\theta _{o_k } \bigr)+\alpha _{R_e } \bigl(\theta _{o_k } \bigr) = 0 .
\end{gathered}
\end{equation}

\step{Part 6}
The next case to consider is that where $\hat{o}$ is not fixed by $\iota $. As
before, use k for $\hat{o}$'s generation number. In this case, the effect of
$\iota $ is to map $\tau _{\hat{o}}$ to $\mathbb{R}_{\iota
(\hat{o})}$ and this map has the schematic form
\begin{equation}\label{eq6.43}
\tau _{\hat{o}}   \to   \tau _{\hat{o}} - 2\pi
\frac{\alpha _{Z_{\hat{o}} } (\theta _{\hat{o}} )}{\alpha_{Q_{e(\hat{o})} } (\theta _{\hat{o}} )}-\frac{2\pi }{\alpha
_{Q_{e(\hat{o})} } (\theta _{\hat{o} })}  \sum _{\gamma }
r_{\hat{o}}(\gamma )  \mod(2\pi \mathbb{Z}),
\end{equation}
where $Z_{\hat{o}}$ is an ordered pair of rational
numbers that is determined by the image of $\iota $ in
$\Auth(T)/(\times _{o \in \Vertt(k)}\Aut(T_{o}))$. For
example, when $\hat{o}$ is a first generation vertex, then
$Z_{\hat{o}} = R_{e}$ with $e$ here denoting the edge that
contains both $\hat{o}$ and $\diamondsuit $. Meanwhile, in all cases,
the sum in \eqref{eq6.43} is indexed by the arcs in $\ell_{\hat{o}}$
that are crossed when traveling in the oriented
direction from $\iota ^{- 1}(\upsilon _{\iota (\hat{o})})$ to
$\upsilon _{\hat{o}}$.

\begin{proof}[Proof of \fullref{lem:6.8}]
The automorphisms $\iota $ and $\iota '$ have
identical actions on $\Vertt(k)$ since their actions on this set are determined
by their image in $\Aut(T)/(\times _{o \in \Vertt(k)}$ $\Aut(T_{o}))$. This
understood, if $\iota $ fixes a given $\hat{o}  \in  \Vertt(k)$, then so does
$\iota' $ and it follows from \eqref{eq6.42} that $\iota $ and $\iota '$ have
identical images in $\mathbb{Z}/(n_{\hat{o}}\mathbb{Z})$.

Suppose next that $\hat{o} \in  \Vertt(k)$ is not fixed by $\iota $.
To study this case, note first that any $j \in   \times _{o \in \Vertt(k)}\Aut(T_{o})$
has an image in each $o \in  \Vertt(k)$
version of $\mathbb{Z}/(n_{o}\mathbb{Z})$, and this image is
denoted in what follows as $j_{o}$, a number from the set $\{0, \ldots , n_{o}-1\}$.
Meanwhile, the $\Auth(T)$ version of \eqref{eq6.13} assigns both $\iota $ and $\iota'$
integers in this same set, these are their respective factors in the
$\mathbb{Z}/(n_{o}\mathbb{Z})$ summand. The latter are denoted
here by $\iota _{o}$ and $\iota '_{o}$. This understood, there
exists $j \in   \times _{o \in \Vertt(k)}\Aut _{o}$ such that
$j\iota j^{- 1}=\iota '$ in $\Aut(T)/(\times _{o \in \Vertt(k + 1)}\Aut _{o})$
when there exists a collection $\{j_{o}   \in\mathbb{Z}/(n_{o}\mathbb{Z})\}_{o \in \Vertt(k)}$ such that
\begin{equation}\label{eq6.44}
\iota '_{o}-\iota _{o} = -\bigl(j_{\iota (o)} - j_{o}\bigr)
\end{equation}
for each $o \in  \Vertt(k)$. To see that this equation is solvable, let
$\hat{o}  \in  \Vertt(k)$, let $M$ denote $\hat{o}$'s orbit under the action of $\iota $,
and let $m$ denote the number of elements in $M$. If $M$ consists only of $\hat{o}$,
then by virtue of what was said in the previous paragraph, $\iota _{o} =\iota '_{o}$
and one can take $j_{o} = 0$. In the case that $m> 1$, then
\eqref{eq6.44} is solvable if and only if
\begin{equation}\label{eq6.45}
\sum _{o \in M} (\iota _{o}-\iota '_{o}) = 0  \mod(2\pi n_{\hat{o}}\mathbb{Z}).
\end{equation}
To prove \eqref{eq6.45}, remark that $\iota ^{m}$ and $\iota '^{m}$
both fix $\hat{o}$ and so must be equal in
$\mathbb{Z}/(n_{\hat{o}}\mathbb{Z})$. As $\iota ^{m}$ acts
in $\mathbb{Z}/(n_{\hat{o}}\mathbb{Z})$ as $\sum _{o \in M}
\iota _{o}$ and $\iota '^{m}$ acts as $\sum _{o \in M}\iota'_{o}$,
the equality in \eqref{eq6.45} follows.
\end{proof}

\begin{proof}[Proof of \fullref{prop:6.7}]
Suppose that $G \subset  \Aut(T)$
stabilizes some point. The first observation here stems from \fullref{prop:6.4}:
Since $G$ is isomorphic to its image in $\Aut _{\diamondsuit }=\mathbb{Z}/(n_{\diamondsuit }\mathbb{Z})$,
it must be a cyclic group with one
generator. To see what this means, let $e$ denote the distinguished incident
edge to $\diamondsuit $, now fixed by $\Aut(T)$. Let $\iota    \in G$ be the
generator. As before, use $\iota $ to also denote the lift to $\Auth(T)$ and
let $b$ denote the point in \eqref{eq6.15} that is fixed by this lift.

Let $(z, N)$ denote $\iota $'s image in \eqref{eq6.32}. The assumption that $\iota $
fixes $b$ has the following consequence: The factor in $\mathbb{R}_{ - }$ is
fixed if and only if
\begin{equation}\label{eq6.46}
N = \frac{{k_ - } }{ {m_ - }}Q_{e},
\end{equation}
where $k_{ - }   \in   \mathbb{Z}$ and where $m_{ - }$ is
the greatest common divisor of the pairs that comprise $Q_{e}$.
Meanwhile, $b$'s factor in $\mathbb{R}_{\diamondsuit }$ is then fixed
if and only if
\begin{equation}\label{eq6.47}
\frac{{\iota _\diamondsuit}}{{n_\diamondsuit}}  +
\frac{{k_ - } }{ {m_ - }}  + z = 0,
\end{equation}
where $\iota _{\diamondsuit }   \in  \{0, \ldots , n_{\diamondsuit}-1\}$
is $\iota $'s image in $\Aut _{\diamondsuit }=\mathbb{Z}/(n_{\diamondsuit }\mathbb{Z})$.
This requires that
\begin{equation}\label{eq6.48}
\biggl(\frac{{k_ - } }{ {m_ - }} + z\biggr) = -\frac{{\iota
_\diamondsuit } }{ {n_\diamondsuit }}  \mod(\mathbb{Z}).
\end{equation}
To proceed, suppose next that $\iota $ fixes some vertex $\hat{o} \in T- \diamondsuit $
and that $\hat{o}$ is in generation $k  \ge 1$.
Let $\{o_{1}, \ldots , o_{k} = \hat{o}\}$ denote the vertices in
the $\hat{o}$ version of \eqref{eq6.42}. By virtue of \eqref{eq6.35},
these equations now imply that $\iota $ fixes $\tau_{\hat{o}}$ if and only if
there are integers $\{c_{1},\ldots , c_{k}\}$ such that
\begin{equation}\label{eq6.49}
\begin{gathered}
\biggl(\frac{{\iota _{o_1 } } }{ {n_{o_1 } }}-\frac{{\iota
_\diamondsuit } }{ {n_\diamondsuit }} - c_{1}\bigg)  \alpha _{Q_{e(o_1 )}
} \bigl(\theta _{o_1 } \bigr) + \biggl(\frac{{\iota _\diamondsuit } }{
{n_\diamondsuit }}+\frac{k }{ {m_ - }} + z\biggr)\alpha
_{Q_e } \bigl(\theta _{o_1 } \bigr) = 0.
\\
\biggl(\frac{{\iota _{o_2 } } }{ {n_{o_2 } }}-\frac{{\iota _{o_1 }
} }{ {n_{o_1 } }} - c_{2}\biggr)  \alpha _{Q_{e(o_2 )} } \bigl(\theta _{o_2
} \bigr) + \biggl(\frac{{\iota _{o_1 } } }{ {n_{o_1 } }} -
\frac{{\iota _\diamondsuit } }{ {n_\diamondsuit }} - c_{1}\biggr)\alpha _{Q_{e(o_1 )} }
\bigl(\theta _{o_2 } \bigr)
\\
+ \biggl(\frac{{\iota
_\diamondsuit } }{ {n_\diamondsuit }}+\frac{{k_ - } }{ {m_ -
}} + z\biggr)\alpha _{Q_e } \bigl(\theta _{o_2 } \bigr) = 0,\ \text{ and so on through }
\\
\biggl(\frac{{\iota _{o_k } } }{ {n_{o_k } }}-\frac{{\iota _{o_{k
- 1} } } }{ {n_{o_{k - 1} } }} - c_{k}\biggr)  \alpha _{Q_{e(o_k )} }
\bigl(\theta _{o_k } \bigr)+\cdots + \biggl(\frac{{\iota
_{o_2 } } }{ {n_{o_2 } }}-\frac{{\iota _{o_1 } } }{ {n_{o_1 }
}} - c_{2}\biggr)  \alpha _{Q_{e(o_2 )} } \bigl(\theta _{o_k } \bigr)
\\
+\biggl(\frac{{\iota _{o_1 } } }{ {n_{o_1 } }}-\frac{{\iota
_\diamondsuit } }{ {n_\diamondsuit }} - c_{1}\biggr)  \alpha _{Q_{e(o_1 )}
} \bigl(\theta _{o_k } \bigr) + \biggl(\frac{{\iota _\diamondsuit } }{
{n_\diamondsuit }}+\frac{{k_ - } }{ {m_ - }} + z\biggr)\alpha _{Q_e } \bigl(\theta _{o_k } \bigr) = 0.
\end{gathered}
\end{equation}
Here, a given version of $\iota _{o}$ denotes $\iota $'s image in
$\mathbb{Z}/(n_{o}\mathbb{Z}) = \{0, \ldots , n_{o}-1\}$. Coupled with \eqref{eq6.47},
this set of equations is satisfied if and only if each of the $k$ versions of
$\iota _{o}/n_{o}$ that appear here are identical. The preceding fact
together with \eqref{eq6.48} implies that $\iota $ generates a canonical diagonal
subgroup of $\Aut(T)$.

The remaining assertions of \fullref{prop:6.7} follow directly from the
preceding analysis with the various versions of the formula in \eqref{eq6.43}. The
details here are straightforward and so left to the reader.
\end{proof}


\setcounter{section}{6}
\setcounter{equation}{0}
\section{Proof of Theorems~\ref{thm:6.2} and \ref{thm:6.3}}\label{sec:7}

As the heading indicates, the purpose of this section is to
supply the proofs of the main theorems from the previous section.
The arguments to this end are much like those used to prove
\fullref{thm:3.1}. In fact, there are many places where the arguments
transfer in an almost verbatim form and, except for a comment to
this effect, these parts are left to the reader. In any event, to
start, use $\mathfrak{X}$ in what follows to denote the map that is defined in
\fullref{sec:6c}. \fullref{thm:6.2} is proved by establishing the following
about $\mathfrak{X}$:

\itaubes{7.1}
{\sl The map $\mathfrak{X}$  lifts on some neighborhood of any
given  element in ${\mathcal{M}^{*}}_{\hat{A},T}$
 as a local diffeomorphism from ${{\mathcal{M}^{*}}_{\hat{A},T}}^{\Lambda }$  onto an open set in
$\mathbb{R}  \times O_{T}$.}

\item
\textsl{The map $\mathfrak{X}$  is 1--1  into $\mathbb{R}\times O_{T}/\Aut(T)$.}

\item
$\mathcal{R} \cap {\mathcal{M}^{*}}_{\hat{A},T}$ \textsl{is sent to
$\mathbb{R}\times  (O_{T}- \hat{O}_{T})/\Aut(T)$  and its complement to}
$\mathbb{R} \times  \hat{O}_{T}/\Aut(T)$.

\item
\textsl{The map $\mathfrak{X}$  is proper onto $\mathbb{R}\times O_{T}/\Aut(T)$.}
\eit
Here, ${{\mathcal{M}^{*}}_{\hat{A},T}}^{\Lambda }$ is defined as in
\fullref{thm:6.3} to be the space of equivalence class of
triples $(C_{0}, \phi , T_{C})$ where $(C_{0}, \phi )$ defines a
point in ${\mathcal{M}^{*}}_{\hat{A},T}$ and $T_{C}$ is a
correspondence of $T$ in $(C_{0},\phi )$.

The first subsection below proves \fullref{thm:6.3} that given $\mathfrak{X}$ supplies
\fullref{thm:6.2}'s diffeomorphism.

\subsection{The local structure of the map $\mathfrak{X}$}\label{sec:7a}

The arguments given in this subsection justify the first point in
\eqreft71 and the assertions of \fullref{thm:6.3}. To start,
note that \fullref{lem:6.5} asserts that the map $\mathfrak{X}$ lifts to define
a continuous map from ${{\mathcal{M}^{*}}_{\hat{A},T}}^{\Lambda }$
to $\mathbb{R} \times O_{T}$. This understood, \fullref{prop:5.1} and
\fullref{lem:5.4} imply that the lift is smooth and locally 1--1. The
first point in \eqreft71 follows directly from this last conclusion.

\begin{proof}[Proof of \fullref{thm:6.3} for the map $\mathfrak{X}$]
For the purposes of
this proof, assume that the map $\mathfrak{X}$ is one of \fullref{thm:6.2}'s
diffeomorphisms. The proof of \fullref{thm:6.3} amounts to no
more than verifying the asserted properties of $\mathfrak{X}$. This task is left
to the reader with the following comment: Given the form of the
parametrizations from \fullref{def:2.1}, all of these
properties are direct consequences of the definitions given in
\fullref{sec:6c}.
\end{proof}

\subsection{Why the map $\mathfrak{X}$ is 1--1 from ${\mathcal{M}^{*}}_{\hat{A},T}$ to
$\mathbb{R}\times O_{T}/\Aut(T)$}\label{sec:7b}

The proof that $\mathfrak{X}$ is 1--1 applies \fullref{lem:4.1} as in
\fullref{sec:4b}. To start, suppose that $(C_{0}, \phi )$ and
$({C_{0}}', \phi ')$ have the same image in $\mathbb{R}  \times
O_{T}/\Aut(T)$. By assumption, $T$ has respective correspondences,
$T_{C}$ and ${T_{C}}'$, in $(C_{0}, \phi )$ and in $(C_{0}, \phi ')$
and these can be fixed so that $(C_{0}, \phi )$ and $({C_{0}}',
\phi')$ have the same image in $\mathbb{R}  \times O_{T}$. This is
because any correspondence of $T$ in $(C_{0}, \phi )$ can be
obtained from any other by composing the original with a suitable
automorphism of $T$. And, noted in \fullref{lem:6.6}, such a
change in the correspondence changes the assigned point in
$\mathbb{R}  \times O_{T}$ by the action of the relevant
automorphism.

To say more, the choices in Parts 1 and 2 of \fullref{sec:6c}
should be used to identify $O_{T}$ with the space in \eqref{eq6.9}.
With the identification understood, the assignments to $(C_{0}, \phi)$ and
$(C_{0}, \phi ')$ of their respective points in
\begin{equation}\label{eq7.2}
\mathbb{R}_{ - }  \times \bigl(\mathbb{R}_{\diamondsuit } \times \Delta
_{\diamondsuit }\bigr)  \times \bigl(\times _{o \in \mathcal{V} }(\mathbb{R}_{o}
\times   \Delta _{o})\bigr)
\end{equation}
can be made so as to agree. This is done in a
sequential fashion by first arranging that the respective $\mathbb{R}_{ - }$
assignments agree. For the latter purpose, let $e$
denote the distinguished edge in the distinguished
 $\Aut _{\diamondsuit }$ orbit $\hat{E}$. The assignments to $\mathbb{R}_{- }$
 can be made equal by suitably choosing the respective
pararameterizations of $e$'s component in the $C$ and $C'$ versions of
$C_{0}- \Gamma $. With the $\mathbb{R}_{ - }$ assignments equal, the
next step is to arrange so that the $\mathbb{R}_{\diamondsuit }$
assignments agree. Since the $\mathbb{Z}\cdot Q_{\hat{E}}$ subgroup
of $\mathbb{Z}  \times   \mathbb{Z}$ acts trivially on $\mathbb{R}_{- }$
and since its action on $\mathbb{R}_{\diamondsuit }$ changes
the assignment by the action of $2\pi \mathbb{Z}$, the $C'$ version of
the lift to $\mathbb{R}$ that defines its point in
$\mathbb{R}_{\diamondsuit }$ can be chosen to make it agree with $C$'s
assigned point. When applied inductively, versions of this last
argument prove that any given $o \in \mathcal{V} $ version of the two
assignments in $\mathbb{R}_{o}$ can be made to agree. Here, the
induction moves from any given $o \in \mathcal{V} $ to the vertices
that share its edges in $T_{o}$.

Given that $C$ and $C'$ have the same assignments in \eqref{eq7.2}, then minor
modifications of the arguments used in Part 2 of \fullref{sec:4b} prove
that the conditions in \eqreft41 are met in this case.

The appeal to \fullref{lem:4.1} also requires the condition in
\eqreft42. To see why this condition holds, suppose that $o$ is
a mutlivalent vertex in $T$ and that $\gamma    \subset
\underline{\Gamma }_{o}$ is an arc. Now let $e$ denote one of the
edge labels on $\gamma $ and let $e'$ denote the other. The choices
made for the $(C_{0}, \phi )$ assignment to the space in
\eqref{eq7.2} give parametrizations to both the $e$ and $e'$
components of $C_{0}-\Gamma $. These define $C_{0}$ versions of the
functions $w_{e}$ and $w_{e'}$ whose difference along the interior
of $\gamma $'s image in $\Gamma $ is described by some $N= (n, n')$
version of \eqref{eq2.14} and \eqref{eq2.15}. There are
corresponding ${C_{0}}'$ versions of $w_{e}$ and $w_{e'}$, and the
point here is that their difference is also described by this same $N=
(n, n')$ version of \eqref{eq2.14} and \eqref{eq2.15}. Indeed, such
is the case because the integer $N$ is obtained by using
\fullref{lem:2.3} to compare the respective canonical
parametrizations of the $e'$ component of either version of $C_{0}-
\Gamma $ and its ${C_{0}}'$ counterpart with the parametrization
that gives the $N= 0$ case of \eqref{eq2.14} and \eqref{eq2.15}
across $\gamma $. Since the same integer pair appears in both the
$C_{0}$ and ${C_{0}}'$ versions of \eqref{eq2.14} and
\eqref{eq2.15}, so $\hat {w}_{e}=\hat {w}_{e'}$ along $\gamma $.

Given that the conditions in \eqreft41 and \eqreft42 have
been met, the graph $G$ is well defined. If $G \ne  ${\o}, then
\fullref{lem:4.1} asserts that ${C_{0}}' = C_{0}$ and that $\phi'$
is obtained from $\phi $ by a constant translation along the
$\mathbb{R}$ factor of $\mathbb{R}  \times (S^1  \times  S^2)$. This
means that $\phi =\phi'$ since they share the same assignment in the
$\mathbb{R}$ factor of $\mathbb{R}  \times (S^1 \times  S^2)$. To
see that $G \ne ${\o}, return to the respective $C_{0}$ and
${C_{0}}'$ versions of \eqref{eq6.27}. Subtraction of the ${C_{0}}'$
version from the $C_{0}$ version finds the value 0 for the integral
of $\hat {w}$ around a non-empty union of constant $\theta $ circles
in $C_{0}- \Gamma $. Either $\hat {w} = 0$ on some of these circles,
in which case $G \ne ${\o}, or else $\hat {w}$ is negative on some
and positive on others. But this last case also implies that $G \ne
${\o} since the complement in $C_{0}$ of the critical point set of
$\cos\theta $ is path connected.

\subsection{The images of ${\mathcal{M}^{*}}_{\hat{A},T}\cap   \mathcal{R}$  and
${\mathcal{M}^{*}}_{\hat{A},T}- \mathcal{R}$.}\label{sec:7c}

The first order of business here is to explain why
${\mathcal{M}^{*}}_{\hat{A},T}   \cap   \mathcal{R}$ is mapped to the
image in $\mathcal{R}  \times O_{T}/\Aut(T)$ of the points where
$\Aut(T)$ acts with non-trivial stablizer. For this purpose, suppose
that $(C_{0}, \phi )$ defines a point in
${\mathcal{M}^{*}}_{\hat{A},T}$ and that there is a non-trivial
group of holomorphic diffeomorphisms of $C_{0}$ that fix $\phi $.
Let $\psi $ denote an element in this group. Let $T_{C}$ denote a
correspondence for $T$ in $(C_{0}, \phi )$.  Then 
$T_{C}$ defines a point for $(C_{0}, \phi )$ in
$\mathbb{R} \times O_{T}$. Meanwhile, $\psi $ defines a new
correspondence for $T$ in $(C_{0}, \phi )$ as follows: If $e$ is an
edge of $T$ and $K_{e}   \subset C_{0}- \Gamma $ the component that
originally corresponds to $e$, then $\psi ^{- 1}(K_{e})$ gives the
component for $e$ from the new correspondence. Likewise, if $o$ is a
multivalent vertex and $\gamma $ is an arc in $\underline {\Gamma}_{o}$,
then the new correspondence assigns to $\gamma $ the $\psi$--inverse image
of what is assigned $\gamma $ by the old
correspondence. As noted in \fullref{lem:6.6}, the change of the
correspondence changes the point in $\mathbb{R}  \times O_{T}$ that
is assigned to $(C_{0}, \phi )$ by the action of $\iota $.
Meanwhile, \fullref{lem:6.5} asserts that the new assigned point
is, after all, the same as the original. Thus the assigned point in
$\mathbb{R}  \times O_{T}$ is fixed by $\iota $.

The argument as to why ${\mathcal{M}^{*}}_{\hat{A},T}- \mathcal{R}$
is mapped to $\mathcal{R}  \times \hat{O}_{T}/\Aut(T)$ is much like
the argument in \fullref{sec:4d}. To start, suppose that
$(C_{0}, \phi )$ defines a point in ${\mathcal{M}^{*}}_{\hat{A},T}-\mathcal{R}$.
Fix a correspondence, $T_{C}$, for $T$ in $(C_{0}, \phi)$.
According to \fullref{lem:6.5}, this gives $(C_{0}, \phi )$ a
point in $O_{T}$ that defines the image of its equivalence class in
$O_{T}/\Aut(T)$. Now suppose that $\iota    \in  \Aut(T)$ fixes this
point in $O_{T}$. The assertion to prove is that $\iota $ is the
identity element. The proof has three steps.

\substep{Step 1}
The isomorphism $\iota $ can be
used to change the correspondence $T_{C}$ to a new correspondence
for $T$ in $(C_{0}, \phi )$. The latter is denoted in what follows
as ${T_{C}}^{\iota }$ and it is defined by the following two
conditions: To state the first, let $e$ denote an edge in $T$ and
use $K_{e}$ to denote the component of $C_{0}- \Gamma $ that is
labeled by $e$ using $T_{C}$. Meanwhile use ${K^{\iota }}_{e}$ to
denote the component that is labeled by $e$ using the new
correspondence. Define the collection $\{{K_{e}}^{\iota }\}$ so that
${K^{\iota }}_{\iota (e)} = K_{e}$. To state the second condition,
let $\gamma $ denote an arc in some version of $\underline {\Gamma}_{o}$.
Then the arc in $\Gamma $ that corresponds via $T_{C}$ to
$\gamma $ corresponds via ${T_{C}}^{\iota }$ to $\iota (\gamma )$.

Now, the assignment to $C$ of a point in \eqref{eq7.2} using the
correspondence $T_{C}$ required parametrizations for each component
of $C_{0}- \Gamma $. Recall that these were assigned in a sequential
manner starting with an arbitrary choice for the component that
corresponds to the distinguished incident edge to $\diamondsuit $.
Such a choice and the assignment in $\mathbb{R}_{\diamondsuit }$
assigned parametrizations to all incident edges to $\diamondsuit $.
The sequential nature of the process appears when the
parametrizations were assigned to the components of $C_{0}- \Gamma $
that are labeled by the incident edges to a vertex $o \in \mathcal{V}$.
In particular, the point in $\mathbb{R}_{o}$ and the
parametrization for the component labeled by the edge that connects
$o$ to $T- T_{o}$ give the parametrizations for the components that
are labeled by the edges that connect $o$ to $T_{o}-o$.

Make the choices that assign a point in \eqref{eq7.2} for $C$ using
the correspondence $T_{C}$. Let $b$ denote this point in
\eqref{eq7.2}. Meanwhile, a second set of choices can be made so
that $b$ is also the assigned point in \eqref{eq7.2} for $C$ using the
correspondence ${T_{C}}^{\iota }$. In this way, each component of
$C_{0}- \Gamma $ receives two parametrizations; one in its
incarnation as $K_{e}$ and the other as ${K^{\iota }}_{\iota (e)}$.
This understood, define a map, $\psi _{e}\co K_{e}   \to  {K^{\iota}}_{e }$ as follows: 
Let $\phi _{e}$ denote the parametrizing map from
the relevant cylinder to $K_{e}$, and let ${\phi ^{\iota }}_{e}$
denote the corresponding map to ${K^{\iota }}_{e}$. Note that the
parametrizing cylinders are the same. Set $\psi _{e} \equiv
{\phi ^{\iota }}_{e} \circ (\phi _{e})^{- 1}$.

\substep{Step 2}
This step proves that there is a
continuous map $\psi \co  C_{0}   \to C_{0}$ whose restriction to any
given $K_{e}$ is the map $\psi _{e}$. The proof has two parts; the
first verifies that any given $\psi _{e}$ extends continuously to
the closure of $K_{e}$. The second verifies that corresponding
extensions agree where the closures overlap.

To start the first part, remark that $\psi _{e}$ extends
continuously to the closure of $K_{e}$ when the following is true:
Let $o \in e$ be a multivalent vertex and let $\upsilon $ be any
vertex in $\ell _{oe}$. Then the $\mathbb{R}/(2\pi \mathbb{Z})$
coordinate of the point on the $\sigma  = \theta _{o}$ boundary of
the parametrizing cylinder that corresponds via $\phi _{e}$ to
$\upsilon $ is the same as the $\mathbb{R}/(2\pi \mathbb{Z})$
coordinate of the point that corresponds via ${\phi ^{i}}_{e}$ to
$\upsilon $. Here, one need check this condition at only one vertex
since the condition in \eqref{eq6.37} makes the difference between
the relevant $\mathbb{R}/(2\pi \mathbb{Z})$ values into a constant
function on the vertex set of $\ell _{oe}$.

To check this condition, consider first the case when either $o=\diamondsuit $ and
$e$ is $\diamondsuit $'s distinguished edge, or
else $o \in   \mathcal{V} $ and $e$ is the edge that connects $o$ to
$T- T_{o}$. In these cases, $\ell _{oe}$ contains the distinguished
vertex $\upsilon _{o}   \in \underline {\Gamma }_{o}$. In
particular, the value of the $\mathbb{R}/(2\pi \mathbb{Z})$
coordinate of the point that corresponds via $\phi _{e}$ to
$\upsilon _{o}$ is the reduction modulo $2\pi \mathbb{Z}$ of the
point, $b$, that $C$ is assigned in \eqref{eq7.2} using the
correspondence $T_{C}$. Meanwhile, the analogous ${\phi ^{\iota}}_{e}$
version of this $\mathbb{R}/(2\pi \mathbb{Z})$ value is the
reduction modulo $2\pi \mathbb{Z}$ of the point that $C$ is assigned
in \eqref{eq7.2} using the correspondence ${T_{C}}^{\iota }$. Since
the latter point is also b, so the desired equality holds for the
relevant $\mathbb{R}/(2\pi \mathbb{Z})$ values.

Consider next the case where either $o= \diamondsuit $ and $e$ is
not the distinguished edge, or else $o \in   \mathcal{V} $ and $e$
connects $o$ to $T_{o}-o$. To analyze this case, remember that a
concatenating path set has been chosen whose first path starts at
$\upsilon _{o}$ and whose last path ends at a vertex on $\ell_{oe}$.
Let $\mu (e)  \subset \underline {\Gamma }_{o}$ denote the
path that is obtained from the concatenating path set by identifying
the final arc in all but the last path with the initial arc in the
subsequent path. Note that $\mu (e)$ is a concatenated union of arcs
whose last vertex is on $\ell _{oe}$. In particular, the
$\mathbb{R}/(2\pi \mathbb{Z})$ coordinate that corresponds via
$\phi_{e}$ to this last vertex is
\begin{equation}\label{eq7.3}
\tau _{o}+\frac{2\pi }{\alpha _{Q_e } (\theta _o )}\sum _{\gamma
\in \mu (e)}  \pm r_{o}(\gamma ) \mod (2\pi \mathbb{Z}),
\end{equation}
where the notation is as follows: First, $\tau _{o}$ is
$\mathfrak{b}$'s coordinate in the $\mathbb{R}_{o}$ factor in \eqref{eq7.2} and
$r_{o}$ is $\mathfrak{b}$'s $\Delta _{o}$ factor. Meanwhile, the sum is over the
arcs that are met sequentially in $\mu (e)$ as it is traversed from
$\upsilon _{o}$ to its end, and where the + sign appears with an arc
if and only if it is crossed in its oriented direction. In this
regard, note that a given arc can appear more than once in
\eqref{eq7.3}, and with different signs in different appearances.

Of course, the analogous formula gives $\mathbb{R}/(2\pi \mathbb{Z})$
coordinate that corresponds via ${\phi ^{\iota }}_{e}$ to the final
vertex on $\mu _{e}$. Thus, the $\phi _{e}$ and ${\phi ^{\iota}}_{e}$
versions of the relevant $\mathbb{R}/(2\pi \mathbb{Z})$
coordinate agree.

To verify that the extensions of $\{\psi _{e}\}$ agree on the locus
$\Gamma    \subset C_{0}$, suppose that $o \in T$ is a multivalent
vertex, that $\gamma    \subset \underline {\Gamma }_{o}$ is an arc,
and that $e$ and $e'$ are $\gamma $'s labeling edges. The
correspondence $T_{C}$ identifies $\gamma $ with a component of the
complement in $\Gamma $ of $\theta $'s critical points. Denote this
component by $\gamma ^{\delta }$. Of course, $\gamma $ also
corresponds via ${T_{C}}^{\iota }$ to some other component,
$\gamma^{\iota \delta }= (\iota ^{- 1}(\gamma ))^{\delta }$. The
extension of $\psi _{e}$ maps $\gamma ^{\delta }$ to $\gamma ^{\iota\delta }$
as a diffeomorphism because the $\mathbb{R}/(2\pi\mathbb{Z})$
coordinate of a point that corresponds via $\phi _{e}$
to either of its end vertices is the same as the corresponding ${\phi^{\iota }}_{e}$ coordinate.
Likewise, $\psi _{e'}$ maps $\gamma^{\delta }$ diffeomorphically onto $\gamma ^{\iota \delta }$. The
agreement between these two diffeomorpisms then follows directly
from the formula in \fullref{def:2.1}. In this regard, keep in
mind that the 1--form $\surd 6\cos \theta d\varphi  -(1-3\cos^{2}\theta ) dt$
pulls back to the $\sigma =\theta _{o}$
circle in the parametrizing cylinder for $K_{e}$ as the $Q =Q_{e}$
version of $\alpha _{Q}(\theta _{o}) dv$ while its pull-back to the
circle in the $K_{e'}$ cylinder is the $Q =Q_{e'}$ version.

\substep{Step 3}
Let $\phi \co  C_{0} \to \mathbb{R}\times (S^1  \times  S^2)$ denote
the tautological map onto $C$. With $\psi $ now defined, the plan
is to prove that $\phi  \circ \psi = \phi $. Since $(C_{0},\phi )
\notin \mathcal{R}$,
this then means that $\psi$ is the identity map and so $\iota $ is trivial.

The argument for this employs \fullref{lem:4.1}. To conform with
the notation used in \fullref{lem:4.1}, let ${C_{0}}'$ denote
$C_{0}$ and ${T_{C}}^{\iota }$ as $T_{C'}$. This understood, the fact
that $\psi $ is continuous implies that the respective
parametrizations $\phi _{e}$ and ${\phi ^{\iota }}_{e}$ are compatible
in the sense of the definition in \eqreft41. Next, let
$(a_{e},w_{e})$ denote the functions that appear in the $\phi _{e}$
version of \eqreft25 and let $({a_{e}}', {w_{e}}')$ denote the
functions that appear in the ${\phi ^{\iota }}_{e}$ version of
\eqreft25. Now set $\hat{a}_{e}   \equiv  a_{e}- {a_{e}}'$ and
$\hat {w}_{e} \equiv  w_{e}- {w_{e}}'$. As it turns out, there is a
continuous function on the complement in $C_{0}$ of the $\cos\theta
$ critical points whose pull-back from $C_{0}- \Gamma $ via the maps
$\{\phi_{e}\}$ gives the collection $\{\hat {w}_{e}\}$. This is to
say that the assumption in \eqreft42 is satisified. Indeed, the
proof that the collection $\{\hat {w}_{e}\}$ comes from a continuous
function is identical in all but cosmetics to the proof of the
analogous assertion in the second to last paragraph of the preceding
subsection.

Let $\hat {w}$ denote the function on $C_{0}$ that gives the
collection $\{\hat {w}_{e}\}$ and introduce $G$ to denote the zero
locus of $\hat {w}$ in the complement of the $\cos\theta $ critical
points. According to \fullref{lem:4.1}, either $G = ${\o} or the
closure of $G$ is all of $C_{0}$. Now, $G \ne  ${\o} since the
respective assignments to $C_{0}$ using $T_{C}$ and ${T_{C}}^{\iota }$
in \eqref{eq7.2} agree, and so they agree in the $\mathbb{R}_{ - }$
factor in particular. Thus $G$'s closure is $C_{0}$ and so $\hat {w}
\equiv  0$. Thus $\phi  \circ \psi $ is obtained from $\phi $ via a
constant translation along the $\mathbb{R}$ factor of $\mathbb{R}
\times (S^1  \times  S^2)$. Moreover, this constant translation must
be the identity since the two maps have the same image set. Hence
$\phi  \circ \psi =\phi $ and so $\iota $ must be the identity
automorphism.

%
%

\subsection{Why the map to $\mathbb{R}\times\hat{O}_{T}/\Aut(T)$ is proper}\label{sec:7d}

The proof that \fullref{sec:6c}'s map is proper is given here in three parts, each
the analog of the corresponding part of \fullref{sec:4e}.

\step{Part 1} This part of the argument proves the
${\mathcal{M}^{*}}_{\hat{A},T}$ version of \fullref{prop:4.6}. This
is as follows:

\begin{proposition} \label{prop:7.1}

Let $\{(C_{0j}, \phi _{j})\}_{j =1,2,\ldots }$  denote an infinite sequence of pairs that defines a
sequence in $ \mathcal{M}^{*}_{\hat{A},T}$ with
convergent image in $\mathbb{R}\times O_T/\Aut(T)$.  There
exists a subsequence, hence renumbered consecutively from 1,  and a
finite set, $\Xi $, of pairs of the form $(S, n)$  where $n$  is a
positive integer and $S$  is an irreducible, pseudoholomorphic,
multiply punctured sphere; and these have the following properties
with respect to the sequence of subsets $\{C_{i} \equiv \phi _{i}(C_{0i})\}$  in
$\mathbb{R}\times (S^1  \times  S^2)$:

\bit

\item
$\lim_{j \to \infty }  \int _{C_j } \varpi =
\sum_{(S,n)\in\Xi} n \int _{S} \varpi $
for each compactly supported 2--form $\varpi $.

\item
The following limit exists and is zero:
\begin{equation}\label{eq7.4}
\lim_{j \to \infty } \Bigl(\sup_{z \in C_j } \dist\bigl(z,  \cup _{(S,n)
\in \Xi } S\bigr) + \sup_{z \in \cup _{(S,n) \in \Xi } S}
\dist(C_{j}, z)\Bigr).
\end{equation}

\eit
\end{proposition}

Granted for the moment \fullref{prop:7.1}, the proof that
\fullref{sec:6c}'s map is proper is obtained by the following
line of reasoning: Start with a sequence $\{(C_{0i}, \phi _{i})\}$ with convergent image in
$\mathbb{R}\times O_T/\Aut(T)$. The next part of this subsection proves that
\fullref{prop:7.1}'s set $\Xi $ contains only one element.
Let $(S, n)$ denote this element. The third part of the subsection
proves the following:

\itaubes{7.5}
\textsl{If $n = 1$,  then $S \in   \mathcal{M}_{\hat{A},T}$  and $\{C_{j}\}$  converges to $S$  in
${\mathcal{M}_{\hat{A},T}}$.}

\item
 \textsl{If $n > 1$, then there exists a pair $(S_{n}, \phi )$  where $\phi $  sends $S_{n}$  onto $S$
 as an $n$--fold, branched cover; and this pair defines a point in ${\mathcal{M}^{*}}_{\hat{A},T}$
 that is the limit point of the image in ${\mathcal{M}^{*}}_{\hat{A},T}$  of $\{(C_{0j},\phi _{j})\}$.}
\eit
\noindent Thus, a convergent subsequence exists for any sequence in
${\mathcal{M}^{*}}_{\hat{A},T}$ with convergent image in
$\mathbb{R}\times  O_T/\Aut(T)$ and this property characterizes
a proper map.

It is assumed in the rest of this subsection that a correspondence has been
chosen for $T$ in each $(C_{0j}, \phi _{j})$.

\begin{proof}[Proof of \fullref{prop:7.1}]
As with \fullref{prop:4.6}, all but \eqref{eq7.4}
follows from  \cite[Proposition~3.7]{T3}. In this regard, \eqref{eq7.4} is the version
of \cite[Proposition~3.7]{T3} where the compact set involved, $K$, is replaced by
 $K= \mathbb{R}  \times (S^1  \times  S^2)$. In any event, assume
that \eqref{eq7.4} does not hold so as to derive some patent nonsense. This
derivation has six steps.

\substep{Step 1}
This step proves that the conclusions in
\fullref{lem:4.7} hold if \eqref{eq7.4} does not hold. The argument here is almost the same
as that used for the \fullref{sec:4e}'s version of the lemma. To start, note that
an argument in \fullref{sec:4e} for \fullref{lem:4.7} works as well here to prove that its
conclusions hold except possibly in the case that all subvarieties from $\Xi$
are $\mathbb{R}$--invariant cylinders and that one of the two points in \eqreft4{19}
hold.

To rule out the \eqreft4{19} cases, note first that the smallest vertex angle from
 $T$ must be zero in the \eqreft4{19} cases. The argument is essentially that from
\fullref{sec:4e}: The convergence of the image of $\{(C_{0j},\phi _{j})\}$ in the
$\mathbb{R}$ factor of $\mathbb{R}\times O_T/\Aut(T)$ precludes a non-zero smallest angle because there
would otherwise be a sequence in ${\mathcal{M}^{*}}_{\hat{A},T}$
with the following properties: First, its $j$'th element is $(C_{0}$,
${\phi _{j}}')$ where ${\phi _{j}}'$ is obtained from $\phi _{j}$ by a
$j$--dependent but constant translation along the $\mathbb{R}$ factor
of $\mathbb{R}  \times (S^1  \times  S^2)$. Second, the limit data
set for $\{{\phi _{j}}'(C_{0j})\}$ from
\cite[Proposition~3.7]{T3} has a subvariety that is not an $\mathbb{R}$--invariant
cylinder, has the same $\theta $ infimum as each $C_{j}=\phi_{j}(C_{0j})$.
Third, there is no non-zero constant $b$ for one of
the corresponding ends that makes \eqreft23 hold.

Granted the preceding, assume that the smallest vertex angle from
$T$ is equal to zero. Let $\hat{E}'$ denote the $\Aut(T)$ orbit of edges
in $T$ that is used in \fullref{sec:6c} to define the image of
$\{(C_{0j},\phi _{j})\}$ in the $\mathbb{R}$ factor of
$\mathbb{R}  \times \hat{O}_{T}/\Aut(T)$. Suppose first that
$\hat{E}'$ corresponds to a set of disjoint disks in each $C_{0j}$
that intersect the $\theta  = 0$ cylinder at their center points.
Note that this case occurs only when $\hat{A}$ lacks $(1,\ldots)$
elements. To preclude this \eqreft4{19} case, first associate to
each $C_{0j}$ the set of its $\theta  = 0$ points that correspond to
the edges in $\hat{E}'$. Let $t_{j}$ denote this set. Note that no
sequence whose $j$'th element is from $t_{j}$ can remain bounded as
$j\to   \infty $. This is to say that the set of $|s| $ coordinates
for such a sequence cannot be bounded. Indeed, a bounded sequence
here would force $\Xi $ to have a pair whose subvariety is the
$\theta  = 0$ cylinder. Granted \eqreft4{19}, each $C_{0j}$ would
have an end where the $|s|    \to   \infty $ limit of $\theta $ is
zero. But there are no such ends when $\hat{A}$ lacks $(1,\ldots)$
elements.

Meanwhile, note that as the image of $\{(C_{0j}, \phi _{j})\}$ in
the $\mathbb{R}$ factor of $\mathbb{R}  \times \hat{O}_{T}/\Aut(T)$
converges, so the sequence whose $j$'th component is the sum of the s
values at the points in $t_{j}$ also converges. Granted the
conclusions of the preceding paragraph, this can occur only if there
are two sequences whose $j$'th element is in $t_{j}$, and these are
such that $s$ tends to $\infty $ on one and to $-\infty $ on the
other. However, this conclusion leads afoul of \eqreft4{19} and
the lack of $(1,\ldots)$ elements in $\hat{A}$ when one considers
that there is a path in $C_{0j}$ between the $j$'th point in the one
sequence and the $j$'th point in the other.

To finish the story for the current version of \fullref{lem:4.7},
assume now that $\hat{A}$ has some $(1,\ldots)$ element and so the
edges in $\hat{E}'$ correspond to ends in any given $C_{0j}$ where
$\lim_{| s| \to \infty }  \theta  = 0$. In this case, the image of
$C_{0j}$ in the $\mathbb{R}$ factor of $\mathbb{R} \times \hat{O}_{T}/\Aut(T)$ is,
up to a positive constant factor, minus the
sum of the logarithms of the versions of the constant $\hat {c}$
that appear in \eqref{eq1.9} for the ends in $C_{0j}$ that
correspond to the edges in $\hat{E}'$. Thus, the sequence of such sums
converge. As a consequence, the sequence whose $j$'th element is the
minimal contribution to the $j$'th sum can not diverge towards
$+\infty $, nor can the sequence whose $j$'th element is the maximal
contribution to the $j$'th sum diverge towards $-\infty $. This then
gives the following nonsense: As in the analogous case from the
original proof of \fullref{lem:4.7}, a new sequence can be
constructed with the following mutually incompatible properties:
First, its $j$'th element is $(C_{0j}, {\phi _{j}}')$ where ${\phi_{j}}'$
is obtained from $\phi _{j}$ by a constant, but $j$--dependent
translation along the $\mathbb{R}$ factor in $\mathbb{R}\times (S^1  \times  S^2)$.
Second, the data set from the
corresponding $\{{\phi _{j}}'(C_{0j})\}$ version of
  \cite[Proposition~3.7]{T3} provides a subvariety that
is not the $\theta  = 0$ cylinder, but has an end where the $|s| \to
\infty $ limit of $\theta $ is 0 and is such that the respective
integrals of $\frac{1 }{ {2\pi }} dt$ and $\frac{1 }{ {2\pi }}d\varphi $
about its constant $|s| $ slices have the form $\frac{1}{ m}p$ and $\frac{1 }{ m}p'$ where
$(p, p')$ is the integer pair that
comes from the elements in $\hat{A}$ that correspond to $\hat{E}'$,
and where $m \ge 1$ is a common divisor of this pair. Finally, there
is no non-zero $\hat {c}$ that for the resulting verions of
\eqref{eq1.9}. The argument here for the values of the integrals of
$\frac{1 }{ {2\pi }} dt$ and $\frac{1 }{ {2\pi }}d\varphi $ is
almost verbatim that used in the analogous part of
\fullref{sec:4}'s proof of \fullref{lem:4.7}. Indeed, the
latter argument works because the same domain in $(0, \pi )  \times
\mathbb{R}/(2\pi \mathbb{Z})$ parametrizes all $\hat{E}'$ labeled
components of all $C_{0j}$ version of $C_{0}- \Gamma $; thus this
same domain parametrizes all of the $\hat{E}'$ labeled components of
each version of $C_{0}- \Gamma $ from the translated sequence.

\substep{Step 2}
\fullref{lem:4.7} provides the
$\mathbb{R}$--invariant cylinder $S_{*}$, and given $\varepsilon  >0$,
the real number $s_{0}$, the sequences $\{s_{j - }\}$ and
$\{s_{j + }\}$ and the component $C_{j*}$ of the s $ \in  [s_{j - }, s_{j + }]$
part of each large $j$ version of $C_{0j}$. The
first point to make here is that $\theta $ is neither 0 nor $\pi $
on $S_{*}$. Indeed, were this otherwise, then the mountain pass
lemma would find a non-extremal critical point of $\theta $ on each
large $j$ version of $C_{0j}$ in $C_{j*}$. In particular, the
corresponding sequence of critical values would converge as $j \to \infty $
to either 0 or $\pi $, and this is nonsense since the
critical values on any one $C_{0j}$ are identical to those on any
other.

Thus, $\theta $ on $S_{*}$ must be some angle in $(0, \pi )$. Denote
the latter by $\theta _{*}$. \cite[Lemma~3.9]{T3} can be
used in the present context to draw the following conclusions: Given
$\delta _{0} > 0$, there exist $j$--independent constants $\delta \in(0, \delta _{0})$ and
$R_{ + }, R_{ - }   \ge 0$ such that when $j$ is large,

\itaubes{7.6}
$\theta$ \textsl{has values both greater than $\theta _{*}+\delta $
and less than $\theta _{*}-\delta $  where $s=s_{j - }+R_{ - }$
in} $C_{j*}$.

\item
 \textsl{Either $| \theta -\theta _{*}| $  is strictly greater than $\delta $  where $s=s_{j +}-R_{ + }$
 in $C_{j*}$  or else $\theta $  takes values both
greater than $\theta _{*}+\delta $
 and $\theta _{*}-\delta $  where $s=s_{j + }-R_{ + }$  in $C_{j*}$.}
\eit
As no generality is lost by choosing $R_{\pm }$ so that both the
$s=s_{j - }+R_{ - }$ and $s=s_{j + }-R_{ + }$ are regular values
of s on each $C_{j*}$ and that neither locus in $C_{j*}$ contains a
$\theta $ critical point. Such choices for $R_{\pm }$ are assumed in
what follows.

By contrast, the fourth point in \fullref{lem:4.7} implies the
existence of a sequence $\{\delta _{j}\}$ with limit zero such that
$| \theta -\theta _{*}|  < \delta _{j}$ where $s=\frac{1 }{2}(s_{j - }+s_{j + })$.
Granted all of this, the argument for
\fullref{prop:4.6} in \fullref{sec:4e} works here to
prove that $\theta _{*}$ is either a
critical point of $\theta $ on $C_{0j}$ or else the $|s|    \to \infty $
value of $\theta $ on an end in $C_{j0}$ whose version of
\eqref{eq2.4} has integer $n_{(\cdot )} = 0$. In either case,
$\theta _{*}$ is the angle of some multivalent vertex in $T$ and
$C_{j*}$ must intersect some graph from the collection $\{\Gamma_{o}\}$.
As $C_{j*}$ is connected and $\theta $ has small variation
on $C_{j*}$, it can intersect at most one such graph. Let $o \in T$
denote the labeling vertex.

\substep{Step 3}
Let $s_{*}$ denote a regular value
of s on $C_{j*}$ that is within 1 of $\frac{1 }{ 2}(s_{j - }+s_{j+ })$,
has transversal intersection with the $\theta =\theta _{*}$
locus, and is such that $| \theta -\theta _{*}|  < \delta _{j}$ on
the $s=s_{*}$ slice of $C_{j*}$. By virtue of
  \cite[Proposition~3.7]{T3}, the $s=s_{*}$ slice of
$C_{j*}$ is not null-homologous in the radius $\varepsilon $ tubular
neighborhood of $S_{*}$. Thus this slice is not null-homologous in
the complement of the $\theta  = 0$ and $\theta =\pi $ loci in
$C_{0j}$. Even so, the following is true: If $j$ is large, then this
$s=s_{*}$ slice is not homologous in the $\theta  \notin \{0, \pi \}$ part of
$C_{0j}$ to a union of suitably oriented
slices of ends of $C_{0j}$ where the $|s|    \to   \infty $ limit of
$\theta $ is $\theta _{*}$. Indeed, were this not the case, then
\eqreft75 demands the impossible: The union of the $s=s_{*}$
slice of $C_{j*}$ and a collection of very large, but constant $|s|$
slices of $C_{0j}$ is the boundary of a subset of $C_{j}$'s model
curve where $\theta \notin \{0, \pi \}$ and whose interior has a local extreme point
of $\theta $. This is impossible because, as noted in \fullref{sec:2a}, the only
local minima or maxima of $\theta$ on a pseudoholomorphic subvariety
occur where $\theta $ is respectively 0 or $\pi $.

By the way, a component of the $s=s_{*}$ slice of $C_{j*}$ that is
homologically non-trivial in the radius $\varepsilon $ neighborhood
of $S_{*}$ must intersect the $C_{0j}$ version of the locus
$\Gamma_{o}$ when $j$ is large. To explain, introduce the 1--form
$x=(1-3\cos^{2}\theta _{*}) d\varphi -\surd 6 \cos\theta _{*} dt$. The
latter is exact in the radius $\varepsilon $ tubular neighborhood of
$S_{*}$ so has zero integral over any given component of the
$s=s_{*}$ slice of $C_{j*}$. Now, let $\eta $ denote a component of
the $s=s_{*}$ slice of $C_{j*}$. If $| \theta -\theta _{*}|  > 0$ on
$\eta $, then $\eta $ must sit as an embedded circle in some
component of the $C_{0j}$ version of $C_{0}- \Gamma $. If $\eta $ is
homologically non-trivial in this component, then it is homotopic
there to a constant $\theta $ slice. Thus, the integral of $x$ over
$\eta $ is $\pm \alpha _{Q}(\theta _{*})$ where $Q$ is the integer
pair that is assigned to the edge in $T$ that labels $\eta $'s
component of $C_{0}- \Gamma $. Thus, $\alpha _{Q}(\theta _{*}) = 0$.
However, this is nonsense since \eqreft75 requires that
$\theta=\theta _{*}$ on the closure of this component.

\substep{Step 4}
To say more about the $s=s_{*}$
slice of $C_{j*}$ requires a digression to point out some features
of the small but positive $\delta $ versions of the $| \theta
-\theta _{*}| < \delta $ neighborhood in $C_{0j}$ of the locus
$\Gamma _{o}$.

The first point here is that such a neighborhood is homeomorphic to
a multiply punctured sphere. In particular, the first homology of
such a neighborhood is canonically isomorphic to the first homology
of the graph ${\underline {\Gamma }^{*}}_{o}$. This is to say that a
system of generators for the $| \theta -\theta _{*}| < \delta $
can be obtained as follows: First take a large $|s| $ slice of each
end of $C_{0j}$ that corresponds to a vertex on $\underline {\Gamma}_{o}$.
These form a set, $\{\ell ^{*\upsilon }\}$, where the label
$\upsilon $ runs through the vertices in $\underline {\Gamma }_{o}$
with non-zero integer label. Add to these the set, $\{{\ell^{*}}_{oe}\}$,
where the label $e$ runs through the incident edges to
$o$ and ${\ell ^{*}}_{oe}$ denotes the $| \theta -\theta _{*}|
=\frac{1 }{ 2}\delta $ slice of $e$'s component of the $C_{0j}$
version of $C_{0}-\Gamma $. The collection $\{{\ell ^{*}}_{oe}\} \cup
\{\ell ^{*\upsilon }\}$ then generates the homology of this
neighborhood of $\underline {\Gamma }_{o}$ subject to the one
constraint in \eqref{eq2.22}.

The second point concerns the 1--form $x = (1-3\cos^{2}\theta _{*})
d\varphi -\surd 6 \cos\theta _{*} dt$ where $\theta $ is near
$\theta _{*}$. As remarked previously, this form is exact near
$S_{*}$. To elaborate on this, write $S_{*}$ as $\mathbb{R}\times
\gamma _{*}$ where $\gamma _{*}   \subset  S^1  \times  S^2$ is the
relevant $\theta =\theta _{*}$ Reeb orbit. Now, $x$ is, of course,
pulled back from $S^1  \times  S^2$ and its namesake on $S^1  \times
S^2$ can be written on the radius $\varepsilon $ tubular
neighborhood of $\gamma _{*}$ as $df$ where $f$ is a smooth function
that vanishes on $\gamma _{*}$. Over the whole of the $\theta
=\theta _{*}$ locus in $S^1  \times  S^2$, the form $x$ can be
written as $df_{*}$ where $f_{*}$ is a multivalued function that
agrees with an $\mathbb{R}$--valued lift near $\gamma _{*}$ that
agrees with $f$. The function $f_{*}$ is constant on each $\theta
=\theta _{*}$ Reeb orbit and its values distinguish the various
$\theta =\theta _{*}$ Reeb orbits. In this regard, the values of
$f_{*}$ are in $\mathbb{R}/(2\pi \kappa _{*}\mathbb{Z})$ where
\begin{equation}\label{eq7.7}
{\kappa _{*}}^{2 } =
\frac{{(1 + \cos^4\theta_* )} }{ {(p^{2} + p'^{2}\sin ^2\theta_ * )}}.
\end{equation}
Here $(p, p')$ is the relatively prime integer pair that determines $\theta_{*}$ via \eqref{eq1.8}.

\substep{Step 5}
This step concerns the integral of
the 1--form $x$ over any given component of the $| \theta -\theta_{*}|  > 0$
portion of the $s=s_{*}$ slice of $C_{j*}$. To say more,
let $\eta $ denote such a component. Since $x= df$ near $\eta $
and $\eta $ is close to $S_{*}$ where $f$ is zero, so the integral
of $x$ over $\eta $ has absolute value no greater than some
 $j$ and $\eta $ independent multiple of $\delta _{j}$.

To see the implications of this result, note that $\eta $ sits in
some component of the $C_{j}$ version of $C_{0}- \Gamma $ whose
label is an incident edge to $o$. Fix a parametrization of this
component, and the closure of $\eta $ then corresponds to an
embedded path in the closed parametrizing cylinder whose two
endpoints are on the $\sigma =\theta _{*}$ boundary and whose
interior lies where $0 < | \theta -\theta _{*}|<\delta _{j}$.
As such, this version of $\eta $ is homotopic rel its end points to
an embedded path, $\mathfrak{p}(\eta )$, in the $\sigma =\theta _{*}$ circle of
the parametrizing cylinder.

The path $\mathfrak{p}(\eta)$ may or may not pass through some missing points
on the $\sigma =\theta _{*}$ circle. Let $\mathfrak{p}_{0}(\eta )   \subset $
$\mathfrak{p}(\eta)$ denote the complement of any such missing points. Thus,
$\mathfrak{p}_{0}(\eta )$ corresponds to a properly embedded, disjoint set of
paths in $\Gamma _{o}$ with two endpoints in total, these the
endpoints of the closure of $\eta $. By virtue of the second point
in the preceding step, the integral of $x$ over $\mathfrak{p}_{0}(\eta )$ is
also bounded in absolute value by a $j$ and $\eta $ independent
multiple of $\delta _{j}$. Moreover, with a suitable orientation,
the 1--form $x$ is positive on $\mathfrak{p}_{0}(\eta )$. Thus, the integral of
$x$ over any subset of $\mathfrak{p}_{0}(\eta )$ is also bounded by the same
multiple of $\delta _{j}$.

\substep{Step 6}
Let $V$ now denote the set of
components of the $| \theta -\theta _{*}|  > 0$ part of the
$s=s_{*}$ locus in $C_{j*}$. Let $\sigma _{j}$ denote $ \cup _{\eta\in V} \mathfrak{p}_{0}(\eta )$.
This is the image in $\Gamma _{o}$ via a
proper, piecewise smooth map of a finite, disjoint set of copies of
$S^1$ and $\mathbb{R}$. However, by virtue of what was said at the
end of the previous step, if an arc, $\gamma $, from $\Gamma_{o}$
is contained in $\sigma _{j}$, then the integral of $x$ over
$\gamma $ is bounded by a $j$--independent multiple of $\delta _{j}$.
As a consequence, no large $j$ version of $\sigma _{j}$ contains the
whole of any arc in $\Gamma _{o}$. Here is why: The integral of $x$
over an arc in $\Gamma _{o}$ gives the value on the arc of
$(C_{0j}, \phi _{j})$'s assigned point in the symplex $\Delta_{o}$.
Were $\gamma $ in $\sigma _{j}$, then this assigned point
would give an $\mathcal{O}(\delta _{j})$ value to $\gamma $, and
were such the case for an infinite set of large $j$ versions of
$(C_{0j}, \phi _{j})$, then the image of $\{(C_{0j}, \phi_{j})\}$ in $\Delta _{o}$ could not converge.

In the case $\Gamma _{o}=\underline {\Gamma }_{o}$, then $\sigma_{j}$
is compact. Since no arc is contained in $\sigma _{j}$, it
defines the zero homology class. According to the what was said in
Step 4, it therefore defines the zero homology class in the
$|\theta -\theta _{*}| < \delta $ neighborhood of $\Gamma _{o}$. At
the same time, $\sigma _{j}$ is homologous in the $| \theta -\theta_{*}|< \delta $
neighborhood of $\Gamma _{o}$ to the $s=s_{*}$
slice of $C_{j*}$ and Step 3 found that latter's class is
definitively not zero.  This nonsense proves that $\Gamma _{0}$
can not be compact.

Suppose now that $\Gamma _{o}$ is not the whole of $\underline
{\Gamma }_{o}$. In this case, the closure of $\sigma _{j}$ in
${\underline {\Gamma }^{*}}_{o}$ is the image via the collapsing map
of the image, ${\sigma ^{*}}_{j}$, of a map from a finite set of
circles into ${\underline {\Gamma }^{*}}_{o}$. Since this ${\sigma ^{*}}_{j}$ contains no arc inverse image from $\Gamma _{o}$, its
homology class must be a multiple of sums of those generated by the
vertex labeled loops from the set $\{\ell ^{*\upsilon }\}$. At the
same time, the discussion in Step 4 identifies $H_{1}(\underline
{\Gamma }^{*}; \mathbb{Z})$ with the first homology of the
$|\theta -\theta _{*}|<\delta $ neighborhood of $\Gamma _{o}$ and
with this understood any such ${\sigma ^{*}}_{j}$ is homologous modulo
the classes from $\{[\ell ^{*\upsilon }]\}$ to the $s=s_{*}$ slice
of $C_{j*}$. This then means that the $s=s_{*}$ slice of $C_{j*}$ is
homologous to some union of constant $|s| $ slices of the
$\theta=  \theta _{*}$ ends of $C_{0j}$. However, as noted in Step 3,
this is definitively not the case. This last bit of nonsense
completes the proof of \fullref{prop:7.1}.
\end{proof}

\step{Part 2}
This part of the subsection proves that $\Xi $ has
but a single element. The proof borrows much from the discussion in
Part 2 of \fullref{sec:4e}. In particular, it starts out just
the same by assuming that $\Xi $ has more than one element so as to
derive some patent nonsense. In this regard, note that the
conclusions from \fullref{prop:7.1} require that $\Xi $ have
at least one element that is not an $\mathbb{R}$--invariant cylinder
and the three steps that follow explain why there is at most one
element in $\Xi $ of this sort. Granted that only one subvariety
from $\Xi $ is not an $\mathbb{R}$--invariant cylinder the argument
for precluding $\mathbb{R}$--invariant cylinders from $\Xi $ can be
taken verbatim from Part 2 of \fullref{sec:4e}.

\substep{Step 1}
Choose an infinite subsequence from
$\{(C_{0j}, \phi _{j})\}$ with the following property: Let $o$
denote a multivalent vertex in $T$ and $\upsilon $ a vertex in
$\Gamma _{o}$. For each j, let $x_{j}$ denote the image in $C_{j}$
of the critical point of $\theta $ that corresponds to $\upsilon $.
Then either $\{x_{j}\}$ converges, or $\{s(x_{j})\}$ is unbounded
and strictly increasing or strictly decreasing. Here is the second
property: Agree henceforth to relable this subsequence by the
integers starting at 1.

Define as in Part 2 of \fullref{sec:4e} the subvariety
$\Sigma\subset   \mathbb{R}  \times (S^1  \times  S^2)$ to be the union of
the subvarieties from $\Xi $ that are not $\mathbb{R}$--invariant
cylinders. Let $Y \subset   \Sigma $ denote the critical points of
$\cos(\theta )$ on the irreducible components of $\Sigma $, the
singular points in the subvariety $ \cup _{(S,n) \in \Xi } S$,
and the limit points of the sequences $\{x_{j}\}$ as defined in the
previous paragraph. So defined, $Y$ is a finite set.

\substep{Step 2}
Suppose two subvarieties from $\Xi$ are not $\mathbb{R}$--invariant cylinders. To obtain nonsense from
this assumption, let $S$ and $S'$ denote distinct, subvarieties from
$\Xi $ that are not $\mathbb{R}$--invariant. The argument in the
fourth paragraph of Part 2 in \fullref{sec:4e} that ruled out
two non-$\mathbb{R}$ invariant subvarieties applies here to prove
that $S$ and $S'$ can be chosen to so that there is a value of $\theta$
that is taken simultaneously on $S$ and $S'$. No generality is lost
by choosing this angle to be in $(0, \pi )$, to be distinct from
$\theta $'s values on $Y$, and to be distinct from the angles that are
assigned to the vertices in
 $T$. This being the case, then for each $j$, there are points $z_{j}$ and
${z_{j}}'$ in $C_{0j}$ on which $\theta $ has this same value and such
that the sequence $\{z_{j}\}$ converges to a point in $S$ while
$\{{z_{j}}'\}$ converges to one in $S'$.

As $C_{0j}$ is irreducible, there is a path in $C_{0j}$ that runs from
$z_{j}$ to ${z_{j}}'$. For the present purposes, some paths are better than
others. In particular there exists $R > 1$ and for each $j$, such a path,
$\gamma _{j}$, with the following properties:

\itaubes{7.8}
\textsl{The image of $\gamma _{j}$  in $\mathbb{R}\times(S^1\times S^2)$  avoids the radius
$\frac{1}{R}$  balls about any point in $Y$.}

\item
\textsl{$|s| \le R$ on $\gamma _{j}$.}
\eit
The existence of $R$ and $\{\gamma _{j}\}$ is proved in the
subsequent two steps. Of course, granted \eqreft78, then $S$ and
 $S'$ must coincide because \fullref{prop:7.1} and
\cite[Proposition~3.7]{T3} put the whole of every large
 $j$ version of $\gamma _{j}$ very close to only one subvariety from $\Xi $.

\substep{Step 3}
The proof of the first point in \eqreft78
starts with the assertion that there exists $\varepsilon  > 0$ such that the
following is true:

\qtaubes{7.9}
\textsl{Let $o$ denote a multivalent vertex in $T$.  If $j$ is large, then no arc in the $C_{0j}$  version
of $\Gamma _{o}$  lies entirely in the inverse image of a single radius $\varepsilon $  ball in
$\mathbb{R}  \times (S^1  \times S^2)$.}
\endqtaubes

\noindent The first point of \eqreft78 follows directly from this assertion and
\cite[Lemma~3.10]{T3}.

To prove the assertion, suppose that $\gamma $ is an arc in $\Gamma_{o}$
whose image in $C_{j}$ is in a radius $\rho $ ball. The
integral of the 1--form $(1-\cos^{2}\theta _{o}) d\varphi -\surd 6\cos\theta _{o} dt$ over
$\gamma $ is then bounded by a fixed
multiple of $\rho $. Thus, if $\rho $ is small, so $(C_{0j}, \phi_{j})$'s
assigned point in the simplex $\Delta _{o}$ assigns a small
value to $\gamma $. Since the sequence $\{(C_{0j}, \phi _{j})\}$
has convergent image in $\Delta _{o}$, all such values enjoy a
positive, $j$--independent lower bound.

The proof of the second point in \eqreft78 requires the following assertion:

\qtaubes{7.10}
\textsl{There exists $R > 1$  such that when $j$ is large, then no arc in the
$C_{j}$  version of $\Gamma _{o}$  lies entirely where $|s| \ge  R$.}
\endqtaubes

\noindent To see why, note that by virtue of \fullref{prop:7.1}, if $R$ is such
that the $|s|    \ge R$ part of $\Sigma $ is contained in the ends
of $\Sigma $, then any arc in $\Gamma _{o}$ where $|s| $ is
everywhere larger than $R$ is very close to an
$\mathbb{R}$--invariant cylinder where $\theta =\theta _{o}$. Now, as
explained in Part 1, the closed 1--form $(1-\cos^{2}\theta _{o})
d\varphi -\surd 6\cos\theta _{o}$ can be written as $df$ on some
fixed neighborhood of such a cylinder with $f$ being the pull-back
from a neighorhood of the corresponding Reeb orbit of a smooth
function that vanishes on the Reeb orbit. This then means that the
integral of this 1--form over any $|s| \ge R$ arc is very small if
$R$ is very large and, with $R$ chosen, then $j$ is sufficiently
large. Thus, the point assigned to $(C_{0j}, \phi _{j})$ in
$\Delta _{o}$ gives a very small value to such an arc. Since the
assigned values enjoy a $j$--independent, positive lower bound, there
is an upper bound to the minimum value of $|s| $ on any arc from any
large $j$ version of $\Gamma _{o}$.

The assertion in \eqreft7{10} allows the large $j$ versions of
$\gamma _{j}$ to cross the $C_{0j}$ version of the locus $\Gamma
\subset $ where $|s| $ enjoys a $j$--independent upper bound.
\fullref{prop:7.1} then implies that the large $j$ versions
of $\gamma _{j}$ can be chosen so that $|s| $ also enjoys a
$j$--independent upper bound on the portions of $\gamma _{j}$ in the
$C_{0j}$ version of $C_{0}- \Gamma $. To elaborate on this, keep in
mind that \fullref{prop:7.1} finds a lower bound to $|s| $
on any large $j$ version of $\gamma _{j}$ at angles that are
uniformly bounded away from the $|s|    \to   \infty $ limits of
$\theta $ on the ends of $\Sigma $. To choose $\gamma _{j}$ with a
$j$--independent upper bound for $|s| $ as $\theta $ nears such a
limit, first set the stage by letting $\theta _{*}$ denote the angle
in question and write $(1-3\cos^{2}\theta _{*}) d\varphi-\surd 6\cos\theta _{*} dt$ as
$df_{*}$ where $f_{*}$ is the multivalued
function that is described in the fourth step in Part 1. As $f_{*}$
is constant on the $\theta =\theta _{* }$ Reeb orbits, so the ends of
$\Sigma $ where $\lim_{| s| \to \infty }  \theta =\theta _{*}$
account for only a finite set of values for $f_{*}$. Thus, $|s| $ on
$\gamma _{j}$ will enjoy a $j$--independent upper bound if $\gamma_{j}$
is chosen so that it approaches and crosses the $\theta=\theta _{*}$
locus where $f_{*}$ is uniformly far from its values
on the ends of $\Sigma $ where $\theta _{*}$ is the $|s|\to\infty $ limit of $\theta $.
Granted \eqreft7{10}, such a version
of $\gamma _{j}$ can be chosen in a component of the $C_{0j}$
version of $C_{0}- \Gamma $ using a parametrization as depicted in
\eqreft25. In this regard, the case where $f_{*}$ approaches a
constant as $\theta \to \theta _{*}$ on a given component can
be avoided since it occurs if and only if the part of the component
where $\theta $ is nearly $\theta _{*}$ is entirely in the large
$|s| $ part of some end of $C_{0j}$.

\step{Part 3}
Let $(S, n)$ denote the single element in $\Xi $. This
last part of the story explains why \eqreft75 is true. The story
here starts by making the following point: Because the image of
$\{(C_{0j}, \phi _{j})\}$ converges in $O_{T}/\Aut(T)$, the choices
in Part 2 of \fullref{sec:6c} can be made for each $(C_{0j},\phi _{j})$
so that the resulting sequence in $(\times _{o}  \Delta
_{o})  \times (\mathbb{R}_{ - }  \times   \mathbb{R}_{\diamondsuit
})  \times [\times _{o \in \mathcal{V} }  \mathbb{R}_{o}]$ also
converges. These choices are assumed in what follows.

Let $S_{0}$ denote the model curve for $S$ and let $\phi _{0}$
denote the tautological, almost everywhere 1--1 map from $S_{0}$ onto
 $S$ in $\mathbb{R}\times (S^1  \times  S^2)$. The complement of the
inverse image of $Y$ in $S_{0}$ is embedded in $\mathbb{R}  \times
(S^1  \times  S^2)$ and so has a well defined normal bundle, this
denoted by $N$ in what follows.

Now, fix $\varepsilon  > 0$ but very small and with the following
considerations: First, the various points in $Y$ are pairwise
separated by a distance that is much greater than $\varepsilon $.
View the latter distance as $\mathcal{O}(1)$ relative to
$\varepsilon $. Second, given a point  $p \in Y$, there is only one
value of $\theta $ on the set $Y$ that is also a value of $\theta $
on the radius $\varepsilon $ ball about $p$, this being $\theta (p)$.
Finally, the $|s| \ge 1/\varepsilon $ portion of $S_{0}$ is far
out in the ends of $S_{0}$ and is far from any point that maps to $Y$.

Given such $\varepsilon $, let $S_{\varepsilon }$ denote the portion
of the $|s| \le 1/\varepsilon $ part of $S_{0}$ that is mapped
by $\phi _{0}$ to where the distance from $Y$ is at least
$\varepsilon ^{4}$. There is a subdisk bundle, $N_{\varepsilon } \subset  N$,
over a submanifold of $S$ that has $S_{\varepsilon }$
in its interior and an exponential map $ e\co  N_{\varepsilon }   \to
\mathbb{R} \times (S^1\times  S^2)$ with the following properties:
First, it embeds $N_{\varepsilon }$ as a tubular neighborhood of
this larger submanifold. Second, it embeds each fiber as a
pseudoholomorphic disk. Finally, the function $\theta $ is constant
on each such fiber disk. In what follows, $N_{\varepsilon }$ is not
distinguished notationally from its image in $\mathbb{R}  \times
(S^1  \times S^2)$ via $e$.

If $\varepsilon $ is fixed in advance and $j$ is large, then $\phi_{j}(C_{0j})$
intersects $N_{\varepsilon }$ as an immersed
submanifold such that the composition of $\phi _{j}$ followed by the
projection to $S_{\varepsilon }$ defines a degree $n$, unramified
covering map from ${\phi _{j}}^{- 1}(N_{\varepsilon })$ to
$S_{\varepsilon }$. Let $\pi _{j}$ denote the latter map. Granted
that $\pi _{j}$ is a covering map, then \eqreft75 follows by
analyzing the behavior of $\phi _{j}(C_{0j})$ where $|s| \ge 1/\varepsilon $
and also where the distance to $Y$ is less than
$\varepsilon $. This analysis is presented in fourteen steps. The
first six describe the parts of the large $j$ versions of $C_{0j}$
that map to where the distance to $Y$ is less than $\varepsilon $.

\substep{Step 1}
Let $z$ denote a point in $S_{0}$
that maps to $Y_{*}$ and let $B \subset   \mathbb{R}  \times (S^1
\times  S^2)$ denote the radius $\varepsilon $ ball centered on $z$'s
image. Let $\theta _{*}$ denote the value of $\theta $ at $z$ and
assume here that $\theta _{*} > 0$. If $z$ is a critical point of
$\theta $, let $m$ denote $\deg_{z}(d\theta )$, otherwise set $m=0$.
Fix an embedded circle, $\nu \subset S_{0}$, around $z$ so that its
image lies in $B$, so that it intersects the $\theta =\theta _{*}$
locus transversally in $2m+2$ points,  and so that all points in
$\nu $ are mapped by $\phi _{0}$ to where the distance to $z$'s image
is greater than $\varepsilon ^{2}$. Let $\delta $ denote the maximum
of $| \theta -\theta _{*}| $ on $\nu $. It follows from
\eqref{eq2.11} that $\nu $ can be chosen so that $\delta    \le c(z)\cdot \varepsilon ^{2}$
where $c(z)$ depends only on $z$. With
$\varepsilon $ small, the image of $\nu $ in $S$ has $\theta$--preserving
preimages via $\pi _{j}$ in every large $j$ version of
$C_{0j}$. Let $\nu '$ denote such a circle. Let $k$ denote the degree
of $\pi _{i}$ on $\nu '$. The loop $\nu '$ intersects the
$\theta=\theta _{*}$ locus in $C_{0j}$ transversely in $k(2m+2)$ points.
The upcoming parts of the story explain why
this is possible when $\varepsilon $ is small and $j$ is large only
if $\nu'$ bounds an embedded disk in ${\phi _{j}}^{-1}(B)$ that contains a single
critical point of $\theta $, one where $d\theta $ vanishes with degree $m$.

\substep{Step 2}
A digression is need for some preliminary
constructions. For this purpose, fix attention on a component, $\mu $, of
the $\theta -\theta _{*} > 0$ portion of $\nu'$. The interior of
$\mu $ lies in a component of $C_{0j}$'s version of $C_{0}- \Gamma$.
After parametrizing the latter, $\mu $ can be viewed as an embedded path
in the corresponding parametrizing cylinder with both endpoints on the
$\theta =\theta _{*}$ circle but otherwise disjoint from this
circle. Remark next that $\mu $ is homotopic rel its endpoints in the
parametrizing cylinder to a path, $\mu _{*}$, on the $\theta =\theta _{*}$ circle.
In fact, the concatenation of $\mu $ and $\mu_{*}$
bounds a topologically embedded disk in the parametrizing
cylinder. Use $D_{\mu }$ to denote both the interior of this disk and its
image in $C_{j}$. If 0 $ \le   \theta -\theta _{*}$ on $\mu $,
then $0 < \theta -\theta _{*}<\delta $ on $D_{\mu }$,
and if $0\ge \theta -\theta _{*}$ on $\mu $, then $0 > \theta-\theta _{*}> -\delta $ on $D_{\mu }$.

As is explained next, the path $\mu _{*}$ must be contained in
${\phi_{j}}^{- 1}(B)$ when $\varepsilon $ is small and $j$ is large. To see
why, introduce the 1--form $x \equiv  (1-3\cos^{2}\theta _{*})
d\varphi -\surd 6 \cos\theta _{*} dt$. This form is positive on
$\mu _{*}$ and this leads to nonsense if $\varepsilon $ is small,
$j$ is large and $\mu _{*}$ exits $B$. To demonstrate the nonsense,
let $B_{1}$ denote the ball centered at $z$'s image in $\mathbb{R}
\times (S^1  \times  S^2)$ whose radius is the minimum of 1 and half
the distance between $z$'s image and any other point of $Y$. Introduce
the function $f$ on $B_{1}$ that obeys $df=x$ and that vanishes
on the image of $z$. The integral of $x$ over $\mu $ is the difference
between the values of $f$ on $\mu $'s endpoints, and so is bounded
in absolute value by a $j$--independent multiple of $\varepsilon $.
Since $\mu _{*}$ and $\mu $ are homotopic real boundary in the
parametrizing cylinder, the integral of $x$ over $\mu _{*}$ is
likewise bounded in absolute value by $\varepsilon $.

Keeping this in mind, note that $x$ is positive on the $\theta=\theta _{*}$
locus in $S \cap (B_{1}- B)$. In particular, when
$\varepsilon $ is small, there is a positive and $\varepsilon$--independent
number that is less than the integral of $x$ over each
such component. When $j$ is large, \fullref{prop:7.1}
requires that this same number serve as a lower bound to the
integral of $x$ in $C_{0j}   \cap   \phi ^{- 1}(B_{1}- B)$. Now,
if a large $j$ version of $\mu _{*}$ leaves ${\phi _{j}}^{- 1}(B)$,
then it must leave ${\phi _{j}}^{- 1}(B_{1})$ also since a large $j$
version of $\mu _{*}$ in ${\phi _{j}}^{- 1}(B_{1}- B)$ is a
$\theta $--preserving preimage of a portion of the $\theta =\theta_{*}$ locus in
$S_{\varepsilon }$. Thus, the
$\mathcal{O}(\varepsilon )$ integral of
 $x$ over the large $j$ versions of $\mu _{*}$ precludes it leaving ${\phi_{j}}^{- 1}(B)$.

There are $2k(m+1)$ versions of $\mu $, $\mu _{*}$ and $D_{\mu }$, one
for each component of the part of $\nu'$ where $| \theta -\theta_{*}|  > 0$.
Each such $\mu _{*}$ is in ${\phi_{j}}^{- 1}(B)$.

\substep{Step 3}
The fact that each $\mu _{*}$ is
in ${\phi _{j}}^{-1}(B)$ implies the following: If $\varepsilon  >0$
is chosen sufficently small, then any sufficiently large $j$
versions of $\nu'$ is homotopically trivial in $C_{0j}$. To
elaborate, note first that any loop in the $| \theta -\theta _{*}| < \delta $
part of $C_{0j}$ is homotopic to one that sits entirely
in the $\theta  =\theta _{*}$ locus except at points far out on
ends of $C_{0j}$ where the $|s|    \to   \infty $ limit of $\theta $
is $\theta _{*}$. The fact that any given version of $\mu _{*}$ lies
entirely in ${\phi _{j}}^{- 1}(B)$ implies that $\nu' $ is homotopic
to the loop in the $\theta =\theta _{*}$ circle that is obtained by
the evident concatenation of the $k(2m+2)$ versions of $\mu _{*}$.
This is a loop, $\nu _{*}$, that lives entirely in the $\theta =\theta _{*}$ locus.

If $\theta _{*}$ is not the angle of a multivalent vertex in $T$, then
$\nu _{*}$ lies in some component of $C_{0j}- \Gamma $. Let
 $e$ denote its labeling edge from $T$. If $\nu _{*}$ is not
null-homotopic, then the integral of $x \equiv  (1-3\cos^{2}\theta_{*})
d\varphi -\surd 6 \cos\theta _{*} dt$ over $\nu _{*}$ is a
non-zero, integer multiple of the $Q =Q_{e}$ version of $\alpha_{Q}(\theta _{*})$.
However, as $\phi _{j}(\nu ')$ sits in a small
ball, the integral of $x$ over $\nu'$ is therefore tiny if
$\varepsilon $ is small, and so the integral of $x$ over $\nu _{*}$
must be zero.

If $\theta _{*}$ is a multivalent vertex angle from $T$ and if $\nu_{*}$
is not null homotopic, then it must contain at least one arc
from some version of $\Gamma _{o}$. Such an arc thus sits in ${\phi_{j}}^{- 1}(B)$.
However, if an arc is mapped into $B$, then the
integral of $x \equiv  (1-3\cos^{2}\theta _{*}) d\varphi -\surd 6 \cos\theta _{*} dt$
over the arc is smaller than a
$j$--independent, constant multiple of $\varepsilon $. This means
that the arc is assigned a very small number by $(C_{0j}, \phi_{j})$'s image in
$\times _{o}  \Delta _{o}$ if $\varepsilon $ is
small and $j$ is large. However, such arc assignments enjoy a
$j$--independent, positive lower bound, and so there is no arc in $D'$
if $\varepsilon $ is too small and $j$ is large. Thus, $\nu _{*}$ is
null-homotopic and so is $\nu '$.

\substep{Step 4}
As each version $\mu _{*}$ is in
$\phi ^{- 1}(B)$, it follows that all version of $D_{\mu }$ approach
$\nu '$ from the same side, and therefore the union of the $2k(m+1)$
version of $D_{\mu }$ define an embedded disk in $C_{0j}$ with
boundary $\nu' $. Let $D'$ denote this disk. As $| \theta -\theta
_{*}| \le   \delta $, this disk must lie in $\phi ^{- 1}(B)$ when
$\delta $ is small and $j$ is large. Indeed, were $D'$ to exit $\phi
^{- 1}(B)$, it would have to intersect a $\theta $--preserving
preimage in $C_{0j}$ of an analog of $\nu $ around some other point
in $S_{0}$ that maps to the same point as $z$. Since $D'$ is embedded,
it would then contain this circle and thus contain a point where
$|\theta -\theta _{*}|> \delta $.

As constructed, the complement in $D'$ of the $\theta =\theta _{*}$ locus is the
union of the $2k(m+1)$ versions of $D_{\mu }$. Since all
$D_{\mu }$ are disks, this implies that the $\theta =\theta _{*}$ locus in
$D'$ is connected. Moreover, if either $m$ or $k$ is greater than 1,
then there are at least 4 such disks and so there is a critical point of
$\theta $ inside $D'$. If both $m$ and k equal 1, then there is no critical
point of $\theta $ in $D'$.

If $\varepsilon $ is small and $j$ is large, there can be at most
one such critical point in $D'$. To see this, note first that any two
$\theta $ critical points in $D'$ are joined in $D'$ by a constant
$\theta $ path since the $\theta =\theta _{*}$ locus in $D'$ is
connected. A constant $\theta $ path between two $\theta $ critical
points is a union of arcs in some $C_{0j}$ version of a graph from
the set $\{\Gamma _{(\cdot )}\}$. However, as noted in the previous
step, no such arc can be mapped by $\phi _{j}$ into $B$ if
$\varepsilon $ is small.

Granted that there is but one $\theta $ critical point in $D'$, then the
degree of vanishing of $d\theta $ there must be $k\cdot (m+1) - 1$ since
there are $2k(m+1)$ disk components to the complement of the $\theta =\theta _{*}$
locus in $D'$. Indeed, this follows from \eqref{eq2.11}.

\substep{Step 5}
If $n \ge 2$, then the loop
$\nu $ can have more than one $\theta $--preserving preimage. Suppose
$\nu '$ and $\nu''$ are two distinct $\theta $--preserving preimages
in a large $j$ version of $C_{0j}$. As is explained next, no path in
$C_{0j}$ from $\nu '$ to $\nu''$ is contained entirely in ${\phi_{j}}^{- 1}(B)$
if $\varepsilon $ is small and $j$ is taken very
large. To see why this is the case, note that were such a path to
exist, the variation of $\theta $ on it would be bounded by a
$j$--independent multiple of $\varepsilon $. As a consequence, $\nu'$ and $\nu''$
would have to intersect the same component of the
$\theta =\theta _{*}$ locus in $C_{0j}$ and so there would be a path
on this locus between them. Now, the disk $D'$ is contained in ${\phi_{j}}^{- 1}(B)$ and
$\nu''$ cannot be in $D'$ as the $\theta $ values
on $\nu''$ are identical to those on $\nu $. Thus this hypothetical
path from $\nu'$ to $\nu''$ would have to travel from $\nu'$ on
the portion of the $\theta =\theta _{*}$ locus that avoids $D'$. This
part of the locus is in the tubular neighborhood $N_{\varepsilon }$
and as it projects to the $\theta  =\theta _{*}$ locus in
$S_{\varepsilon }$, it can not hit a $\theta $--preserving preimage
of $\nu $ before it exits the larger ball $B_{1}$.

\substep{Step 6}
This step describes the large $j$
versions of $\phi _{j}(C_{0j})$ near points where $S$ interects the
$\theta  = 0$ locus. A similar story can be told for the points near
the $\theta =\pi $ locus. To start, suppose that $z \in S_{0}$ is a point
where $\theta $ is 0. In this case, choose a loop
$\nu $ in $\phi _{0}^{- 1}(B)$ that bounds a disk in $S_{0}$ whose
center is $z$. In this regard, choose $\nu $ so that the maximum value
of $\theta $ on $\nu $ is much less than that of $\theta $ on the
boundary of $S \cap B$. Take $j$ large and let $\nu ' \subset C_{0j}$
again denote a $\theta $--preserving preimage of $\nu $. As
before, this is an embedded circle. As can be seen using
\fullref{prop:7.1} and the $C_{0j}$ versions of the
parametrizations from \fullref{def:2.1}, the circle $\nu'$ is
embedded in a component of the $C_{0j}$ version of $C_{0}- \Gamma $
whose closure contains a point where $\theta  = 0$. As a
consequence, $\nu'$ bounds a disk in this closure on which the
maximum of $\theta $ is achieved on $\nu'$. Let $D'$ denote this
disk. Were $D'$ to leave ${\phi _{j}}^{- 1}(B)$, its $\phi _{j}$--image
would by necessity intersect the boundary of $B$ very near some
component of $S \cap B$. But, a large $j$ version of such a disk
would then have interior points where $\theta $ was larger than its
maximum on $\nu '$. Thus $D'$ is in ${\phi _{j}}^{- 1}(B)$.

Here is one more point: Let k denote the degree of the covering map
from $\nu'$ to $\nu $. If the local intersection number at $z$
between $S_{0}$ and the $\theta  = 0$ locus is denoted as
$q_{z}$, then the local intersection number between $D'$ and the
$\theta  = 0$ locus is $k\cdot q_{z}$.

\substep{Step 7}
Let $E \subset S_{0}$ denote an end, let $\theta _{*}$
denote the $|s|    \to   \infty $ limit of $\theta $ on $E$, and let
$n_{E}$ denote the integer that appears in $E$'s version of
\eqref{eq2.4}. Assume until told otherwise that $\theta _{*}   \in (0, \pi )$.
Take $R_{0} \gg  1$ so that $|s| $ takes the value
$R_{0}$ on $E$ and so that $R_{0}$ is much greater than the value of
$|s| $ on any point from the set $Y$. As can be seen from
\eqref{eq2.4}, it is also possible to choose $R_{0}$ so that the
$|s|    \ge R_{0}$ portion of the $\theta =\theta _{*}$ locus
consists of $2n_{E}$ properly embedded, half open arcs on which $ds$
restricts without zeros. Take $R_{0}$ so that this last condition
holds. In what follows, it is assumed that $1/\varepsilon $ is much
greater than $R_{0}$.

It follows from $E$'s version of \eqref{eq2.4} that when
$\varepsilon $ is sufficiently small, there exists an embedded loop, $\nu $,
in the $|s| \in [\frac{1 }{ {2\varepsilon }}+2, \frac{1 }{ \varepsilon }-2]$ part of
$E$ with the following
properties: First, $\nu $ is homotopic to the $|s|  =\frac{1 }{ {2\varepsilon }}$ slice of $E$.
Second, $\nu $ intersects the $\theta =\theta _{*}$ locus transversely, and in $2n_{E}$
points. Third, the maximum on $\nu $ of $| \theta -\theta _{*}| $ is less than $\varepsilon $.
Use $\delta $ in what follows to denote this maximum.

If $j$ is sufficiently large, then $\nu $ has $\theta $--preserving
preimages in $C_{j}$. Let $\nu'$ denote such a preimage. Thus, $\nu'$ also
intersects the $\theta =\theta _{*}$ locus transversely, and
the maximum of $| \theta -\theta _{*}| $
on $\nu '$ is also $\delta $. Finally, with k denoting the degree of the
covering map $\pi _{j}$ on $\nu '$, then $\nu' $ intersects the $\theta  =\theta _{*}$
locus $2kn_{E}$ times, each in a transversal
fashion. Note that by virtue of \fullref{prop:7.1}, the loop $\nu '$ is
homotopically non-trivial in the complement in $C_{0j}$ of the $\theta\in  \{0, \pi \}$ locus.

Let $\mu $ denote the closure of a component of the $| \theta-\theta _{*}|  > 0$
portion of $\nu'$. This $\mu $ can be
viewed as an embedded path in a parametrizing cylinder for the component of
the $C_{0j}$ version of $C_{0}- \Gamma $ that contains $\mu $'s
interior. Viewed in the parametrizing cylinder, $\mu $ lies in the interior
save for its two boundary points. In the parametrizing cylinder, $\mu $ is
homotopic rel boundary to a path, $\mu _{*}$, that runs between the
two endpoints of $\mu $ on the $\theta =\theta _{*}$ circle.
This is to say that the concatenation of $\mu $ and $\mu _{*}$
bounds a topological disk in the parametrizing cylinder. Let $D_{\mu }$
denote the interior of this disk, and also its image in $C_{0j}$. As before,
$0 < \theta -\theta _{*}<\delta $ on the disk $D_{\mu }$ when
$0 \le   \theta -\theta _{*}$ on $\mu $, and $0 > \theta-\theta _{*}> \delta $ on
$D_{\mu }$ when $0 \ge   \theta-\theta _{*}$ on $\mu $.

The parametrizing cylinder that contains $\mu _{*}$ may or may not
have missing points on its $\theta =\theta _{*}$ circle. In the
case that some such points lie in $\mu _{*}$, let $\mu _{*0}   \subset   \mu _{*}$
denote their complement. The latter has
a corresponding image in $C_{0j}$. The concatenation of $\mu _{*0}$ and $\mu $ is a
piecewise smooth, properly embedded submanifold
of $C_{0j}$ that bounds the closure of the disk $D_{\mu }$. In this regard,
the complement of the $\theta $--critical points in $\mu _{*0}$ is
smoothly embedded, and a component of this complement that lacks an endpoint
of $\mu $ is the whole of the interior of an arc in some graph from the
$C_{0j}$ version of the collection $\{\Gamma _{(\cdot )}\}$

\substep{Step 8}
When $\varepsilon $ is small, and with $\varepsilon $ chosen,
 $j$ is sufficiently large, the disk $D_{\mu }$ lies entirely in the
$|s| > \frac{1 }{ {2\varepsilon }}$ portion of $C_{0j}$. To see
why, note first that if $R_{0}$ is large, then the
$|s|    \ge R_{0}$ part of any end in $S_{0}$ where $\lim_{| s| \to\infty }  \theta =\theta _{*}$
lies in a small radius tubular
neighborhood of a $\theta =\theta _{*}$ pseudoholomorphic cylinder.
In this tubular neighborhood, the 1--form
$x \equiv (1-3\cos^{2}\theta_{*}) d\varphi -\surd 6 \cos\theta _{*}dt$
is the pull-back by the
projection to $S^1\times S^{2}$ of an exact form, $df$, where
$f$ is a function that vanishes on the Reeb orbit and is constant on
any nearby $\theta  =\theta _{*}$ Reeb orbit. This understood,
the integral of $x$ along any part of the $|s| \ge \frac{1 }{{2\varepsilon }}> R_{0}+3$
portion of the $\theta =\theta _{*}$
locus in any given end of $S_{0}$ is very small in absolute value
with the bound going to zero as $\varepsilon \to  0$. On the
other hand, there exists $\kappa  > 0$ with the following
significance: If $\gamma $ is a connected portion of the $\theta=\theta _{*}$ locus
in an end of $S_{0}$, and if $\gamma $ runs
from where $|s|  =\frac{1 }{ {2\varepsilon }}$ to $|s|  =R_{0}+1$,
then the integral of $x$ over $\gamma $ is greater than $\kappa $.

If $\varepsilon $ is sufficiently small, then the integral of $x$ over $\mu $
will be less than $\kappa $ since the integral is the difference between the
values of $f$ at the two endpoints. This then means that the integral over
$\mu _{*}$ of the pull-back of $x$ to the parametrizing cylinder is
also less than $\kappa $. In this regard, note that $x$ pulls back as $\alpha
_{Q}(\theta _{*})dv$ where $Q$ is the integer pair that is
associated to the edge label from $\mu $'s component of $C_{0}- \Gamma $.

Now, if $j$ is very large and $|s| \le \frac{1 }{{2\varepsilon }}$ on $\mu _{*0}$,
then this part of $\mu _{*0}$ must
be a $\theta $--preserving preimage in $C_{0j}$ of a component of the
$\theta =\theta _{*}$ locus in $S_{0}$ in an end of $S_{0}$
where $\lim_{| s| \to \infty }  \theta =\theta _{*}$. As all such
components run from where $|s| > \frac{1 }{ {2\varepsilon }}$ to
where $|s|  =R_{0}$, so must $\mu _{*}$. Granted the conclusions
of the preceding paragraph, \fullref{prop:7.1} demands that
$x$ have integral greater than $\kappa $ on $x_{*}$. Hence,
$\mu_{*0}$ can not enter the $|s|    \le   \frac{1 }{{2\varepsilon }}$ part of $C_{0j}$.
This implies that the disk $D_{\mu }$ is also forbidden from the
$|s|    \le   \frac{1 }{{2\varepsilon }}$ part of $C_{0j}$.

\substep{Step 9}
There are $2k\cdot n_{E}$ versions each of $\mu $, $\mu _{*}$ and $D_{\mu }$. Since
$|s|  > \frac{1 }{{2\varepsilon }}$ on all versions of $\mu _{*0}$, they all leave
$\nu' $ from the same side of $\nu' $; and this implies that the
closure in $C_{0j}$ of the union of the $2k\cdot n_{E}$ versions of
$D_{\mu }$ is a submanifold of $C_{0j}$ with a piecewise smooth
boundary. One boundary component is $\nu '$, and were there more,
they would sit inside the union of the $2k\cdot n_{E}$ versions of
$\mu _{*0}$. However, when $\frac{1 }{ {2\varepsilon }}$ is large
and $j$ also, then $\nu '$ is the only boundary component. Indeed,
the last remarks of Step 7 imply that any second component would
necessarily contain the whole interior of an arc in some graph from
the $C_{0j}$ version of the collection $\Gamma _{0}$. By virtue of
\fullref{prop:7.1}, this arc would sit very close to the end
$E$ when $\varepsilon $ is small and when $j$ is very large. As
explained in the preceding step, taking $\varepsilon $ small and $j$
large makes the integral of the 1--form $x$ over such an arc less
than any given, positive number. But the integral of $x$ over an arc
from $C_{0j}$'s version of $\{\Gamma _{(\cdot )}\}$ enjoys a
$j$--independent, positive lower bound since the image of the
sequence $\{(C_{0j}, \phi _{j})\}$ in $\times _{o}\Delta _{o}$
converges.

Let $C'$ denote the closure of the union of these $2k\cdot n_{E}$ versions
of $D_{\mu }$. When $\varepsilon $ is small and $j$ is large, this
smooth submanifold of $C_{0j}$ is homeomorphic to the complement of
the origin in the closed unit disk; thus a half open cylinder with
boundary $\nu'$. Indeed, $C'$ is not a disk because $\nu'$ is
homologically non-trivial in the complement of the $\theta  = 0$ and
$\theta =\pi $ loci. Here is why $C'$ has but one puncture: Each
puncture corresponds to an end of $C_{0j}$ where the $|s|    \to
\infty $ limit of $\theta $ is $\theta _{*}$. Thus, each corresponds
to a vertex on a graph from $C_{0j}$'s version of $\{\Gamma _{(\cdot)}\}$.
Since $\theta $ is nearly $\theta _{*}$ on $C'$ and $C'$ is
connected, all such vertices are on the same graph. Let $o$ denote
the corresponding vertex. Since the complement of the $\theta=\theta _{*}$
locus in $C'$ is a union of disks, the part of $\Gamma_{o}$ in $C'$ has connected
closure in $\Gamma _{o}$. Thus, any two
vertices in $\Gamma _{o}$ that label ends in $C'$ are joined by an arc
in $\Gamma _{o}$ whose interior lies entirely in $C'$. As argued in
the preceding paragraph, there are no such arcs when $\varepsilon $
and $j$ are large.

The argument just used explains why $C'$ has no critical points of $\theta $
on $C_{0j}$.

\substep{Step 10}
The conclusions of the preceding
step imply that the end $E \subset S_{0}$ where the $|s| \to \infty$
limit of $\theta $ is in $(0, \pi )$ corresponds to a set of ends
of $C_{0j}$ that are all very close to $E$ in $\mathbb{R} \times
(S^1\times  S^2)$. This collection is in 1--1 correspondence with the
$\theta $--preserving preimages of $\nu $ in the sense that each
pre-image bounds a properly embedded, half open cylinder in $C_{0j}$
whose large $|s| $ part coincides with the large $|s| $ part of its
corresponding end. If $\nu'$ is a $\theta $--preserving preimage of
$\nu $, let $E' \subset C_{0j}$ denote the corresponding end. If $k$
is the degree of the projection from $\upsilon '$ to $\upsilon $,
then the integer $n_{E'}$ in the $E'$  version of \eqref{eq2.4} is
k$\cdot n_{E}$.

Consider next the behavior of $C_{0j}$ near an end $E \subset S_{0}$
where the $|s|    \to   \infty $ limit of $\theta $ is 0. The
pair $(p, p')$ are used in what follows to designate the integers that appear
in $E$'s version of \eqref{eq1.9}.

To start story in this case, take $R_{0} \gg  1$ so that $|s| $
takes the value $R_{0}$ on $E$ and so that $R_{0}$ is much greater
than the value of $|s| $ on any point in $Y$ and any $\theta  = 0$
point in $S$. In particular, choose $R_{0}$ so that \eqref{eq1.9}
describes the $|s|    \ge R_{0}$ portion of $E$. It is assumed here that $\varepsilon $ is
such that $\frac{1 }{ {2\varepsilon }} \gg  R_{0}$.

Now let $\nu $ denote a constant $\theta $ slice of $E$ that lives
where $|s|    \in [\frac{1 }{ {2\varepsilon }}+2, \frac{1 }{\varepsilon }-2]$.
Let $\delta  > 0$ denote the value of $\theta $
on $\nu $. As before, $\delta < \varepsilon $ when $\varepsilon $
is small. When $j$ is very large, the circle $\nu $ has $\theta$--preserving
preimages in $C_{0j}$. Let $\nu'$ denote such a
preimage, and let $k$ denote the degree of the restriction of $\pi_{j}$ as a map from $\nu'$ to $\nu $.

As can be seen using \fullref{prop:7.1} and the $C_{0j}$
versions of the parametrizations from \fullref{def:2.1}, any
small $\varepsilon $ and large $j$ version of $\nu'$ bounds a
cylinder in $C_{0j}$ on which $\theta $ is strictly positive but
limits to zero as $|s|    \to   \infty $. As such, the large $|s| $
part of this cylinder is the large $|s| $ part of an end of $C_{0j}$
whose associated integer pair is $(kp, kp')$. A cylinder in $C_{0j}$
of the sort just described is denoted in what follows by $C'$.

The part of $C_{0j}$ that maps near an end of $S$ where the $|s| \to
\infty $ limit of $\theta $ is $\pi $ looks much the same as the
description just given for the part near an end of $S$ where the
$|s|    \to   \infty $ limit of $\theta $ is 0. To summarize: Each
end of $S_{0}$ where $\lim_{| s| \to \infty }  \theta $ is either 0
or $\pi $ corresponds to one or more ends of each large $j$ version
of $C_{0j}$. If the given end of $S_{0}$ is characterized by the
4--tuple $(\delta =  \pm 1, \varepsilon =\pm , (p, p'))$,
then each of the associated ends of $C_{0j}$ is characterized by a
4--tuple with the same $\delta $ and $\varepsilon $, and with an
integer pair that is some positive multiple of $(p, p')$. Moreover,
these multiples from the associated ends to each $S_{0}$ end add up
to the integer $n$.

\substep{Step 11}
The step considers assertion in
\eqreft75 for the case that the integer $n$ is 1. To argue this
case, let $T_{S}$ denote the graph that is assigned to the pair
$(S_{0}, \phi )$. Granted that $T_{S}$ is isomorphic to $T$, it
follows that $(S_{0}, \phi )$ defines a point in
${\mathcal{M}_{\hat{A},T}}$ and \fullref{prop:7.1} asserts
that $\{(C_{0j}, \phi _{j})\}$ converges to $(S_{0}, \phi )$ in
${\mathcal{M}_{\hat{A},T}}$. As noted in \fullref{thm:1.3}, this
means that $\{(C_{0j}, \phi _{j})\}$ converges to $(S_{0}, \phi)$ in
${\mathcal{M}^{*}}_{\hat{A},T}$ which is the desired conclusion.

With the preceding understood, what follows explains why $T_{S}$ is
isomorphic to $T$. The explanation starts with a summary of results
from the 10 steps just completed. The first point here is that the
ends of $S_{0}$ and those of the large $j$ versions of $C_{0j}$
enjoy a 1--1 correspondence whereby corresponding pairs determine the
same 4--tuple and map very near each other in
$\mathbb{R}  \times (S^1 \times  S^2)$ when $j$ is large. The second point is that the
respective sets of $\theta =0$ points in $S_{0}$ and in the
large $j$ versions of $C_{0j}$ enjoy a 1--1 correspondence whereby
corresponding pairs have identical local intersection numbers and
map very near each other in this locus when $j$ is large. A similar
correspondence exists between the respective sets of $\theta  =\pi $
points in $S_{0}$ and in $C_{0j}$. Granted these two points,
it follows that $(S_{0}, \phi )   \in   \mathcal{M}_{\hat{A}}$.

In the case that the integer $n$ is 1, the set of non-extremal
critical points of $\theta $ on $S_{0}$ enjoy a 1--1 correspondence
with the analogous set in any large $j$ version of $C_{0j}$. This
correspondence is such that partnered critical points have the same
critical value and the same degree of vanishing of $d\theta $.
Moreover, corresponding critical points map very close to each other
when $j$ is large. Granted these conclusions and those of the first
paragraph of this Step 11, it follows that the respective vertices
in the graph $T_{S}$ and in the graph $T$ enjoy a 1--1 correspondence
that partners pairs with equal angle.

To compare the edges of the graphs $T_{S}$ and $T$, remark that by
virtue of \fullref{prop:7.1}, the components of $S_{0}-\Gamma $
enjoy a 1--1 correspondence between those of any large $j$
version of $C_{0j}- \Gamma $ that pairs components that map very
close to each other in $\mathbb{R}  \times (S^1  \times  S^2)$.
Moreover, the respective ranges of $\theta $ on paired components
are identical, and the respective integrals of $\frac{1 }{ {2\pi }}dt$ and of
$\frac{1 }{ {2\pi }}d\varphi $ about the constant $\theta$
slices of paired components are equal. It follows from this that
the edges of $T_{S}$ and $T$ enjoy a 1--1 correspondence that is
consistent with the aforementioned vertex correspondence and
preserves integer pair assignments.

It remains now to consider the respective $T_{S}$ and $T$ versions
of the graphs $\{\underline {\Gamma }_{(\cdot )}\}$ that are
assigned to paired, multivalent vertices. For this purpose, let $o$
denote a multivalent vertex in $T$ and also its partner in $T_{S}$.
The correspondence between the respective sets of non-extremal
critical points of $\theta $ in $S_{0}$ and in $C_{0j}$ together
with that between the respective sets of ends in $S_{0}$ and in
$C_{0j}$ defines a 1--1 correspondence between the vertices in the
$S_{0}$ and $C_{0j}$ versions of $\Gamma _{o}$. Moreover, the latter
correspondence pairs vertices with the same labels and with the same
number of incident half-arcs.

The conclusions of Step 4 imply that any compact portion of the
interior of any arc in the $S_{0}$ version of $\Gamma _{o}$ has a
unique $\theta $--preserving preimage in any sufficiently large $j$
version of $C_{0j}$. This correspondence pairs the arcs in the
$S_{0}$ and $C_{0j}$ versions of $\Gamma _{o}$ so that partnered
arcs have the same edge pair labels. The conclusions of Steps 4 and
9 together imply that this arc correspondence is consistent with the
just described pairing of the vertices of $\underline {\Gamma}_{o}$.
As a result, the two versions of $\Gamma _{o}$ are
isomorphic via an isomorphism that respects the correspondences
described previously between the respective edge and vertex sets of
$T_{S}$ and those of $T$.

Taken together, these correspondences describe the desired isomorphism
between $T_{S}$ and $T$.

\substep{Step 12}
This step establishes the $n>1$
cases of \eqreft75. To start the argument for this case, fix
$\varepsilon $ to be very small and note that each large $j$ version
of $\pi _{j}$ defines ${\phi _{j}}^{- 1}(N_{\varepsilon })$ as a
degree $n$, proper covering space over $S_{\varepsilon }$. The first
task in this step is to explain why the sequence whose $j$'th
component is the supremum over ${\phi _{j}}^{- 1}(N_{\varepsilon })$
of the ratio $| \bar {\partial }\pi _{j}| /| \partial \pi _{j}| $
limits to zero as $j \to   \infty $. To see this, let $z \in S_{\varepsilon }$.
As explained in the first point of Step 2 in
\fullref{sec:5c}, there is a neighborhood of $\phi (z)$ in
$\mathbb{R}  \times (S^1  \times S^2)$ with complex valued
coordinates $(x, y)$ and a disk $D \subset S_{\varepsilon }$ with
center $z$ such that the $y = 0$ slice is the $\phi $--image of $D$, the
$x =$ constant slices are the fibers over $D$ of the bundle
$N_{\varepsilon }$, and the 1--forms in \eqref{eq5.2} span
$T^{1,0}(\mathbb{R}  \times (S^1  \times  S^2))$ over the coordinate
patch. With $D$ identified via $\phi $ with the $y = 0$ slice, the
function $x$ restricts as a holomorphic coordinate on D, and $\pi_{j}$
on ${\pi _{j}}^{- 1}(D)$ is the composition $x \circ \phi _{j}$.
This understood, it follows that the ratio $| \bar {\partial }\pi_{j}| /|\partial \pi _{j}| $
on ${\pi _{j}}^{- 1}(D)$ is the ratio $| \bar {\partial }x| /| \partial x|  =| \sigma | $.
As $\sigma $ vanishes where $y = 0$, \fullref{prop:7.1}
implies that the sequence whose $j$'th element is the supremum over
${\pi _{j}}^{- 1}(D)$ of $| \sigma | $ converges to zero as $j \to \infty $.

The assertion that the sequence whose $j$'th element is the supremum
over ${\pi _{j}}^{- 1}(S_{\varepsilon })$ of $| \bar {\partial }\pi
_{j}| /| \partial \pi _{j}| $ limits to zero as $j \to   \infty $
follows from this analysis on disks by virtue of the fact that
$S_{\varepsilon }$ is compact.

The next task is to extend each large $j$ version of $\pi _{j}$ as a
degree $n$, ramified cover of the whole of $C_{0j}$ to $S_{0}$ so that
\begin{equation}\label{eq7.11}
\lim_{j \to \infty } \sup_{C_{0j} } \bigl(| \bar {\partial }\pi
_{j}| /| \partial \pi _{j}|\bigr) = 0.
\end{equation}
To do so, remark that the complement in $C_{0j}$ of ${\phi _{j}}^{-1}(N_{\varepsilon })$
consists of a disjoint union of disks and
cylinders. The set of disks is partitioned into subsets that are
labeled by the disk components of $S_{0}- S_{\varepsilon }$.
Likewise, the set of cylinders is partitioned into subsets that are
labeled by the cylindrical components of $S_{0}- S_{\varepsilon }$.

What with the conclusions from Steps 9 and 10, the desired extension
of each large $j$ version of $\pi _{j}$ over the cylindrical
components of $C_{0j}- {\pi _{j}}^{- 1}(S_{\varepsilon })$ is
obtained by copying in an almost verbatim fashion the construction
that is described in Step 4 of \fullref{sec:5c}. In this regard,
note that this construction extends $\pi _{j}$ as an \underline{unramified}
cover over the cylindrical components of $C_{0j}- {\pi_{j}}^{- 1}(S_{\varepsilon })$.

The map $\pi _{j}$ is extended over the disk components of $C_{0j}-\phi _{j}(N_{\varepsilon })$
momentarily. Granted that this has
been done, here is how to complete the argument for \eqreft75.
Remark first that by virtue of \eqref{eq7.11}, the large $j$
versions of $\pi _{j}$ can be used to pull-back the complex
structure from $S_{0}$ and so define a new complex structure on
$C_{0j}$ that makes $\pi _{j}$ into a holomorphic map. Let $S_{nj}$
denote $C_{0j}$ with this new complex structure. Note that the pair
$(S_{n,j}, \phi \circ \pi _{j})$ defines an equivalence class in
${\mathcal{M}^{*}}_{\hat{A},T}$. Furthermore, the points that are
defined by the pairs $(S_{0j}, \phi  \circ \pi _{j})$ and
$(C_{0j}, \phi _{j})$ are very close in
${\mathcal{M}^{*}}_{\hat{A},T}$. Indeed, to see this, take $\psi $
in \eqref{eq1.24} to be the identity map. \fullref{prop:7.1}
asserts that $\phi  \circ \pi _{j}$ and $\phi _{j} \circ \psi $ are
very close when $j$ is large. Meanwhile, $r(\psi )=| \bar {\partial
}\pi _{j}| /\partial \pi _{j}| $, and this is very small when $j$
is large by virtue of \eqref{eq7.11}.

To finish the argument, note that the ramification points for the
map $\pi _{j}$ converge in $S_{0}$ as $j \to   \infty $ since none
occur in the cylindrical components of $S_{0}- S_{\varepsilon }$.
Thus, the sequence of complex curves $\{S_{nj}\}$ has a subsequence
that converges to a complex curve. Let $S_{n}$ denote the latter.
The corresponding subsequence of holomorphic maps from the sequence
$\{\pi _{j}\}$ likewise converges to a degree $n$, holomorphic,
ramified covering, $\pi \co  S_{n}   \to S_{0}$. This pair $(S_{n},\phi  \circ \pi )$
defines an equivalence class in
${\mathcal{M}^{*}}_{\hat{A},T}$ that is a limit of the sequence that
is defined by the pairs $\{(S_{0j}, \phi \circ \pi _{j})\}$.
Granted what was said in the preceding paragraph, the point defined
by $(S_{n}, \phi  \circ \pi )$ is necessarily the limit of the
sequence $\{(C_{0j}, \phi _{j})\}$.

The final two steps explain how $\pi _{j}$ is extended as a branched
cover over the disk components of each large $j$ version of $C_{0j}-
{\pi _{j}}^{- 1}(S_{\varepsilon })$ so as to satisfy \eqref{eq7.11}.

\substep{Step 13}
To start, let $z$ denote the center
point of a disk from $S_{0}- S_{\varepsilon }$ and suppose first
that $\theta (z)$ is neither 0 nor $\pi $. Let $B$ denote the radius
$\varepsilon $ ball centered at $z$. Introduce the function $r$ on $B$
as defined in \eqref{eq2.9}. As noted in the ensuing discussion,
\begin{equation}\label{eq7.12}
dr = J\cdot d\theta  + r d\theta
\end{equation}
on $B$ where $|r|    \le  c\cdot \dist(\cdot , \phi (z))$
with $c$ being a constant.

Now, let $D \subset S_{0}$ denote a disk with center $z$ whose
boundary is in $S_{\varepsilon }$ and whose $\phi $ image is in $B$.
According to Property 5 from \fullref{sec:2b}, if $\varepsilon $
is small, there is a holomorphic coordinate on $D$ such that the
function $r+i\theta $ defined on $B$ pulls back as indicated in
\eqref{eq2.11} with $\theta _{*}$ set equal to $\theta (z)$ and with
$m \ge 0$ denoting $\deg_{z}(d\theta )$. This understood, let
$x \equiv   \theta + ir - \theta (z)$. As a consequence
of \eqref{eq2.11}, $\tau    \equiv  x^{1 / (m + 1)}$ defines a class
$C^{1}$, complex valued coordinate on D when $\varepsilon $ is small.
Moreover, by virtue of \eqref{eq7.12},
\begin{equation}\label{eq7.13}
|\bar {\partial }\tau | \le c \rho   | \partial \tau |
\end{equation}
at the points whose image in $B$ has distance $\rho $ or less from $\phi (z)$.
Here, c is independent of $\rho $ when $\rho $ is small.

To continue, let $D' \subset C_{0j}$ denote the disk that bounds a
given component of ${\pi _{j}}^{- 1}(\partial D)$ and let $k$ denote the
degree of $\pi _{j}$ as a map from $\partial D'$ to $\partial D$.
When $j$ is large, $D'$ contains a single disk component of
$C_{0j}-{\phi _{j}}^{- 1}(N_{\varepsilon })$ and is mapped into B by $\phi_{j}$.
Also, $D'$ contains inside it a single critical point of
$\theta $ on $C_{0j}$. The aforementioned Property 5 in
\fullref{sec:2b} provides $D'$ a holomorphic coordinate such that
\eqref{eq2.11} holds with $k\cdot (m+1)-1$ replacing $m$. This
understood, there is a complex valued constant, $x_{j}$, such that
$x-x_{j}$ has a $k\cdot (m+1)$ fold root on $D'$. This is to say that
$\tau '  \equiv  (x- x_{j})^{1 / k(m + 1)}$ defines a $C^{1}$,
complex valued function on $D'$ when $\varepsilon $ is small and $j$
is large. In fact, it follows from \eqref{eq7.12} and the $D'$ version
of \eqref{eq2.11} that $\tau '$ defines a $C^{1}$, complex coordinate
on $D'$ and that \eqref{eq7.13} holds with $\tau '$ replacing $\tau $.

Let $\theta _{j}$ be such that the value of $\theta $ at the $\tau ' = 0$
point in $D'$ is $\theta (z)+\theta _{j}$. thus, $\theta _{j} = 0$ in
the case that $m> 0$. Next, define $\psi \co  D' \to  D$ by declaring that
\begin{equation}\label{eq7.14}
\psi \ast \tau  - \theta _{j}  =\tau '^{k}.
\end{equation}
This $C^{1}$ map realizes $D'$ as a degree $k$ branched cover over $D$ with
a single branch point. Moreover, this map restricts near $\partial D'$
so as to have the following property: The maps $\phi \circ \psi$
and $\phi _{j}$ send any given point to the same constant $\theta$,
pseudoholomorphic submanifold in $B$. Indeed, such is the case by
virtue of the fact that $r$ as well as $\theta $ are constant on any
$\theta =$  constant submanifold in $\mathbb{R} \times (S^1  \times
S^2)$. It follows from this last conclusion that $\tau '$ can be
changed via multiplication by a $k$'th root of unity so that it agrees
with $\pi _{j}$ near $\partial D'$. Thus, $\psi $ extends $\pi _{j}$
over $D'$ as a $C^{1}$, ramified cover.

Here is one last point about this extension of $\pi _{j}$: By virtue
of \eqref{eq7.13} and its $\tau '$ analog, $| \bar {\partial }\psi | \ll | \partial \psi | $
over the whole of $D'$, and the sequence
whose $j$'th element is the supremum of $| \bar {\partial }\psi | /|\partial \psi | $
over the $j$'th version of $D'$ has limit zero as $j\to   \infty $.
Since $\psi $ is differentiable and smooth save at
$\psi ^{- 1}(z)$, it has a deformation that extends $\pi _{j}$ over
 $D'$ as a smooth map with $| \bar {\partial }\pi _{j}| \ll |
\partial \pi _{j}| $ at all points. Moreover, these extensions can
be made for each large $j$ version of $\psi $ so that the resulting
sequence of supremums of $| \bar {\partial }\pi _{j}| /| \partial \pi _{j}| $
has limit zero as $j \to \infty$.

\substep{Step 14}
To finish the story about $\pi_{j}$'s extension to the disk components of
$C_{0j}- \pi _{j}^{-1}(S_{\varepsilon })$, suppose now that $z$ is a point in a disk
component of $S_{0}- S_{\varepsilon }$ where $\theta $ is zero. Let
B denote the radius $\varepsilon $ ball in $\mathbb{R}  \times (S^1
\times S^2)$ centered at $\phi (z)$. The ball $B$ has smooth complex
coordinates $(x, y)$ where
\begin{equation}\label{eq7.15}
x = \sin\theta  \exp\biggl(-i\biggl(\varphi -\frac{{\surd 6\cos \theta } }{
{(1 - 3\cos ^2\theta )}}\hat {t}\biggr)\biggr) \text{ and } y = s - i\hat {t};
\end{equation}
here $\hat {t}$ is the $\mathbb{R}$--valued lift of the
function $t$ to $B$ that vanishes at $\phi (z).$ These coordinates are
such that each $x=$ constant disk is pseudoholomorphic and the $x =0$
disk lies entirely in the $\theta  = 0$ cylinder. In this regard,
there is a complex valued function $\sigma $ on B that vanishes at
$x = 0$ and is such that $dx + \sigma d\bar {x}$ spans
$T^{1,0}(\mathbb{R}  \times (S^1  \times  S^2))$ over the whole of $B$.

Now, let $D \subset S_{0}$ denote a disk centered at $z$ with
boundary in $S_{\varepsilon }$ that is mapped by $\phi $ into $B$. Let
$q$ denote the intersection number between $D$ and the $\theta  = 0$
locus. Since $\phi $ is holomorphic, there is a holomorphic
coordinate, $u$, on $D$ such that $x$ pulls back via $\phi $ as
$u^{q}+\mathcal{O}(|u| ^{q + 1})$. Thus, $\tau    \equiv  x^{1 / q}$
defines a $C^{1}$, complex coordinate on $D$. Moreover, $\tau $ obeys
\eqref{eq7.13} where the constant $c$ is furnished by the expression
for $x$ in \eqref{eq7.15}.

Now fix some very large $j$ and let $D' \subset C_{0j}$ denote a
disk that bounds a component of ${\pi _{j}}^{- 1}(\partial D)$. Let $k$
denote the degree of $\pi _{j}$'s restriction to $\partial D'$. Applying
the argument just given to $D'$ finds that $\tau '  \equiv  x^{1 /
kq}$ defines a $C^{1}$, complex coordinate on $D'$ that obeys its own
version of \eqref{eq7.13} with the constant $c$ used for the $\tau $
version. Granted all of this, define a map, $\psi \co  D' \to  D$ by
requiring that $\psi *\tau =\tau '^{k}$. Arguing as in the case
where $\theta (z)  \ne  0$ finds that such a map $\psi $ can be
modified by multiplication by a $k$'th root of unity so as to provide
a $C^{1}$ extension of $\pi _{j}$ over $D'$. Moreover, this extension
has a smooth perturbation as a degree $k$ ramified cover with
$| \bar{\partial }\pi _{j}| \ll | \partial \pi _{j}| $. Finally, these
extensions can be made for each large $j$ so that the sequence whose
$j$'th element is the supremum of $| \bar {\partial }\pi _{j}| /|\partial \pi _{j}| $
limits to zero as $j \to   \infty $.


\setcounter{section}{7}
\setcounter{equation}{0}
\section{The strata and their classification}\label{sec:8}

This section completes the story started in \fullref{sec:5} by describing in more
detail the strata of ${\mathcal{M}^{* }}_{\hat{A}}$. The following
is a brief summary of the contents: The first subsection describes each
stratum as a fiber bundle over a product of simplices whose typical fiber is
some ${\mathcal{M}^{* }}_{\hat{A},T}$. This result is stated as
\fullref{prop:8.1}. \fullref{thm:6.2} and \fullref{prop:8.1}
thus give an explicit picture of any given component of any given stratum in
${\mathcal{M}^{* }}_{\hat{A}}$.

Meanwhile, \fullref{sec:8b} describes necessary and sufficient conditions on a
graph $T$ that insure a non-empty ${\mathcal{M}^{* }}_{\hat{A},T}$.
These are stated as \fullref{prop:8.2}. The final subsection proves
\fullref{prop:6.1}, this the assertion that the homotopy type of a
graph $T$ arises from at
most one component of at most one stratum of ${\mathcal{M}^{* }}_{\hat{A}}$.
Together, Propositions~\ref{prop:6.1} and~\ref{prop:8.2} provide a complete
classification of the components of the strata that comprise ${\mathcal{M}^{* }}_{\hat{A}}$.

\subsection{The structure of a stratum}\label{sec:8a}

As in Sections~\ref{sec:5a}, let $\mathcal{S} _{B,c,\mathfrak{d}}$
denote a stratum of ${\mathcal{M}^{* }}_{\hat{A}}$ and let
$\mathcal{S} \subset\mathcal{S} _{B,c,\mathfrak{d}}$ denote a
connected component. A graph of the sort introduced in
\fullref{sec:6a} from any equivalence class in $\mathcal{S} $
has some $m$ distinct, multivalent vertex angles that do not
arise via \eqref{eq1.8} from any $(0,+,\ldots)$ element in
$\hat{A}$ nor from an element in $B$. This understood, a
function, $\mathfrak p$, from $\mathcal{S} $ to the $m$'th symmetric product
of $(0, \pi )$ is defined as follows: The value of $\mathfrak p$ on the
equivalence class defined by $(C_{0}, \phi )$ is the set of
$\theta $--values of the $m$ compact but singular $\theta $ level
sets in $C_{0}$. If non-empty, then the inverse image via $\mathfrak p$ of
any given point in the $m$'th symmetric product of $(0,\pi )$ is
a version of ${\mathcal{M}^{* }}_{\hat{A},T}$ where $T$ is has
precisely $m$ distinct, multivalent vertex angles that do not
come via \eqref{eq1.8} from an integer pair of any $(0,+,\ldots)$
element in $\hat{A}$ nor any element in $B$.

With the preceding understood, consider:

\begin{proposition}\label{prop:8.1}

If non-empty, then $\mathcal{S} $ is fibered by $\mathfrak{p}$ over a product of open simplices.
\end{proposition}

\begin{proof}[Proof of  \fullref{prop:8.1}]
It follows from \fullref{lem:5.4} with \eqref{eq2.4},
\eqref{eq2.11} and the implicit function theorem that the map $\mathfrak{p}$ is a submersion.
Granted this, then \fullref{thm:6.2} implies that $\mathfrak{p}$ fibers $\mathcal{S} $ over its
image in the $\mathfrak{m}$'th symmetric product of $(0, \pi )$.

To picture the image of $\mathcal{S} $ via $\mathfrak p$, return to the definition in
\fullref{sec:5a} of $\mathcal{S} _{B,c,\mathfrak{d}}$.
The definition involved the subspace $\mathcal{S}_{B,c } \subset{\mathcal{M}^{* }}_{\hat{A}}$
whose elements are defined by pairs $(C_{0},\phi )$ where the following
two conditions are met: First, $C_{0}$ has precisely $c$ critical points of
$\theta $ where $\theta $ is neither 0 nor $\pi $. Second, the ends in
$C_{0}$ that correspond to elements in $B$ are the sole convex side ends of
$C_{0}$ where the $|s|\to\infty $ limit of $\theta $ is
neither 0 nor $\pi $ and whose version of \eqref{eq2.4} has a strictly positive
integer $n_{E}$. Letting $d = N_{ + }+|B| +c$ and $I_{d}$ the
$d$'th symmetric product of $(0, \pi )$, then $\mathcal{S} _{B,c,\mathfrak{d}}$ consists of
the inverse image of the stratum in $I_{d}$ indexed by d via the map that
assigns to a given $(C_{0}, \phi )$ the angles of the critical points in
$(0, \pi )$ of $\theta $ as well as the angles in $(0, \pi )$ that are
$|s|\to\infty $ limits of $\theta $ on the concave side
ends in $C_{0}$ and the ends that correspond to 4--tuples from $B$. As in
\fullref{sec:5a}, denote the map from $\mathcal{S} _{b,c}$ to $I_{d}$ as $f$.

Only $m$ angles in the image of $f$ can vary and their values
define the map $\mathfrak p$. To say more, let $\Lambda _{ + ,B}$ denote
the angles that are defined via \eqref{eq1.8} by the integer
pairs from the $(0,+,\ldots)$ elements in $\hat{A}$ and from the
elements in $B$. The complement of $\Lambda _{ + ,B}$ in $(0, \pi
)$ is a union of open segments. The image of $f$ consists of the
set $\Lambda _{ + ,B}$ and then $m$ angles that are distributed
$(0, \pi )-\Lambda _{ + ,B}$. Each of the latter angles can vary
as a function on $\mathcal{S} $, but only in a single component
of $(0, \pi )-\Lambda _{ + ,B}$. However, keep in mind that two
such angles in the same component can not coincide. Furthermore,
the angles from $f$ that lie in a given component of $(0, \pi
)-\Lambda _{+ ,B}$ need not sweep out the whole component as
functions on $\mathcal{S} $. The picture just drawn implies that
the image of $\mathfrak p$ is a product of simplices.

As is explained below, there are additional constraints on the range of
variation of the angles from $f$.
\end{proof}

\subsection{An existence theorem}\label{sec:8b}

A component of any given stratum of ${\mathcal{M}^{* }}_{\hat{A}}$ has an
associated homotopy type of graph. Though yet to be proved,
\fullref{prop:6.1} asserts that these homotopy types classify the components of
the strata. The next proposition gives necessary and sufficient conditions
for a given homotopy type of graph to arise as the label of a component of
some stratum of ${\mathcal{M}^{* }}_{\hat{A}}$. In this
proposition and subsequently, the angle that is assigned to a given vertex $o$
of a graph $T$ is denoted by $\theta _{o}$. As before, when $e$ is an edge of
$T$, then $e$'s assigned integer pair is denoted by $Q_{e}$ or $(q_{e}, {q_{e}}')$.

\begin{proposition} \label{prop:8.2}

Let $T$ denote a graph of the sort that is described in
\fullref{sec:6a}. Then the space ${\mathcal{M}^{* }}_{\hat{A},T}$ is
non-empty if and only if the following conditions are met: Let $e$ denote an
edge of $T$ and let $o$ and $o'$ denote the vertices at the ends of $e$ where the
convention taken has $\theta _{o' }>\theta _{o}$. Then the $Q= Q_{e}$ version of
$\alpha _{Q}$ is positive on $(\theta _{o},\theta _{o' })$, and it is zero
at an endpoint if and only if the
corresponding vertex is monovalent and labeled by a $(0,-,\ldots)$ element from
$\hat{A}$.

\end{proposition}

To comment on these conditions, remark that if $e$ is an edge,
then its integer pair, $Q_{e}$, may or may not define an angle
via \eqref{eq1.8}. In any event, at least one of $\pm Q_{e}$
defines such an angle. In the case that $Q_{e}$ defines an angle,
denote the latter as $\theta _{e}$, and in the case that $-Q_{e}$
defines an angle, denote the latter by $\theta _{ -e}$. These
angles are the zeros of the function $\alpha _{Q}$. Moreover,
since the derivative of $\alpha _{Q}$ is positive at $\theta
_{e}$ and negative at $\theta _{ - e}$, the following conditions
are necessary and sufficient for $\alpha _{Q}$ to be positive on
$(\theta _{o}, \theta_{o}')$:

\itaubes{8.1}
\textsl{$\theta _{o}>\theta _{e}$ if only $\theta _{e}$ is defined.}

\item
\textsl{$\theta _{o' }<\theta _{ - e}$ if only $\theta _{ - e}$ is defined.}

\item
\textsl{$\theta _{e}<\theta _{o}<\theta _{o' }<\theta _{ -
e}$ if both $\theta _{e}$ and $\theta _{ - e}$ are defined and $\theta _{e}>\theta _{ - e}$.}

\item
\textsl{Either $\theta _{o' }<\theta _{ - e}$ or $\theta _{o}>\theta _{e}$ if both
$\theta _{e}$ and $\theta _{ - e}$ are defined and $\theta _{ - e}<\theta _{e}$.}
\end{itemize}
\noindent In this regard, note that $\theta _{ - e}$ is not defined when
$q_{e} > 0$ and $q_{e}/| {q_{e}}'| <\sqrt {\frac{3}{2}} $. On the
otherhand, $\theta _{e}$ is not defined when $q_{e}< 0$ and
$-q_{e}/| {q_{e}}'| <\sqrt {\frac{3}{2}} $. If both are defined,
then $\theta _{e}<\theta _{ - e}$ if ${q_{e}}'$ is positive, and
$\theta _{e}>\theta _{ - e}$ if ${q_{e}}'$ is negative.

Note that angle inequalities can be interpreted directly in terms
of integer pairs. To elaborate, suppose that $\theta _{Q}$ is
defined via \eqref{eq1.8} from an integer pair $(q, q')$ and that
$\theta _{P}$ is defined from another integer pair, $(p, p')$.
Then

\itaubes{8.2}
\textsl{If $p$ is negative and $q'$ is positive, then $\theta _{P}>\theta _{Q}$.}

\item
\textsl{If $p'$ and $q'$ have the same sign, then $\theta _{P}>\theta _{Q}$ if $p'q - pq' < 0$.}
\end{itemize}
\noindent With \eqreft82 in hand, the conditions in \eqreft81 can be written directly in terms
of the data given in $\hat{A}$ and $T$.

\begin{proof}[Proof of  \fullref{prop:8.2}]
The beginning of \fullref{sec:2a} explains why
the stated conditions are necessary. The strategy to prove that the stated
conditions are sufficient is as follows: A given graph $T$ is approximated by
a sequence of graphs, all mutually homotopic, and chosen so that the
corresponding versions of ${\mathcal{M}^{* }}_{\hat{A},(\cdot )}$
are nonempty. A sequence is constructed whose $j$'th element is from the $j$'th
version of ${\mathcal{M}^{* }}_{\hat{A},(\cdot )}$. The latter
sequence is then seen to converge to an element in ${\mathcal{M}^{* }}_{\hat{A},T}$.
The four steps that follow give the details. In
this regard, note that the convergence arguments are very much the same as
those in \fullref{sec:7d} and so only the novel points are noted.

\substep{Step 1}
To start, say that a graph $T$ of the sort that is described in
\fullref{sec:6a} is generic when it has the following
properties: All multivalent vertices are either bivalent or
trivalent, the trivalent vertex angles are pairwise distinct, and
they are distinct from all bivalent vertex angles. \cite[Theorem~1.3]{T3}
asserts that ${\mathcal{M}^{* }}_{\hat{A},T}$ is nonempty when $T$ is
generic and obeys the stated conditions in
\fullref{prop:8.2}.

\substep{Step 2}
Let $T$ denote a graph from \fullref{sec:6a} and
let $o \in T$ denote a bivalent vertex. Consider modifying $T$ by replacing
the graph $\underline {\Gamma }_{o}$ by a `less valent' graph,
${\underline {\Gamma }_{o}}'$. This is done as follows: Suppose first that
$\underline {\Gamma }_{o}$ has a vertex with valency greater than 4. Let $\upsilon $
denote the latter. The new graph, ${\underline {\Gamma }_{o}}'$, is obtained
from $\underline {\Gamma }_{o}-\upsilon $ by attaching the now
dangling incident half-arcs that are incident to $\upsilon $ to the vertices
in a graph with two vertices and one arc between them. Three of the dangling
half-arcs are attached to one of these two vertices and the remaining
half-arcs are attached to the other. Now care must be taken here with the
choice of the first three so as to insure that the arcs in the new graph
have consistent labels by pairs of edges. In particular, this is done as
follows: Take the first incident half-arc to be oriented as an incoming arc,
and let $(e, e')$ denote its edge pair label. The second incident half-arc
must be the arc that follows the first in $\ell _{oe}$. Note that it must
also be distinct from the first arc. As a consequence, its pair label has
the form $(e, e'')$. This second incident half-arc is outgoing. The third
incident half-arc should be the arc that follows the first on $\ell_{oe' }$. I
t must also be distinct arc from the first arc. Thus,
its edge label has the form $(e''', e')$. It too is outgoing. Note that the
case $e''' = e$ is allowed. The new arc that runs between the two new vertices
is oriented as an incoming half-arc and labeled by the edge pair $(e''',e'')$.
One of the two new vertices should be labeled with the integer 0, the
other with the integer that labels $\upsilon $. Note that there is a
completely analogous construction that has all arc orientations reversed.
The operation just described can be performed on any vertex with valency
greater than four.

Suppose next that $\underline {\Gamma }_{o}$ has only 4--valent and
bivalent vertices, but at least one 4--valent vertex with a non-zero integer
label. Let $\upsilon $ denote one of the latter vertices. In this case,
${\underline {\Gamma }_{o}}'$ is obtained from $\underline {\Gamma
}_{o}-\upsilon $ in the manner just described. One of the
vertices in the new graph is bivalent and the other is 4--valent. The
bivalent vertex should be given $\upsilon $'s integer label and the other
should be labeled with 0.
\end{proof}

Consider now:

\begin{lemma}\label{lem:8.3}
Suppose that $T'$ is obtained from $T$ by modifying one version of
$\underline {\Gamma }_{(\cdot )}$ as just described. Then
${\mathcal{M}^{* }}_{\hat{A},T}$ is nonempty if and only if
${\mathcal{M}^{* }}_{\hat{A},T'}$ is nonempty.
\end{lemma}

\begin{proof}[Proof of  \fullref{lem:8.3}]
Suppose first that ${\mathcal{M}^{* }}_{\hat{A},T' }$ is nonempty. Let $o \in T$ denote the
vertex involved and $\upsilon $ the vertex in $\underline {\Gamma }_{o}$.
Consider a sequence, $\{\lambda _{j}\}$, in the $T'$ version of the
space in \eqref{eq6.15} whose coordinates are $j$--independent but for the coordinate
in the $T'$ simplex $\Delta _{o}'$. In the latter, the coordinate for the
arc between the two vertices in ${\underline {\Gamma }_{o}}'$ that replace
$\upsilon $ should converge to zero. The remaining arc coordinates should
converge as $j  \to\infty $ so as to define a point in $\Delta _{o}$.
Use a $j$--independent value in $\mathbb{R}$ and the image of $\{\lambda_{j}\}$ in
$O_{T' }/\Aut (T')$ to define a sequence of equivalence
classes in ${\mathcal{M}^{* }}_{\hat{A},T' }$. The
discussion in \fullref{sec:7d} can be repeated now with a minor modification to
prove that the sequence in ${\mathcal{M}^{* }}_{\hat{A},T'}$ converges to an
element in ${\mathcal{M}^{* }}_{\hat{A},T}$.
The salient modification replaces \eqreft79 and \eqreft7{10} with the assertion that
only the one arc in ${\underline {\Gamma }_{o}}'$ that does not come from
$\underline {\Gamma }_{o}-\upsilon $ can correspond to an arc in
$C_{0j}$ that either lies entirely in a radius $\varepsilon $ ball or where $|s|\ge R$.
\end{proof}

Suppose next that ${\mathcal{M}^{* }}_{\hat{A},T}$ is nonempty.
The fact that ${\mathcal{M}^{* }}_{\hat{A},T' }$ is
nonempty follows using \fullref{lem:5.4} with \eqref{eq2.4}, \eqref{eq2.11} and the implicit
function theorem.

\substep{Step 3}
Suppose now that each version from $T$ of
$\underline {\Gamma }_{(\cdot )}$ has only bivalent or 4--valent vertices
and that all 4--valent vertices are labeled by zero. Thus, no modifications
as described in Part 2 are possible. However, suppose that $o$ is a
multivalent vertex in $T$ and that $\underline {\Gamma }_{o}$ contains two
or more vertices with one being a 4--valent vertex. In this case, $T$ is
modified to produce a graph $T'$ as follows: To start, let
$\upsilon\in \underline {\Gamma }_{o}$ denote a 4--valent vertex in
$\underline {\Gamma}_{o}$. The arc segments that are incident to $\upsilon $ are
labeled by pairs of edges, but there are either 3 or 4 edges in total involved. In any
case, two connect $o$ to respective vertices whose angles are either both
smaller or both larger than $\theta _{o}$. What follows assumes the
former; the argument in the latter case is omitted since it differs only
cosmetically.

Let $e$ and $e'$ denote the two edges that connect $o$ to vertices with larger
angle. The complement in $T$ of the interiors of $e$ and $e'$ is disconnected,
with three components, these denoted as $T_{e}$, $T_{e' }$ and $T_{- }$.
The graph $T_{ - }$ contains o, while $T_{e}$ contains the vertex
opposite $o$ on the edge $e$ and $T_{e' }$ contains the vertex opposite
 $o$ on the edge $e'$. The graph $T'$ is the union of $T_{e}$, $T_{e' }$,
$T_{ - }$ and a trivalent graph, $Y$, with one vertex and three edges. The two
vertices at the top of the $Y$ are identified in $T'$ with the respective $e$ and
 $e'$ vertices in $T_{e}$ and $T_{e' }$. The vertex at the bottom of
the $Y$ is identified in $T'$ with the vertex $o$ in $T_{ - }$. The vertex at the
center of the $Y$ is denoted by $o'$, and its angle, $\theta _{o' }$,
is slightly larger than $\theta _{o}$. The corresponding graph,
$\underline {\Gamma }_{o' }$, is a figure 8 where the two small circles are
the versions of $\ell _{o' (\cdot )}$ that are labeled by the two
top edges of the $Y$ graph. Meanwhile, the loop that traces the figure 8 is
the version of $\ell _{o' (\cdot )}$ that is labeled by the bottom
edge in the $Y$ graph.

Let $\hat{o}$ denote the vertex in $T'$ that corresponds to $o \in T_{ - }$.
The vertex angle of $\hat{o}$ is that of $o$, thus $\theta _{o}$. The graph
$\underline {\Gamma }_{\hat{o}}$ is obtained from $\underline {\Gamma }_{o}-\upsilon $
as follows: Attach the incoming arc
segment to $\upsilon $ with label $e$ to the outgoing arc segment with label
$e'$. Likewise, attach the incoming arc segment with label $e'$ to that outgoing
arc segment with label $e$. Then, replace $e$ and $e'$ in all arc labels by the
label of the bottom edge in the figure $Y$.

As can be readily verified, a graph $T'$ as just described satisfies the
conditions in \fullref{prop:8.2} if $T$ does, and if the vertex angle
$\theta_{o' }$ is sufficiently close to $\theta _{o}$. This understood,
consider:

\begin{lemma}\label{lem:8.4}
Suppose that $T'$ is a graph as just described, with $\theta _{o' }$
very close to $\theta _{o}$. Then ${\mathcal{M}^{* }}_{\hat{A},T}$
is nonempty if and only if ${\mathcal{M}^{* }}_{\hat{A},T' }$ is nonempty.
\end{lemma}
\noindent This lemma is proved momentarily.

\fullref{prop:8.2} is a corollary to Lemmas~\ref{lem:8.3} and~\ref{lem:8.4}
together with \cite[Theorem~1.3]{T3}. The reason is that any given $T$ can be
sequentially modifed as described first by \fullref{lem:8.3} and then as in
\fullref{lem:8.4} so as to obtain a graph that is generic in the sense that is used by Step 1.

\substep{Step 4}
This step contains the following proof.

\begin{proof}[Proof of \fullref{lem:8.4}]
Suppose first that ${\mathcal{M}^{* }}_{\hat{A},T'}$ is nonempty when $\theta _{o' }$
is sufficiently close to $\theta _{o}$. Take a sequence of angles
$\{\theta _{j}\}_{j = 1,2,\ldots }$ that converges to $\theta_{o}$ from above with
each very close to $\theta _{o}$. Let
$\{T_{j}\}_{j = 1,2,\ldots }$ denote a corresponding sequence of
graphs where the $j$'th version is $T'$ with $\theta _{o' }=\theta_{j}$.
The plan is to define a corresponding sequence in ${\mathcal{M}^{* }}_{\hat{A}}$
whose $j$'th element is a point in the $T'$ = $T_{j}$
version of $\mathcal{M}^{\ast }_{\hat{A},T' }$ by using
\fullref{thm:6.2} and a point in the $T_{j}$ version of the space in \eqref{eq6.15}. For
this purpose, it is necessary to first make the choices that are described
in Parts 1 and 2 of \fullref{sec:6c}. Since the $T_{j}$'s are pairwise homotopic,
these choices can be made for all at once. Granted that this is done, choose
a corresponding sequence, $\{\lambda _{j}\}$, with $\lambda _{j}$ a
point in the $T_{j}$ version of the space in \eqref{eq6.15}. This sequence should be
chosen so that all of the factors are independent of the index $j$. Now define
a corresponding sequence in ${\mathcal{M}^{* }}_{\hat{A}}$ whose
$j$'th element is in the $T' = T_{j}$ version of ${\mathcal{M}^{* }}_{\hat{A},T' }$
and is obtained from $\lambda _{j}$
using the $T_{j}$ version of \fullref{thm:6.2} with some $j$--independent choice for
the $\mathbb{R}$ factor. Arguments that are much like those in \fullref{sec:7d} can
be employed to prove that such a sequence converges in ${\mathcal{M}^{* }}_{\hat{A}}$
and that the limit is in some ${\mathcal{M}^{* }}_{\hat{A},T''}$ where $T''$ is a graph of a
rather special sort. In particular, if $T''$ is not isomorphic to T, then it
has a vertex, $\hat{o}$, with angle $\theta _{o}$ such that the replacement of
$\underline {\Gamma }_{\hat{o}}$ with $\underline {\Gamma }_{o}$
makes a graph that is isomorphic to $T$. As is explained next, a careful
choice for the constant sequence $\{\lambda _{j}\}$ gives a version of
$T''$ that is isomorphic to $T$.

Care must be taken only with the coordinates of $\lambda _{j}$ in the
$\mathbb{R}_{\hat{o}} \times \Delta _{\hat{o}}$ and
$\mathbb{R}_{o' }$ factors in \eqref{eq6.15}. To specify the latter, return
to the construction of the graph $\underline {\Gamma }_{\hat{o}}$
in $T'$ from $\underline {\Gamma }_{o}$. Let $\hat{e}$ denote the bottom edge of
the $Y$--graph portion of $T'$. There is a map from $\underline {\Gamma
}_{\hat{o}}$ to $\underline {\Gamma }_{o}$ that is 1--1 except for
two points on $\ell _{\hat{o}\hat{e} }$ that are both mapped
to the vertex $\upsilon $. Let $\pi $ denote the latter map. By assumption,
there is another vertex besides $\upsilon $ in $\ell _{oe}\cup\ell_{oe' }$.
For the sake of argument, suppose $\ell _{oe}$ has a
second vertex, then let $\upsilon _{0}$ denote the vertex on $\ell_{oe}$
that starts the arc in $\ell _{oe}$ that ends at $\nu $.
Meanwhile, the abstract version of the loop $\ell _{o' \hat{e} }$ from the
figure 8 graph $\underline {\Gamma }_{o' }$ has
two vertices. Let $\upsilon _{1}$ and $\upsilon _{2}$ denote the latter
where the convention has the arc that starts at $\upsilon _{1}$ and ends
at $\upsilon _{2}$ mapping to $\ell _{o' ,e}$ in $\underline {\Gamma }_{o' }$.

The next step to choosing the coordinates of $\lambda _{j}$ requires the
introduction of the map from $\Delta _{o}$ to $\Delta _{\hat{o}' }$
that is defined so that  $r \in\Delta _{o}$ sends
$\gamma\subset \underline {\Gamma }_{\hat{o}}$ to $\sum
_{\gamma '  \subset \pi (\gamma )}r(\gamma ')$. Fix
$r\in\Delta _{o}$ and use its image under this map for $\lambda_{j}$'s
factor in $\Delta _{\hat{o}}$. Also, fix $\tau\in \mathbb{R}_{\hat{o}}$.
To define the coordinate of $\lambda _{j}$
in the $\mathbb{R}_{o}$ factor, observe first that the concatenating path set
for the loop $\ell _{\hat{o}\hat{e} }$ in $\underline {\Gamma}_{\hat{o}}$
can be used to assign a value in $\mathbb{R}$ to the
vertex $\pi ^{-1}(\upsilon _{0})$ from any given $\tau\in \mathbb{R}_{\hat{o}}$.
This value is defined by starting with $\tau $
and adding or subtracting suitable multiples of the values given by the
image of $r$ in $\Delta _{\hat{o}}$ to the arcs on a certain path
from the distinguished vertex in $\underline {\Gamma }_{\hat{o}}$
to $\pi ^{-1}(\upsilon _{0})$. In particular, the path uses the
concatenating path set to get to $\ell _{\hat{o}\hat{e} }$
and then proceeds in the oriented direction on $\ell _{\hat{o}\hat{e} }$
to the vertex $\pi ^{-1}(\upsilon _{0})$. All
of this is done so as to be compatible with the parametrizing algorithm as
described in \fullref{sec:2}. Let $\tau _{0}$ denote the $\mathbb{R}$ value defined
in this way from the pair $(\tau , r)$. Now let $\tau _{1}$ denote the
result of adding to $\tau _{0}$ the value that the image of $\mathbb{R}$
assigns to the arc in $\ell _{\hat{o}\hat{e} }$ that starts
at $\pi ^{-1}(\upsilon _{0})$. Use $\frac{1}{2}(\tau
_{0}+\tau _{1})$ for $\lambda _{j}$'s coordinate in $\mathbb{R}_{o' }$.

Arguments that are much like those used in \fullref{sec:7d} prove that the
sequence $\{\lambda _{j}\}$ as just described defines a sequence in
${\mathcal{M}^{* }}_{\hat{A}}$ whose limit is in ${\mathcal{M}^{* }}_{\hat{A},T}$.
The proof that ${\mathcal{M}^{* }}_{\hat{A},T' }$ is nonempty if
${\mathcal{M}^{* }}_{\hat{A},T}$ is obtained using \fullref{lem:5.4}
with \eqref{eq2.4}, \eqref{eq2.11} and the implicit function theorem.
\end{proof}

\subsection{Proof of Proposition \ref{prop:6.1}}\label{sec:8c}

In order to prove \fullref{prop:6.1}, it is necessary to return to the milieu
of \fullref{prop:8.1} and obtain a more refined picture of the image of the map
$\mathfrak{p}$. For this purpose, remember that there are $\mathfrak{m}$ angles in the image of the
map $f$ that can vary on a given stratum component $\mathcal{S} $. These angles
are distributed amongst the various components of $(0, \pi )-\Lambda _{ + ,B}$.
If $T$ and $T'$ are homotopic graphs, and if $o \in T$ and $o' \in T'$
are corresponding vertices, then either $\theta _{o} = \theta _{o' }$ and this
angle is in $\Lambda _{ + ,B}$, or
else $\theta _{o}$ and $\theta _{o'  }$ are in the same component
of $(0, \pi )-\Lambda _{ + ,B}$.

With this last point in mind, suppose that $T_{1}$ and $T_{2}$ are homotopic
graphs with the following property: Let $\theta _{1}\in  (0, \pi)-\Lambda _{ + ,B}$
and suppose that there are multivalent
vertices in $T_{1}$ with angle $\theta _{1}$. Let $V \subset T_{1}$
denote the latter set, and also the corresponding set in $T_{2}$. Suppose
that the vertices in the $T_{2}$ version of $V$ are assigned angle
$\theta_{2}>\theta _{1}$. In addition, assume that both the $T = T_{1}$ and
 $T = T_{2}$ versions of ${\mathcal{M}^{* }}_{\hat{A},T}$ are
nonempty. For each angle $\theta  \in (\theta _{1},\theta _{2})$, let $T_{\theta }$
denote the version of $T$ that is obtained
from $T_{1}$ by assigning the angle $\theta $ to the vertices in $V$. Thus,
each $T_{\theta }$ is homotopic to $T$.

\fullref{prop:6.1} is now a consequence of the following:

\begin{lemma}\label{lem:8.5}
If both the $T_{1}$ and $T_{2}$ versions of ${\mathcal{M}^{* }}_{\hat{A},T}$
are nonempty, then such is the case for each $T = T_{\theta }$ version in the case that
$\theta\in  [\theta _{1}, \theta _{2}]$.
\end{lemma}

\begin{proof}[Proof of  \fullref{lem:8.5}]
Consider the following scenario:

\begin{scenario}\label{scn:1}

An angle $\sigma $ lies in $\theta _{1}$'s component of $(0, \pi)-\Lambda _{ + ,B}$,
the $T = T_{\sigma }$ version of ${\mathcal{M}^{* }}_{\hat{A},T}$ is empty, but all
$T = T_{\theta }$ versions of ${\mathcal{M}^{* }}_{\hat{A},T}$ are nonempty when
$\theta\in  [\theta _{1}, \sigma )$.

\end{scenario}
\noindent Note that \fullref{prop:8.1} finds some such $\sigma $ when the $T = T_{1}$
version of ${\mathcal{M}^{* }}_{\hat{A},T}$ is nonempty. As is
explained momentarily, Scenario~\ref{scn:1} occurs if and only if $\sigma $ coincides
with $\theta _{ - e}$ in the case that $e$ is an edge that connects a vertex
from the $T_{1}$ version of $V$ to a vertex with angle less than $\theta_{1}$.

Here is a second scenario:

\begin{scenario}\label{scn:2}
An angle $\sigma '$ lies in $\theta _{1}$'s component of $(0, \pi)-\Lambda _{ + ,B}$, the
$T = T_{\sigma ' }$ version of ${\mathcal{M}^{* }}_{\hat{A},T}$ is empty, but all
$T = T_{\theta }$ versions of ${\mathcal{M}^{* }}_{\hat{A},T}$ are nonempty when
$\theta\in (\sigma ', \theta_{2}]$.

\end{scenario}
\noindent A cosmetic modification to the arguments given below to prove the assertion
just made about Scenario~\ref{scn:1} prove the following: Scenario~\ref{scn:2} occurs if and
only if $\sigma '$ coincides with $\theta _{e}$ in the case that $e$ is an
edge that connects a vertex from the $T_{2}$ version of $V$ to a vertex with
angle greater than $\theta _{2}$.

Note that Scenario~\ref{scn:1} precludes Scenario~\ref{scn:2}, and vice versa.
Indeed, were both scenarios to occur, then both $T_{1}$ and $T_{2}$ would be in violation of
the conditions in \fullref{prop:8.2}. \fullref{lem:8.5} is a consequence of this fact
that the two scenarios cannot both occur.

To explain the assertion about Scenario~\ref{scn:1}, note first that if
$\sigma  = \theta _{ - e}$ with $e$ as described, then the $T = T_{\sigma }$ version
of ${\mathcal{M}^{* }}_{\hat{A},T}$ is empty because $T_{\sigma }$
violates the conditions stated in \fullref{prop:8.2}. Suppose next that $\sigma$
is some as yet undistinguished angle that gives Scenario~\ref{scn:1}. By virtue of
the fact that the various $\theta\in  [\theta _{1}, \sigma )$
versions of $T_{\theta }$ differ only in their vertex angle labels, the
choices that are made in Parts 1 and 2 of \fullref{sec:6c} can be made in a
$\theta $--independent fashion. This then identifies all $T = T_{\theta }$
versions of the space that is depicted in \eqref{eq6.15}. Fix some element in this
space, $\lambda $, and a real number, $s_{0}$. Next, let $T = T_{\theta }$
and use $(s_{0}, \lambda )$ in $\mathbb{R}\times O_{T}/\Aut (T)$ to
define via \fullref{thm:6.2} a point in this same $T = T_{\theta }$ version of
${\mathcal{M}^{* }}_{\hat{A},T}$. Let $c_{\theta }$ denote the
latter sequence. As will now be explained, arguments much like those used in
\fullref{sec:7d} prove the following: The sequence $\{c_{\theta }\}$ converges
as $\theta\to\sigma $ in ${\mathcal{M}^{* }}_{\hat{A}}$
to a point in the $T = T_{\sigma }$ version of ${\mathcal{M}^{* }}_{\hat{A},T}$
unless $\sigma =\theta _{ - e}$ with
 $e$ an edge that connects a vertex in $V$ to a vertex with angle less than
$\theta _{1}$. Indeed, all of the arguments in \fullref{sec:7d} can be made
in this situation except possibly those in Step 3 in \fullref{sec:7d} from the
proof of \fullref{prop:7.1}. Moreover, the arguments in Step 3 from the proof
of \fullref{prop:7.1} can also be made if the angle $\sigma $ (this is the
angle to use for $\theta _{* }$ in Step 3 of \fullref{sec:7d}) is
such that $\alpha _{Q}(\sigma ) > 0$ when $Q$ is the integer pair for any
edge that is incident to a vertex in $V$.
\end{proof}

\setcounter{section}{8}
\setcounter{equation}{0}
\section{Geometric limits}\label{sec:9}

The purpose of this last section is to indicate how the various codimension
1 strata fit one against another to make the whole of ${\mathcal{M}^{* }}_{\hat{A}}$.
The resulting picture of ${\mathcal{M}^{* }}_{\hat{A}}$ is by no means complete,
and perhaps not very illuminating. However, what follows should indicate how tools from the
previous sections can be used to add missing details.

This story starts with the codimension 0 strata, and so let $\mathcal{S} $ denote
a component of such a stratum in ${\mathcal{M}^{* }}_{\hat{A}}$.
What follows summarizes some of results about $\mathcal{S} $ from the previous
sections. To start, introduce $k$ to denote
$N_{ - }+\hat {N}+\text{\c{c}}_{ - }+\text{\c{c}}_{ + }-2$.
The component $\mathcal{S} $ lies in a stratum of the form
$\mathcal{S} _{B,c,\mathfrak{d}}$ where $B = ${\o} and $c = k$, and $d$ is the partition of
$N_{+ }+k$ with the maximal number of elements. As a consequence,
\fullref{prop:5.1}'s integer $m$ is equal to $k$ also. Thus, if
$T$ is a graph that arises from
an element in $\mathcal{S} $, then $T$ has $k$ trivalent vertices, with no two angles
identical and none an angle from an integer pair of any $(0,+,\ldots)$ element in
$\hat{A}$. If $o$ is a trivalent vertex, then $\Gamma_{o}$ is compact, a
figure~8 with one vertex. The other multivalent
vertices are bivalent. If $o$ is a bivalent vertex in $T$, then
$\underline {\Gamma }_{o}$ is a circular graph whose vertices correspond to the
$(0,+,\ldots)$ elements in $\hat{A}$ with integer pairs that
define $o$'s angle, $\theta _{o}$, via \eqref{eq1.8}.

This all translates into the following geometry: Suppose that $(C_{0}, \phi )$
defines an element in $\mathcal{S} $. Then the pull-back of $\theta $ to
$C_{0}$ has non-degenerate critical points; the critical values are pairwise
distinct, and none is an $|s|\to\infty $ limit of
$\theta $ on a concave side end of $C_{0}$ where $\lim_{| s| \to\infty }\theta\in(0, \pi )$.
If $E$ is such a concave side end,
then $E$'s version of \eqref{eq2.4} has integer $n_{E} = 1$. On the otherhand, if $E$ is
a convex side end where $\lim_{| s| \to \infty }\theta\in (0, \pi )$,
then the integer $n_{E}$ is zero. Finally, $C_{0}$ has
transversal intersections with the $\theta  = 0$ and $\theta =\pi $
cylinders.

To picture $\mathcal{S} $, note first that the image of
\fullref{sec:8a}'s map $\mathfrak p$ can be
viewed as a $k$--dimensional product of simplices in $\times _{k}
((0, \pi)-\Lambda _{+, \text{{\o}} })$, this denoted in what follows
as $\Delta ^{k}$. Then \fullref{sec:6c}'s map provides an orbifold
diffeomorphism that identifies $\mathcal{S} $ as $\mathbb{R} \times {\mathbb{O}}$,
where ${\mathbb{O}}$ is fibered by $\mathfrak p$ over $\Delta ^{k}$. In addition, the
typical fiber has the form $O_{T}/\Aut (T)$ with $T$ as just described. Because
all such fibers have homotopic graphs, the fibration
$\mathfrak{p}\co  {\mathbb{O}}\to\Delta ^{k}$ can be trivialized. This is to say that there is an
orbifold diffeomorphism from $\mathcal{S} $ to
\begin{equation}\label{eq9.1}
\mathbb{R}\times {\mathcal{O}}/ \mathcal{A}\times \Delta^{k},
\end{equation}
where ${\mathcal{O}} = O_{T}$ and $\mathcal{A} = \Aut (T)$ for a some fixed graph $T$.

The behavior of the codimension 0 strata near a given codimension 1 stratum
component is rather benign by virtue of the fact that the codimension 1
strata are submanifolds where smooth and suborbifolds otherwise. To
elaborate, remark that the $\mathfrak p$--image of a path in $\mathcal{S} $ to a codimension 1
or larger stratum will limit to a boundary point of the closure of $\Delta^{k}$
in $\times _{k} ((0, \pi )-\Lambda _{ + ,\text{ {\o}}})$.
As explained below, each codimension 1 stratum component in the
closure of $\mathcal{S} $ correspond in a natural fashion to a codimension 1 face
in this closure of $\Delta ^{k}$. In particular, if $\mathcal{S} _{1}$
denotes a codimension 1 stratum component, then $\mathcal{S} _{1}$ fibers over
some $k-1$ dimensional product of simplices, thus some $\Delta ^{k - 1}$. In
particular, $\mathcal{S} _{1}$ is diffeomorphic as an orbifold to
$\mathbb{R}\times{\mathcal{O}}_{1}/\mathcal{A}^{1} \times \Delta ^{k - 1}$
where ${\mathcal{O}}_{1}$ is some $O_{T^1} $ and $\mathcal{A}^{1}$ the
corresponding $\Aut (T^{1})$ with $T^{1}$ some fixed graph. In all cases, the
group $\mathcal{A}^{1}$ has a representation in $\mathbb{Z}/2\mathbb{Z}$, and this
understood, a neighborhood of $\mathcal{S} _{1}$ in ${\mathcal{M}^{*}}_{\hat{A},T}$
is diffeomorphic as an orbifold to
\begin{equation}\label{eq9.2}
\mathbb{R}\times ({\mathcal{O}}_{1}\times  (-1,1))/\mathcal{A}^{1}\times\Delta ^{k - 1},
\end{equation}
where $\mathcal{A}^{1}$ acts on $(-1,1)$ through the multiplicative action of
$\mathbb{Z}/2\mathbb{Z}$ as $\pm 1$. In all cases, the stratum $\mathcal{S} _{1}$
appears in \eqref{eq9.2} as the locus where the coordinate in the $(-1,1)$ factor is 0.

As it turns out, there is much more to say about how various codimension 0
and 1 strata fit around the codimension 2 strata in their closure. This
aspect of the stratification is the focus of \fullref{sec:9a} that follows.

Most of the rest of this section focuses on what can be viewed as the
codimension 1 strata in a certain natural compactification of
${\mathcal{M}^{* }}_{\hat{A}}$. To elaborate for a moment on this point,
remark that the identification given by \fullref{thm:6.2} provides a natural
compactification of any given ${\mathcal{M}^{*}}_{\hat{A},T}$,
this obtained from the compactification of $O_{T}$ that replaces each open
simplex in \eqref{eq6.9} with the corresponding closed symplex. This and the
replacement of $\Delta ^{k}$ in \eqref{eq9.1} with its closure defines, up to an
obvious factor of $\mathbb{R}$, a stratified space compactification of
${\mathcal{M}^{* }}_{\hat{A}}$. As it turns out, each added stratum in
this compactification has a natural geometric interpretation in terms of
multiply punctured sphere subvarieties. This interpretation is presented
below in Sections~\ref{sec:9b} for the additional codimension 1 strata in the
compactifications of the various versions of ${\mathcal{M}^{*}}_{\hat{A},T}$.
\fullref{sec:9c} goes on to describe \fullref{sec:1b}'s compactification of
the $N_{ - }+\hat {N}+\text{\c{c}}_{ - }+\text{\c{c}}_{ + } = 2$
versions of $\mathcal{M}_{\hat{A}}$. The added codimension 1 strata for the
compactification of ${\mathcal{M}^{* }}_{\hat{A}}$ when
$N_{ -}+\hat {N}+\text{\c{c}}_{ - }+\text{\c{c}}_{ + } > 2$ are described in
Sections~\ref{sec:9d} and~\ref{sec:9e}.

\subsection{The codimension one strata}\label{sec:9a}

The components of the codimension 1 strata in the case that
$N_{ - }+\hat{N}+\text{\c{c}}_{ + }+\text{\c{c}}_{ - }\equiv k+2$ are characterized in
part by the structure of the graph that arises from a typical element. Here
are the four possibilities for a component of a codimension 1 stratum:

\itaubes{9.3}
\textsl{The graph has $k$ trivalent vertices where precisely one pair have identical angles.
Even so, no trivalent vertex angle comes via \eqref{eq1.8} from an integer pair of a $(0,+,\ldots)$
element in $\hat{A}$.}

\item
\textsl{The graph has $k-2$ trivalent vertices and one 4--valent vertex. No two have the same
angle and none comes via \eqref{eq1.8} from an integer pair of a $(0,+,\ldots)$ element in
$\hat{A}$. In addition, the 4--valent vertex is assigned a graph with two vertices.}

\item
\textsl{The graph has $k-1$ trivalent vertices and none comes via \eqref{eq1.8} from an integer pair
of a $(0,+,\ldots)$ element in $\hat{A}$. Meanwhile, there are $N_{ + }+1$ bivalent vertices.}

\item
\textsl{The graph has $k$ trivalent vertices with pairwise distinct angles, and precisely
one such angle comes via \eqref{eq1.8} from an integer pair of a $(0,+,\ldots)$ element in
$\hat{A}$. The latter has a graph with one vertex labeled with 0.}
\end{itemize}

These four cases correspond to the following sorts of codimension 1 faces in
the symplex $\Delta ^{k}$ that appears in \eqref{eq9.1}. The first and second
points in \eqreft93 arise when two or more trivalent vertex angles lie in the
same component of $(0, \pi )-\Lambda _{ + , \text{{\o}}}$. The
first point can arise when there are no multivalent vertex angles between
the angles of a pair of trivalent vertices from distinct edges in $T$. The
second point can arise when no multivalent vertex angle lies between the
angles of two trivalent vertices that share an edge.

The third point in \eqreft93 can arise when no multivalent vertex angle lies
between the angles of a monovalent and trivalent vertex from a single edge.
The fourth point in \eqreft93 can occur when no multivalent vertex angle lies
between the angles of a bivalent and trivalent vertex from a single edge.

There is a corresponding geometric interpretation to the four strata in
\eqreft93. To say more in this regard, suppose that $(C_{0}, \phi )$ defines
an element in a codimension 1 stratum. If the component is characterized by
the first or second points in \eqreft93, then there are two critical points of
$\theta $ on $C_{0}$ with the same critical value in $(0, \pi )$. In the
case of the first point, the corresponding $\theta $ level sets are
disjoint; and they are not disjoint in the case of the second point. In the
case of the third point, a convex side end version of \eqref{eq2.4} has $n_{E} = 1$.
In the case of the fourth point, a critical value of $\theta $ coincides
with the $|s|\to\infty $ limit of $\theta $ on some concave side end.

As is explained in the final subsection, sequences in a top dimensional
stratum where a trivalent vertex angle limits to 0 or $\pi $ can not converge
to a codimension 1 stratum in ${\mathcal{M}^{* }}_{\hat{A}}$.
Faces of the closure of $\Delta ^{k}$ in $\times _{k} [0, \pi ]$ that
lie on the boundary of $\times _{k}[0, \pi ]$ can give rise to
codimension 2 strata.

The subsequent five parts of this subsection describe how the codimension 0
and the codimension 1 strata fit together inside ${\mathcal{M}^{* }}_{\hat{A}}$.
The proofs of the assertions that are made below
are omitted except for the comment that follows because the arguments in all
cases are lengthy yet introduce no fundamentally new ideas. Here is the one
comment: The proofs use the implicit function theorem to establish the
pictures that are presented below of the relevant strata of
${\mathcal{M}^{* }}_{\hat{A}}$. In this regard, the arguments use the techniques
that have already been introduced in \fullref{sec:7d}, in much the same manner as
they are used in \fullref{sec:7d}, to prove that the picture that is provided by
the implicit function theorem contains a full neighborhood of the stratum in
question.

\step{Part 1}
This part describes a neighborhood of the strata whose elements have
graphs that are described by the first point in \eqreft93. There are two cases
to consider; the distinction is whether the group $\mathcal{A}^{1}$ that
appears in \eqref{eq9.2} is larger than the automorphism group of the graphs that
arise from elements in the nearby codimension 0 strata. To elaborate, the
automorphism group of a graph from a codimension 0 stratum element must fix
all trivalent vertices since these have distinct angles. However, the
automorphism group of the codimension 1 stratum can, in principle,
interchange the two trivalent vertices that share the same angle. Granted
this, let $\mathcal{A}$ denote the automorphism group of a graph from an element
in a nearby codimension zero stratum and $\mathcal{A}^{1}$ denote that for a
graph from an element in the codimension 1 stratum. The two cases under
consideration here are those where $\mathcal{A}^{1}\approx\mathcal{A}$ and
where $\mathcal{A}^{1}$ is the semi-direct product of $\mathbb{Z}/2\mathbb{Z}$
with $\mathcal{A}$. In the former case, $\mathcal{A}^{1}$ acts trivially on the
$(-1, 1)$ factor in \eqref{eq9.2}. In the other case, $\mathcal{A}^{1}$ acts via its
evident projection to $\mathbb{Z}/2\mathbb{Z}$.

To explain how this dichotomy of automorphism groups arises, let $T^{1}$
denote the graph from a typical element in the codimension 1 stratum. Thus,
$\mathcal{A}^{1}$ is isomorphic to $\Aut (T^{1})$. The graph $T^{1}$ has an
$\mathcal{A}^{1}$--invariant, trivalent vertex $o$ with the following property:
Let $T_{1}$, $T_{2}$ and $T_{3}$ denote the closures of the three components
of $T^{1}-$o. The labeling is such that $T_{1}$ and $T_{2}$ each
contain one of the two trivalent vertices with equal angle. If $T_{1}$
is not isomorphic to $T_{2}$, then $\mathcal{A}^{1}\approx\mathcal{A}$.
If $T_{1}$ is isomorphic to $T_{2}$, then the distinguished $\mathbb{Z}/2\mathbb{Z}$
subgroup in $\mathcal{A}^{1}$ switches $T_{1}$ with $T_{2}$.

\step{Part 2}
This part and the next part of the subsection consider the components of
the codimension 1 stratum whose elements have graphs that are characterized
by the second point in \eqreft93. There are also various cases to consider here.
To elaborate, let $T^{1}$ denote such a graph, and let $o \in T^{1}$
denote the 4--valent vertex. Introduce $E_{ - }$ and $E_{ + }$ to denote the
sets of incident edges to $o$ that respectively connect $o$ to vertices with
smaller angle and with larger angle. In the first case, either $E_{ + }$ or
$E_{ - }$ has a single edge. In the second case, both have two edges. Note
that in the first case, the graph $\Gamma _{o}$ associated to $o$ has the
form:
\begin{equation}
\label{eq9.4}\includegraphics{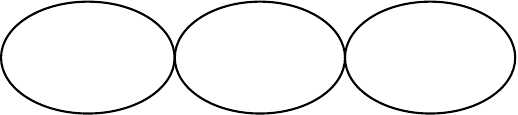}
\end{equation}
when $E_{ + }$ and $E_{ - }$ each have two edges, then the associated graph
$\Gamma _{o}$ can be either the graph in \eqref{eq9.4} or the graph that follows.
\begin{equation}
\label{eq9.5}\includegraphics{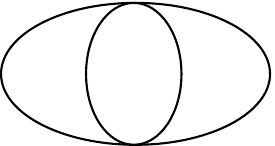}
\end{equation}
This part of the subsection focuses on the case where $E_{ - }$ has three of
$o$'s incident edges. The story when $E_{ + }$ has three edges is identical
but for notation to that told here. Part 3 of the subsection considers the
case when both $E_{ - }$ and $E_{ + }$ have two edges.

To start, label the three edges in $E_{ - }$ as $\{e_{1}, e_{2},e_{3}\}$.
The corresponding versions of $\ell _{o(\cdot )}$ label
the three circles depicted in \eqref{eq9.4}. Note that the central circle is
distinguished, and each of the three edges here can have the central circle
for its version of $\ell _{o(\cdot )}$. This is an important fact in what
follows because it indicates that there can be three distinct codimension 1
stratum components involved.

Suppose that $e_{2}$ occupies the central circle. The image of
$\mathcal{A}^{1}=\Aut(T^{1})$ in $\Aut _{o}$ is either trivial or
$\mathbb{Z}/(2\mathbb{Z})$. In this regard, let $T_{1}$, $T_{2}$ and $T_{3}$ denote
the closures in $T^{1}$ of the respective components of $T^{1}-o$
that contain the interiors of $e_{1}$, $e_{2}$ and $e_{3}$. The case where
the $\Aut _{o}$ image of $\mathcal{A}^{1}$ is $\mathbb{Z}/(2\mathbb{Z})$ arises when
$T_{1}$ and $T_{3}$ are isomorphic.

Let $\mathcal{S} _{1}$ denote the corresponding codimension 1 stratum of
${\mathcal{M}^{* }}_{\hat{A}}$. In either case, a neighborhood of
$\mathcal{S} _{1}$ in ${\mathcal{M}^{* }}_{\hat{A}}$ is diffeomorphic
as an orbifold to the space depicted in \eqref{eq9.2}. In this case, the group
$\mathcal{A}^{1}$ acts on the $(-1,1)$ factor via its image in $\Aut _{o}$, either
trivially or as the $\mathbb{Z}/(2\mathbb{Z})$ action as multiplication by $\pm 1$.

There is more to the picture just presented by virtue of the fact that any
one of the three edges in $E_{ - }$ can label the middle circle in \eqref{eq9.4} and
so there can be from 1 to 3 distinct codimension one strata involved here.
What follows describes how these codimension 1 strata and their
neighborhoods fit together in ${\mathcal{M}^{* }}_{\hat{A}}$.

To start, consider a graph, $T$, as described in the second point of \eqreft93
save that its 4--valent vertex, o, has a version of $\Gamma _{o}$ with only
one vertex. Thus, $\Gamma _{o}$ is the union of three circles that
intersect at a single point. These circles are the $e = e_{1}, e_{2}$ and
$e_{3}$ versions of $\ell _{oe}$. In this case, $\Aut _{o}$ is a subgroup of
$\mathbb{Z}/(3\mathbb{Z})$, thus trivial if $T_{1}$, $T_{2}$ and $T_{3}$ are not
mutually isomorphic. Note that there are two distinct versions of $\Gamma_{o}$
in any case, these are distinguished as follows: Let $\hat{e}$ denote
the single edge in $E_{ + }$. Then $\ell _{o\hat{e} }$ has three
vertices and three arcs. Each arc is labeled by $e_{1}$, $e_{2}$ and
$e_{3}$; and the two versions of $\Gamma _{o}$ are distinguished by the
two possible cyclic orderings of the arcs that comprise $\ell _{o\hat{e} }$.
Thus, there are, in fact, two possibilities for $T$. However, the
two versions are isomorphic when two or more from the collection
$\{T_{j}\}_{j = 1,2,3}$ are isomorphic.

The homotopy type of a graph $T$ as just described labels a codimension 2
stratum component, this diffeomorphic as an orbifold to
$\mathbb{R}\times {\mathcal{O}}_{2}/\mathcal{A}^{2}\times \Delta ^{k - 2}$. Here,
${\mathcal{O}}_{2}$ is diffeomorphic to $O_{T}$ and $\mathcal{A}^{2}$ isomorphic
to $\Aut (T)$.

Now introduce $Z  \subset\mathbb{C}\mathbb{P} ^{1}=\mathbb{C}\cup \infty $
to denote the complement of the three cube roots of $-1$. Thus, $Z$
is a model for a standard `pair of pants'. Let $Z_{1}\subset $Z$ \cap\mathbb{C}\mathbb{P} ^{1}$
denote the three rays that go through 0, $\infty $,
and the respective cube roots of 1. The drawing in \eqref{eq1.27} depicts $Z_{1}$ in
the finite part of $Z$.

There are three cases to consider. In the first, $T_{1}$, $T_{2}$ and
$T_{3}$ are pairwise non-isomorphic. In this case, there are three distinct
codimension 1 strata components involved; and a neighborhood in
${\mathcal{M}^{* }}_{\hat{A}}$ of their union is diffeomorphic as an
orbifold to
\begin{equation}\label{eq9.6}
\mathbb{R}\times {\mathcal{O}}_{2} / \mathcal{A}^{2}\times Z \times \Delta ^{k - 2}.
\end{equation}
Here, the three codimension 1 strata correspond to the loci
$Z_{1}-\{0,\infty \}$, and the two codimension 2 strata correspond to 0 and
$\infty $.

The second case occurs when two of the three graphs from $\{T_{j}\}_{j= 1,2,3}$
are isomorphic. In this case, there can be as few as two distinct
codimension 1 strata components involved. To describe a neighborhood of the
union of these strata, note that the image of $\mathcal{A}^{2}$ in $\Aut _{o}$
in this case is $\mathbb{Z}/2\mathbb{Z}$. Thus, $\mathcal{A}^{2}$ acts on
$\mathbb{C}\cup\infty $ through the action of $\mathbb{Z}/2\mathbb{Z}$ on
$\mathbb{C}\cup\infty $ whose generator sends $z \to\bar {z}$.
Note that this action commutes with the $\mathbb{Z}/2\mathbb{Z}$ action on
$\mathbb{C}\cup\infty $ as $z \to -\bar {z}^{ - 1}$ and whose orbit
space is $\mathbb{R}\mathbb{P} ^{2}$. As a consequence $\mathcal{A}^{2}$ acts on
$\mathbb{R}\mathbb{P} ^{2}$ as well. With all of this understood, a neighborhood
of their union in ${\mathcal{M}^{* }}_{\hat{A}}$ is diffeomorphic
as an orbifold to
\begin{equation}\label{eq9.7}
\mathbb{R}\times ({\mathcal{O}}_{2} \times \bar {Z})/\mathcal{A}^{2}\times\Delta ^{k - 2},
\end{equation}
where $\bar {Z}\subset\mathbb{R}\mathbb{P} ^{2}$ is the image of $Z$. Here,
the two codimension 1 strata correspond to the image in $\mathbb{R}\mathbb{P}^{2}$
of $Z_{1}-\{0,\infty \}$ and the codimension 2 stratum
to the image of $\{0, \infty \}$. In this regard, note that the rays of
$Z_{1}$ through the non-trivial cube roots of 1 have the same image in
$\mathbb{R}\mathbb{P} ^{2}$, this distinct from the ray through 1.

The final case occurs when $T_{1}$, $T_{2}$ and $T_{3}$ are pairwise
isomorphic. To picture this case, introduce the six element permutation
group $G$ of the set $\{1, 2, 3\}$. This group has the $\mathbb{Z}/3\mathbb{Z}$
subgroup of elements that preserves the cyclic order. The quotient of $G$ by
the latter group gives the parity homomorphism $G  \to\mathbb{Z}/2\mathbb{Z}$.
The group $G$ acts on $\mathbb{C}\mathbb{P} ^{1}$ as the group of
complex automorphisms that permutes the cube roots of $-1$. In the latter
guise, $G$ has generators $z \to  1/z$ and $z \to \lambda z$
where $\lambda $ is a favorite, non-trivial cube root of 1. Note that $G$
also permutes the cube roots of 1. In any event, $G$ acts on
$Z \subset \mathbb{C}\mathbb{P} ^{1}$. This understood, a neighborhood of the codimension 2
stratum in ${\mathcal{M}^{* }}_{\hat{A} }$ is diffeomorphic as an orbifold to
\begin{equation}\label{eq9.8}
\mathbb{R}\times \bigl[({\mathcal{O}}_{2}\times  G)/\mathcal{A}^{2}\times _{G} Z\bigr]
\times\Delta ^{k - 2},
\end{equation}
where $\mathcal{A}^{2}$ acts on $G$ on its left side through
$\Aut _{o}=\mathbb{Z}/(3\mathbb{Z})$. Meanwhile, $G$ acts on itself on its right
side and on $Z$ as noted above. The codimension 1 stratum appears in \eqref{eq9.8} as the image of
$Z_{1}-\{0,\infty \}$. The codimension 2 strata appear as the
image of 0 and $\infty $ from $Z$.

\step{Part 3}
What follows here is a description of a neighborhood in ${\mathcal{M}^{* }}_{\hat{A}}$
of the components of the codimension 1 strata whose
elements have graphs that are characterized by the second point in \eqreft93 in
the case that the 4--valent vertex has two edges that connect it to vertices
with larger angle and two that connect it to vertices with smaller angle.
To set the stage, let $T^{1}$ denote the graph in question and $o$ the
4--valent vertex. Let $e_{ - }$ and ${e_{-}}'$ denote the edges in $E_{ - }$
and let $e_{ + }$ and ${e_{+}}'$ denote those in $E_{ + }$. In what follows,
$\alpha _{ - }$, ${\alpha_{-}}'$, $\alpha _{ + }$ and ${\alpha _{+}}'$
denote $\alpha _{Q}(\theta _{o})$ in the case that $Q = Q_{e}$
with $e$ respectively the edges $e_{ - }$, ${e_{-}}'$, $e_{ + }$ and ${e_{+}}'$.
Each of these functions is positive at $\theta _{o}$, and by virtue
of \eqref{eq2.17},
\begin{equation}\label{eq9.9}
\alpha _{ -} + {\alpha_{-}}' = \alpha _{ + }+{\alpha_{+}}'.
\end{equation}
A distinction must now be made between the cases when the unordered sets
$\{\alpha _{ - }, {\alpha_{-}}'\}$ and $\{\alpha _{ + },{\alpha_{+}}'\}$
are distinct, and when they agree. Considered first
as Case 1 is that when these two sets are distinct.

\substep{Case 1}
Make the convention that $\alpha _{ - } \ge {\alpha_{-}}'$ and that
$\alpha _{ + }\ge{\alpha _{+}}'$. At least one of these is a strict inequality.
Since $\alpha _{ - }\ne\alpha _{ + }$, one or the other is larger; and as the description
for the $\alpha _{ - }>\alpha _{ + }$ case is identical to that when
$\alpha _{ + }>\alpha _{ - }$, the former is left to the reader.
Thus, in what follows,
\begin{equation}\label{eq9.10}
\alpha _{ +} > \alpha_{ - } \ge {\alpha_{-}}'> {\alpha_{+}}'.
\end{equation}
As noted briefly above, the graph $\Gamma _{o}$ can be either as depicted
in \eqref{eq9.4} or as in \eqref{eq9.5}. In the case of \eqref{eq9.4}, the assumption in \eqref{eq9.10}
implies that there are only two consistent labelings of the arcs with pairs
of incident edges. These are as follows:

\itaubes{9.11}
\textsl{Both of the middle circle's arcs are labeled by $(e_{ - }, e_{ + })$,
and the other two arcs are labeled by $(e_{ - }, {e_{+}}')$ and $({e_{-}}', e_{ + })$.}

\item
\textsl{Both of the middle circle's arcs are labeled by $({e_{-}}', e_{ + })$,
and the other two arcs are labeled by $({e_{-}}', {e_{+}}')$ and $(e_{ - }, e_{ + })$.}
\end{itemize}

Let $T_{ - }$, ${T_{-}}'$, $T_{ + }$ and ${T_{+}}'$ denote the closures of
the four components of $T^{1}-$o, here labeled so as to indicate
which contains which incident edge. By virtue of \eqref{eq9.10}, the graphs
$T_{ +}$ and ${T_{+}}'$ are not isomorphic. However, $T_{ - }$ and ${T_{-}}'$ may
be isomorphic. In the latter case, the two versions of $T^{1}$ that
correspond to the two labelings in \eqreft9{11} are isomorphic. Otherwise, the two
versions are not isomorphic. In any event, the image of $\Aut (T^{1})$ in
$\Aut _{o}$ is trivial.

Granted what has just been said, a component of a stratum whose typical
element has a graph such as either version of $T^{1}$ as just described is
diffeomorphic as an orbifold to what is depicted in \eqref{eq9.2} with
$\mathcal{A}^{1} = \Aut (T^{1})$ acting trivially on $(-1, 1)$.

When the graph $\Gamma _{o}$ is as depicted in \eqref{eq9.5}, there is, up to
isomorphism, only one way to label the arcs with pairs of edges. From right
to left, the labeling is:
\begin{equation}\label{eq9.12}
\bigl(e_{ - }, e_{ + }\bigr), \quad \bigl(e_{ - }, \quad  {e_{+}}'\bigr),
\quad \bigl({e_{-}}', {e_{+}}'\bigr),\quad
\bigl({e_{-}}', e_{ + }\bigr).
\end{equation}
In the case that $T_{ - }$ is not isomorphic to ${T_{-}}'$, the image of
$\Aut (T^{1})$ in $\Aut _{o}$ is trivial. In the case that these two graphs are
isomorphic, the image is $\mathbb{Z}/2\mathbb{Z}$. In the latter case, the
element $-1$ acts so as to rotate the diagram in \eqref{eq9.5} by $\pi $ radians.

A neighborhood in ${\mathcal{M}^{* }}_{\hat{A}}$ of a component of
a codimension 1 stratum that yields a the graph $T^{1}$ is diffeomorphic as
an orbifold to the space in \eqref{eq9.2}. In this regard, the action of
$\mathcal{A}^{1}= \Aut (T^{1})$ on $(-1, 1)$ is trivial if $T_{ - }$ and ${T_{-}}'$
are not isomorphic. If they are isomorphic, then the action on $(-1, 1)$ is
via its image in $\Aut _{o}$ as the multiplicative action of $\{\pm 1\} =
\mathbb{Z}/2\mathbb{Z}$.

A particularly intriguing point here concerns how the codimension 0 and 1
strata fit around a codimension 2 stratum. The intrigue stems from the fact
that the closures of the respective \eqref{eq9.4} and \eqref{eq9.5} cases for
$\mathcal{S}_{1}$ intersect. In this regard, the relevant graph that describes
the intersection obeys the second point of \eqreft93 save that the graph for the
4--valent vertex has but one vertex. Thus, the latter graph, $\Gamma _{o}$,
consists of three circles that meet at a single point. Such a graph is
obtained from \eqref{eq9.4} by shrinking either of the two arcs in the middle
circle. The same graph is produced by shrinking either arc. However, the two
labeled versions of \eqreft9{11} produce distinctly labeled versions of such a
1--vertex and 3--circle $\Gamma _{o}$. In particular, the respective top and
bottom versions of \eqreft9{11} produce such a $\Gamma _{o}$ whose arcs are
labeled by
\begin{equation}\label{eq9.13}
\bigl(e_{ - }, {e_{+}}'\bigr), \bigl(e_{ - }, e_{ + }\bigr), \bigl({e_{-}}', e_{ + }\bigr)
\qquad\text{and}\qquad \bigl({e_{-}}', {e_{+}}'\bigr), \bigl({e_{-}}', e_{ + }\bigr),
\bigl(e_{ - }, e_{ + }\bigr).
\end{equation}
Meanwhile, \eqref{eq9.5} yields a one vertex and three circle graph in two ways
since either of the arcs labeled with ${e_{+}}'$ can be shrunk. Note that
the shrinking of an $e_{ + }$ labeled arc is prohibited by the assumption in
\eqref{eq9.10}. There are two resulting versions of $\Gamma _{o}$, these
distinguished by the arc labelings in \eqref{eq9.13}.

Note that in the case that $T_{ - }$ and ${T_{-}}'$ are isomorphic, the two
versions of $T$ that correspond to the two versions of $\Gamma _{o}$ with
the respective arc labelings in \eqref{eq9.13} are isomorphic. Otherwise, they are
not isomorphic graphs.

Granted all of this, a picture of a neighborhood in ${\mathcal{M}^{* }}_{\hat{A}}$
of all of these strata is obtained as follows: Let $T$
denote one or the other of the graphs that result from the two cases in
\eqref{eq9.13}, and let ${\mathcal{O}}_{2}$ denote $O_{T}$ and $\mathcal{A}^{2}$ denote $\Aut (T)$.
Next, let $Z \subset\mathbb{C}$ denote the complement of
2 and $-2$. Let $Z_{1}\subset Z$ denote the union of the circles of
radius 1 centered at 2 and $-2$ together with the arc $[-1,1]$ along the real
axis between them. The drawing in \eqref{eq1.28} depicts $Z_{1}$ in $Z$. In the case
that $T_{ - }$ is not isomorphic to ${T_{-}}'$, a neighorhood of the
codimension 1 and 2 strata just described is diffeomorphic as an orbifold to
\begin{equation}\label{eq9.14}
\mathbb{R}\times {\mathcal{O}}_{2} / \mathcal{A}^{2}\times Z \times
\Delta ^{k - 2}.
\end{equation}
Here, the codimension 1 strata correspond to the complement in $Z_{1}$ of 1
and $-1$, while the latter correspond to the codimension 2 strata. In this
regard, subvarieties that map to the arc $(-1,1)  \subset Z_{1}$ have the
4--valent vertex graph in \eqref{eq9.5}. Those on the two circular parts of $Z_{1}$
have graphs as in \eqref{eq9.4}, and the two circles are distinguished by the two
labelings in \eqreft9{11}. Meanwhile, the subvarieties in the codimension 2 strata
that map to +1 are distinguished from those that map to $-1$ by the two
labelings in \eqref{eq9.13}.

In the case that $T_{ - }$ and ${T_{-}}'$ are isomorphic, then a
neighborhood of the strata is diffeomorphic as an orbifold to
\begin{equation}\label{eq9.15}
\mathbb{R}\times {\mathcal{O}}_{2} / \mathcal{A}^{2}\times
\bar {Z}\times \Delta ^{k - 2},
\end{equation}
where $\bar {Z}$ is the quotient of $Z$ via the action by multiplication of
$\{\pm 1\}$ on $\mathbb{C}$.

\substep{Case 2}
Consider now the case that the sets
$\{\alpha _{ - }, {\alpha_{-}}'\}$ and $\{\alpha _{ + },{\alpha_{+}}'\}$ are identical.
Agree to label things so that $\alpha_{ - }\ge{\alpha_{-}}'$ and
$\alpha _{ + }\ge\alpha_{ + }'$. Assume first that these are strict inequalities. In this case,
the version of $\Gamma _{o}$ in \eqref{eq9.4} has but one allowable arc labeling,
this where the middle arcs are labeled by $(e_{ - }, e_{ + })$ and the
outer two by $(e_{ - }, {e_{+}}')$ and $({e_{-}}', e_{ + })$. Shrinking
either of the two middle arcs yields a graph with three circles and one
vertex. The two graphs so obtained have identical edge labels,
$(e_{ - },e_{ + }), (e_{ - }, {e_{+}}')$ and $({e_{-}}', e_{ + })$. The
corresponding two versions of $T$ are thus isomorphic.

In the case that the version of $\Gamma _{o}$ is given by \eqref{eq9.5}, there is,
as before, only one way to label the graph in \eqref{eq9.5} by incident half-arcs.
Note that maps in the corresponding $\Delta _{o}$ must assign a greater
value to the $(e_{ - }, e_{ + })$ arc than to the $({e_{-}}', {e_{+}}')$
arc. Meanwhile, the arc labeled by $(e_{ - }, {e_{+}}')$ must be assigned
the same value as the $({e_{-}}', e_{ + })$ arc. This being the case, the
boundary of $\Delta _{o}$ can be reached either by shrinking the value of
the $({e_{-}}', {e_{+}}')$ arc or by simultaneously shrinking the values
of the $(e_{ - }, {e_{+}}')$ and $({e_{-}}', e_{ + })$ arcs. The face
that corresponds to giving the $({e_{-}}', {e_{+}}')$ arc value zero
corresponds to the codimension 2 stratum whose elements have a graph $T$ as
just described in the preceding paragraph. The other face of $\Delta _{o}$
does not correspond to an element in ${\mathcal{M}^{* }}_{\hat{A}}$.
Indeed, according to \fullref{lem:9.6} to come, points on the latter face
correspond to reducible subvarieties.

Granted this, here is a picture of a neighborhood in ${\mathcal{M}^{* }}_{\hat{A} }$
of these strata. Let $Z \subset\mathbb{C}$ denote
the complement of the circle of radius 1 centered at 1, and let
$Z_{1}\subset Z$ denote the union of this circle with the negative real axis. Let
 $T$ denote the graph that is referred to in the preceding paragraph.
 Set ${\mathcal{O}}_{2} = O_{T}$ and $\mathcal{A}^{2}= \Aut (T)$. Then the strata in
question have a neighborhood in ${\mathcal{M}^{* }}_{\hat{A}}$ that
is diffeomorphic as an orbifold to \eqref{eq9.14} with $Z$ as just described. In this
new version, the codimension 1 strata correspond to $Z_{1}-0$ and
the codimension 2 stratum corresponds to 0. Here, the circular part of
$Z_{1}$ corresponds to the component of the codimension 1 stratum whose
elements have 4--valent vertex graphs as depicted in \eqref{eq9.4}. The negative real
axis corresponds to the component whose elements have 4--valent graphs as
depicted in \eqref{eq9.5}.

The final case to consider is that where
$\alpha _{ - }=\alpha _{ -}' = \alpha _{ + }={\alpha_{+}}'$. In this case, the 4--valent
vertex in the codimension 1 stratum must have the form depicted in \eqref{eq9.5}.
Just one component is involved. The image of the corresponding version of
$\mathcal{A}^{1} = \Aut (T^{1})$ in $\Aut _{o}$ is either trivial,
$\mathbb{Z}/2\mathbb{Z}$ or $\mathbb{Z}/2\mathbb{Z} \times \mathbb{Z}/2\mathbb{Z}$. The
last case occurs when $T_{ - }$ is isomorphic to ${T_{-}}'$ and when
$T_{ +}$ is isomorphic to ${T_{+}}'$. The $\mathbb{Z}/2\mathbb{Z}$ case occurs when
one or the other of these pairs consist of isomorphic graphs, but not both.
The trivial case occurs when neither pair has isomorphic graphs. A
neighborhood in ${\mathcal{M}^{* }}_{\hat{A}}$ is diffeomorphic as
an orbifold to what is depicted in \eqref{eq9.2} where $\mathcal{A}^{1}$ acts on $(1,-1)$
via its image in $\Aut _{o}$, thus via its image in the group
$\mathbb{Z}/2\mathbb{Z} \times \mathbb{Z}/2\mathbb{Z}$. Here, the action of the
latter on $(-1, 1)$ has both the $\mathbb{Z}/2\mathbb{Z}$ generators multiplying by
$-1$. Note that in this case, neither boundary face of the $\Delta _{o}$
factor in ${\mathcal{O}}_{2}$ corresponds to a subvariety in
${\mathcal{M}^{* }}_{\hat{A}}$; both correspond to reducible subvarieties.

\step{Part 4}
This part describes the neighborhood of the codimension 1 strata in
${\mathcal{M}^{* }}_{\hat{A} }$ whose elements have graphs that are
characterized by the third point in \eqreft93. If $(C_{0}, \phi )$ defines a
point on such a stratum, then the extra bivalent vertex in $T_{C}$ has graph
$\underline {\Gamma }_{o}$ that is a single circle with one vertex. As
noted previously, the latter vertex corresponds to a convex side end of
$C_{0}$ where the $|s|\to\infty $ limit of $\theta $ is
neither 0 nor $\pi $, and whose version of \eqref{eq2.4} has $n_{E} = 1$. Let $o$
denote this vertex and $\theta _{o}$ its angle. Elements near this stratum
have a trivalent vertex, $\hat{o}$, that is very nearly $\theta _{o}$. If
$\theta _{\hat{o}}>\theta _{o}$, the $\hat{o}$ version of $E_{ -}$ has two of the
three incident edges; if $\theta _{\hat{o}}<\theta _{o}$, then the corresponding
$E_{ + }$ has two of the three
incident edges. The graphs from elements on the two sides of this stratum
would otherwise be homotopic.

Suppose that $\{(C_{0j}, \phi _{j})\}_{j = 1,2,\ldots }$ is a
sequence in ${\mathcal{M}^{* }}_{\hat{A}}$ with limit $(C_{0}, \phi )$
and if this sequence is not on the stratum in question, then each
$C_{0j}$ has a unique $\theta $ critical point that is very close to $\theta_{o}$.
Let $z_{j}$ denote the latter point. Then $\{\theta(z_{j})\}_{j = 1,2,\ldots }$
converges to $\theta _{o}$ but
$\{s(z_{j})\}_{j = 1,2,\ldots }$ is unbounded from below. To put this
in a colloquial fashion, the position of this one critical point moves as $j \to\infty $
to more and more negative values of $s$.

\step{Part 5}
This part describes the neighborhoods of the codimension 1 strata whose
elements have graphs that are characterized in the fourth point of \eqreft93.
There are now two cases to consider. In the first case, the graph $T$ for
$C_{0}$ has $N_{ + }$ bivalent vertices. This is to say that $\theta $ on
$C_{0}$ has $k$ non-degenerate critical points with one having $\theta $ value
in $\Lambda _{ + }$. However, the constant $\theta $ level set through
this critical point is compact like the others; thus still a figure~8. In
this case, a neighborhood of the stratum is not very interesting and won't
be discussed further except to say that the union of the two abutting
codimension zero strata are described by a version of \eqref{eq9.1}.

The more interesting case occurs when there are $N_{ + }-1$ bivalent
vertices. Assuming that this is the situation, let $o$ denote the trivalent
vertex in $T$ whose angle is in $\Lambda _{ + , \text{{\o}} }$. The
corresponding graph $\underline {\Gamma }_{o}$ is a figure~8 with one
4--valent vertex labeled with $o$ and with all other vertices being bivalent
and having positive integer labels. These bivalent vertices are in 1--1
correspondence with the $(0,+,\ldots)$ elements in $\hat{A}$
whose integer pair defines $\theta _{o}$ via \eqref{eq1.8}.

The various versions of $\underline {\Gamma }_{o}$ are characterized by
the manner in which these bivalent vertices are distributed on the two
circles that comprise the figure~8. In this regard, the possible isomorphism
types for $\underline {\Gamma }_{o}$ are described momentarily with the
help of an ordered pair $(\wp ', \wp '')$ of linearly ordered sets that
partition the set of $(0,+,\ldots)$ elements in $\hat{A}$ whose
integer pairs define $\theta _{o}$ in \eqref{eq1.8}.
For the purposes of this
description, label the incident edges to $o$ as $e$, $e'$ and $e''$ with the
convention that $e'$ and $e''$ connect $o$ to vertices with both angles either
greater than $\theta _{o}$ or both angles less than $\theta _{o}$. The
pair $(\wp ', \wp '')$ defines a version of $\underline {\Gamma }_{o}$
by using $\wp '$ to label the bivalent vertices that are met when
circumnavigating $\ell _{oe' }$ in the oriented direction starting
at the 4--valent vertex. The elements in $\wp ''$ have the analogous
interpretation with regards to $\ell _{oe''}$. Note that
this labeling of the versions of $\underline {\Gamma }_{o}$ can be
redundant. In particular, such will be the case when $\hat{A}$ has multiple
copies of some $(0,+,\ldots)$ whose integer pair gives
$\theta _{o}$ via \eqref{eq1.8}. The labeling is also redundent when $Q_{e}$ and
$Q_{e' }$ agree, for then the respective versions of $\underline {\Gamma }_{o}$
that are defined by $(\wp ', \wp '')$ and $(\wp '',\wp ')$ are identical.

The closures of these various strata components intersect in codimension 2
strata and it is intriguing to see how all of these strata fit together.
Since the story in general can be complicated, attention is restricted in
what follows to the case when $\theta _{o}$ is determined by the integer
pair from a single $(0,+,\ldots)$ element in $\hat{A}$. To
describe this case, let $T$ denote the graph that is described by the fourth
point in \eqreft93 but where the version of $\underline {\Gamma }_{o}$ has a
single 4--valent vertex with positive label. Let $T_{e}$, $T_{e' }$
and $T_{e'' }$ denote the closures of the components of
T$-$o, these labeled so as to indicate which incident edge is in which
component. Now, let $Z \subset\mathbb{C}$ denote the complement of $\{1,-1\}$ and
let $Z_{1}\subset Z$ denote the figure 8 locus where $|z^{2}-1|  = 1$.
This is depicted schematically in \eqref{eq1.26}. In the
case that $T_{e' }$ is not isomorphic to $T_{e'' }$, a neighborhood in
${\mathcal{M}^{* }}_{\hat{A}}$ of the
strata involved is diffeomorphic as an orbifold to the space depicted in
\eqref{eq9.4} with $Z$ as just described, with ${\mathcal{O}}_{2} = O_{T}$, and with
$\mathcal{A}^{2} = \Aut (T)$. In the case at hand, the two codimension 1 strata
are the two components of $Z_{1}-0$, and the codimension 2 stratum
corresponds to 0.

In the case that $T_{e' } = T_{e'' }$, there is
but a single codimension 1 stratum component and one codimension 2 stratum
component involved. To describe a neighborhood in ${\mathcal{M}^{* }}_{\hat{A} }$
of these strata, note that in this case $\Aut (T)$ has
image $\mathbb{Z}/2\mathbb{Z}$ in $\Aut _{o}$. Thus, $\Aut (T)$ acts on $\mathbb{C}$ via
this image and the $\mathbb{Z}/2\mathbb{Z}$ action is via multiplication by
$\{\pm 1\}$ on $\mathbb{C}$. This understood, a neighborhood in
${\mathcal{M}^{* }}_{\hat{A}}$ of the codimension 1 and codimension 2
strata is diffeomorphic as an orbifold to
$\mathbb{R}\times ({\mathcal{O}}_{2}\times  Z)/\mathcal{A}^{2} \times \Delta ^{k - 2}$ where
${\mathcal{O}}_{2}= O_{T}$, $\mathcal{A}^{2} = \Aut (T)$ and $Z$ is as described in
the preceding paragraph. In this case, the codimension 1 stratum is the
image of $Z_{1}-0$ and the codimension 2 stratum is the image of 0.

\subsection{Limits of sequence in ${\mathcal{M}^{*}}_{\hat{A},T}$
 that do not converge in $O_{T}/\Aut (T)$}\label{sec:9b}

The space $O_{T}/\Aut (T)$ has a natural compactification as a stratified
space that is obtained by replacing each open symplex in \eqref{eq6.9} by the
corresponding closed symplex. All of the added points in this
compactification label subvarieties that are geometric limits of elements in
${\mathcal{M}^{*}}_{\hat{A},T}$. The two parts of this subsection
describe the various limits for sequences that converge to the added
codimension 1 strata of the compactification. Analogous assertions for the
codimension greater than 1 added strata when
$N_{ - }+\hat {N}+\text{\c{c}}_{ - }+\text{\c{c}}_{ + } > 2$
are more complicated and so left to the more
industrious readers.

The proofs of the lemmas that appear are omitted but for some sporadic
comments. This is because the argument in each case would lengthen an
already long exposition; in any event, but for straighforward modifications,
each argument repeats the those that appear in \fullref{sec:7d}.

\step{Part 1}
To set the scene for this part of the subsection, a multivalent vertex
$o\in T$ has been fixed along with an arc $\gamma\subset \underline {\Gamma }_{o}$.
In addition, a sequence $\{\lambda _{j}\}_{j =1,2,\ldots }$ has been specified in $T$'s
version of the space depicted in
\eqref{eq6.9} that is constant but for the $\Delta _{o}$ factor. In this regard,
the $\Delta _{o}$ part of $\lambda _{j}$ maps $\gamma $ to a number,
$r_{j}$, such that the resulting sequence $\{r_{j}\}_{j = 1,2,\ldots}$
converges to zero. Meanwhile, the analogous sequences that are defined by
the remaining arcs in $\underline {\Gamma }_{o}$ converge as  $j \to \infty $
to positive numbers. With $\{\lambda _{j}\}$ so chosen, fix
$s_{0}\in\mathbb{R}$ and then use the image of $(s_{0}, \lambda_{j})$ in
$\mathbb{R}\times O_{T}/\Aut (T)$ with the inverse of the map
from \fullref{sec:6c} to define a sequence in ${\mathcal{M}^{*}}_{\hat{A},T}$.
This scenario is assumed implicitly in the statements of Lemmas~\ref{lem:9.1}--\ref{lem:9.4}
that follow.

By the way, the arc $\gamma $ can not start and end at the same vertex
because each element in $\Delta _{o}$ assigns the same value to any given
arc that starts and ends at the same vertex. To explain, note that such an
arc defines a non-trivial class in $H_{1}(\underline {\Gamma }_{o}; \mathbb{Z})$,
and the latter can be written as $\sum _{e}c_{e} [\ell_{oe}]$ with $\{c_{e}\}$
a collection of integers. Let $Q = \sum_{e}c_{e}Q_{e}$. As a consequence of \eqref{eq6.6},
each element in $\Delta_{o}$ must assign $\alpha _{Q}(\theta _{o})$ to $\gamma $.

Here is one comment with regards to proving Lemmas~\ref{lem:9.1}--\ref{lem:9.4}:
\fullref{prop:7.1} is still valid under the assumptions just made. In fact,
what follows are assumptions that guarantee the validity of \fullref{prop:7.1}.

\qtaubes{9.16}
{\sl
Let $\{\lambda _{j}\}_{j = 1,2,\ldots }$ denote a sequence from $T$'s version of the
space in \eqref{eq6.15} that is constant but for the $\Delta _{o}$ factor. Meanwhile,
let $\{r_{j}\}$ denote the corresponding sequence of
$\Delta _{o}$ factors. This sequence of maps should converge so that the following is
true: Let $\mathcal{A}$ denote the set of arcs in $\underline {\Gamma }_{o}$
for which $\lim_{j \to \infty }r_{j}(\cdot )$ is zero. Then
$ \cup _{\gamma \in \mathcal{A}}\gamma $ is simply connected.}
\endqtaubes

As illustrated by what is discussed in Case 2 of Part 3 in \fullref{sec:9a}, there
may be codimension 1 faces of $\Delta _{o}$ whose maps assign 0 to two or
more arcs from $\underline {\Gamma }_{o}$. As explained in Part 2 of this
subsection, limits to these faces violate the preceding assumption. In any
event, such faces are considered in Part 2.

In the various cases described below, a certain graph, $T'$, is defined from $T$
by collapsing $\gamma $ to a point and suitably interpreting the result. In
all of the cases, $T'$ is a graph that has the properties stated in
\fullref{sec:6a}. In what follows, $\hat{A}'$ denotes the corresponding
asymptotic data set.

In the meantime, a stratum in the compactified version of $O_{T}/\Aut (T)$
arises as the image of the face of the compact version of $\Delta _{o}$
whose maps send $\gamma $ to zero. This stratum has a natural interpretation
as $O_{T' }/\Aut (T')$. In each case that follows, the explanation of
this interpretation is straightforward and left to the reader. Granted this
interpretation, the sequence $\{\lambda _{j}\}$ in each of the cases
that follow has a well defined limit point in $O_{T' }/\Aut (T')$.
This point is denoted by $[\lambda _{0}]$. The $T'$ versions of \fullref{thm:6.2}
and the inverse of the map from \fullref{sec:6c} assign a point in
${\mathcal{M}^{* }}_{\hat{A}}$ the pair $(s_{0}, [\lambda _{0}])$.
Use $(C_{0}, \phi _{0})$ to denote the latter point in ${\mathcal{M}^{* }}_{\hat{A}}$.

In the first lemma that follows, $T'$ is identical to $T$ but for the graph
assigned to $o$. The latter, ${\underline {\Gamma}_{o}}'$, is obtained from
$\underline {\Gamma }_{o}$ by removing the interior of the arc $\gamma $
and replacing its two vertices with a single vertex whose label is the sum
of the integers that label the vertices on $\gamma $.

\begin{lemma}\label{lem:9.1}
Assume that at least one vertex on $\gamma $ is labeled by the integer 0.
In this case, the sequence in ${\mathcal{M}^{* }}_{\hat{A}}$ converges in
${\mathcal{M}^{* }}_{\hat{A}}$ to the point defined by $(C_{0}, \phi _{0})$
in ${\mathcal{M}^{* }}_{\hat{A},T' }$.
\end{lemma}

The next lemma, assumes that the two vertices on $\gamma $ have either both
positive or both negative integer assignments. Let $m_{1}$ and $m_{2}$
denote their respective integer assignments. The two vertices on $\gamma $
correspond to respective 4--tuples $a_{1}$ and $a_{2}$ in $\hat{A}$ of the form
$(0, \varepsilon , | m_{1}|  P)$ and $(0, \varepsilon ,| m_{2}|  P)$ where $P$
is the relatively prime integer pair that
defines $\theta _{o}$ via \eqref{eq1.8} and $\varepsilon  = \sign(m_{1})$. Let
$\hat{A}'$ denote the asymptotic data set that is obtained from $\hat{A}$ by first
removing both $a_{1}$ and $a_{2}$, and then adding $a' = (0, \varepsilon ,| m_{1}+m_{2}| P)$.
Meanwhile, let $T'$ denote the graph that
is identical to $T$ but for the graph assigned to $o$. This new assignment,
${\underline {\Gamma}_{o}}'$, is obtained from $\underline {\Gamma }_{o}$
by removing the interior of $\gamma $ and replacing its two vertices by a
single vertex whose integer assignment is $m_{1}+m_{2}$.

\begin{lemma}\label{lem:9.2}

With the circumstances as just indicated, the sequence converges to the element defined by
$(C_{0}, \phi )$ in ${\mathcal{M}^{*}}_{\hat{A}',T'}$ in the following sense:
The sequence in ${\mathcal{M}^{*}}_{\hat{A},T}$ comes from a sequence
$\{(C_{0j},\phi _{j})\}_{j = 1,2,\ldots }$ such that

\begin{itemize}

\item
$\lim_{j \to \infty }\smallint _{C_{0j} } \phi _{j}^*\varpi = \int _{C_0 }\phi ^*\varpi $
for each compactly supported 2--form $\varpi $.

\item
The following limit exists and is zero:
\begin{equation*}
\lim_{j \to \infty } \Bigl[\sup_{z \in C_{0j} } \dist\bigl(\phi _{j}(z), \phi
(C_{0})\bigr) + \sup_{z \in C_0 }  \dist\bigl(\phi _{j}(C_{0j}), \phi (z)\bigr)\Bigr].
\end{equation*}

\end{itemize}
\noindent Moreover, there exists an end, $E  \subset C_{0}$, that gives the 4--tuple $a'$
and has the following significance:  Given $R \gg 1$ and any sufficiently large $j$,
there is a smooth, proper embedding, $\psi _{j}$, from the complement in
$C_{0}$ of the  $|s| > R$ portion of $E$ into $C_{0j}$; and these are such that

\begin{itemize}

\item
The complement of the image of $\psi _{j}$ is a properly embedded, thrice punctured
sphere that contains two ends of $C_{0j}$, these supplying the elements
$a_{1}$ and $a_{2}$ to $\hat{A}$.

\item
$\lim_{j \to \infty } \dist(\phi _{j} \circ \psi _{j}(z), \phi
(z)) = 0$ for all $z$ in the domain of $\psi _{j}$.

\item
$|\bar {\partial }\psi _{j}|  \ll | \partial \psi
_{j}| $ and  $\lim_{j \to \infty } \sup_{\domain(\psi _j )}|
\bar {\partial }\psi _{j}| /| \partial \psi _{j}|  = 0$.

\end{itemize}

\end{lemma}

The final two lemmas in this Part 1 assume that both vertices on $\gamma $
have nonzero integer assignments with differing signs. Let $m_{1}$ and
$-m_{2}$ denote the two integers with the convention here that $m_{1} > 0$
and $m_{2 }> 0$. The two vertices on $\gamma $ correspond to 4--tuples
$a_{1}$ and $a_{2}$ in $\hat{A}$ where $a_{1} = (0,+, m_{1}P)$ and $a_{2} =(0,-, m_{2}P)$.
Here, as before, $P$ is the relatively prime pair that
defines $\theta _{o}$ via \eqref{eq1.8}.

The following lemma assumes that $m_{1}\ne m_{2}$. Let $\varepsilon = \sign(m_{1}-m_{2})$
and $m = | m_{1}-m_{2}| $. For this
case, let $\hat{A}'$ denote the asymptotic data set that is obtained from $\hat{A}$
by removing both $a_{1}$ and $a_{2}$, and replacing them with $a' = (0,\varepsilon , mP)$.
As in the previous cases, $T'$ denotes the graph that is
identical to $T$ but for the graph assigned to $o$; and the latter is obtained
from $\underline {\Gamma }_{o}$ by removing the interior of $\gamma $ and
replacing its two vertices by a single vertex whose integer assignment is
$m_{1}+m_{2}$.

\begin{lemma}\label{lem:9.3}
With the circumstances as just indicated, the sequence converges to the element defined by
$(C_{0}, \phi )$ in ${\mathcal{M}^{*}}_{\hat{A}',T'}$ in the following sense:
There is an $\mathbb{R}$--invariant cylinder, $S$, at angle $\theta _{o}$ and a sequence
$\{(C_{0j}, \phi_{j})\}$ that defines the given sequence, and these are such that

\begin{itemize}

\item
$\lim_{j \to \infty }\int _{C_{0j} } \phi _{j}^*\varpi =\smallint _{C_0 }\phi ^*\varpi
+ \min(m_{1}, m_{2})\int _{S}\varpi $ when $\varpi $ has compact support

\item
The following limit exists and is zero:

\begin{equation*}
\lim_{j \to \infty } \Bigl[\sup_{z \in C_{0j} } \dist\bigl(\phi _{j}(z), \phi
(C_{0})\cup  S\bigr) + \sup_{x \in \phi (C_0 ) \cup S}  \dist\bigl(\phi
_{j}(C_{0j}), x\bigr)\Bigr].
\end{equation*}
\end{itemize}
Moreover, there exists an end, $E  \subset C_{0}$, that gives the 4--tuple $a'$
and has the following significance:  Given $R \gg 1$ and any sufficiently large $j$,
there is a smooth, proper embedding, $\psi _{j}$, from the complement in $C_{0}$ of the
$|s|  > R$ portion of $E$ into $C_{0j}$; and these are such that

\begin{itemize}

\item
The $|s|\to\infty $ limit of the constant $|s| $ slices of $E$ converge as a degree $m$
multiple cover of the Reeb orbit in $S^1\times S^{2}$ that defines the cylinder $S$.

\item
The complement of the image of $\psi _{j}$ is a properly embedded, thrice
punctured sphere that contains two ends of $C_{0j}$, these supplying the elements
$a_{1}$ and $a_{2}$ to $\hat{A}$.

\item
The constant $|s| $ slices of the latter ends converge as $|s|\to\infty $
as multiple covers of a $\theta =\theta_{o}$ Reeb orbit, $\nu _{j}$;
and the resulting sequence, $\{\nu _{j}\}$, of Reeb orbits converges as $j  \to\infty $
to the Reeb orbit that defines $S$.

\item
$\lim_{j \to \infty } \dist(\phi _{j} \circ \psi _{j}(z), \phi
(z)) = 0$ for all $z$ in the domain of $\psi _{j}$.

\item
$|\bar {\partial }\psi _{j}|  \ll | \partial \psi_{j}| $ and
$\lim_{j \to \infty } \sup_{\domain(\psi _j )}|
\bar {\partial }\psi _{j}| /| \partial \psi _{j}|  = 0$.

\end{itemize}
\end{lemma}

\fullref{lem:9.4} that follows assumes that the integers $m_{1}$ and $m_{2}$ are
equal. In this case, $\hat{A}'$ is obtained from $\hat{A}$ by removing both $a_{1}$
and $a_{2}$. Nothing is added. The versions of $T'$ in this case depend on
certain properties of $\underline {\Gamma }_{o}$. There are three cases.
In the first case, the graph $\underline {\Gamma }_{o}$ has only two
vertices and both are bivalent. In this case, the vertex $o$ has just two
incident edges and $T'$ is obtained from $T$ by removing the vertex $o$ and
concatenating these two edges as one. In the second case, $\underline {\Gamma }_{o}$
has more than two vertices, but the two vertices on
$\gamma $ are bivalent. In this case, $T'$ differs from $T$ only in the graph
assigned to $o$. Here, the $T'$ version is obtained from $\underline {\Gamma
}_{o}$ by removing the whole of $\gamma $ and concatenating the resulting
free ends as a single arc in the new graph. This arc is labeled by the same
pair of edges that label $\gamma $. In the third case, $\underline {\Gamma
}_{o}$ may or may not have two vertices, but at least one of the vertices
on $\gamma $ has valency 4 or more. In this case, the interior of $\gamma $
is removed from $\underline {\Gamma }_{o}$, and its two vertices are
replaced by a single vertex whose integer assignment is zero. In any of
these cases, write the sum of the valencies of the two vertices on $\gamma $
as $k+4$.

\begin{lemma}\label{lem:9.4}
With the circumstances as just indicated, the sequence converges to the element defined by
$(C_{0}, \phi )$ in ${\mathcal{M}^{*}}_{\hat{A}',T'}$ in the following sense:
There is an $\mathbb{R}$--invariant cylinder, $S$, at angle $\theta _{o}$ and a sequence
$\{(C_{0j}, \phi_{j})\}$ that defines the given sequence, and these are such that

\begin{itemize}

\item
$\lim_{j \to \infty }\int _{C_{0j} } \phi _{j}^*\varpi =\int _{C_0 }\phi ^*\varpi
+ m_{1}\smallint _{S}\varpi $ when $\varpi $ has compact support

\item
The following limit exists and is zero:
\begin{equation*}
\lim_{j \to \infty } \Bigl[\sup_{z \in C_{0j} } \dist\bigl(\phi _{j}(z), \phi
(C_{0})\cup  S\bigr) + \sup_{x \in \phi (C_0 ) \cup S}  \dist\bigl(\phi
_{j}(C_{0j}), x\bigr)\Bigr].
\end{equation*}
\end{itemize}

Moreover, there exists a point $z  \in C_{0}$ with the following significance:
Given small but positive $\varepsilon $ and any sufficiently large $j$,
there is a smooth, proper embedding, $\psi _{j}$, from the complement in
$C_{0}$ of the radius $\varepsilon^{2}$ disk about $z$ in $C_{0}$ into $C_{oj}$;
and these are such that:

\begin{itemize}

\item
If $k = 0$, then $d\theta | _{z}\ne 0$. If $k > 0$, then $z$ is a critical point of
$\theta $ and $\deg(d\theta| _{z})= k$.  In either case, $\phi (z)  \in  S$.

\item
The complement of the image of $\psi _{j}$ is a properly embedded, thrice punctured
sphere that contains two ends of $C_{0j}$, these supplying the elements
$a_{1}$ and $a_{2}$ to $\hat{A}$.

\item
The constant $|s| $ slices of the latter ends converge as $|s|\to\infty $ as multiple
covers of a $\theta =\theta_{o}$ Reeb orbit, $\nu _{j}$; and the resulting sequence,
$\{\nu _{j}\}$, of Reeb orbits converges as $j  \to\infty $ to the Reeb orbit that defines $S$.
The third end of this thrice punctured sphere is mapped by $\phi_{j}$ to the radius
$\varepsilon $ ball in $\mathbb{R}\times  (S^1\times S^2)$ about $\phi(z)$.

\item
$\lim_{j \to \infty } \dist(\phi _{j} \circ \psi _{j}(z), \phi
(z)) = 0$ for all $z$ in the domain of $\psi _{j}$.

\item
$|\bar {\partial }\psi _{j}|  \ll | \partial \psi
_{j}| $ and  $\lim_{j \to \infty } \sup_{\domain(\psi _j )}|
\bar {\partial }\psi _{j}| /| \partial \psi_{j}|  = 0$.

\end{itemize}
\end{lemma}

\step{Part 2}
The preceding four lemmas require that $\{\lambda _{j}\}$ provide a
convergent sequence with non-zero limit to each arc but one in
$\underline {\Gamma }_{o}$. As seen in the previous subsection, there can arise
situations where a codimension 1 face of $\Delta _{o}$ can not be reached
by such a sequence. This can happen when the constraints in \eqref{eq6.6} require
some collection of arcs in $\underline {\Gamma }_{o}$ to have identical
assignments for each  $r \in\Delta _{o}$.

To say more about this phenomena, note that when $r \in \Delta _{o}$ and a vertex
$\upsilon\in \underline {\Gamma}_{o}$ are given first, and then a number with
sufficiently small absolute
value is chosen, a new point in $\Delta _{o}$ is obtained from r by adding
the chosen number to $r$'s value on each outward pointing arc at $\upsilon $
while subtracting the number from $r$'s value on each inward pointing arc.
This observation has the following implication: Let $\gamma $ and $\gamma '$
denote arcs in $\underline {\Gamma }_{o}$. Then $r(\gamma )- r(\gamma ')$ is independent
of  $r \in\Delta _{o}$ if and only if $\gamma$ and $\gamma '$ share the same
starting vertex and also the same ending vertex.

To see what $r(\gamma )- r(\gamma ')$ can be in this case, remark that
the loop defined by traversing $\gamma $ in its oriented direction and
$\gamma '$ in reverse defines a primitive homology class from
$H_{1}(\underline {\Gamma }_{o}; \mathbb{Z})$ and so can be written as a
linear combination $\sum _{e}c_{e} [\ell _{oe}]$ where
$\{c_{e}\}$ is a collection of integers that are defined modulo adding a
constant to each $e \in $ $E_{ + }$ version of $c_{e}$ and subtracting the
same constant from each $e \in E_{ - }$ version. This implies that
\begin{equation}\label{eq9.17}
r(\gamma )- r(\gamma ') = 2\pi\sum _{e}c_{e}\alpha _{Q_e} (\theta _{o}).
\end{equation}
Since the collection $\{c_{e}\}$ are integers, this last identity has
the following implication:

\qtaubes{9.18}
\textsl{The values of $r(\gamma )$ and $r(\gamma ')$ are equal for all  $r \in\Delta _{o}$}
\textsl{for at most two values of $\theta _{o}$ unless $\sum _{e}c_{e}Q_{e} = 0$.}
\endqtaubes

Indeed, with  $Q \equiv\sum _{e}c_{e}Q_{e}\ne 0$, then
$r(\gamma )- r(\gamma ') = 2\pi\alpha _{Q}(\theta _{o})$;
and this is zero when $Q  \ne 0$ if and only if either $Q$ or $-Q$ defines the
angle $\theta _{o}$ via \eqref{eq1.8}.

The following lemma about the set of integers $\{c_{e}\}$ plays a role
in what follows.

\begin{lemma}\label{lem:9.5}

Let $\gamma $ and $\gamma '$ denote arcs that share the same starting vertex and share
the same ending vertex. Then the homology class of the loop that is obtained by traversing
$\gamma $ in its oriented direction and then $\gamma '$ in reverse can be written as
$\sum_{e} c_{e} [\ell _{oe}]$ where  $c_{e} = 1$ or 0 when $e  \in  E_{ - }$ and where
$c_{e} =-1$ or 0 if $e \in E_{ + }$. Moreover, $c_{e}\ne 0$ for at least one $e  \in E_{ - }$
and for at least one $e  \in E_{ + }$.

\end{lemma}

\begin{proof}[Proof of  \fullref{lem:9.5}]
Let $\mathbb{Z}$ denote the free $\mathbb{Z}$--module
generated by the arcs in $\underline {\Gamma }_{o}$, and let $\eta $
denote an arc. Introduce the homomorphism $f^{\eta }\co \mathbb{Z}\to \mathbb{Z}$
that sends $\eta $ to 1 and all other arcs to 0. Now, let $(e,\hat{e})$ denote the
edges that label $\gamma $ with the convention that $e\in E_{ - }$ and $\hat{e}\in E_{ + }$.
As $f^{\gamma }(\gamma  -\gamma ') = 1$, it follows that $c_{e}+c_{\hat{e} } = 1$ since only
$\ell _{oe}$ and $\ell _{o\hat{e} }$ contain $\gamma $. As a
consequence, \eqref{eq2.17} can be used to choose the coefficients $\{c_{(\cdot)}\}$ so as to
make $c_{e} = 1$ and $c_{\hat{e} } = 0$. Now let $(e',\hat{e}')$ denote the edges
that label $\gamma '$ with $e' \in E_{ - }$ and $\hat{e}'  \in E_{ + }$.

As is explained next, $c_{e' }$ can not equal $-1$. To see why such
is the case, introduce $\upsilon $ to denote the vertex where $\gamma $ and
$\gamma '$ start. An incident half-arc to $\upsilon $ is called a $1_{ - }$
half-arc when $c_{(\cdot )}= -1$ on the $E_{ - }$ component of its
labeling pair. Meanwhile, a $1_{ + }$ half-arc has $c_{(\cdot )}= 1$
on the $E_{ + }$ component of its labeling pair. Note that there are an even
number of $1_{ - }$ half-arcs and and even number of $1_{ + }$ half-arcs.
Also, every $1_{ - }$ half-arc except $\gamma '$ in the case that
$c_{e' } = -1$ is a $1_{ + }$ half-arc. On the other hand, a $1_{ +}$ half-arc is also a
$1_{ - }$ half-arc because $f^{\gamma ' }(\gamma -\gamma ') = -1$. Granted these
obervations, then parity considerations forbid $c_{e' }$ from having value $-1$.

A very similar argument now proves that $c_{e' } = 0$ and $c_{\hat{e} ' } = 1$.
To see how this comes about, say that an incident
half-arc to $\upsilon $ is a $0_{ - }$ half-arc if $c_{(\cdot )}= 0$
on the $E_{ - }$ component of its labeling pair. Meanwhile, a $0_{ + }$
half-arc has $c_{(\cdot )}= 0$ on the $E_{ + }$ component of its
labeling pair. There are an even number of $0_{ - }$ half-arcs and and even
number of $0_{ + }$ half-arcs. Also, every $0_{ + }$ half-arc except $\gamma$
is a $0_{ - }$ half-arc. On the other hand, every $0_{ - }$ half-arc is a
$0_{ + }$ half-arc except $\gamma '$ if $c_{e' } = 0$. Granted these
obervations, then parity considerations require $c_{e' } = 0$ and
then $c_{\hat{e} ' } = -1$ because $f^{\gamma ' }(\gamma -\gamma ')= -1$.

Having established that $c_{e} = 1$, $c_{\hat{e} } = 0$, $c_{e' } = 0$ and
$c_{\hat{e} ' } = -1$, define a sequence of sets,
$\mathcal{A}_{1}\subset\mathcal{A}_{2}\subset\cdots
\subset\mathcal{A}^{\gamma ,\gamma ' }$ of arcs
in $\underline {\Gamma }_{o}-\{\gamma $, $\gamma '\}$ as
follows: The arcs that comprise $\mathcal{A}_{1}$ are the arcs from
$\ell_{oe}\cup\ell _{o\hat{e} ' }-\{\gamma, \gamma '\}$ Those in any
$k \ge 2$  version of $\mathcal{A}_{k}$ are
labeled by an edge from $e$ that also labels an arc from $\mathcal{A}_{k - 1}$.
If $e$ labels an arc in $\mathcal{A}^{\gamma ,\gamma ' }$, then
$c_{e} = 1$ when $e \in E_{ - }$, and $c_{e} = -1$ when $e \in E_{ +}$.
If $e$ does not label an arc from $\mathcal{A}^{\gamma ,\gamma ' }$,
then $c_{e} = 0$. As a parenthetical remark, note that the sets
$\mathcal{A}^{\gamma ,\gamma ' }$ and $\mathcal{A}^{\gamma ' ,\gamma}$
are complementary: An arc from $\underline {\Gamma }_{o}-\{\gamma , \gamma '\}$
in the former is not in the latter and vice versa.
\end{proof}

To lay more groundwork for the assertions to come, suppose that an ordered
pair of distinct vertices in $\underline {\Gamma }_{o}$ have been
specified and let $\mathbb{A}$ denote a set of arcs that start at the first
vertex of the pair, end at the second and are all assigned equal value by
every map in $\Delta _{o}$. Other arcs can connect these vertices if there
is a map in $\Delta _{o}$ that assigns them values that differ from its
value on the arcs in $\mathbb{A}$. It is assumed in what follows that
$\mathbb{A}$ has two or more arcs. Now, define a new graph, ${\underline {\Gamma}_{o}}'$,
from $\underline {\Gamma }_{o}$ by identifying the
arcs in $\mathbb{A}$ to a single vertex. This single vertex, $\upsilon '$, is
labeled by the sum of the integers that label the two vertices on the arcs
in $\mathbb{A}$. To obtain the other labels for ${\underline {\Gamma}_{o}}'$,
declare that the collapsing map from $\underline {\Gamma }_{o}$ to
${\underline {\Gamma}_{o}}'$ induce an isomorphism between
$\underline {\Gamma }_{o}- \cup _{\gamma \in \mathbb{A}}\gamma $ and
${\underline {\Gamma}_{o}}'-\upsilon '$ that preserves all
integer labels to vertices and the edge label pairs of the arcs.

Since $\mathbb{A}$ has more than one arc, the graph ${\underline {\Gamma}_{o}}'$
has the wrong Euler number for a graph that arises from a pair
$(C_{0}, \phi )$ with $C_{0}$ a multiply punctured sphere. Indeed, let $n$
denote the number of incident edges to $o$ and let a denote the number of arcs
in $\mathbb{A}$. The Euler number of ${\underline {\Gamma}_{o}}'$ is $1-n+(a-1)$
where as $1-n$ is the correct number were ${\underline {\Gamma}_{o}}'$ to
come from a multiply punctured sphere. As a consequence, the image in
$H_{1}({\underline {\Gamma}_{o}}'; \mathbb{Z})$ of the classes $\{[\ell_{oe}]\}$
must satisfy $a > 1$ constraints, thus more than the one that is
depicted in \eqref{eq2.17}. Of course, these extra constraints are generated by a
set whose elements are labeled by distinct, unordered pairs of arcs from
$\mathbb{A}$. In this regard, the constraint labeled by such a pair,
$\{\gamma , \gamma '\}$, asserts that
\begin{equation}\label{eq9.19}
\sum _{e}c_{e} [\ell _{oe}] = 0 \textsl{ in } H_{1}({\underline {\Gamma}_{o}}'; \mathbb{Z}),
\end{equation}
where the collection $\{c_{e}\}$ comes from the $\{\gamma , \gamma '\}$
version of \fullref{lem:9.5}.

The next task is to make some sense out of this. To start, fix an arc
$\gamma\in\mathbb{A}$ and let $e$ denote the edge from $E_{ - }$ that
labels $\gamma $. Define sets
$\mathbb{A}^{\gamma }_{1}\subset\mathbb{A}^{\gamma }_{2}\subset\cdots
\subset\mathbb{A}^{\Gamma }$ of arcs in ${\underline {\Gamma}_{o}}'$ as follows: The
arcs from $\ell _{oe}$ comprise $\mathbb{A}^{\gamma }_{1}$. Meanwhile,
those in the $k \ge 2$  version of $\mathbb{A}^{\gamma }_{k}$ are labeled
by an edge in $E_{ + }$ that also labels an arc in $\mathbb{A}^{\gamma }_{k- 1}$.
Note that $\mathbb{A}^{\gamma }\cap\mathbb{A}^{\gamma ' } = ${\o} if $\gamma\ne\gamma '$.
To see this, first return to the
proof of \fullref{lem:9.5} to see that
$\mathbb{A}^{\gamma }\subset\mathbb{A}^{\gamma ,\gamma ' }$.
Since $\mathbb{A}^{\gamma ,\gamma ' }$ is disjoint from
$\mathbb{A}^{\gamma ' ,\gamma }$, so
$\mathbb{A}^{\gamma }$ and $\mathbb{A}^{\gamma ' }$ are also disjoint.

Let $E^{\gamma }$ denote the set of $o$'s incident edges that label arcs in
$\mathbb{A}^{\gamma }$. Note that $E^{\gamma }$ is disjoint from $E^{\gamma' }$
in the case that $\gamma\ne\gamma '$. Indeed, such is
the case since the former set have $c_{(\cdot )}=\pm 1$ in the
$\{\gamma $, $\gamma '\}$ version of \fullref{lem:9.5}, while the latter set
have $c_{(\cdot )}= 0$ in this same version. Granted this, let
$E^{\gamma }_{\pm }$ denote $E_{\pm }\cap E^{\gamma }$. It follows
from \fullref{lem:9.5} that the union of the arcs in $\mathbb{A}^{\gamma }$ define a
subgraph, ${\underline{\Gamma}_{o}}^{\gamma}\subset {\underline {\Gamma}_{o}}'$,
whose $\mathbb{Z}$ homology is generated by the collection
$\{\ell _{o\hat{e} }: \hat{e}  \in E^{\gamma }\}$ subject to the one constraint,
\begin{equation}\label{eq9.20}
\sum_{\hat{e} \in E^\gamma _ - } [\ell _{o\hat{e} }] - \sum _{\hat{e}
\in E^\gamma _ + } [\ell _{o\hat{e} }] = 0.
\end{equation}
The various $\gamma\in\mathbb{A}$ versions of \eqref{eq9.20} generate the
constraint in \eqref{eq2.17} and the full collection of the various $\{\gamma ,\gamma '\}$
versions of the constraint in \eqref{eq9.19}.

Return now to the graphs in the collection $\{{\underline{\Gamma}_{o}}^{\gamma}\}$.
Remark first that any two such graphs with
distinct arc labels intersect only at $\upsilon '$. On the other hand,
$\cup _{\gamma \in \mathbb{A}}{\underline{\Gamma}_{o}}^{\gamma}$ is
the whole of ${\underline {\Gamma}_{o}}'$. These obervations follow from
Property 3 from Part 3 of \fullref{sec:2c}.

Each arc in ${\underline{\Gamma}_{o}}^{\gamma}$ already has an edge
pair label, and all but one vertex has an integer label. The one as yet
unlabeled vertex is the one that maps to $\upsilon '$ in ${\underline {\Gamma}_{o}}'$.
What follows explains how to deal with the latter so that the
result, ${\underline{\Gamma }_o}^\gamma  $, comes from the label of a vertex
in a graph as described in a version of \fullref{sec:6a}. To start, define
\begin{equation}\label{eq9.21}
Q^{\gamma }\equiv\sum _{\hat{e} \in E^\gamma _ - } Q_{e}-\sum
_{\hat{e} \in E^\gamma _ + } Q_{e}.
\end{equation}
It follows from \eqreft9{18} that $\theta _{o}$ is defined via \eqref{eq1.8} by some
relatively prime integer pair if $Q^{\gamma }$ is non-zero; and then
$Q^{\gamma }$ is a multiple of this pair. Let $\hat {m}^{\gamma }$ denote
the latter multiple when $Q^{\gamma }$ is non-zero, and set
$\hat{m}^{\gamma } = 0$ otherwise. Meanwhile, let $m^{\gamma }$ denote the sum
of the integers that label the vertices in ${\underline{\Gamma}_{o}}^{\gamma}-\upsilon '$.

Now there are two cases to consider. If $\hat {m}^{\gamma } - m^{\gamma}$
is non-zero or if the vertex in ${\underline{\Gamma}_{o}}^{\gamma}$
that maps to $\upsilon '$ has more than two incident edges, then the vertex
that maps to $\upsilon '$ is viewed in ${\underline{\Gamma }_o}^\gamma $ as
an honest vertex with integer label $\hat {m}^{\gamma } - m^{\gamma }$.
If the vertex in ${\underline{\Gamma}_{o}}^{\gamma}$ that maps to
$\upsilon '$ is bivalent and if $\hat {m}^{\gamma } - m^{\gamma } = 0$,
then this vertex is invisible in ${\underline{\Gamma }_o}^\gamma  $ and its
incident arc are viewed as a single arc.

The corresponding \fullref{sec:6a} graph is denoted in what follows as
$T^{\gamma}$, and it is characterized as follows: There is a multivalent vertex,
$o^{\gamma }$, in $T^{\gamma }$ whose angle is $\theta _{o}$ and graph
label is ${\underline{\Gamma }_o}^\gamma  $. Meanwhile, $T^{\gamma }-o^{\gamma }$
is isomorphic to the union of the components of $T-o$
that contain the interiors of the edges from the set $E^{\gamma }$. The
corresponding asymptotic data set for $T^{\gamma }$ is denoted below by
$\hat{A}^{\gamma }$.

All of this background is assumed in the upcoming \fullref{lem:9.6}. In addition,
this lemma assumes that a sequence, $\{\lambda _{j}\}_{j = 1,2\ldots}$,
has been specified in $T$'s version of the space depicted in \eqref{eq6.15} with
the following properties: The sequence is constant but for the $\Delta_{o}$ factor.
Let $\{r_{j}\}$ denote the corresponding sequence in the
$\Delta _{o}$ factor. Then $\{r_{j}\}$ converges and $\mathbb{A}$ is the
set of arcs in $\underline {\Gamma }_{o}$ where $\lim_{j \to \infty }r_{j}(\cdot )$
is zero. With $\{\lambda _{j}\}$ so chosen, fix a
sequence $\{s_{j}\} \in\mathbb{R}$ and then use the image of
$(s_{j}, \lambda _{j})$ in $\mathbb{R}\times O_{T}/\Aut (T)$ with the
inverse of the map from \fullref{sec:6c} to define a sequence in
${\mathcal{M}^{*}}_{\hat{A},T}$.

\begin{lemma}\label{lem:9.6}

With the circumstances as just indicated, there exists the following:
First, a pair, $(S^{\gamma }, \phi ^{\gamma })$, that defines an element in each
$\gamma\in \mathbb{A}$ version of the space $\mathcal{M}^{* }_{\hat{A}^\gamma ,T^\gamma } $.
Second, a subsequence of $\{\lambda_{j}\}$, hence renumbered consecutively from 1,
and a sequence $\{(C_{j0}, \phi _{j})\}$ that defines the corresponding subsequence in
${\mathcal{M}^{*}}_{\hat{A},T}$. Third, a finite set, $\Xi $, of pairs of the form
$(S, \phi )$ where each $(S, \phi )\in\Xi $ is one of the following

\begin{itemize}

\item
An $\mathbb{R}$--invariant cylinder at an angle in $\Lambda _{\hat{A}}$ or at the angle
$\theta _{o}$.

\item
An element from the set $\{(S^{\gamma }, \phi ^{\gamma })\}_{\gamma \in \mathbb{A}}$.
\end{itemize}

Here is the significance: Given a compact set $K\subset\mathbb{R}\times(S^1\times S^2)$, then

\begin{itemize}

\item
$\lim_{j \to \infty }\int _{C_{j0} } \phi _{j}^*\varpi =
\sum _{(S,\phi ) \in \Xi }\int _{S}\phi ^*\varpi $
for each 2--form $\varpi $ with compact support in $K$.

\item
The following limit exists and is zero:
\begin{equation*}
\lim_{j \to \infty } \Bigl(\sup_{z \in \phi _j ^{ - 1}(K)} \dist\bigl(\phi
_{j}(z),  \cup _{(S,\phi ) \in \Xi }\phi (S)\bigr) + \sup_{z \in \cup
_{(S,\phi ) \in \Xi } \phi ^{ - 1}(K)}  \dist\bigl(\phi _{j}(C_{j0}), \phi (z)\bigr)\Bigr).
\end{equation*}

\end{itemize}
\noindent Moreover, there is a subset $\mathbb{A}_{0}\subset\mathbb{A}$ and two versions of
$\Xi $ that arises from suitable choices of $\{s_{j}\}$ such that the first contains
only and all of the $\gamma\in \mathbb{A}_{0}$ versions of $(S^{\gamma }, \phi ^{\gamma })$,
and the second contains only and all of the $\gamma\in\mathbb{A}-\mathbb{A}_{0}$ versions.

\end{lemma}

A story can be told in the manner of Lemmas~\ref{lem:9.1}--\ref{lem:9.4}
that gives a much more detailed account of the convergence of the sequence in question. In
particular, a part of this story describes a subvariety in
$\times_{\gamma \in \mathbb{A}}O_{T^\gamma } /\Aut (T^{\gamma })$ that maps
in a proper, finite to one fashion onto the given codimension 1 stratum of
the compactification of $O_{T}/\Aut (T)$. In any event, this more detailed
account is left to the reader except for the following comment: As with
Lemmas~\ref{lem:9.1} and~\ref{lem:9.4}, various cases must be distinguished. In the present
circumstances, these are characterized by

\begin{itemize}

\item
\textsl{the integers that are assigned the two vertices on the arcs in $\mathbb{A}$;}

\item
\textsl{whether or not $\upsilon '$ is represented by a vertex in the various versions of
${\underline{\Gamma }_o}^\gamma  $;}

\item
\textsl{in the cases that $\upsilon '$ is so represented, the integer label of the corresponding vertex.}

\end{itemize}

\subsection{The compactification of $\mathcal{M}_{\hat{A}}$ in
the case that $N_{ - }+\hat {N}+\text{\c{c}}_{ -}+\text{\c{c}}_{ + }= 2$}\label{sec:9c}

The purpose of this subsection is to describe a compactification of the
whole of $\mathcal{M}_{\hat{A}}$ in the case that $\hat{A}$ is an
asymptotic data set with $N_{ - }+\hat {N}+\text{\c{c}}_{ - }+\text{\c{c}}_{ + } = 2$.

The proofs of the various assertions that follow are omitted as they can be
obtained in a straightforward manner using variations of arguments from
\fullref{sec:4} and \fullref{sec:7}. In this regard, \eqreft9{16} should be used to obtain the
conclusions of \fullref{prop:7.1}.

To start the discussion, return to \fullref{thm:1.2} where
$\mathcal{M}_{\hat{A}}$ is described as $\hat{O}^{\hat{A}}/\Aut
^{\hat{A}}$ with $\hat{O}^{\hat{A}}$ sitting in the space $O^{A}$ in
\eqref{eq1.21} as the subset where the $\Aut ^{\hat{A}}$ action is
free. The compactification here is as described in Part 2 of
\fullref{sec:1b}, thus $\underline {O}^{\hat{A}}/\Aut
^{\hat{A}}$ where
\begin{equation}\label{eq9.22}
\underline {O}^{A}\equiv [\mathbb{R}_{ - }\times \Maps(\hat{A}_{ + }; \mathbb{R})]/
\bigl[(\mathbb{Z} \times \mathbb{Z})\times \Maps(\hat{A}_{ + }; \mathbb{Z})\bigr].
\end{equation}
As indicated in Part 2 of \fullref{sec:1b}, the points in
$\underline {O}^{\hat{A}}-\hat{O}^{\hat{A}}$ describe bonafide
subvarieties, but these lie in $\hat{A}' \ne \hat{A}$ versions of
$\mathcal{M}_{\hat{A}' }$. This story is told in the two parts of
the subsection that follow. The first part discusses the points in
$O^{\hat{A}}-\hat{O}^{\hat{A}}$ while the second considers those in
$\underline{O}^{\hat{A}}-O^{\hat{A}}$.

\step{Part 1}
The following summarizes most of the story on $O^{\hat{A}}-\hat{O}^{\hat{A}}$:

\begin{proposition}\label{prop:9.7}
The map described in \fullref{sec:3c} extends to define an
orbifold diffeomorphism between $\mathbb{R}\times O^{\hat{A}}/\Aut^{\hat{A}}$ and ${\mathcal{M}^{*}}_{\hat{A}}$.
\end{proposition}
\noindent This proposition is a corollary of \fullref{thm:6.2}. To elaborate some on how
\fullref{prop:9.7} arises, suppose that $o$ denotes a bivalent vertex in
$T^{\hat{A}}$ and let $\hat{A}_{o}$ denote the set of elements in
$\hat{A}_{ + }$ whose integer pair gives $\theta _{o}$ via \eqref{eq1.8}. Let
Cyc$_{o}$ denote the set of cyclic orderings of $\hat{A}_{o}$. As noted in
\fullref{sec:4c}, the components of $O^{\hat{A}}$ are in 1--1 correspondence
with the points in $\times _{o}$ Cyc$_{o}$. Let $v$ denote a point in the
latter space and let ${O^{\hat{A}}}_{v}$ denote the corresponding
component. Borrowing again from \fullref{sec:4c}, let $\Aut_{o,v}$ denote the group
of permutations of $\hat{A}_{o}$ that preserve the cyclic ordering from $v$
while permuting only elements with identical 4--tuples, and set $\Aut_{v }
\equiv\times _{o}\Aut_{o,v}$. This is the subgroup of $\Aut^{\hat{A}}$ that
preserves ${O^{\hat{A}}}_{v}$. As such, the image of
${O^{\hat{A}}}_{v}$ in $O^{\hat{A}}/\Aut^{\hat{A}}$ is
diffeomorphic as an orbifold to ${O^{\hat{A}}}_{v}/\Aut_{v}$. The
component ${O^{\hat{A}}}_{v}/\Aut_{v}$ corresponds in the present
contex to one of \fullref{sec:5}'s top dimensional strata in ${\mathcal{M}^{* }}_{\hat{A}}$.

To say something more about the points in $O^{\hat{A}}-\hat{O}^{\hat{A}}$, suppose that $v$ is as above and that
$\Aut_{v}$ has a nontrivial, canonical $\mathbb{Z}/k\mathbb{Z}$ subgroup. In this case,
${O^{\hat{A}}}_{v }$ will have points where the $\Aut^{\hat{A}}$
action is not free; and, of course, these give points in $O^{\hat{A}}/\Aut^{\hat{A}}$ that are not in
$\hat{O}^{\hat{A}}/\Aut ^{\hat{A}}$. Now, let $G$ denote a nontrivial subgroup of some
canonical $\mathbb{Z}/k\mathbb{Z}$ subgroup of $\Aut_{v}$. Define ${\hat{A}_{+}}'$
to be the quotient of $\hat{A}_{ + }$ by $G$. Thus, ${\hat{A}_{+}}'$ is obtained
from $\hat{A}_{ + }$ by replacing each $G$ orbit in $\hat{A}_{ + }$ by a single
point. Note that the given cyclic ordering $v$ for $\hat{A}_{ + }$ induces one
for ${\hat{A}_{+}}'$; the latter is denoted in what follows by $v'$.

Meanwhile, both $\text{\c{c}}_{ - }$ and $\text{\c{c}}_{ + }$ must vanish when
$\Aut ^{\hat{A}}$ has a nontrivial, canonical $\mathbb{Z}/k\mathbb{Z}$
subgroup. Moreover, $\hat{A}-\hat{A}_{ + }$ must consist of two 4--tuples that
are of the form $(0,-,Q)$ or $(\pm 1,\cdot ,Q)$ where $k$ evenly divides both
components of $Q$. This understood, augment ${\hat{A}_{+}}'$ with the 4--tuples
that are obtained from the latter two by replacing Q with $\frac{1}{{|G|}}Q$. Here, $|G|$ denotes the size of the
group $G$. Use $\hat{A}'$ to denote this augmented set. Note that $\hat{A}'$ defines a
graph, $T^{\hat{A}' }$, of the sort described in \eqreft1{15} that
satisfies the condition stated in \eqreft1{16}. In this regard, the vertices of
$T^{\hat{A}'}$ enjoy an angle preserving 1--1 correspondence
with those of $T^{\hat{A}}$. This correspondence induces one between
the respective edges that has the following property: If $Q$ is an integer
pair from an edge in $T^{\hat{A}}$, then $\frac{1}{{|G|}}Q$ is the integer pair of its partner
in $T^{\hat{A}' }$.

Now, let ${O^{\hat{A}}}_{v,G}$ denote the subset of points in
${O^{\hat{A}}}_{v}$ whose $\Aut_{v}$ stabilizer is $G$. Then the
quotient of of ${O^{\hat{A}}}_{v,G}$ by $\Aut_{v}$ is diffeomorphic to
${\hat{O}^{\hat{A}'}}_{v' }/\Aut_{v' }$. Here, $\Aut_{v' }$ is the subgroup of $\Aut ^{\hat{A}' }$ that
preserves the cyclic order of ${\hat{A}_{+}}'$ that is
inherited from $v$'s ordering of $\hat{A}_{ + }$. This diffeomorphism is induced
by a map from ${O^{\hat{A}}}_{v}$ to ${\hat{O}^{\hat{A}'}}_{v' }$ that intertwines the action of $\Aut_{v}$ with that of
$\Aut_{v'}$. To describe the desired map, fix a point in each of the $G$ orbits in $\hat{A}_{ + }$. Now,
let $\lambda$ denote a point in $\mathbb{R}_{-}\times \Maps(\hat{A}_{ + }; \mathbb{R})$ whose image lies in
${O^{\hat{A}}}_{v,G}$. Define from $\lambda $ a point $\lambda '\in\mathbb{R}_{ - }\times
\Maps({\hat{A}_{+}}';
\mathbb{R})$ by taking
its $\mathbb{R}_{ - }$ factor to be $\frac{1}{{| G| }}$
times that of $\lambda $ while defining its value on any element in
${\hat{A}_{+}}'$ to be the value that $\lambda $ assigns to the chosen inverse
image in $\hat{A}_{ + }$. This assignment of $\lambda '$ to $\lambda $ defines
a smooth map from ${O^{\hat{A}}}_{v,G}$ to $O^{\hat{A}' }_{v' }$ that intertwines the $\Aut_{v}$ action with the
$\Aut_{v' }$ action. The image of this map is ${\hat{O}^{\hat{A}'}}_{v' }$ and the induced map,
$\pi\co {O^{\hat{A}}}_{v,G}/\Aut_{v}\to {\hat{O}^{\hat{A}'}}_{v' }/\Aut_{v' }$, is the desired diffeomorphism

With the preceding understood, consider now the following scenario: Let
$\lambda _{0}\in {O^{\hat{A}}}_{v,G}$ and let $\{\lambda_{j}\} \in \hat{O}^{\hat{A}}_{v}$ denote a sequence that
converges to $\lambda _{0}$. Fix $s_{0}\in\mathbb{R}$ and let $(C,\phi)$ denote the point
in $\mathcal{M}_{\hat{A}' }$ that is
defined by $(s_{0}$, $\pi (\lambda _{0}))$. Because $C$ has genus zero,
there is a unique holomorphic covering of $C$ with the following properties:
First, the covering space is a punctured sphere and the group of deck
transformations is $\mathbb{Z}/(| G| \mathbb{Z})$. Second, the
covering is trivial over each concave side end of $C$ where the $|s|\to\infty $ limit of $\theta $ is
in $(0,\pi)$. Third, the covering restricts over the other two ends of $C$ as a connected covering
space. Use $C_{0}$ to denote the covering space and let $\phi _{0}\co C_{0}\to\mathbb{R}\times
(S^1\times S^2)$ denote the composition of the map $\phi$ with the holomorphic covering map to $C$. Thus,
$(C_{0}$, $\phi _{0})$ defines a point in ${\mathcal{M}^{* }}_{\hat{A}}$. Meanwhile, use
$\{(C_{j0}$, $\phi _{j})\}$ to denote the sequence in $\mathcal{M}_{\hat{A}}$ that corresponds via
the map from \fullref{sec:3} to the sequence $\{(s_{0}, \lambda _{j})\}$. \fullref{thm:6.2}
implies that $\{(C_{j0},\phi _{j})\}$ converges in ${\mathcal{M}^{* }}_{\hat{A}}$ to $(C_{0},\phi _{0})$.

\step{Part 2}
This part of the subsection discusses the geometric significance of the
points in $\underline {O}^{\hat{A}}-O^{\hat{A}}$.
The story here starts with the definition of an $\Aut ^{\hat{A}}$
invariant stratification of $\underline{O}^{\hat{A}}$. In this
regard, a given stratum of the stratification is labeled by a partition,
$\wp$, of $\hat{A}_{ + }$ of the following sort: Each partition subset
consists of elements whose integer pair components define the same angle via \eqref{eq1.8}.
The stratum ${\mathcal{O}}_{\wp } \subset \underline{O}^{\hat{A}}$ comes via \eqref{eq9.22} from the
subset maps from $\hat{A}_{ + }$ to $\mathbb{R}$ that have the following two properties: First,
the map assigns the same value in $\mathbb{R}/(2\pi \mathbb{Z})$ to elements from the same partition
subset. Second, the map assigns distinct values in $\mathbb{R}/(2\pi \mathbb{Z})$ to elements from
distinct partition subsets when their associated
integer pair components define the same angle via \eqref{eq1.8}. The set $O^{\hat{A}}$ consists of
the components of the strata where $\wp $ consists
solely of single element sets. In general, the stratum labeled by $\wp $ has
dimension $n_{\wp }+2$, where $n_{\wp }$ denotes the number of sets that
comprise the partition $\wp$.

Associated to $\wp $ is an asymptotic data set, $\hat{A}'$; this is defined as
follows: The integers $(\text{\c{c}}_{ - }, \text{\c{c}}_{ + })$ for $\hat{A}$ are used for
$\hat{A}'$ as are any 4--tuples of the form $(\pm 1,\ldots)$ or $(0,-,\ldots)$. Meanwhile, the 4--tuples
from ${\hat{A}_{+}}'$ are in 1--1 correspondence with the partition subsets of $\wp$. In
particular, each is of the form $(0,+,\ldots)$. The integer
pair component for the 4--tuple labeled by a given partition subset is the
sum of those that label its elements. Thus, all such integer pairs define
the same angle via \eqref{eq1.8}.

The next order of business is to define a certain map from ${\mathcal{O}}_{\wp}$ to $O^{\hat{A}' }$.
For this purpose, let $\lambda $ denote a point in $\mathbb{R}_{-}\times \Maps(\hat{A}_{ + };
\mathbb{R})$. Define $\lambda '\in\mathbb{R}_{ - }\times \Maps(\hat{A}'_{+}; \mathbb{R})$ by using
$\lambda$'s factor in $\mathbb{R}_{ - }$ for the
$\mathbb{R}_{ - }$ factor of $\lambda '$, and by using $\lambda $'s value on
any element from any given partition subset for the value of $\lambda '$ on
the corresponding 4--tuple in $\hat{A}_{\wp + }$. This map induces a
diffeomorphism between ${\mathcal{O}}_{\wp }$ and $O^{\hat{A}' }$
and thus defines a map, $\pi $, from ${\mathcal{O}}_{\wp }/\Aut^{\hat{A}}$ to $O^{\hat{A}' }
/\Aut ^{\hat{A}' }$. Note here that $\Aut^{\hat{A}}$ is, in all cases, a subgroup of $\Aut^{A' }$.

Points in $O^{\hat{A}'}$ where $\Aut ^{\hat{A}'}$ acts freely parametrize $\mathcal{M}_{\hat{A}' }$. As was
just asserted in \fullref{prop:9.7}, the whole of $\mathbb{R}\times O^{\hat{A}' }/\Aut ^{\hat{A}' }$
parametrizes ${\mathcal{M}^{*}}_{\hat{A}'}$.

With all of this as background, consider now the following scenario: Fix
$\lambda _{0}\in{\mathcal{O}}_{\wp}$ and $s_{0 }\in\mathbb{R}$;
and let $(C_{0},\phi)$ define the point in $\mathcal{M}^{*}_{\hat{A}'}$ that is defined from $(s_{0},\pi
(\lambda _{0}))$ as given by the $\hat{A}'$ version of \fullref{prop:9.7}. Use
the asymptotic data from the ends of $C_{0}$ to define a 1--1 correspondence
between this set of ends and the 4--tuples in $\hat{A}'$.

In the mean time, let $\{\lambda _{j}\} \in O^{\hat{A}}$
denote a sequence that converges to $\lambda _{0}$; and let $\{(C_{j0}$,
$\phi _{j})\}$ denote the sequence in ${\mathcal{M}^{* }}_{\hat{A}}$ that is defined from
$\{(s_{0},\lambda _{j})\}$ as described in \fullref{prop:9.7}.

\begin{lemma}\label{lem:9.8}
The following limit exists and is zero:
\begin{equation*}
\lim_{j \to \infty } \Bigl[\sup_{z \in C_{0j}} \dist\bigl(\phi _{j}(z), \phi
(C_{0})\bigr) + \sup_{z \in C_0 } \dist\bigl(\phi _{j}(C_{0j}), \phi (z)\bigr)\Bigr].
\end{equation*}
Moreover, let $R$ be such that the $|s|> R$  portion of $C_{0}$
lies in the ends of $C_{0}$. Define $C_{R}\subset C_{0}$ to be the complement in $C_{0}$
of the $|s|> R$ portion of those ends whose associated  $\hat{A}'$  label corresponds to a
partition subset with size greater than 1. Then, for all $j$ sufficently large, there exists an
embedding, $\psi _{j}\co C_{R}\to C_{j0}$, with the following properties:
\begin{enumerate}
\item[$\bullet$] $\lim_{j \to \infty }\dist(\phi _{j} \circ \psi _{j}(z)$, $\phi(z)) = 0$ for all
$z$ in the domain of $\psi _{j}$.
\item[$\bullet$] $|\bar {\partial }\psi _{j}| \ll | \partial \psi_{j}|$ and
$\lim_{j \to \infty }\sup_{\domain(\psi _j )}|\bar{\partial }\psi _{j}|/|\partial \psi_{j}|  = 0$.
\item[$\bullet$] The complement in $C_{0j}$ of the image of $\psi _{j}$ consists of a disjoint
union of multiply punctured spheres, each embedded in a tubular neighborhood of an
$\mathbb{R}$--invariant cylinder so that the projection to the cylinder defines
an orientation preserving, ramified covering map with finitely many ramification points.
\end{enumerate}
\end{lemma}

\subsection{The codimension 1 strata in the compactification of
${\mathcal{M}^{*}}_{\hat{A}}$}\label{sec:9d}

This subsection describes the codimension 1 strata in the compactification
of ${\mathcal{M}^{* }}_{\hat{A}}$ in the case that $N_{ - }+\hat
{N}+\text{\c{c}}_{ - }+\text{\c{c}}_{ + } > 2$. To start, remark that there are two
sorts of codimension 1 strata in these cases. Strata of the first sort arise
from the codimension 1 strata in the compactifications of those
$O_{T}/\Aut (T)$ that appear in the various versions of \eqref{eq9.1}. The geometric
interpretation of these strata are described in Lemmas~\ref{lem:9.2}--\ref{lem:9.4}. In this
regard, keep in mind that only circular versions of $\underline {\Gamma}_{o}$ are involved in
the cases when the added strata are of codimension 1.

The second sort of stratum arises by adding certain dimension $k-1$ faces to
the simplex product $\Delta ^{k}$ in \eqref{eq9.1}. This section and the final
section explore the geometry of the subvarieties in the codimension 0 strata
that are near this second kind of added strata.

To start the discussion, let $\mathcal{S}$ denote the codimension 0 stratum
involved. The space $\Delta ^{k}$ for $\mathcal{S} $ is a $k$--dimensional product
of simplices in $\times _{k}((0,\pi )-\Lambda _{+,\text{{\o} }})$. The compactification
of $\mathcal{S}$ contains an added codimension 1 stratum when the following occurs: The homotopy type
of graph for $\mathcal{S}$ has an edge, $e$, with a trivalent vertex whose angle
values on $\mathcal{S} $ limit to the angle defined via \eqref{eq1.8} by $Q_{e}$, or to
that defined by --$Q_{e}$. Denote the latter angle by $\theta _{* }$;
it is the angle that defines the relevant codimension 1 face in the closure
of $\Delta ^{k}$. The relevant trivalent vertex is denoted in what follows
by $o$.

The upcoming \fullref{lem:9.9} describes some of the basic features of sequence in
$\mathcal{S} $ whose image in $\Delta ^{k}$ approach the face just described.
To set the stage for the lemma, let $T$ denote a graph whose homotopy type
arises from the elements in $\mathcal{S} $. The other two of $o$'s incident edges
are denoted by $e'$ and $\hat{e}$ with the one distinguished from the other in the
following manner: The second vertices on $e$, $e'$ and $\hat{e}$ have angles that
bracket $\theta _{o}$ in the interval $(0,\pi )$; and under the
circumstances, the second angle on $e$ is on the same side of $\theta _{o}$
as the second angle on one of the other two edges. The convention takes this
other edge to be $e'$.

Let $T_{e}$ denote the closure of the component of $T-o$ that contains
the interior of $e$, and let $T'$ denote the component of $T-\inti(e)$ that
contains $o$. Thus, $T_{e}$ and $T'$ are closed subgraphs in $T$ that intersect
only at $o$. Both can be viewed as graphs that are described by versions of
\fullref{sec:6a}. To elaborate, the labeling of $T_{e}$ is as follows: This graph
has a monovalent vertex, $o_{e}$, at angle $\theta _{* }$ such that
$T_{e}-o_{e}$ is isomorphic to the component of $T-o$ that
contains the interior of e. Meanwhile, $T'$ has a bivalent vertex, $o'$, at
angle $\theta _{* }$ with the property that $T'-o'$ is
isomorphic to the other two components of $T-o$. Let $\hat{A}'$ denote the
asymptotic data set that is defined by $T'$ and let $\hat{A}_{e}$ denote the
corresponding set that comes from $T_{e}$.

Because the graphs that arise from the elements in $\mathcal{S} $ are all
homotopic to a single graph, the choices that are offered in Parts 1 and 2
of \fullref{sec:6c} can be made once so as to hold for all fibers of the
projection map in \eqref{eq9.1} to $\Delta ^{k}$. This done, consider a sequence
in the space depicted in \eqref{eq9.1} whose factors are constant save for the
coordinate in $o$'s factor of $\Delta ^{k}$. The latter angle for the $j$'th
element of the sequence is denoted $\theta _{j}$, and the corresponding
sequence of angles converges with limit $\theta _{* }$. The inverse
of the map in \fullref{sec:6c} takes such a sequence and provides a sequence in
$\mathcal{S} $ with no limit in ${\mathcal{M}^{* }}_{\hat{A}}$. Let
$\{(C_{0j},\phi _{j})\}_{j = 1,2,\ldots}$ denote a sequence of
pairs that defines this sequence in $\mathcal{S} $.

With the preceding as background, consider:

\begin{lemma}\label{lem:9.9}
Under the circumstances as just described, there exists the following:
\begin{enumerate}
\item[$\bullet$] Pairs $({C_{0}}', \phi')$ and $(C_{e0},\phi _{e})$
that define respective points in the $(\hat{A}', T')$ and
$(\hat{A}_{e},T_{e})$ versions of ${\mathcal{M}^{*}}_{(\cdot,\cdot)}$.

\item[$\bullet$] Sequences $\{{r'}_{j}\}_{j = 1,2,\ldots }$ and $\{r_{ej}\}_{j = 1,2,\ldots }$
of real numbers, the former increasing without limit and the latter decreasing without limit.

\item[$\bullet$] Sequences $\{{s'}_{j}\}_{j = 1,2,\ldots }$ and $\{s_{ej}\}_{j = 1,2,\ldots }$
of real numbers, such that the former converges and the latter is decreasing and unbounded,
or else the former is decreasing and unbounded while the latter converges.
\item[$\bullet$] For all sufficiently large $j$, a proper embedding, ${\psi'}_{j}$, into
$C_{0j}$ from the complement in ${C_{0}}'$ of the $s > {r'}_{j}$  part of end in ${C_{0}}'$
that corresponds to the vertex $o'\in T'$.
\item[$\bullet$] For all sufficently large $j$, a proper embedding, $\psi _{ej}$, into $C_{0j}$
from the complement in $C_{e0}$ of the $s < -r_{ej}$ part of the end in
$C_{e0}$ that corresponds to the vertex $o_{e}\in T_{e}$.
\end{enumerate}
Here is the significance: First, the image via ${\psi'}_{j}$ of the $s = {r'}_{j}$ boundary of
its domain coincides with the image via $\psi _{ej}$ of the $s = r_{ej}$ boundary of the
latter's domain. Otherwise, the images of the two maps are disjoint. Second,
\begin{enumerate}
\item[$\bullet$] $\lim_{j \to \infty }\dist(\phi _{j}\circ {\psi'}_{j}(z)$,
${\phi'}_{j}(z)) = 0$ for any fixed $z\in {C_{0}}'$. Here, ${\phi'}_{j}$ is obtained from
$\phi '$ by composing with the translation by ${s'}_{j}$ along the $\mathbb{R}$ factor of
$\mathbb{R}\times (S^1\times S^2)$.
\item[$\bullet$] $|\bar {\partial }{\psi'}_{j}| \ll |\partial {\psi'}_{j}|$ and  $\lim_{j \to \infty }
\sup_{domain({\psi'}_{j} )} | \bar {\partial }{\psi'}_{j}| /| \partial {\psi'}_{j}|  = 0$.
\end{enumerate}
Finally,
\begin{enumerate}
\item[$\bullet$] $\lim_{j \to \infty }\dist(\phi _{j} \circ \psi _{ej}(z)$,
$\phi_{ej}(z)) = 0$ for any fixed $z\in C_{e0}$. Here, $\phi _{ej}$
is obtained from $\phi _{e}$ by composing with the translation by $s_{ej}$ along the $\mathbb{R}$
factor of $\mathbb{R}\times (S^1\times S^2)$.
\item[$\bullet$] $|\bar {\partial }\psi _{ej}| \ll | \partial \psi_{ej}| $ and
$\lim_{j \to \infty }\sup_{\domain(\psi _{ej} )}|\bar {\partial }\psi _{ej}| /| \partial \psi
_{ej}|  = 0$.
\end{enumerate}
\end{lemma}

More can be said about the sorts of pairs $({C_{0}}', \phi ')$ and
$(C_{0e},\phi _{e})$ that appear here. For example, if the vertex $o_{e}\in T_{e}$ is fixed by
$\Aut (T_{e})$, then the subset of such pairs in
\begin{equation}\label{eq9.23}
{\mathcal{M}^{*}}_{\hat{A}',T'} \times {\mathcal{M}^{*}}_{\hat{A}_e ,T_e }
\end{equation}
is a smooth, codimension 2 suborbifold that can be defined as a
fiber product using the maps that are depicted in \cite[(2.19) and
(2.20)]{T3}. The task of filling in the details is left to
the reader.

The story told by \fullref{lem:9.9} has analogues in higher codimensions. This can
occur when some version of $\underline {\Gamma }_{(\cdot )}$ has an
arc that begins and ends on the same vertex. By way of elaboration, suppose
that $o$ is a multivalent vertex in $T$ and that
$\gamma\subset \underline {\Gamma }_{o}$ is such an arc. As noted previously, the maps
in $\Delta _{o}$ are identical on $\gamma $; they send $\gamma $ to
$\alpha _{Q}(\theta _{o})$ where $Q =\sum _{e}c_{e}Q_{e}$
with the constants $\{c_{e}\}$ defined by writing $\gamma $'s homology
class as $\sum _{e}c_{e}[\ell _{oe}]$. Because the maps in $\Delta_{o}$ must send
$\gamma $ to a positive number, situations arise where the
angle $\theta _{o}$ varies on graphs homotopic to $T$ so that $\alpha_{Q}(\theta _{o})$ limits to zero.
\fullref{lem:9.9} concerns the case that
$\underline{\Gamma }_{o}$ is a figure 8 graph with a single 4--valent
vertex. However, the event just described can occur in a more general
context.

But for one comment, the general context is left for the reader to describe.
Here is the comment: The techniques used in Steps 2 and 3 from the proof of
\fullref{prop:8.2} in \fullref{sec:8b} can be used to prove that the event in
question occurs only in the case that ${\mathcal{M}^{*}}_{\hat{A},T}$ lies in the closure of
a codimension zero stratum of the sort that was
just discussed. In particular, the angle where $\alpha _{Q}$ is zero
defines a codimension 1 face of the space $\Delta ^{k}$ that appears in
the corresponding version of \eqref{eq9.1}; this a face that provides one of the
added strata in the closure of ${\mathcal{M}^{* }}_{\hat{A}}$.

\subsection{Faces of $\Delta^{k}$ where $\theta  = 0$ or $\pi $}\label{sec:9e}

There are versions of ${\mathcal{M}^{* }}_{\hat{A}}$ that have top
dimensional strata whose elements define a graph with an edge that has a
trivalent vertex on one end and a $\theta  = 0$ or $\theta =\pi $ vertex
on the other. There can be strata in this case where the trivalent vertex
angle limits to 0 or $\pi $ as the case may be. This subsection discusses
the behavior of the subvarieties where the graph has a very small trivalent
vertex angle. The discussion for the case where the angle is close to $\pi $
is not given since it is essentially identical to what follows.

Let $\mathcal{S} $ denote the top dimensional stratum involved. In this case, one
of the faces in $\Delta ^{k}$ is characterized by a $\theta  = 0$
assignment to a trivalent vertex. To start the story, note that as before,
the choices in Parts 1 and 2 of \fullref{sec:6c} can be made once so as to hold
for all graphs that arise from elements in $\mathcal{S} $. This understood, let $o$
denote the multivalent vertex involved and let $e$ denote its edge. The
assumption here is that there are no other vertex angles between $\theta_{o}$ and 0.

Fix a graph from an element in $\mathcal{S} $ to call $T$, and then
construct a sequence of homotopic graphs, $\{T_{j}\}_{j = 1,2,\ldots}$, where $T_{j}$ varies
from $T$ only in the angle assigned to $o$. The latter
angle is denoted now as $\theta _{j}$ and the assumption in what follows
is that $\lim_{j \to \infty }\theta _{j} = 0$. The choices from Parts 1
and 2 of \fullref{sec:6c} identify all $T = T_{j}$ versions of \eqref{eq6.15}. This
understood, fix a point, $\lambda $, in \eqref{eq6.15} and let $(C_{0j}, \phi_{j})$ denote a
pair that gives rise to the image of $\lambda $ in the $T = T_{j}$ version of
${\mathcal{M}^{*}}_{\hat{A},T}$ via the
inverse of the map from \fullref{sec:6c}.

There are now two cases to consider. In the first, two of $o$'s incident edges
connect $o$ to respective vertices with larger angle. In the second case, $e$
and another edge connect $o$ to angle 0 vertices. The two parts to this
subsection discuss these two cases.

\step{Part 1} Let $e'$ and $e''$ denote the two incident edges that connect $o$ to vertices
with larger angles. It must be assumed here that neither the $Q = Q_{e' }$ nor $Q = Q_{e''}$ version
of $\alpha _{Q}$ has a zero between 0 and $\theta _{o}$. Since $\alpha _{Q}(\theta _{o}) > 0$, this
occurs when $Q$ does not define an angle via \eqref{eq1.8} or, if it does,
then so does --$Q$, the former is greater than the latter and both are greater
than $\theta _{o}$.

The limiting behavior of $\{(C_{0j}, \phi _{j})\}$ is described
below in terms of a pair of graphs, $T'$ and $T''$, that are obtained from $T$ in
the following manner: Let $T_{e'}$ and $T_{e''}$
denote the closures in $T$ of the respective components of $T-o$ that
contain the interiors of $e'$ and $e''$. The graph $T'$ is characterized as
follows: It has an edge, $\hat{e}'$, with a $\theta  = 0$ monovalent vertex such
that the closure of $T'-\hat{e}'$ in $T'$ is isomorphic to the
closure of $T_{e' }-e'$ in $T$. In addition, $Q_{\hat{e}'} = Q_{e' }$. There is a similar characterization of
$T''$: It has an edge, $\hat{e}''$, with a monovalent vertex such that the closure
of $T''-\hat{e}''$ in $T''$ is isomorphic to that of $T_{e''}-e''$ in $T$. Moreover,
$Q_{\hat{e}''} = Q_{e'}$. Let $\hat{A}'$ and $\hat{A}''$ denote the respective asymptotic
data sets for the graphs $T'$ and $T''$. Note that these respective sets are
determined from $T$ and $T'$ via the rules in \eqreft63.

\begin{lemma}\label{lem:9.10}
With the scenario just described, there exist exist the following:
\begin{enumerate}
\item[$\bullet$] Pairs $({C_{0}}', \phi ')$ and $({C_{0}}'', \phi'')$ that
define respective elements in ${\mathcal{M}^{*}}_{\hat{A}',T'}$ and
${\mathcal{M}^{* }}_{\hat{A}'' ,T''}$.
\item[$\bullet$] Sequences $\{{s'}_{j}\}_{j = 1,2,\ldots }$ and $\{{s''}_{j}\}_{j = 1,2,\ldots }$
of real numbers.
\item[$\bullet$] For all sufficiently large $j$, a proper embedding,
${\psi'}_{j}$, into $C_{0j}$ of the complement in ${C_{0}}'$ of the $\theta < 2\theta _{j}$
part of the component of the ${C_{0}}'$ version of $C_{0}-\Gamma $
that corresponds to the edge $\hat{e}' \subset T'$.
\item[$\bullet$] For all sufficently large $j$, a proper embedding, ${\psi ''}_{j}$, into $C_{0j}$
of the complement in ${C_{0}}''$ of the $\theta  < 2\theta _{j}$ part of the component of the ${C_{0}}''$
version of $C_{0}-\Gamma $ of $C_{e0}$ that corresponds to the edge $\hat{e}''\subset T''$.
\end{enumerate}
Here is the significance: First, the images of
${\psi'}_{j}$ and ${\psi ''}_{j}$ are disjoint and their complement in
$C_{0j}$ is the $\theta  < 2\theta _{j}$ part of the closure of the components of the
$C_{0j}$ version of $C_{0}-\Gamma $ that correspond to the edges $e$, $e'$ and $e''$. Second,
\begin{enumerate}
\item[$\bullet$] $\lim_{j \to \infty }\dist(\phi _{j} \circ {\psi'}_{j}(z)$,
${\phi'}_{j}(z)) = 0$ for any fixed $z\in {C_{0}}'$. Here, ${\phi'}_{j}$ is obtained from $\phi '$
by composing with the translation by ${s'}_{j}$ along the $\mathbb{R}$ factor of
$\mathbb{R}\times (S^1\times S^2)$.
\item[$\bullet$] $|\bar {\partial }{\psi'}_{j}| \ll | \partial {\psi'}_{j}|$ and $\lim_{j \to \infty }
\sup_{\domain({\psi'}_{j} )} | \bar {\partial }{\psi'}_{j}| /| \partial {\psi'}_{j}|  = 0$.
\end{enumerate}
Finally,
\begin{enumerate}
\item[$\bullet$] $\lim_{j \to \infty }\dist(\phi _{j} \circ {\psi ''}_{j}(z)$, $\phi ''_{j}(z)) = 0$
for any fixed $z\in {C_{0}}'$. Here, $\phi ''_{j}$ is obtained from $\phi ''$ by composing
with the translation by ${s''}_{j}$ along the $\mathbb{R}$ factor of $\mathbb{R}\times (S^1\times S^2)$
\item[$\bullet$] $|\bar {\partial }{\psi ''}_{j}| \ll | \partial {\psi''}_{j}| $ and
$\lim_{j \to \infty }\sup_{\domain({\psi''}_{j})}| \bar {\partial }{\psi ''}_{j}|/|
\partial {\psi ''}_{j}|  = 0$.
\end{enumerate}
\end{lemma}

As in the case with \fullref{lem:9.9}, the set of pairs in ${\mathcal{M}^{*}}_{\hat{A}' ,T'}
\times{\mathcal{M}^{* }}_{\hat{A}'' , T''}$ that arise can be described rather explicitly.

\step{Part 2} In this case, the limit of $\{(C_{0j},\phi _{j})\}$ defines a
point in ${\mathcal{M}^{* }}_{\hat{A}}$ only in the case that both
integer pairs that connect $o$ to $\theta  = 0$ vertices have the form $(0,\cdot)$.
The resulting stratum in ${\mathcal{M}^{* }}_{\hat{A}}$ has codimension 2 rather than codimension 1.
Without this integer pair
condition, the limit is described with the help of a point in some $\hat{A}'\ne \hat{A}$ moduli
space ${\mathcal{M}^{*}}_{\hat{A}'}$.

In any case, to say more about the limit, introduce the graph $T'$ that is
characterized as follows: Let $\hat{e}$ denote the edge in $T$ that connects the
vertex $o$ to a larger angled vertex. The graph $T'$ has an edge, $\hat{e}'$, with a
monovalent vertex at angle 0 such that the closure in $T'$ of $T'-\hat{e}'$ is isomorphic to the
closure in $T$ of the component of $T-o$ that contains the interior of $T-\hat{e}$. In addition,
$Q_{\hat{e}'} = Q_{\hat{e} }$. Let $\hat{A}'$ denote the asymptotic data set that is determined by $T'$.

\begin{lemma}\label{lem:9.11}
With the scenario just described, there exist exist the following:
\begin{enumerate}
\item[$\bullet$] A pair $({C_{0}}', \phi ')$ that defines an element in
${\mathcal{M}^{*}}_{\hat{A}',T'}$.
\item[$\bullet$] A sequence $\{s_{j}\}_{j = 1,2,\ldots }$ of real numbers.
\item[$\bullet$] For all sufficiently large $j$, a proper embedding, ${\psi'}_{j}$, into $C_{0j}$
of the complement in ${C_{0}}'$ of the $\theta  < 2\theta _{j}$ part of the component of the
${C_{0}}'$ version of $C_{0}-\Gamma $ that corresponds to the edge $\hat{e}'\subset T'$.
\end{enumerate}
Here is the significance: First, the complement of the image of ${\psi'}_{j}$
 in $C_{0j}$ is the $\theta  < 2\theta _{j}$ part of the closure of the components of the
 $C_{0j}$ version of $C_{0}-\Gamma$ that correspond to the edges $e$, $e'$ and $e''$. Second,
\begin{enumerate}
\item[$\bullet$] $\lim_{j \to \infty }\dist(\phi _{j} \circ {\psi'}_{j}(z)$, ${\phi'}_{j}(z)) = 0$
for any fixed $z\in {C_{0}}'$. Here, ${\phi'}_{j}$ is obtained from $\phi '$ by composing with
the translation by $s_{j}$ along the $\mathbb{R}$ factor of $\mathbb{R}\times  (S^1\times S^2)$.
\item[$\bullet$] $|\bar {\partial }{\psi'}_{j}| \ll | \partial {\psi'}_{j}| $ and
$\lim_{j \to \infty }\sup_{\domain({\psi'}_{j})} | \bar {\partial }{\psi'}_{j}| /
| \partial {\psi'}_{j}|  = 0$.
\end{enumerate}
In the case that both $Q_{e }$ and $Q_{e' }$ have first component zero, the sequence $\{s_{j}\}$ can be
taken to be constant and the domain of ${\psi'}_{j}$ can be taken to be the whole of ${C_{0}}'$.
In this case, the range of ${\psi'}_{j}$ is the whole of $C_{0j}$. In particular, the sequence in
${\mathcal{M}^{* }}_{\hat{A} }$ defined by $\{(C_{0j}, \phi _{j})\}$ converges in
${\mathcal{M}^{* }}_{\hat{A}}$ to the point defined by $({C_{0}}', \phi')$.
\end{lemma}

Techniques of \cite[Section~3]{T3} can be used to prove that every element in
${\mathcal{M}^{*}}_{\hat{A}',T'}$ arises as a
limit in the sense just described of a sequence in the stratum $\mathcal{S}$.

\bibliographystyle{gtart}
\bibliography{link}
\end{document}